\documentclass[%
    a4paper,%
    12pt,%
    twoside,%
    openright,%
    halfparskip,%
    cleardoubleempty,%
    bigheadings,%
    titlepage,%
    headsepline%
]{scrreprt}

\usepackage[lmargin=3.5cm,rmargin=1.7cm,tmargin=3.85cm,bmargin=4.5cm]{geometry}

\usepackage{amsmath,amsfonts,amsthm,mathrsfs,amssymb}
\usepackage{mathtools}
\usepackage[all]{xy}
\usepackage{enumerate}
\usepackage{graphicx}
\usepackage{bm}
\usepackage{calc}
\usepackage{float}
\usepackage{mdwlist}
\usepackage{makeidx} 
\usepackage[latin1]{inputenc}
\usepackage[pdftitle={PhD. Thesis},pdfauthor={Nicol\'as Botbol},pdfcreator={LaTeX with hyperref},pdfsubject={poner titulo aca},pdfkeywords={poner keywords aca}]{hyperref}
\usepackage{color}

\definecolor{LinkColor}{rgb}{0,0,1} 
\hypersetup{colorlinks=true,linkcolor=LinkColor,citecolor=LinkColor,filecolor=LinkColor,menucolor=LinkColor,pagecolor=LinkColor,urlcolor=LinkColor}

\theoremstyle{plain}
\newtheorem{thm}{Theorem}[section]
\newtheorem{cor}[thm]{Corollary}
\newtheorem{lem}[thm]{Lemma}

\newtheorem{prop}[thm]{Proposition}

\theoremstyle{definition}
\newtheorem{defn}[thm]{Definition}
\newtheorem{exmp}[thm]{Example}
\newtheorem{rem}[thm]{Remark}

\newtheorem{nota}[thm]{Notation}

\newcommand\dto{\dashrightarrow}
\newcommand\hto{\hookrightarrow}
\newcommand\lto{\longrightarrow}
\newcommand\nto{\stackrel}

\def\NN{\mathbb{N}}
\def\ZZ{\mathbb{Z}}
\def\kk{\mathbb{K}}
\def\PP{\mathbb{P}}
\def\RR{\mathbb{R}}
\def\CC{\mathbb{C}}
\def\AA{\mathbb{A}}
\def\NNbm{{\bm{N}}}
\def\MMbm{{\bm{M}}}

\newcommand\Sc{\mathscr{S}}
\newcommand\Hc{\mathscr{H}}
\newcommand\Zc{\mathcal{Z}}

\newcommand\Cc{\mathcal{C}}  
\newcommand\Tc{\mathscr{T}}

\newcommand\Res{\mathrm{Res}}

\newcommand\conv{\mathrm{conv}}
\newcommand\Syz{\mathrm{Syz}}
\def\ker{\mathrm{ker}}

\def\endd{\mathrm{end}}
\newcommand\coker{\mathrm{coker}}
\newcommand\Proj{\mathrm{Proj}}
\newcommand\proj{\mathrm{Proj}}
\def\deg{\mathrm{deg}}

\newcommand\Sym{\mathrm{Sym}}
\newcommand\Rees{\mathrm{Rees}}
\newcommand\ann{\mathrm{ann}}
\newcommand\sat{\mathrm{sat}}
\newcommand\indeg{\mathrm{indeg}}
\newcommand\Hom{\mathrm{Hom}}

\newcommand\im{\mathrm{im}}
\newcommand\length{\mathrm{length}}
\def\dim{\mathrm{dim}}
\newcommand\res{\textnormal{Res}}
\newcommand\depth{\mathrm{depth}}

\newcommand\supp{\textnormal{supp}}
\newcommand\codim{\textnormal{codim}}
\newcommand\Spec{\textnormal{Spec}}
\newcommand\spec{\textnormal{Spec}}
\newcommand\Biproj{\textnormal{Biproj}}
\newcommand\Mproj{\textnormal{Multiproj}}
\newcommand\projA{\text{Proj}(A)}
\newcommand\tor{\textnormal{Tor}}

\newcommand\gr{\textnormal{gr}}

\newcommand\ext{\textnormal{Ext}}

\newcommand\ord{\textnormal{ord}}
\newcommand\SIA{\textnormal{Sym} _A (I)}
\def\rae{\mathfrak S}

\newcommand\X{\textbf{X}}
\newcommand\T{\textbf{T}}

\def\s{\textbf{s}}

\def\f{\textbf{f}}

\def\h{\textbf{h}}
\def\x{\textbf{x}}

\newcommand\Nc{\mathcal{N}}
\newcommand\Supp{\mathrm{Supp}}

\newcommand\gen[1]{\left\langle#1\right\rangle}
\newcommand\set[1]{\left\{#1\right\}}
\newcommand\gens[1]{\left\lgroup#1\right\rgroup}
\def\div{\textnormal{div}}
\newcommand\mat{\textnormal{Mat}}
\newcommand\relint{\textnormal{relint}}

\def\l.{\mathcal{L}_{\bullet}}
\def\ff.{\mathcal{F}_{\bullet}}
\def\a.{\mathcal{A}_\bullet}
\def\b.{\mathcal{B}_\bullet}

\def\k.{\mathcal{K}_{\bullet}}
\def\M.{\mathcal{M}_\bullet}
\def\Z.{\mathcal{Z}_\bullet}
\def\Zi{\mathcal{Z}_\bullet(f_i,g_i)}
\def\B.{\mathcal{B}_\bullet}

\def\OO{\mathcal O}
\def\aa{\mathcal A}
\def\LL{\mathcal L}

\def\iii{\mathscr I}
\def\jjj{\mathscr J}

\def\ZZZ{\mathcal Z}

\def\UU{\mathcal U}
\def\zz{\mathcal Z}
\def\EE{\mathcal E}
\def\LL{\mathcal L}
\def\BB{\mathcal B}
\def\HH{\mathcal H}
\def\SS{\mathcal S}
\def\KK{\mathcal K}

\def\EEE{\mathscr E}

\def\pp{\mathfrak{p}}
\def\mm{\mathfrak{m}}

\def\kkk{\kappa}

\def\pp{\mathfrak{p}}
\def\qq{\mathfrak{q}}
\def\mm{\mathfrak{m}}
\def\aaa{\mathfrak{a}}
\def\bb{\mathfrak{b}}

\def\k.{\mathcal{K}_{\bullet}}
\def\kpp{\mathcal{K}_{\bullet \bullet}}
\def\kt.{\k.(\textbf{T};A[\textbf{T}])}
\def\ki.{\k.(\textbf{f};A[\textbf{T}])}
\def\BB{\mathcal B}
\def\KK{\mathcal K}
\def\P1{\PP^1}

\def\ffi{f_i,g_i}
\def\fxi{f_i\cdot Y_i-g_i\cdot X_i}

\def\SSup{\mathfrak S}
\def\Region{{\mathfrak R _B}}
\def\C.{C_\bullet}
\def\D{\textnormal{\bf C}\textnormal{l}(\Xc)}
\def\DT{\textnormal{\bf C}\textnormal{l}(\Tc)}

\def\ee{\mathbf e}

\def\G{\textnormal{\bf G}}

\newcommand{\bfgamma}{\bm{\gamma}}
\newcommand{\bfrho}{\bm{\rho}}
\newcommand\cd{\textnormal{cd}}
\newcommand\grade{\textnormal{grade}}
\newcommand\SIR{\textnormal{Sym}_R (I)}
\newcommand\RIR{\textnormal{Rees}_R (I)}
\def\Xc{\mathscr X}
\def\Bc{\mathcal B}
\def\Mc{\mathcal M}
\newcommand\rad{\textnormal{rad}}

\newcommand\sti{``scheme-theoretic image" }

\newcommand\fs{f_1,\hdots,f_n}
\newcommand\ts{T_1,\hdots,T_n}
\newcommand\xs{x_1,\hdots,x_n}

\newcommand\Xs{X_1,\hdots,X_n}
\newcommand\Xss{X_1,\hdots,X_{n-1}}

\newcommand\as{a_1,\hdots,a_n}

\newcommand\Ank{\AA^n_\kk}
\newcommand\pnk{\PP^{n}_\kk}

\newcommand\pnnk{\PP^{n-1}_\kk}

\newcommand\pnna{\PP^{n-1}_A}

\newcommand\polnaT{A[T_1,\hdots,T_n]}
\newcommand\polnkT{\kk[T_1,\hdots,T_n]}

\newcommand\RIA{\Rees_A (I)}

\newcommand\ria{\Rees_A (I)_+}
\newcommand\SIIA{\Sym_{A/I} (I/I^2)}
\newcommand\sym{\Sym}
\newcommand\BlIA{Bl_\iii(\projA)}

\def\F.{F_\bullet}

\newcommand\reg{\textnormal{reg}}
\newcommand\greg{\textnormal{greg}}

\newcommand\pd{\textnormal{pd}}

\newcommand\ra{\rightarrow}
\newcommand\fin{\textnormal{end}}
\newcommand\sg{\mathrm{sg}}

\def\Zi{\Z.(f_i,g_i)}

\newcommand\spann[1]{\left(#1\right)}
\newcommand\colimit[1]{\underset{#1}{\underset{\lto}{\text{lim}}} \ }
\newcommand\ot{\leftarrow}
\newcommand\paren[1]{\left(#1\right)}

\makeindex

\floatstyle{ruled}
\newfloat{algorithm}{htbp!}{loa}[chapter]
\floatname{algorithm}{Algorithm}

\begin{document}
\thispagestyle{empty}

\begin{titlepage}
\begin {center}
{\scriptsize \begin{tabular}{cc}
 \includegraphics[scale=.25]{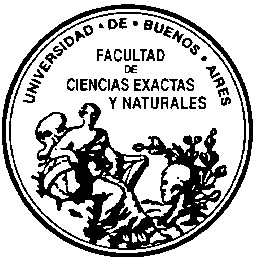}	& \includegraphics[scale=.25]{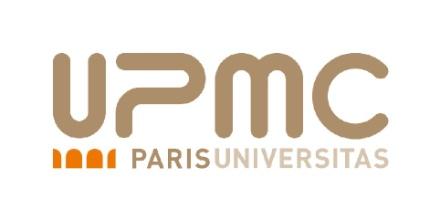} \\
UNIVERSIDAD DE BUENOS AIRES			& UNIVERSITE PIERRE ET MARIE CURIE \\
Facultad de Ciencias Exactas y Naturales	& Sciences Math\'ematiques de Paris Centre \\
Departamento de Matem\'atica			& Institut de Math\'ematiques de Jussieu \\
\end{tabular}}

\vspace{.3cm}
 
\textbf{PhD.\ THESIS}

from: 

\textbf{UNIVERSIDAD DE BUENOS AIRES \&\\ UNIVERSIT\'E PIERRE ET MARIE CURIE}

\vspace{.3cm}

Author:

\textbf{Nicol\'as Botbol}

\vspace{.2cm}

For the grade of:

\textbf{DOCTOR de la UNIVERSIDAD DE BUENOS AIRES,\\ en el \'area de Ciencias Matem\'aticas}

\textbf{DOCTEUR de l'{}UNIVERSIT\'E PIERRE ET MARIE CURIE,\\ sp\'ecialit\'e en Sciences Math\'ematiques.}

\vspace{.2cm}

Title:

\textbf{IMPLICITIZATION OF RATIONAL MAPS}

\end {center}
\vspace{.2cm}                 

Defended the 29th September 2010 at 14h, UPMC, Paris, France.         

Jury composed by:

{\small \begin{tabular}{ll}
Marc Chardin  		& Thesis Advisor \\
Alicia Dickenstein	& Thesis Advisor \\
David Cox			& Rapporteurs \\
Jean-Pierre Jouanolou	& Rapporteurs \\
Monique Lejeune-Jalabert	& Examinateurs \\
Joseph Oesterl\'e		& Examinateurs \\
Laurent Bus\'e		& Examinateurs
\end{tabular}}

\end{titlepage}

\newpage

\cleardoublepage

\thispagestyle{empty} \cleardoublepage

\thispagestyle{empty} \cleardoublepage

\thispagestyle{empty} \cleardoublepage

\begin{center}
 \begin{minipage}{.8\textwidth}
\begin{center}
\textbf{IMPLICITIZATION OF RATIONAL MAPS.}
\end{center}

\medskip
\small

\quad Motivated by the interest in computing explicit formulas for resultants and discriminants initiated by B\'ezout, Cayley and Sylvester in the eighteenth and nineteenth centuries, and emphasized in the latest years due to the increase of computing power, we focus on the implicitization of hypersurfaces in several contexts. Implicitization means, given a rational map $f:\AA^{n-1}\dto \AA^n$, to compute an implicit equation $H$ of the closed image $\overline{\im(f)}$. This is a classical problem and there are numerous approaches to its solution (cf.\ \cite{SC95} and \cite{Co01}). However, it turns out that the implicitization problem is computationally difficult.

\quad Our approach is based on the use of linear syzygies by means of approximation complexes, following \cite{BuJo03}, \cite{BC05}, and \cite{Ch06}, where they develop the theory for a rational map $f:\PP^{n-1}\dto\PP^n$. Approximation complexes were first introduced by Herzog, Simis and Vasconcelos in \cite{HSV} almost 30 years ago.

\quad The main obstruction for this approximation complex-based method comes from the bad behavior of the base locus of $f$. Thus, it is natural to try different compatifications of $\AA^{n-1}$, that are better suited to the map $f$, in order to avoid unwanted base points. With this purpose, in this thesis we study toric compactifications $\Tc$ for $\AA^{n-1}$. First, we view $\Tc$ embedded in a projective space. Furthermore, we compactify the codomain inside $(\P1)^n$, to deal with the case of different denominators in the rational functions defining $f$. We also approach the implicitization problem considering the toric variety $\Tc$ defined by its Cox ring, without any particular projective embedding. In all this cases, we blow-up the base locus of the map and we approximate the Rees algebra $\RIA$ of this blow-up by the symmetric algebra $\SIA$. We provide resolutions $\Z.$ for $\SIA$, such that $\det((\Z.)_\nu)$ gives a multiple of the implicit equation, for a graded strand $\nu\gg 0$. Precisely, we give specific bounds $\nu$ on all these settings which depend on the regularity of $\SIA$. We also give a geometrical interpretation of the possible other factors appearing on $\det((\Z.)_\nu)$.

\quad Starting from the homogeneous structure of the Cox ring of a toric variety, graded by the divisor class group of $\Tc$, we give a general definition of \textsl{Castelnuovo-Mumford regularity} for a polynomial ring $R$ over a commutative ring $k$, graded by a finitely generated abelian group $\G$, in terms of the support of some local cohomology modules. As in the standard case, for a $\G$-graded $R$-module $M$ and an homogeneous ideal $B$ of $R$, we relate the support of $H_B^i(M)$ with the support of $\tor_j^R(M,k)$.

 \end{minipage}
\end{center}
 \cleardoublepage

\begin{center}
 \begin{minipage}{.8\textwidth}
\begin{center}
\textbf{IMPLICITISATION D'{}APPLICATIONS RATIONNELLES.}
\end{center}

\medskip
\small

\quad Motiv\'e par la recherche de formules explicites pour les r\'esultants et les discriminants,  qui remonte au moins aux travaux de B\'ezout, Cayley et Sylvester au XVIII\`eme et XIX\`eme si\`ecles et a donn\'e lieu \`a de nouveaux d\'eveloppements dans les derni\`eres ann\'ees en raison de l'augmentation de la puissance de calcul, on se concentre sur l'implicitisation des hypersurfaces dans plusieurs contextes. Implicitisation signifie calculer une \'equation implicite $H$ de l'image ferm\'ee $\overline{\im(f)}$, \'etant donn\'ee une application rationnelle $f:\AA^{(n-1)}\dto \AA^n$. C'est un probl\`eme classique et il y a de nombreuses approches (cf.\ \cite{SC95} et \cite{Co01}). Toutefois, il s'av\`ere que le probl\`eme d'{}implicitisation est difficile du point de vue du calcul.

\quad Notre approche est bas\'ee sur l'utilisation des syzygies lin\'eaires au moyen des complexes d'{}approximation, en suivant \cite{BuJo03}, \cite{BC05}, et \cite{Ch06}, o\`u ils d\'eveloppent la th\'eorie pour une application rationnelle $f:\PP^{(n-1)}\dto \PP^n$. Les complexes d'{}approximation ont d'abord \'et\'e introduits par Herzog, Simis et Vasconcelos dans \cite{HSV} il y a presque 30 ans.

\quad L'obstruction principale de la m\'ethode des complexes d'{}approximation vient du mauvais comportement du lieu base de $f$. Ainsi, il est naturel d'essayer diff\'erentes compatifications de $\AA^{(n-1)}$, qui sont mieux adapt\'ees \`a $f$, afin d'\'eviter des points base non d\'esir\'es. A cet effet, dans cette th\`ese on \'etudie des compactifications toriques $\Tc$ de $\AA^{(n-1)}$. Tout d'abord, on consid\`ere $\Tc$ plong\'ee dans un espace projectif. En outre, on compactifie le codomaine dans $(\P1)^n$, pour faire face aux cas des d\'enominateurs diff\'erents dans les fonctions rationnelles qui d\'efinissent $f$. On a \'egalement abord\'e le probl\`eme implicitisation lorsque la vari\'et\'e torique $\Tc$ est d\'efinie par son anneau de Cox, sans un plongement projectif particulier. Dans tous ces cas, on \'eclate  le lieu base de $f$ et on approche l'alg\`ebre de Rees $\RIA$ par l'alg\`ebre sym\'etrique $\SIA$. On fournit des r\'esolutions $\Z.$ de $\SIA$, telle que $\det((\Z.)_\nu)$ donne un multiple de l'\'equation implicite, pour $\nu \gg 0$. Pr\'ecis\'ement, on donne des bornes sp\'ecifiques $\nu$ dans tous ces cas qui d\'ependent de la r\'egularit\'e de $\SIA$. On donne aussi une interpr\'etation g\'eom\'etrique des autres facteurs possibles qui apparaissent dans $\det((\Z.)_\nu)$.

\quad Motiv\'e par la structure homog\`ene de l'anneau Cox d'une vari\'et\'e torique, gradu\'ee par le groupe de classes de diviseurs de $\Tc$, on donne une d\'efinition g\'en\'erale de  \textsl{r\'egularit\'e de Castelnuovo-Mumford} pour un anneau de polyn\^omes $R$ sur un anneau commutatif $k$, gradu\'e par un groupe ab\'elien de rang fini $\G$, en termes du support de certains modules de cohomologie locale. Comme dans le cas standard, pour un $R$-module $M$ $\G$-gradu\'e et un id\'eal homog\`ene $B$ de $R$, on lie le support de $H_B^i(M)$ avec le support de $\tor_j^R(M,k)$.

 \end{minipage}
\end{center}
 \cleardoublepage

\begin{center}
 \begin{minipage}{.8\textwidth}
\begin{center}
\textbf{IMPLICITACI\'ON DE APLICACIONES RACIONALES.}
\end{center}

\medskip
\small

\quad Motivados por el inter\'es en el c\'alculo de f\'ormulas expl\'icitas  para resultantes y discriminantes que viene desde B\'ezout, Cayley y Sylvester en los siglos  XVIII y XIX, y enfatizado en los \'ultimos a\~nos por el aumento del poder de c\'omputo, nos concentramos en la implicitaci\'on de hipersuperficies en diversos contextos. Por implicitaci\'on entendemos que, dada una aplicaci\'on racional $f:\AA^{n-1}\dto \AA^n$, calculamos una ecuaci\'on impl\'icita $H$ de la clausura de la imagen $\overline{\im(f)}$. \'Este es un problema cl\'asico con numerosas aproximaciones para su soluci\'on (cf.\ \cite{SC95} y \cite{Co01}). A pesar de esto,  el problema de implicitaci\'on es computacionalmente dif\'icil.

\quad Nuestro enfoque se basa en el uso de sicigias lineales mediante complejos de aproximaci\'on, siguiendo \cite{BuJo03}, \cite{BC05}, y \cite{Ch06}, donde los autores desarrollan la teor\'ia para una aplicaci\'on racional $f:\PP^{n-1}\dto\PP^n$. Los complejos de aproximaci\'on fueron introducidos por primera vez por Herzog, Simis y Vasconcelos en \cite{HSV} hace casi 30 a\~nos.

\quad La principal obstrucci\'on para este m\'etodo basado en complejos de aproximaci\'on proviene del mal comportamiento del lugar base de $f$. Luego, es natural buscar diferentes compactificaciones de $\AA^{n-1}$, que est\'en mejor adaptadas a la aplicaci\'on $f$, con el fin de evitar puntos base no deseados. Con este objetivo, en esta tesis estudiamos compactificaciones t\'oricas $\Tc$ para $\AA^{n-1}$. Primero, vemos a $\Tc$ sumergida en un espacio proyectivo. M\'as a\'un, compactificamos el codominio en $(\P1)^n$, para tratar el caso en que las funciones racionales que definen a $f$ tengan diferentes denominadores. Tambi\'en abordamos el problema de implicitaci\'on considerando la variedad t\'orica $\Tc$ definida por su anillo de Cox, sin una inmersi\'on proyectiva particular. En todos estos casos, explotamos el lugar base de $f$ y aproximamos al \'algebra de Rees de este blow-up $\RIA$, mediante el \'algebra sim\'etrica $\SIA$. Proveemos resoluciones $\Z.$ de $\RIA$ tales que $\det((\Z.)_\nu)$ da un m\'ultiplo de la ecuaci\'on impl\'icita, para una capa graduada $\nu\gg 0$. M\'as precisamente, en todos estos casos damos cotas para $\nu$ que dependen de la regularidad de $\SIA$. Tambi\'en damos una interpretaci\'on  geom\'etrica para los posibles factores extras que aparecen en $\det((\Z.)_\nu)$.

\quad Comenzando desde la estructura homog\'enea del anillo de Cox de la variedad t\'orica, graduado por el grupo de clases de divisores de $\Tc$, damos una definici\'on general de la \textsl{regularidad de Castelnuovo-Mumford} para anillos de polinomios $R$ sobre un anillo conmutativo $k$, graduado por un grupo abeliano $\G$ finitamente generado, en t\'ermino de los soportes de algunos m\'odulos de cohomolog\'ia local. Tal como en el caso est\'andar, dado un $R$-m\'odulo $M$ $\G$-graduado y un ideal homog\'eneo $B$ de $R$, relacionamos el soporte de $H_B^i(M)$ con el soporte de $\tor_j^R(M,k)$.

 \end{minipage}
\end{center}
 \cleardoublepage

\begin{center}

  {\ }\vfill
 \begin{minipage}{.7\textwidth}
 \textit{ 
\begin{itemize} 
 \item [I would] like to thank
 \item [] a mi directora Alicia et \`a mon directeur Marc, for all they have done for me, teaching, helping, suggesting. It was simultaneously a pleasure and a honor;
 \item [] a mis amigos de la facultad, por acompa\~narme y faci\-litarme el trabajo durante estos a\~nos, sin su ayuda este trabajo no hubiera sido posible. Al resto de mis amigos, por su apoyo incondicional. A Ale;
 \item [] \`a mes amis conus en France, car vous avez faites qu'\^etre loin de chez moi soit agreable. Avec qui j'{}ai bien travaill\'e et amus\'e;
 \item [] to all mathematicians that helped me constructing what I have done in this science;
 \item [] A mi familia, a mis padres y a mi hermana. A Flor.
 \item [] \qquad Gracias.
\end{itemize}
}
\end{minipage}
  {\ }\vfill
  {\ }\vfill
\end{center}

\cleardoublepage

\chapter*{Introduction}
\label{ch:introduction}

The interest in computing explicit formulas for resultants and discriminants goes back to B\'ezout, Cayley, Sylvester and many others in the eighteenth and nineteenth centuries. It has been emphasized in the latest years due to the increase of computing power. Under suitable hypotheses, resultants give the answer to many problems in elimination theory, including the implicitization of rational maps. In turn, both resultants and discriminants can be seen as the implicit equation of a suitable map (cf.\ \cite{DFS07}). Lately, rational maps appeared in computer-engineering contexts, mostly applied to shape modeling using computer-aided design methods for curves and surfaces.
 
Rational algebraic curves and surfaces can be described in several different ways, the most common being parametric and implicit representations. Parametric representations describe the geometric object as the image of a rational map, whereas implicit representations describe it as the set of points verifying a certain algebraic condition, e.g.\ as the zeros of a polynomial equation. Both representations have a wide range of applications in Computer Aided Geometric Design (CAGD), and depending on the problem one needs to solve, one or the other might be better suited. It is thus interesting to be able to pass from parametric representations to implicit equations. This is a classical problem and there are numerous approaches to its solution (a good historical overview on this subject can be seen in \cite{SC95} and \cite{Co01}). However, it turns out that the implicitization problem is computationally difficult. 

A promising alternative suggested in \cite{BD07} is to compute a so-called \textit{matrix representation} instead, which is easier to compute but still shares some of the advantages of the implicit equation. Let $\KK$ be a field. For a given hypersurface $\Hc \subset \PP^n$, a matrix $M$ with entries in the polynomial ring $\kk[X_0,\ldots,X_n]$ is called a {\slshape representation matrix} of $\Hc$ if it is generically of full rank and if the rank of $M$ evaluated in a point of $\PP^n$ drops if and only if the point lies on $\Hc$ (see Chapter \ref{ch:toric-emb-pn}, also cf.\ \cite{BDD08}). Equivalently, a matrix $M$ represents $\Hc$ if and only if the greatest common divisor of all its minors of maximal size is a power of the homogeneous implicit equation $F \in \kk[X_0,\ldots,X_n]$ of $\Hc$. 

In the case of a planar rational curve $\Cc$ given by a parametrization of the form $\AA^1 \stackrel{f}{\dashrightarrow} \AA^2$, $s \mapsto \left(\frac{f_1(s)}{f_3(s)},\frac{f_2(s)}{f_3(s)}\right)$, where $f_i \in \kk[s]$ are coprime polynomials of degree $d$ and $\kk$ is a field, a (linear) syzygy (or moving line) is a linear relation on the polynomials $f_1,f_2,f_3$, i.e.\ a linear form $L = h_1X_1+h_2X_2+h_3X_3$ in the variables $X_1,X_2,X_3$ and with polynomial coefficients $h_i \in \kk[s]$ such that $\sum_{i=1,2,3} h_i f_i =0$. We denote by $\Syz (f)$ the set of all those linear syzygies forms and for any integer $\nu$ the graded part $\Syz(f)_\nu$ of syzygies of degree at most $\nu$. To be precise, one should homogenize the $f_i$ with respect to a new variable and consider $\Syz (f)$ as a graded module here. It is obvious that $\Syz(f)_\nu$ is a finite-dimensional $\kk$-vector space of dimension $k=k(\nu)$, obtained by solving a linear system. Let $L_1,\ldots,L_k$ be a basis of $\Syz(f)_\nu$. If $L_i=\sum_{|\alpha|=\nu}s^\alpha L_{i,\alpha}(X_1,X_2,X_3)$, we define the matrix $M_\nu=(L_{i,\alpha})_{1\leq i\leq k, |\alpha|=\nu}$, that is, the coefficients of the $L_i$ with respect to a $\kk$-basis of $\kk[s]_\nu$ form the  columns of the matrix. Note that the entries of this matrix are linear forms in the variables $X_1,X_2,X_3$ with coefficients in the field $\kk$. Let $F$ denote the homogeneous implicit equation of the curve and $\deg(f)$ the degree of the parametrization as a rational map. Intuitively, $\deg(f)$ measures how many times the curve is traced. It is known that for $\nu \geq d-1$, the matrix $M_\nu$ is a representation matrix; more precisely: if $\nu=d-1$, then $M_\nu$ is a square matrix, such that $\det(M_\nu)=F^{\deg(f)}$. Also, if $\nu \geq d$, then $M_\nu$ is a non-square matrix with more columns than rows, such that the greatest common divisor of its minors of maximal size equals $F^{\deg(f)}$. In other words, one can always represent the curve as a square matrix of linear syzygies. One could now actually calculate the implicit equation. We overview this subject more widely in Section \ref{moving cosas}.

For surfaces, matrix representations have been studied in \cite{BDD08} for the case of $2$-dimensional projective toric varieties, and we will analyze it in detail in Chapter \ref{ch:toric-emb-pn}. Previous work had been done in this direction, with two main approaches: One allows the use of quadratic syzygies (or higher-order syzygies) in addition to the linear syzygies, in order to be able to construct square matrices, the other one only uses linear syzygies as in the curve case and obtains non-square representation matrices.

The first approach using linear and quadratic syzygies (or moving planes and quadrics) has been treated in \cite{Co03} for
base-point-free homogeneous parametrizations and some genericity assumptions, when $\Tc=\PP^2$. The authors of \cite{BCD03} also treat the case of toric surfaces in the presence of base points. In \cite{AHW05}, square matrix representations of bihomogeneous parametrizations, i.e.\ $\Tc=\PP^1 \times \PP^1$, are constructed with linear and quadratic syzygies, whereas \cite{KD06} gives such a construction for parametrizations over toric varieties of dimension 2. The methods using quadratic syzygies usually require additional conditions on the parametrization and the choice of the quadratic syzygies is often not canonical. 

The second approach, developed in more detail in Section \ref{implicit con CA}, even though it does not produce square matrices, has certain advantages, in particular in the sparse setting that we present. In previous publications, this approach with linear syzygies, which relies on the use of the so-called approximation complexes has been developed in the case $\Tc=\PP^n$, see for example \cite{BuJo03}, \cite{BC05}, and \cite{Ch06}, and $\Tc=\PP^1 \times \PP^1$ in \cite{BD07} for  bihomogeneous parametrizations of degree $(d,d)$. However, for a given affine parametrization $f$, these two varieties $\Tc$ are not necessarily the best choice of a compactification, since they do not always reflect well the combinatorial structure of the polynomials defining the parametrization. We extend the method to a much larger class of varieties, namely toric varieties of dimension $2$ (cf. \cite{BDD08}, see also \ref{sec:equation}). We show that it is possible to choose a ``good'' toric compactification of $(\AA^*)^2$ depending on the input polynomials, which makes the method applicable in cases where it failed over $\PP^2$ or $\PP^1 \times \PP^1$. Also, it is significantly more efficient, leading to smaller representation matrices. 

Later, in \cite{Bot09}, see Chapter \ref{ch:toric-emb-pn}, we gave different compactifications for the domain and the codomain of an affine rational map $f$ that parametrizes a hypersurface in any dimension and we show that the closure of the image of this map (with possibly some other extra hypersurfaces) can be represented by a matrix of linear syzygies, relaxing the hypothesis on the base locus. More generally, we compactify $\AA^{n-1}$ into an $(n-1)$-dimensional projective arithmetically Cohen-Macaulay subscheme of some $\PP^N$.  We studied one particular interesting compactification of $\AA^{n-1}$ which is the toric variety associated to the Newton polytope of the polynomials defining $f$. 

In \cite{Bot08} and \cite{Bot09} we considered a different compactifications for the codomain of $f$, $(\PP^1)^n$ as is detailed in Chapter \ref{ch:toric-emb-p1xxp1}. We study the implicitization problem in this setting. This new perspective allow to deal with parametric rational maps with different denominators. Precisely, given $f=(\frac{f_1}{g_1},\hdots,\frac{f_n}{g_n}):\AA^{n-1}\dto \AA^{n}$, we can naturally consider a map $\phi=(({f_1}:{g_1})\times\cdots({f_n}:{g_n})):\PP^{n-1}\dto (\P1)^{n}$ (cf.\ \cite{Bot08}). As we have remarked before, $\PP^{n-1}$ need not be the best compactification of the domain of $f$, thus, in \cite{Bot09} we extended this method the setting $\phi:\Tc\dto (\P1)^{n}$ where $\Tc$ is any arithmetically Cohen-Macaulay closed subscheme of some $\PP^N$. In this last context, we gave sufficient conditions, in terms of the nature of the base locus of the map, for getting a matrix representation of its closed image, without involving extra hypersurfaces (cf.\ Chapter \ref{ch:toric-emb-p1xxp1}).

In order to avoid a particular embedding of $\Tc$ in $\PP^N$, we focused on the study of implicitization problem for toric varieties given by its Cox ring  (see Section \ref{CoxRing} or \cite{Cox95}).
This leaded to adapting the technique based on approximation complexes for more general graded rings and modules. In Chapter \ref{ch:CastelMum} we give a definition of \textsl{Castelnuovo-Mumford regularity} for a commutative ring $R$ graded by a finitely generated abelian group $G$, in terms of the support of some local cohomology modules. A very interesting example is that of Cox rings of toric varieties, where the grading is given by the Chow group of the variety acting on a polynomial ring. Thus, this allows to study the implicitization problem for general arithmetically Cohen Macaulay toric varieties without the need of an embedding, as we do in Chapter \ref{ch:toric-pn}.

\newpage

\section*{Organization}

\begin{itemize}
 \item[]\textbf{Ch.\ $1$:} Preliminaries on elimination theory and approximation complexes.
 \item[]\textbf{Ch.\ $2$:} Preliminaries on toric varieties.
 \item[]\textbf{Ch.\ $3$:} Implicitization for $\varphi:\Tc\dto \PP^n$, by means of an embedding $\Tc\subset \PP^N$.
 \item[]\textbf{Ch.\ $4$:} Implicitization for $\phi:\Tc\dto (\P1)^n$, by means of an embedding $\Tc\subset \PP^N$.
 \item[]\textbf{Ch.\ $5$:} Algorithmic approach for Chapters \ref{ch:toric-emb-pn} and \ref{ch:toric-emb-p1xxp1}, and examples.
 \item[]\textbf{Ch.\ $6$:} Castelnuovo-Mumford regularity for $\G$-graded rings, for $\G$ abelian group.
 \item[]\textbf{Ch.\ $7$:} Implicitization for $\phi:\Tc\dto \PP^n$, where $\Tc$ is defined by the Cox ring.
 \item[]\textbf{Ch.\ $8$:} Algorithm for $\varphi:\Tc\dto \PP^3$ following Chapter \ref{ch:toric-emb-pn}.
 \item[]\textbf{Ch.\ $9$:} Algorithm for $\varphi:\Tc\dto \PP^3$ following Chapter \ref{ch:toric-pn}.
\end{itemize}

\begin{center}
 \includegraphics{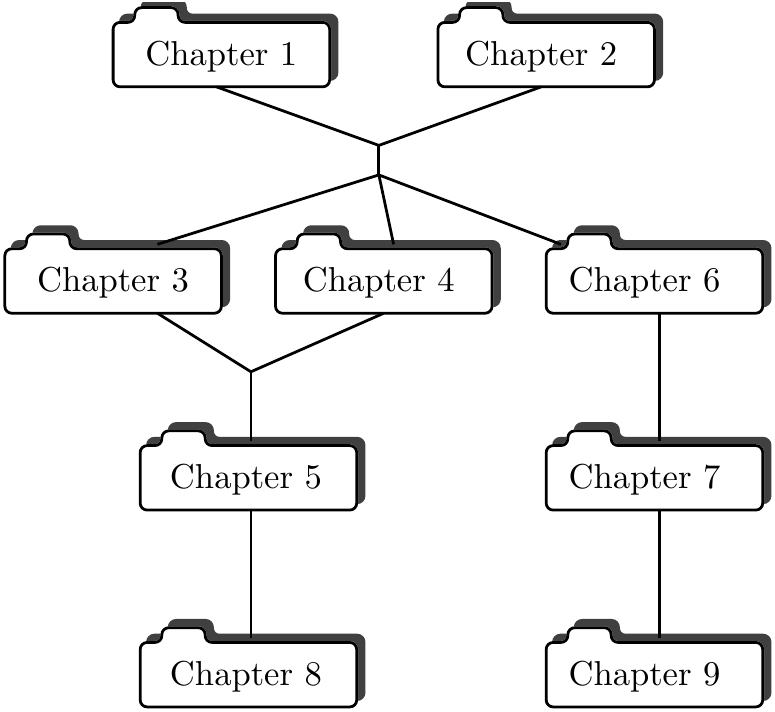}
\end{center}

In Chapter \ref{ch:elimination} we give a fast overview of the original technique of computing implicit equations for projective rational maps by means of approximation complexes. Indeed, we introduce in Section \ref{CA} the notion of approximation complexes and of blow-up algebras in Section \ref{blow-ap-algebras}, and we give basic results that we will use later in this thesis. As it was mentioned, this approach with linear syzygies was first formulated for this purpose in \cite{BuJo03} an later improved in \cite{BC05}, \cite{Ch06} and \cite{BCJ06}. We give a more detailed outline of this method in Section \ref{implicit con CA}.

\medskip

Chapter \ref{ch:toric-varieties} is mainly devoted to give an introduction to toric varieties. We recall some results that we will need later, in order to generalize the implicitization methods for toric compactifications. We develop this idea in Chapters \ref{ch:toric-emb-pn}, \ref{ch:toric-emb-p1xxp1} and \ref{ch:toric-pn}.

\medskip

In Chapters \ref{ch:toric-emb-pn} and \ref{ch:toric-emb-p1xxp1} we adapt the method of approximation complexes to computing an implicit equation of a parametrized hypersurface, focusing on different compactifications of the domain $\Tc$ and of the codomain ($\PP^n$ and $(\P1)^n$). We will always assume that $\Tc$ is a $(n-1)$-dimensional closed subscheme of $\PP^N$ with graded and Cohen-Macaulay $n$-dimensional coordinate ring $A$.

\medskip

In Chapter \ref{ch:toric-emb-pn}, we focus on the implicitization problem for a rational map $\varphi:\Tc \dto \PP^n$ defined by $n+1$ polynomials of degree $d$. We extend the method to maps defined over an $(n-1)$-dimensional Cohen-Macaulay closed scheme $\Tc$, embedded in $\PP^N$, emphasizing the case where $\Tc$ is a toric variety. We show that we can relax the hypotheses on the base locus by admitting it to be a zero-dimensional almost locally complete intersection scheme. Implicitization in codimension one is well adapted in this case, as is shown in Section \ref{sec2setting} and \ref{sec3Pn}, following the spirit of many papers in this subject: \cite{BuJo03}, \cite{BCJ06}, \cite{BD07}, \cite{BDD08} and \cite{Bot08}.

\medskip

In order to consider more general parametrizations given by rational maps of the form $f=(\frac{f_1}{g_1},\hdots,\frac{f_n}{g_n})$ with different denominators $g_1,\hdots,g_n$, we develop in Chapter \ref{ch:toric-emb-p1xxp1} the study of the $(\PP^1)^n$ compactification of the codomain.  With this approach, we study following \cite{Bot08} and \cite{Bot09}, the method of implicitization of projective hypersurfaces embedded in $(\P1)^n$. As in Chapters \ref{ch:elimination} and \ref{ch:toric-emb-pn}, we compute the implicit equation as the determinant of a complex which coincides with the gcd of the maximal minors of the last matrix of the complex, and we make deep analysis of the geometry of the base locus.

\medskip

In Chapter \ref{ch:algorithmic-approach} we exemplify the results of Chapters \ref{ch:toric-emb-pn} and \ref{ch:toric-emb-p1xxp1}, and we study in a more combinatorial fashion the size of the matrices obtained. We analyze, in both settings, how taking an homothety of the Newton polytope $\Nc(f)$ can modify the size of the matrices $M_\nu$. We present several examples comparing our results with the previous ones. First, we show in a very sparse setting the advantage of not considering the homogeneous compactification of the domain when denominators are very different. We extend in the second example this idea to the case of a generic affine rational map in dimension $2$ with fixed Newton polytope.
In the last example we give, for a parametrized toric hypersurface of $(\P1)^n$, a detailed analysis of the relation between the nature of the base locus of a map and the extra factors appearing in the computed equation. We finish this section by giving an example of how the developed technique can be applied to the computation of sparse discriminants.

\medskip

In order to avoid a particular embedding of $\Tc$ in $\PP^N$, we focus in Chapter \ref{ch:toric-pn} on the study of the implicitization
problem for toric varieties given by its Cox ring (see Section \ref{CoxRing} or the original source in \cite{Cox95}). Motivated by this, in Chapter \ref{ch:CastelMum} we give a definition of \textsl{Castelnuovo-Mumford regularity} for a commutative ring $R$ graded by a finitely generated abelian group $G$, in terms of the support of some local cohomology modules.

\medskip

In Chapter \ref{ch:CastelMum} we give a definition of \textsl{Castelnuovo-Mumford regularity} for a commutative ring $R$ graded by a finitely generated abelian group $G$, in terms of the support of some local cohomology modules. This generalizes \cite{HW04} and \cite{MlS04}. With this purpose, we distinguish an ideal $B$ of $R$, and we determine subsets of $G$ where the $G$-graded modules $H^i_B(R)$ are supported, this is, elements $\gamma\in G$ where $H^i_B(R)_\gamma\neq 0$.
Also, we study the regularity of some particular rings, in particular, polynomial rings $\ZZ^n$-graded, and we show that in these cases this notion of regularity coincides with the usual one. A very interesting example is that of Cox rings of toric varieties, where the grading is given by the Chow group of the variety acting on a polynomial ring (cf.\ \cite{Cox95}).

Lately, we establish, for a $G$-graded $R$-module $M$, a relation between the supports of the modules $H^i_B(M)$ and the support of the Betti numbers of $M$, generalizing the well-known duality for the $\ZZ$-graded case.

\medskip

In Chapter \ref{ch:toric-pn} we present a method for computing the implicit equation of a hypersurface given as the image of a rational map $\phi: \Tc \dashrightarrow \PP^n$,  where $\Tc$ is an arithmetically Cohen-Macaulay toric variety defined by its Cox ring (see Section \ref{CoxRing}). In Chapters \ref{ch:toric-emb-pn} and \ref{ch:toric-emb-p1xxp1}, the approach consisted in embedding the space $\Tc$ in a projective space. The need of this embedding comes from the necessity of a $\ZZ$-grading in the coordinate ring of $\Tc$, in order to study its regularity. The aim of this chapter is to give an alternative to this approach: we study the implicitization problem directly, without an embedding in a projective space, by means of the results of Chapter \ref{ch:CastelMum}. Indeed, we deal with the multihomogeneous structure of the coordinate ring $S$ of $\Tc$, and we adapt the method developed in Chapters \ref{ch:elimination}, \ref{ch:toric-emb-pn} and \ref{ch:toric-emb-p1xxp1} to this setting. The main motivations for our change of perspective are that it is more natural to deal with the original grading on $\Tc$, and that the embedding leads to an artificial homogenization process that makes the effective computation slower, as the number of variables to eliminate increases.

\medskip

Chapter \ref{ch:ch-algor-ToricP2} and Chapter \ref{ch:ch-algor-multigr} are devoted to the algorithmic approach of  both cases studied in Chapters \ref{ch:toric-emb-pn} and \ref{ch:toric-pn}. We show how to compute the sizes of the representation matrices obtained in both cases by means of the Hilbert functions of the coordinate ring $A$ and of its Koszul cycles.

\tableofcontents

\chapter[Preliminaries on elimination theory]{Preliminaries on elimination theory}
\label{ch:elimination}


\section{Introduction}

In this chapter we give a short summary of the articles written by Laurent Bus\'e, Marc Chardin  and Jean-Pierre Jouanolou on implicitization of projective hypersurfaces by means of approximation complexes \cite{BuJo03,BC05,Ch06,BCJ06}. There are many branches on mathematics and computer sciences where implicit equations of hypersurfaces are used and, hence, implicitization problems are involved. One of them is the interest in computer aided design (cf.\ \cite{Hof89,GK03}).

In the beginning of the $80$'s, Hurgen Herzog, Aron Simis and Wolmer V. Vasconcelos developed the so called \textsl{Approximation Complexes} (cf.\ \cite{HSV1,HSV2,Vas1}) for studying the syzygies of the conormal module (cf.\ \cite{SV81}). 

In elimination theory \textsl{approximation complexes} were used for the first time by Laurent Bus\'e and Jean-Pierre Jouanolou in $2003$ in order to propose a new alternative to the previous methods (see \cite{BuJo03}). This new tool generalized the work of Sederberg and Cheng, on \textsl{``moving lines"} and \textsl{``moving surfaces"} introduced a few years before in \cite{SC95,CSC98,ZSCC03}, giving also a theoretical framework.

The spirit behind the method based on approximation complexes consists in doing elimination theory by taking determinant of a graded strand of a complex. This idea is similar to the one used for the computation of a Macaulay resultant of $n$ homogeneous polynomials $F_1,\hdots,F_n$ in $n$ variables, by means of taking determinant of a graded branch of a Koszul complex. 

This resultant spans the annihilator of the quotient ring of $A[\Xs]$ by $I=(F_1,\hdots,F_n)$ in big enough degree (bigger than its regularity). This annihilator can also be related to the MacRae invariant of the coordinate ring $A[\Xs]/I$ in the same degree $\nu$. This theoretical method can become effective through the computation of the determinant of the degree-$\nu$-strand of the Koszul complex of $\{F_1,\hdots,F_n\}$ (see \cite{Nor76,MRae65,GKZ94,KMun}).

\medskip

In this case, we wish to give a closed formula for the implicit equation of the image of a rational map $\phi:\PP^{n-2}\dashrightarrow \PP^{n-1}$, over a field $\kk$. We will assume at first that this image defines a hypersurface in $\PP^{n-1}$, and hence, $\phi$ is generically finite. 

It is well known that a map between schemes gives rise to a map of rings that we will denote by $h:\kk[\ts]\to A:=\kk[\Xss]$. We will focus on computing the kernel of this map $h$ which is a principal prime ideal of the polynomial ring $\kk[\ts]$, and hence it describes the closed image of $\phi$.

%
%
\section{The image of a rational map as a scheme}

We will describe henceforward in this chapter how to compute the implicit equation of the closed image of a rational map $\phi:\PP^{n-2} \dashrightarrow \PP ^{n-1}$ following the ideas of L. Bus\'e, M. Chardin and J.-P. Jouanolou. Let $\kk$ be a commutative ring and $A$ a $\ZZ$-graded $\kk$-algebra. We will assume that $\phi = (\fs)$, where the polynomials $f_i\in A$ are homogeneous of the same degree $d$ for all $i=1,\hdots n$. Let $h$ be a morphism of graded $\kk$-algebras defined by
\begin{equation}\label{def h}
 h:\kk[T_1,\hdots,T_n] \to A, \qquad T_i \mapsto f_i .
\end{equation}
The map $h$ induces a morphism of $\kk$-affine schemes
\begin{equation}\label{def mu}
 \mu :\bigcup D(f_i) \to \bigcup D(T_i)=\AA^n_\kk\setminus \{0\},
\end{equation}
where $D(f_i):=\{\pp \in \Spec(A) : f_i\notin \pp\}$ is an open set of $\Spec(A)$.

Also, given $\{f_i\}_{i=1,\hdots n}$ homogeneous of degree $d$, $h$ is a graded morphism of graded algebras (where the grading is given by $\deg(T_i)=1$ for all $i=1\hdots,n$). Hence, $h$ induces a morphism of $\kk$-projective schemes 
\begin{equation}\label{def phi}
 \phi :\bigcup D_+(f_i) \to \bigcup D_+(T_i)=\PP^{n-1}_\kk,
\end{equation}
where $D_+(f_i):=\{\pp \in \projA : f_i\notin \pp\}$ is an open set of $\projA$. 

Denote by $D(\textbf{f}):=\bigcup D(f_i)$ and $D_+(\textbf{f}):=\bigcup D_+(f_i)$, the sets of definition of $\mu$ and $\phi$ respectively, also $D(\textbf{f})=\Spec(A)\setminus V(f_1,\hdots,f_n)$ and $D_+(\textbf{f})=\projA \setminus V(f_1,\hdots,f_n)$.
\medskip

Before getting into the results, we give some notations.

\bigskip
\begin{defn}\label{defsat}
We will denote by $R$ the polynomial ring $\polnkT$, and let $I$ and $J$ be ideals of $R$ and $M$ an $R$-module. Define
\begin{enumerate}
\item $\ann(J)=\{f\in R : f\cdot J=0\}$, the annihilator of $J$;
\item $(I:_R J)=\{f\in R : f\cdot J\subset I \}$, the colon ideal of $I$ by $J$;
\item $(I:_R J^\infty) = \bigcup_{n\in \NN} (I:_R J^n)$, the saturation of $I$ by $J$, also written $TF_{J}(I)$;
\item $H^0_J(M)=\{m\in M : m\cdot J^n =0, \forall n \gg 0\}$, the $0$-th local cohomology group of $M$ with support on $J$.
\end{enumerate}
\end{defn}

\begin{thm}[{\cite[Thm 2.1]{BuJo03}}]\label{sti}
Let $\iii$ and $\jjj$ be the affine and projective sheafification of $\ker(h)$. We have that 
\[
 V(\iii)|_{\Ank\setminus \{0\}} = V(\ker(h)^\sim)|_{\Ank\setminus \{0\}} =V((\ker(h):(\ts)^\infty)^\sim)|_{\Ank\setminus \{0\}}
\]
and similarly with $V(\jjj)$.
\end{thm}

\begin{lem}[{\cite[Rem 2.2]{BuJo03}}]\label{kerhsat} We have 
\[
 TF_{(\ts)}(\ker(h))=\{p\in A[\ts] : p(\fs)\in H^0_{(\fs)}(A)\}.
\]
In particular, when $H^0_{(\fs)}(A)=0$, $\ker(h)=TF_{(\ts)}(\ker(h))$; this means that $\ker(h)$ is saturated with respect to $(\ts)$ in $\kk[\ts]$.  
\end{lem}

Recall that if $I$ and $J=(g_1,\hdots,g_s)$ are ideals of $R$, then $(I:_R J^\infty)$, is defined as $\bigcup_{m\in \NN} (I:_R J^m)=\{f\in R : \exists m\in \NN, f.(g_1,\hdots,g_s)^n\subset I \}$. We have that

\begin{rem}\label{Isat}
\[
 (I:_R J^\infty)=\{f\in R : \exists m\in \NN, f.g_i^m\in I \ \forall i\}.
\]
\end{rem}

This is due to the fact that $(g_1^m,\hdots,g_s^m)\subset J^m$ and if $f\in J$, $f=\sum_{i=1}^s \alpha_j g_j$. Thus, $f^{m(s-1)+1}=(\sum_{i=1}^s \alpha_i g_i)^{m(s-1)+1}=\sum_{\sum i_j=m(s-1)+1} \alpha_{(i_1,\hdots,i_s)}g_1^{i_1}\cdots g_s^{i_s}$ that clearly belongs to $(g_1^m,\hdots,g_s^m)$. Hence, $J^{m(s-1)+1}\subset(g_1^m,\hdots,g_s^m)$.

Recall that $\phi: \projA \to  \pnnk$ is the map induced by 
\[
 h: \kk[\ts] \to A.
\]
Let $U:=D_+(\textbf{f})$ be the open subscheme of definition of $\phi$, and $Z:=V(\fs)$ be the closed subscheme of $\projA$ where the sections $\fs$ vanish. We will blowup $\projA$ along $Z$.

\medskip
We will denote by $\pi_1$ and  $\pi_2$ the two natural projections, 
\[
 \xymatrix@1{\BlIA\ \ar@{^{(}->}[rr]\ar[d]^{\pi_1}\ar[rd]^{\pi_2}& & \projA \times_\kk \pnnk = \pnna\\ \projA\ar@{-->}[r]^{\phi} & \pnnk & &}
\]
The restriction of $\pi_2$ to $\Omega:=\pi_1^{-1}(U)$ coincides with $\phi \circ \pi_1$.

\begin{defn}\label{algRees} 
 Let $\RIA:=\sum_{i\geq 0} I^it^i$ be the Rees algebra of $I=(\fs)$. Let $\polnaT \to A[t]$ be the map of $A$-algebras defined by $T_i \mapsto f_it$, in such a way that $\deg(T_i)=(1,0)$ and $\deg(f_i)=(0,d)$, hence $t$ is of total degree $1-d$.
\end{defn}

Thus, there is a short exact sequence 
$0\to J \to \polnaT \to \RIA \to 0$, where $J=\ker(\polnaT \to A[t])$, namely, $\RIA \cong \frac{\polnaT}{J}$.

\begin{prop}\label{pi2yphi}
The following diagram is commutative
\[
 \xymatrix@1{\Omega \ar[d]_{\pi_1|_\Omega}\ar[rd]^{\pi_2} \\ D_+(\textbf{f})\ar[r]^{\phi} & \pnnk}
\]
where $D_+(\textbf{f})\subset \projA $, $\Omega:=\pi^{-1}(D_+(\textbf{f})) \subset \BlIA$ and $\pi_1|_\Omega$ corresponds to the restriction of $\pi:\BlIA \to \projA$ to the open set $\Omega$.
\end{prop}

One important difficulty is the deep understanding of the difference between $I$ and $J$. We will give a short example to illustrate this relation.

\begin{exmp}\label{egR=S}
Let $A$ be a commutative noetherian ring, $f,g\in A$ and $\Rees_A(f,g)=A[ft,gt]\subset A[t]$. 

Invert $f$ and define $B=A[f^{-1}][X,Y]$. Let $X'=f^{-1}X\in B$ and hence we get $B=A[f^{-1}][X',Y]$. The element $gX'-Y\in B$ spans $\ker(B=A[f^{-1}][X',Y] \to A[f^{-1}][t])$, defined as $X'\mapsto t$ and  $Y\mapsto gt$. Since $B$, $gX-fY$ and $gX'-Y$ coincide, $f$ is not a zero divisor modulo $gX-fY$ in $A[X,Y]$. We see that $(f,gX-fY)$ is a regular sequence in $A[X,Y]$. Hence, the complex
\[
 \k.(gX-fY,f):\xymatrix@1{0\ar[r] & A\ar[rr]^{(-f,gX-fY)} & & A^{2}\ar[rr]^{(gX-fY,f)^{t}} & & A\ar[r]& 0}
\]
is acyclic.
Thus the first homology group of $\k.(gX-fY,f)$, $(f:gX-fY)/(f)$, vanishes. Hence, if $(f,g)$ is a regular sequence, then the kernel of the map $A[X,Y] \to \Rees_A(f,g)$ defined by $X \mapsto ft$ and $Y \mapsto gt$ is spanned by $gX-fY$. That is $\Rees_A(f,g) \cong A[X,Y]/(gX-fY)$.
\end{exmp}

We conclude that if $I$ is spanned by a regular sequence (of length $2$), then the Rees algebra $\RIA$ is isomorphic to the symmetric algebra $\SIA$, defined as 
\[
 \SIA =\bigoplus_{n\geq 0} I^{\otimes n}/(x\otimes y -y\otimes x)_{x,y \in I}.
\]

This can be generalized to a sequence $(\fs)$ of length $n$. In the general case we get that the ideal of relations $J$ is spanned by the $2\times 2$-minors of
\[
 \begin{pmatrix} f_1 & \cdots & f_n\\ T_1 & \cdots & T_n \end{pmatrix}.
\]
We will deepen our understanding of the relationship between the symmetric algebra and the Rees algebra in the following section. We will also see that in the particular context of implicitization theory applied to rational maps defined over a projective scheme, this situation is never reached. Precisely, we cannot hope that the symmetric algebra and the Rees algebra coincide, we can at most ask when they coincide modulo their torsion at the maximal ideal $\mm=(X_1,\hdots,X_n)$.


\section{Blow-up algebras}\label{blow-ap-algebras}

Henceforward let $\kk$ be an infinite integral domain with unity and let $A$ be a commutative $\NN$-graded $\kk$-algebra. Take $I=(\fs)$ an homogeneous ideal of $A$, where $f_i$ is an homogeneous element of degree $d$. We will write $I^n$ for the usual multiplication of $n$ elements of $I$ for $n\geq 0$, and $I^0:=A$. Denote $I^{\otimes n}:=I\otimes_A\cdots\otimes_A I$ $n$ times for $n\geq 0$, where $I^{\otimes 0}:=A$. In this part we will study presentations for the algebras $\RIA$ and $\gr$, and the relation with the symmetric algebras $\SIA$ and $\SIIA$. All these algebras  
\begin{enumerate}
\item $\RIA= \bigoplus_{n\geq 0} I^n$;
\item $\SIA =\bigoplus_{n\geq 0} I^{\otimes n}/(x\otimes y -y\otimes x)_{x,y \in I}$;
\item $\gr_A(I)=\bigoplus_{n\geq 0} I^n/I^{n+1}\cong A/I\otimes_A \RIA$;
\item $\SIIA =\bigoplus_{n\geq 0} (I/I^2)^{\otimes n}/(x\otimes y -y\otimes x)_{x,y \in I}\cong A/I\otimes_A \SIA$.
\end{enumerate}
are called blow-up algebras, because they are closely related to the blow-up of a ring along an ideal.

\subsection{Rees algebras and symmetric algebras of an ideal}\label{RIASIA}

The first idea for giving equations to describe the Rees algebra $\RIA$, is by means of the linear syzygies of $I=(\fs)$. Precisely, there is a presentation homogeneous ideal $J=J_1+J_2+\cdots$ which represents the equations of $\RIA$, where $J_r$ is the module spanned by the syzygies of $r$-products of $\fs$.

Assume $I$ is of finite presentation $0\to Z \to A^n \stackrel{\epsilon}{\to} I \to 0$, where $Z=\{(\as) :  \sum a_if_i=0\}$ is the module of syzygies of $I$.

The map $\epsilon$, induces a surjective morphism $\alpha:A[\ts]\to \SIA$, defined in degree $1$ by $\alpha(T_i)=f_i$. Denote $J':=\ker(\alpha)$. Then, there is a presentation for $\SIA$: 
\begin{equation}\label{presentSIA} 
 0\to J' \to A[\ts] \stackrel{\alpha}{\to} \SIA \to 0.
\end{equation}
It can be shown that the ideal $J'$ is generated by the linear form $\sum_i a_iT_i$ such that $\sum_i a_if_i=0$, 

Consider now the following presentation of the Rees algebra:
\begin{equation}\label{presentRIA} 
 0\to J\to A[\ts]\stackrel{\beta}{\to} \RIA\to 0,
\end{equation}
where the map $\beta:A[\ts]\to \RIA$ is $A$-linear and defined by $\beta(T_i)=f_i$. Clearly the ideal $J$ is an homogeneous ideal and its component of degree $1$ is $J_1$, which is the $A$-module of linear forms $\sum a_iT_i$ such that $\sum a_if_i=0$. Thus $J'$ is spanned by $J_1$.

Closely related to this presentation of $\RIA$ there is one for the associated graded ring of $I$, $\gr_A(I)$, coming from the $I$-adic filtration $\cdots \subset I^{n+1}\subset I^n \subset \cdots \subset I^2 \subset I$ in $A$. Namely, since $\RIA \cong A[\ts]/J$, there is an exact sequence
\begin{equation}\label{presentgr}
 0\to J+I\to A[\ts]\to \gr_A(I)\to 0. 
\end{equation}
We describe $J$ in terms of a presentation of $I$.

When $I$ is generated by a regular sequence $\{\fs\}$, the Rees algebra coincides with the symmetric algebra, and the ideals $J$ and $J'$ are spanned by the $2\times 2$-minors of the matrix $M=\left( \begin{array}{ccc} f_1 & \cdots & f_n\\ T_1 & \cdots & T_n \end{array}\right)$. 
 
Let $S$ be a polynomial ring $I'$ an ideal of $S$, and take $A=S/I'$. Let $I$ be an ideal of $A$. It is shown in \cite{Vas1} that

\begin{prop}\label{ReesElimin}
Let $\fs\in S$ be $n$ homogeneous polynomials of the same degree that span $I$. Consider $S[\ts]$. Then $\RIA \cong S[\ts]/J,$ where $J=(T_1-f_1t,\hdots,T_n-f_nt,I')\cap S[\ts]$ and $\gr_A(I)\cong S[\ts]/(\fs,J)$.
\end{prop}

It is a well known fact that $J'=(\sum a_iT_i \ :\ (\as)\in Z),$. Explicitly, $J'=\{\sum g_iT_i, \ :\ g_i=g_i(\ts)\in A[\ts], \text{and } \sum g_i(\ts)f_i=0\}$.

\begin{defn}\label{defLineartype}
The \textsl{relation type} of $I$ is the smallest integer $s$ such that $J=(J_1,\hdots,J_s)$. This number is independent of the generators chosen for $I$  (cf.\ \cite{Vas1}). When $s=1$, we say that $I$ is of \textsl{linear type}.
\end{defn}

Observe that since $\RIA$ is a commutative $A$-algebra, there exists a surjective map $\sigma: \SIA \to \RIA$, given by $\beta_m: \overline{I^{\otimes m}}\to I^m$ defined as $\overline{f_{i_1}\otimes\cdots\otimes f_{i_m}}\mapsto f_{i_1}\cdots f_{i_m}$. From the presentations of \eqref{presentSIA} and \eqref{presentRIA} for $\SIA$ and $\RIA$ respectively we have the following diagram:
\[
 \xymatrix@1{0\ar[r] & J'=(J_1)\ \ar@{^{(}->}[d]\ \ar@{^{(}->}[r] &A[\ts]\ar@{=}[d]\ar[r]^{\alpha} &\SIA\ar@{->>}[d]^{\sigma}\ar[r]&0\\ 0\ar[r] & J\ \ar@{^{(}->}[r] &A[\ts]\ar[r]^{\beta} &\RIA\ar[r]&0}
\]

Denote by $K:=\ker(\sigma)$, hence $K=J/J'$, and $K=0$ iff $I$ is of linear type, equivalently, $\sigma$ is an isomorphism between $\RIA$ and $\SIA$.

\subsection{$d$-sequences}\label{d-seq}

\begin{defn}\label{sucesiones} Let $\textbf{x}=\{\xs\}$ be a sequence of elements of a ring $A$, let $I=(\xs)$ be an ideal of $A$. We say that $\textbf{x}$ is a:
 \begin{enumerate}
  \item \textsl{regular sequence in $M$}, where $M$ is an $A$-module, if: 
  \begin{enumerate}
   \item $(\xs)M\neq M$;
   \item for all $i=1,\hdots,n$, $x_i$ is not a  zero divisor in $M/(x_1,\hdots,x_{i-1})M$.
 \end{enumerate}
 \item \textsl{$d$-sequence} if: 
 \begin{enumerate}
  \item \textbf{x} is a minimal system of generators of $I$;
  \item $((x_1,\hdots,x_i):x_{i+1}x_k)=((x_1,\hdots,x_i):x_k)$ for all $i=1,\hdots,n-1$ and $k\geq i+1$.
 \end{enumerate}
 \item \textsl{relative regular sequence} if $((x_1,\hdots,x_i):x_{i+1})\cap I=(x_1,\hdots,x_i)$ for all $i=1,\hdots,n-1$.
 \item \textsl{proper sequence} if $x_{i+1}H_j(x_1,\hdots,x_i;A)=0$ for all $i=1,\hdots,n-1, \ j>0$, where $H_j(x_1,\hdots,x_i;A)$ denote the $j$-th module of Koszul homology associated to the sequence $\{x_1,\hdots,x_i\}$.
 \end{enumerate}
\end{defn}
These conditions are related in the following way: 
\begin{center}
 regular sequence $\Rightarrow$ d-sequence $\Rightarrow$ relative regular sequence $\Rightarrow$ proper sequence. 
\end{center}

\begin{lem}\label{dseq=>tlineal} Every ideal generated by a $d$-sequence is of linear type.
\end{lem}
\begin{proof}
 See \cite{Vas1}.
\end{proof}
\section{Rees and Symmetric algebras of a rational map}

Assume we have a rational map $\phi:\PP^{n-2} \dashrightarrow \PP ^{n-1}$ defined by homogeneous polynomials $\{f_i\}_{i=1,\hdots n}$ of degree $d$. 
Let $\kk$ be a commutative ring and $A$ a $\ZZ$-graded $\kk$-algebra. Denote by $\iota$ the map that sends $\kk$ in $A_0$. The map $\phi$ defines a morphism of $\kk$-algebras $h:\kk[T_1,\hdots,T_n] \to A$, that maps $T_i \mapsto f_i$. This map defines a morphism of affine schemes $\mu :\bigcup D(f_i) \to \bigcup D(T_i)=\AA^n_\kk-\{0\}$ and a map of projective schemes  $\phi :\bigcup D_+(f_i) \to \bigcup D_+(T_i)=\PP^{n-1}_\kk$.

We have mentioned that $\phi$ also defines a graded map of $A$-algebras defined by $T_i \mapsto f_i\cdot t$, defining the Rees algebra as a quotient of a polynomial ring: $\RIA \cong \frac{\polnaT}{J}$. The ideal $J$ can be described as $J=(T_1-f_1\cdot t,\hdots,T_n-f_n\cdot t)\cap A[\ts]$, using Proposition \ref{ReesElimin}.

Consider the extended Rees algebra $\Rees_{A[t^{-1}]}(I)$ as a sub-$A$-algebra of $A[t,t^{-1}]$. Denote $u:=t^{-1}$, hence, $\eta: A[\ts,u]\to A[u,u^{-1}]$ is defined $T_i \mapsto f_i\cdot u^{-1}$.

\begin{lem}\label{caract J}
If $J=(T_1-f_1\cdot t,\hdots,T_n-f_n\cdot t)\cap A[\ts]$, then $J=((T_1u-f_1,\hdots,T_nu-f_n):u^\infty) \cap A[\ts]$.
\end{lem}
It can be seen that the kernel of the map $h:\kk[T_1,\hdots,T_n] \to A$ defined in \eqref{def h} is given by
\begin{equation}\label{ker h y epsilon} 
\ker(h) = \epsilon^{-1}((T_1-f_1,\hdots,T_n-f_n))=\{g\in \kk[\ts] \ :\ g(\fs)=0\}.
\end{equation}
Writing with $i$ the inclusion map $A[\ts]\hookrightarrow A[\ts,u]$ and by $\theta=i\circ \epsilon$ the composition, we have a description of $\ker(h)$
\begin{lem}
$\ker(h)=\theta^{-1}((T_1u-f_1,\hdots,T_nu-f_n):u^\infty)$.
\end{lem}
In \cite{BuJo03}, the authors also proved that 
\begin{rem}\label{ker h y k=A_0}
If $\kk \subset A_0$ then $\ker(h)=((T_1u-f_1,\hdots,T_nu-f_n):u^\infty)\cap \kk[\ts]$. Moreover, if $\kk=A_0$, $\deg(T_i)=0$ and $\deg(t)=d\geq 1$, then $\kk[\ts]=(A[\ts,u])_0$ and hence, $\ker(h)=((T_1u-f_1,\hdots,T_nu-f_n):u^\infty)_0$.
\end{rem}

Now, we can compute $\ker(h)$ from $\ker(\beta)$, defined in \eqref{presentRIA}.

\begin{prop}\label{ker h y ker beta}
Assume $\iota: \kk \to A$ is the inclusion, then $\ker(h)=\ker(\beta)\cap \kk[\ts]=((T_1u-f_1,\hdots,T_nu-f_n):u^\infty)\cap \kk[\ts]$. Moreover if $I'$ is an ideal of $A$ such that $H^0_{I'}(A)=0$, then $\ker(\beta)=(\ker(\beta):(I')^\infty)$ and hence $\ker(h)=(\ker(\beta):(I')^\infty)\cap \kk[\ts]$.
\end{prop}


\section{Approximation complexes}\label{CA}

\bigskip

Approximation complexes were defined by Herzog, Simis and Vasconcelos in \cite{HSV} almost 30 years ago. We will give here a brief outline on these complexes and some of their basic properties.

Consider the two Koszul complexes over the ring $A=\kk[\Xs]$ associated to the sequences $\fs$ and $\ts$ respectively.  
\[
 \k.(\fs;A[\ts]): \quad\cdots\to \bigwedge^{1}A[\ts]^n \stackrel{d_f}{\to} A[\ts]
\]
that will be denoted by $\k.(\textbf{f};A[\textbf{T}])$, and
\[
\k.(\ts;A[\ts]): \quad\cdots\to \bigwedge^{1}A[\ts]^n \stackrel{d_T}{\to} A[\ts] 
\]
that will be denoted by $\l.$ meaning $ \kt.$.

It is easy to verify that $d_f \circ d_T - d_T \circ d_f=0$ giving rise to a double complex $\kpp (\textbf{f},\textbf{T};A[\textbf{T}])$. In particular, $d_T$ induces a morphism between the cycles $Z_i$, boundaries $B_i$ and homologies $H_i$ of $\ki.$. The complexes obtained having as objects, the cycles $Z_i$, boundaries $B_i$ and homologies $H_i$ of $\ki.$ with the induced differentials $d_t$ are called approximation complexes of cycles, boundaries and homologies respectively, and denoted by $\Z.$, $\B.$, $\M.$ respectively.

\bigskip

It is easy to verify that $H_0(\Z.)=A[\ts]/d_T(\ker(d_f))=\SIA$. Similarly, $H_0(\M.)=\SIIA$. hence, it is important to give acyclicity conditions for the complexes $\Z.$ and $\M.$., in order to provide resolutions to $\SIA$ and $\SIIA$. 
  
\bigskip

One important property of the approximation complexes is the following
\begin{prop}\label{CAindepI}
The modules $H_i(\Z.)$, $H_i(\B.)$ and $H_i(\M.)$ are independent of the generators chosen for $I$, for all $i$.
\end{prop}
\begin{proof}
Proposition 3.2.6 and Corollary 3.2.7 of \cite{Vas1}
\end{proof}

We will denote by $(\Z.)_t$, $(\B.)_t$ and $(\M.)_t$ the $t$-graded strand of the complexes, considering the degree on the variables $T_1,\hdots, T_n$. We will write $\SS_s$ for the component of degree $s$ of $\sym(A^n)$. 

Since $d_T$ has degree $1$ on the variables $T_i$, we get for each $t$ a subcomplex of $\Z.$
\[
 (\Z.)_t: 0\to (Z_n)_t \stackrel{d_T}{\to} (Z_{n-1})_t \stackrel{d_T}{\to} \cdots \stackrel{d_T}{\to} (Z_1)_t \stackrel{d_T}{\to} (Z_0)_t \to 0.
\]

By definition we can rewrite the module $(Z_i)_t$ as $Z_i(\KK)\otimes_A \SS_{t-i}$. Hence we get that
\[
 (\Z.)_t: 0\to Z_n(\KK)\otimes_A \SS_{t-n} \stackrel{d_T}{\to} \cdots \stackrel{d_T}{\to} Z_1(\KK)\otimes_A \SS_{t-1} \stackrel{d_T}{\to} Z_0(\KK)\otimes_A \SS_{t} \to 0.
\]
Similarly, $(\M.)_t: 0\to H_n(\KK)\otimes_A \SS_{t-n} \stackrel{d_T}{\to} \cdots \stackrel{d_T}{\to} H_1(\KK)\otimes_A \SS_{t-1} \stackrel{d_T}{\to} H_0(\KK)\otimes_A \SS_{t} \to 0.$

Finally, we propose a different notation fot the complex $\Z.$ that will be very convenient. Observe that the module $\ZZZ_i$ is an ideal of the $i$-th module of the Koszul complex $\ki.$, where the maps have degree $d$ on the grading of $A$. If we write the complex with the adequate shift, we get
\[
 \ki.: 0\to K_n[-dn]\stackrel{d_f}{\to}K_{n-1}[-d(n-1)] \stackrel{d_f}{\to} \cdots \stackrel{d_f}{\to} K_1[-d]\stackrel{d_f}{\to}A[\ts]\to 0,
\]
Hence, with this notation we have that the complex $\Z.$ has as objects $\ZZZ_i=Z_i(\KK)[di]\otimes_A A[\ts]$. 

\begin{lem}\label{isoZyB} Denote $H_i'(\Z.)$ for $(H_i'(\Z.))_t=(H_i(\Z.))_t$ if $i\geq 0$ and $t>0$; and $(H_0'(\Z.))_0=0$. For all $i$ and all $t$, the conexion morphism $\delta:(H_i(\B.))_t \to (H_i(\Z.))_{t+1}$ induces an isomorphism $\delta':(H_i(\B.))_t \stackrel{\sim}{\to} (H'_i(\Z.))_{t+1}$.
\end{lem}
\begin{proof}
The complex $\l.:=\kt.$ with maps $d_T$ is exact since the sequence $\{\ts\}$ is regular. In particular each homogeneous strand $(\l.)_t$ is acyclic for all positive $t$. Hence, for all $i,t>0$, $(H_i(\B.))_t \stackrel{\delta}{\to} (H_i(\Z.))_{t+1}$ is an isomorphism. Denoting by $\pi$ the right-most (non-zero) map of the long exact sequence of homology we get a short exact sequence
 $0 \to H_0(\B.) \stackrel{\delta}{\to} H_0(\Z.) \stackrel{\pi}{\to} H_0(\l.)\to 0$,
that provides the isomorphism $H_0(\B.)\stackrel{\delta}{\cong} \ker(\pi)$. Moreover, $(H_0(\l.))_t=0$ iff $t=0$ and $(H_0(\l.))_0=A$. Then, we get the conexion morphism $\delta:(H_i(\B.))_t \to (H_i(\Z.))_{t+1}$ induces an isomorphism $\delta':(H_i(\B.))_t \stackrel{\sim}{\to} (H'_i(\Z.))_{t+1}$. 
\end{proof}

By definition of $\Z.$, $\B.$ y $\M.$, for each $t$ we have a graded short exact sequence of complexes 
$0 \to \B. \to \Z. \to \M. \to 0$, giving rise to a long exact sequence in homology. From Lemma \ref{isoZyB}, we get
\begin{equation}\label{SELBZM}
\begin{array}{c}
\cdots \to H_{i+1}(\M.) \stackrel{\Delta}{\to} H'_i(\Z.)(1) \to H_i(\Z.) \to H_i(\M.) \stackrel{\Delta}{\to} H'_{i-1}(\Z.)(1)\to \cdots\\
\qquad \qquad \qquad \qquad \cdots  \to H_1(\M.) \stackrel{\Delta}{\to} H'_0(\Z.)(1) \to H_0(\Z.) \to H_0(\M.)\to 0,
\end{array}
\end{equation} 
where $H_i(\M.) \stackrel{\Delta}{\to} H'_{i-1}(\Z.)$ stands for the composition of the connection morphism in the last exact sequence, with $\delta'$ of Lemma \ref{isoZyB}. We get the following

\begin{prop}\label{HM=0 entonces HZ=0}
If $H_i(\M.)=0$ then $H_i(\Z.)=0$. In particular, if $\M.$ is acyclic, then $\Z.$ is also acyclic.
\end {prop}
\begin{proof}
Using the long exact sequence we get that if $H_{i+1}(\M.)=H_i(\M.)=0$, then $0=H_{i+1}(\M.)\to H'_i(\Z.)(1) \to H_i(\Z.) \to H_i(\M.)=0$, hence $H_i(\Z.)=0$.

Again from the long exact sequence we get $H_i(\Z.)(1) \to H_i(\Z.) \to H_i(\M.)$ is exact for all $t$ and all $i>0$. By hypothesis, $H_i(\M.)=0$, Since $A$ is noetherian, $H_i(\Z.)$ is of finite type. Since the map $H_i(\Z.)(1) \to H_i(\Z.)$ is given by the composition of the isomorphism $\delta'$ with the inclusion $(\B.)_t$ in $(\Z.)_t$, then, we get an isomorphism $H_i(\Z.)(1) \stackrel{\sim}{\to} H_i(\Z.)$. Hence, for all $t$ $(H'_i(\Z.))_{t+1} \stackrel{\sim}{\to} (H_i(\Z.))_t$. Iteratively, from $(H_i(\Z.))_{-1}=0$ we get $(H_i(\Z.))_t=0$ for all $t$.
\end{proof}

From the long exact sequence of homologies 
\[
 \cdots  \to H_1((\M.)_t) \stackrel{\Delta}{\to} H'_0((\Z.)_{t+1}) \stackrel{\lambda}{\to} H_0((\Z.)_t) \to H_0((\M.)_t)\to 0,
\]
we get 
\begin{equation}\label{eqcolasel}
 \cdots  \to H_1((\M.)_t) \stackrel{\Delta}{\to} (\SIA)_{t+1} \stackrel{\lambda}{\to} (\SIA)_t \to (\SIIA)_t\to 0,
\end{equation}
where $\Delta$ is the connecting mapping (composed by $\delta'$) and $\lambda$ is the downgrading mapping $\lambda: (\SIA)_{t+1}\cong (H'_0(\Z.))_{t+1} \stackrel{\delta'^{-1}}{\to} (H_0(\B.))_t \hookrightarrow (H_0(\Z.))_t\cong (\SIA)_t,$. 

\medskip

Let us go back to the relation between  Rees algebras and Symmetric algebras. From the long exact sequences arising from the short exact sequences of complexes $0\to \B. \to \Z. \to \M. \to 0$ \eqref{SELBZM}, we get a condition on the map $\sigma: \SIA \to \RIA$ for being an isomorphism, namely, for $I$ to be of linear type.

From the long exact sequence \eqref{eqcolasel} and the short exact sequence $0 \to I^{n+1} \to I^n \to I^n/I^{n+1} \to 0$ we obtain the following commutative diagrama
\[
 \xymatrix@1{H_1(\M.) \ar[r]& \SIA \ar[d]\ar[r]^{\lambda}&  \SIA \ar[d]^{\sigma}\ar[r]^{\pi}& \SIIA\ar[r]\ar[d]^{\gamma} & 0\\
0 \ar[r] & \ria\ar[r] & \RIA \ar[r] & \gr\ar[r] &0.}
\]
where $\ria$ consists on the ideal of $\RIA$ with elements of positive degree.

\begin{prop}\label{HM0entTipoLineal}
If $H_1(\M.)=0$ then $\sigma:\SIA \stackrel{\sim}{\to} \RIA$ is an isomorphism, namely, $I$ is of linear type.
\end{prop}
\begin{proof}
If $H_1(\M.)=0$ for each degree $i$ we get a commutative diagram 
\[
 \xymatrix@1{0 \ar[r]& (\SIA)_{i+1} \ar[d]^{\sigma_{i+1}}\ar[r]^{\lambda}&  (\SIA)_i \ar[d]^{\sigma_i}\\ 0 \ar[r] & (\ria\ar[r])_{i+1} & (\RIA)_i}
\]
where $\sigma_0:A=(\SIA)_0\to (\RIA)_0=A$ is the identity. Since $\sigma_0 \circ \lambda$ is injective, then $\sigma_1$ also is, hence, an isomorphism. Iteratively we get that $\sigma_t$ is an isomorphism for all $t$. 
\end{proof}

\begin{thm}If $A$ is noetherian, and $\sigma:\SIA \to \RIA$ is the map above and $\gamma: \SIIA \to \gr_A(I$ its reduction modulo $I$, then $\sigma$ is an isomorphism iff $\gamma$ is an isomorphism.
\end{thm}

\begin{proof}
Clearly, if $\sigma$ is an isomorphism, then also its reduction modulo $I$. Conversely, from the Snake Lemma applied to the diagram
\[
 \xymatrix@1{0 \ar[r]& K_{i+1}\ar[d]\ar[r] & (\SIA)_{i+1} \ar[d]^{\lambda_{i+1}} \ar[r]& I^{i+1} \ar[r]\ar[d] & \ 0 \ \ \\0 \ar[r]& K_{i}\ar[r] & (\SIA)_i \ar[r] &I^i\ar[r]&\ 0 \ ,}
\]
we get the short exact sequence $0\to K_i/\lambda_{i+1}(K_{i+1}) \to \SIIA_i\to \gr_A(I)_i\to 0$.
By hypotesis $K_i=\lambda_{i+1}(K_{i+1})$ for $i>1$. Since $K$ is a finitely generated ideal of $\SIA$, there exists $n>1$ such that $K_{i+1}=\SIA_1K_i$, for $i\geq n$. Applying $\lambda$ we get $K_i=\lambda(K_{i+1})=\lambda(\SIA_1K_i)=IK_i$.

Localizing and using Nakayama lemma, we get that $K_i=0$ for all $i\geq n$. By descendent induction we can annihilate the rest of the components.
\end{proof}

\section{Acyclicity of approximation complexes}

Assume that $A$ is an $\NN$-graded noetherian ring. Dente by $\mm:=A_+= \bigoplus_{i>0}A_i$. 

\begin{rem}\label{H(K) sop en V(m)}
Write $\k.$ for the Koszul complex $\k.(\texttt{x};A)$. If $I$ and $\mm$ have the same radical then $\supp(H_i(\k.)) \subset V(\mm)$, this is $H_i(\k.)_\pp=0$ for $\pp \neq \mm$. Hence, we also have $\supp(H_i(\M.))\subset V(\mm)$ and $\supp(H_i(\Z.))\subset V(\mm)$.
\end{rem}

Laurent Bus\'e and Jean-Pierre Jouanolou proved in \cite{BuJo03} that:

\begin{prop}\label{acicl sin pb} Let $I=(\xs)$ be an ideal of $A$ such that $\rad(I)=\rad(\mm)$ and $r=\depth(\mm:A)\geq 1$. Then $H_i(\Z.)=0$ for all $i\geq max\{1,n-r\}$. In particular if $n\geq 2$ and $r\geq n-1$, then $\Z.$ is acyclic. 
\end{prop}

This result states acyclicity when the ideals $I$ and $\mm$ have the same radical. Geometrically, if $I$ stands for the base locus ideal of a rational map, this means, that the map is well-defined everywhere. Since the condition $\rad(I)=\mm$ is not ubiquitous, Bus\'e and Jouanolou gave a generalization of this result, in the same article \cite{BuJo03}.

First, given an ideal $J$ of a ring $A$ denote by $\mu(J)$ the minimum number of generators of $J$.

\begin{defn}\label{lci}
Let $I$ be an ideal of a ring $A$. We say that $I$ is a local complete intersection (LCI) in $\proj(A)$ iff for all $\pp\in \spec(A)\setminus V(\mm)$ we have $\mu(I_\pp)=\depth(I_\pp: A_\pp)$. We say that $I$ is an almost local complete intersection (ALCI) in $\proj(A)$ iff for all $\pp\in \spec(A)\setminus V(\mm)$ we have $\mu(I_\pp)+1=\depth(I_\pp: A_\pp)$. 
\end{defn}

\begin{prop}\label{acicl con pb}
Let $I=(\fs)$ be a LCI ideal of $A$. Take $n\geq 2$, and assume that $\depth(\mm: A)\geq n-1$ and $\depth(I: A)= n-2$. Then, the complex $\Z.$ associated to $I$ is acyclic.
\end{prop}

\begin{lem}[{\cite[Lemma 4.10]{BuJo03}}]\label{lema acicl con pb aux 1}
Let $I=(\fs)$ be an ideal of $A$ such that $\depth(\mm: A)> \depth(I:A)=r$. Then $H^0_\mm(H_{n-r}(\k.))=0$.
\end{lem}

\begin{lem}[{\cite[Lemma 4.11]{BuJo03}}]\label{lema acicl con pb aux 2}
Let $I=(\fs)$ be an ideal of $A$. Write $\zeta:=\mu(I)-\depth(I:A)$ and for all $\pp\in \spec(A)\setminus V(\mm)$ we have $\zeta_\pp:=\mu(I_\pp)-\depth(I_\pp:A_\pp)$. Then
\begin{enumerate}
\item for all $i>\zeta$, $H_i(\M.)=0$;
\item for all $\pp\in \spec(A)\setminus V(\mm)$ we have $\zeta>\zeta_\pp$, hence, $H_\zeta(\M.)=H^0_\mm(H_\zeta(\M.))$.
\end{enumerate}
\end{lem}

In \cite{HSV} it is proved that:
\begin{thm}\label{thmProperZacyclic}
Let $A$ be a ring and $I$ an ideal of $A$. Consider the following statements:
\begin{enumerate}
\item $I$ is generated by a proper sequence;
\item the complex $\Z.$ associated to $I$ is acyclic.
\end{enumerate}
Then $(a)$ implies $(b)$. Moreover, if $A$ is local, with maximal ideal $\mm$, with residue infinite field $\kk$, or if $A$ is graded such that $A_0=\kk$ is an infinite field and $\mm:A_+$ generated in degree $1$; then $(a)$ and $(b)$ are equivalent.
\end{thm}


\section{Implicitization}\label{implicitization}

\bigskip

In this section we will overview the implicitization problem in two perspective, focusing on the second one. First, we will briefly introduce the method by Sederberg and Chen, later developed in depth by Bus\'e, Cox and D'{}Andrea. This method consists in the so called theory of \textit{moving curves} and \textit{moving surfaces}. We will see that this is a ``innocent"{} way of approaching a very deep subject that involves sophisticated homological and commutative algebra and geometry.

Second, we will treat the implicitization problem by means of approximation complexes, where we will use all the algebraic tool we exposed in the sections before. This point of view has been developed by Bus\'e, Chardin and Jouanolou since the beginning of this century.

\subsection{Moving curves and  moving surfaces}\label{moving cosas}

In this part, we will sketch some results on moving curves and moving surfaces obtained by Sederberg and Chen in \cite{SC95}, and later more sophistificated approaches by Bus\'e, Cox and D'{}Andrea in \cite{Co01,Co03,DA01, BCD03}.

We will follow the classical notation by D. Cox. For a better reading, we will give a short dictionary. Denote by $s,t,u$ the variables $X_1,X_2,X_3$, $\kk=\CC$ and hence, the ring $A=k[X_1,X_2,X_3]$ or $A=k[X_1,X_2]$ will be $R=\CC[s,t,u]$ or $\CC[s,t]$ respectively. We will write $x,y,z,w$ for $T_1,T_2,T_3,T_4$ and $a,b,c,d$ for the functions $f_1,f_2,f_3,f_4$.  $A,B,C,D$ will denote the syzygies that we have written $a,b,c,d$, namely $A\cdot a+\cdot Bb+C\cdot c+D\cdot d=0$ or $A\cdot a+B\cdot b+C\cdot c=0$, depending on the context. We will denote by $k$ the degree of $A,B,C,D$.

The question we want to reply is: How to get an implicit equation $F$ which defines the curve or the surface given parametrically by $a,b,c,d$.

\subsubsection{Moving curves}\label{Moving curves}

Assume that $\phi: \PP^1_\CC \to \PP^2_\CC$ is a map which has as image a plane curve. We will compute the implicit equation of the image of $\phi$, given by $\phi(s,t)=(a(s,t),b(s,t),c(s,t))$, where $a,b,c\in R=\CC [s,t]$ are homogeneous polynomials of degree $k$. First, assume that $gcd(a,b,c)=1$. Hence, $\phi$ has no base points. Sederberg et.\ al.\ have introduced in \cite{SC95} and \cite{CSC98} the idea of \textsl{moving lines} in $\PP^1$.

Let $x,y,z$ be homogeneous coordinates in $\PP^2$. A moving line consists in an equation 
\[
 A(s,t)x+B(s,t)y+C(s,t)z=0
\]
where $A,B,C\in R$ are homogeneous polynomials of the same degree. We can see the formula above as a family of lines parametrized by $(s,t)\in \PP^1$.

\begin{defn}\label{movline}
We will say that the moving line $A(s,t)x+B(s,t)y+C(s,t)z=0$ follows the parametrization $\phi(s,t)=(a(s,t),b(s,t),c(s,t))$ if $$A(s,t)a(s,t)+B(s,t)b(s,t)+C(s,t)c(s,t)=0$$ for all $(s,t)\in \PP^1$.
\end{defn}

Geometrically, this means that the point $(s,t)$ lies on a line. Algebraically, Definition \ref{movline} says that $A,B,C$ is a syzygy in $a,b,c$, namely $(A,B,C)\in \Syz(a,b,c)$, where $\Syz(a,b,c)\subset R^3$ is the module of syzygies of $(a,b,c)$.

Since $\Syz(a,b,c)$ is a graded module, we write $\Syz(a,b,c)_s$ for its $s$-strand. We will see that $\Syz(a,b,c)_{k-1}$ determines the implicit equation of the image of $\phi$. 

Indeed, consider the Koszul map given by $(a,b,c)$, $R^3_{k-1}\stackrel{(a,b,c)}{\lto}R_{2k-1}$, which has degree $k$. Its kernel is $\Syz(a,b,c)_{k-1}$. Observe that $\dim_\CC(R^3_{k-1})=3k$, $\dim_\CC(R_{2k-1})=2k$. Hence, $\dim_\CC(\Syz(a,b,c)_{k-1})=k$ if and only if the map given by $(a,b,c)$ has maximal rank. Thus, we can get $k$ generator (moving lines) linearly independent following $\phi$. We will denote them by: 
\[
 A_ix+B_iy+C_iz=\sum_{j=0}^{k-1}L_{i,j}(x,y,z)s^jt^{k-1-j}, \qquad i=0,\hdots, k-1,
\]
where the $L_{i,j}(x,y,z)$ are linear forms with coefficients in $\CC$.

One of the main results in this area is the following:

\begin{thm}\label{thm moving lines}
Let $\mathscr C$ be the image of $\phi$, and denote by $e$ its degree. Then $\det(L_{i,j})=\lambda F^e$, where $\lambda\in \CC-\{0\}$ and $F=0$ is the implicit equation of the curve $\mathscr C \subset \PP^2$.
\end{thm}

This can be seen for example in \cite{Co01,Co03}.

Observe that $a,b,c$ heve degree $k$, the curve $\mathscr C$ is defined by $\phi$ which has degree $k/e$, where $e=\deg(\phi)$. Hence, $\deg(F^e)=k$. On the other hand, the determinant of Theorem \ref{thm moving lines} has also degree $k$, since the forms $(L_{i,j})$ are linear.

We will study this with some more algebra. Take $I=(a,b,c)\subset R$. There is an exact sequence
\[
 0\to \Syz(a,b,c)\to R(-k)^3\stackrel{(a,b,c)}{\lto}I\to 0.
\]

In two variables, Hilbert syzygy theorem implies that $\Syz(a,b,c)$ is free. By the Hilbert polynomial we get 
\[
 \Syz(a,b,c)\cong R(-k-\mu_1)\oplus R(-k-\mu_2), \qquad \mu_1+\mu_2=k.
\]
Hence, if we write $\mu=\mu_1\leq \mu_2=k-\mu$, then, there exist syzygies $p,q\in \Syz(a,b,c)$ such that $\Syz(a,b,c)=R.p\oplus R.q$ where the degree of $p$ is $\mu$ and the degree of $\mu$ is $k-\mu$. We say that $\{p,q\}$ is a $\mu$-bases of the parametrization $\phi:\PP^1\to \PP^2$.

Hence, we have the following free presentation of $I$
\begin{equation}\label{eqFreePresI}
 0\to R(-k-\mu_1)\oplus R(-k-\mu_2)\to R(-k)^3\stackrel{(a,b,c)}{\lto}I\to 0.
\end{equation}
The existence of $\mu$-basis has many important consequences, namely,

\begin{prop}
If $\mathscr C$ is the image of $\phi$, $e=\deg(\phi)$ and $p,q$ form a $\mu$-basis of $\phi$. Then, $\Res(p,q)=F^e$, where $F=0$ is the implicit equation of $\mathscr C \subset \PP^2$.
\end{prop}

From the existence of a $\mu$-basis we can get important consequences about the regularity of the ideal $I=(a,b,c)$. From the free presentation \eqref{eqFreePresI} of $I$, we can prove that $\reg(I)=2k-\mu-1$. Hence, a $\mu$-basis determines the regularity of an ideal.

\subsubsection{Moving surfaces}\label{Moving surfaces}

In this part, we will focus on the implicitization problem of surfaces in $\PP^3$. Take $\phi:\PP^2\to \PP^3$, given by homogeneous polynomials $a,b,c,d\in R=\CC[s,t,u]$ of degree $k$. Assume, as before, that $a,b,c,d$ have no common zeroes, that is $\phi$ has no base points.

The analog of moving lines in $\PP^2$ are \textsl{moving planes} in $\PP^3$. A moving plane is an equation
\[
 A(s,t,u)x+B(s,t,u)y+C(s,t,u)z+D(s,t,u)w=0,
\]
where $x,y,z,w$ are homogeneous coordinates in $\PP^3$, and $A,B,C,D$ are elements of $R$ of the same degree.

\begin{defn}\label{movplane}
We say that a moving plane follows the parametrization $\phi$ if $$A(s,t,u)a(s,t,u)+B(s,t,u)b(s,t,u)+C(s,t,u)c(s,t,u)+D(s,t,u)d(s,t,u)=0$$ for all $(s,t,u)\in \PP^2$. That is, if and only if $A,B,C,D\in \Syz(a,b,c,d)$.
\end{defn}

We will see that moving planes are not enough in order to get the implicit equation of the image of $\phi$, it will be necessary the use of \textsl{moving surfaces} of higher degree. In this case, we will consider \textsl{moving quadrics}, which are equations: 
\[
 (s,t,u)x^2+B(s,t,u)xy+\cdots+I(s,t,u)zw+J(s,t,u)w^2=0,
\]
where $A,B,\hdots,I,J$ are homogeneous elements of $R$ of the same degree. A moving quadric follows the parametrization when $A,B,\hdots,I,J\in \Syz(a^2,ab,\hdots,cd,d^2)\subset R^{10}$.

Moving planes and moving quadrics can be obtained as 
\[
 MP: R^4_{k-1}\stackrel{(a,b,c,d)}{\lto}R_{2k-1},\quad \mbox{and}
\]
\[
MQ: R^{10}_{k-1}\stackrel{(a^2,ab,\hdots,cd,d^2)}{\lto}R_{3k-1}
\]

Observe that $\dim_\CC(R_{2k-1})=k(2k+1)$ and $\dim_\CC(R^4_{k-1})=2k(k+1)$. Hence, the space of moving planes has dimension $2k(k+1)-k(2k+1)=k$ iff the map $MP$ has maximal rank. Similarly, the space of moving quadrics has dimension $(k^2+7k)/2$ iff $MQ$ has maximal rank.

\begin{rem}\label{remMovSurfacesM}
 Remark that each moving plane gives place to four moving quadrics, obtained by multiplication by the four variables $x,y,z,w$. Hence, if $MP$ and $MQ$ have maximal rank, then there are exactly $(k^2+7k)/2-4k=(k^2-k)/2$ moving quadric linearly independent not coming from moving planes. Taking these $(k^2-k)/2$ moving quadrics and the $k$ moving planes, we build a matrix $M$ of size $(k^2+k)/2\times (k^2+k)/2$, where:
\begin{enumerate}
\item $k$ rows correspond to the $k$ moving planes of degree $k-1$;
\item $(k^2-k)/2$ rows come from the moving quadrics of degree $k-1$.
\end{enumerate}
\end{rem}

We get a similar result to Theorem \ref{thm moving lines}:

\begin{thm}
Let $\phi:\PP^2\to \PP^3$ be a rational map without base points, given by $\phi(s,t,u)=(a(s,t,u),b(s,t,u),c(s,t,u),d(s,t,u))$. Assume $\phi$ admits exactly $k$ linearly independent moving planes of degree $k-1$ following the parametrization. Then, the image of $\phi$ is given by $\det(M)=0$, where $M$ is the matrix in Remark \ref{remMovSurfacesM}.
\end{thm}

We can rewrite this as follows. Let $\phi$ be a rational map given by homogeneous polynomials $f_1,f_2,f_3,f_4$ of degree $k$. Write $x_1, x_2, x_3$ for the variables $s,t,u$ and by $t_1,t_2,t_3,t_4$ the variables $x,y,z,w$. Then, we write the moving planes as polynomials 
\[
 a_1(x_1, x_2, x_3)t_1+a_2(x_1, x_2, x_3)t_2+a_3(x_1, x_2, x_3)t_3+a_4(x_1, x_2, x_3)t_4
\]
and the moving quadrics as
\[
 a_{1,1}(x_1, x_2, x_3)t_1^2+a_{1,2}(x_1, x_2, x_3)t_1t_2+\cdots + a_{3,4}(x_1, x_2, x_3)t_3t_4+a_{4,4}(x_1, x_2, x_3)t_4^2,
\]
where the $a_i$ and the $a_{i,j}$ are homogeneous polynomials. If we take $k$ moving planes $L_1,\hdots, L_k$ and $l=(k^2-k)/2$ moving quadrics $Q_1,\hdots,Q_l$ of degree $k-1$ following the parametrization, we obtain a square matrix $M$ corresponding to the map of $\CC[x,y,z,w]$-modules
\[
 \begin{array}{ccc}
\bigoplus_{i=1}^k \CC[x,y,z,w]\oplus \bigoplus_{j=1}^l \CC[x,y,z,w]	&\to	&  \CC[s,t,u]_{k-1}\otimes _\CC \CC[x,y,z,w]\\
(p_1,\hdots,p_k,q_1,\hdots,q_l)						&\mapsto& \sum_{i=1}^k p_iL_i+\sum_{j=1}^l q_jQ_j
 \end{array}
\]

It can be shown that is always possible to chose $L_1,\hdots, L_d$ and $Q_1,\hdots,Q_l$ such that $\det(M)\neq 0$ everywhere, and whose zeroes give the implicit equation of the image of $\phi$ raised to its degree. Again, we identify $\CC[s,t,u]_{k-1}$ with $\CC^l$, which permits ``hiding'' the variables $s,t,u$ in order to get expressions that only depend on $x,y,z,w$. 

\subsection{Implicitization by means of approximation complexes}\label{implicit con CA}

Recall from our first sections, let $\kk$ be a commutative ring, $h$ a graded ring of $\kk$-graded algebras, defined as:
\[
 h:\kk[\ts] \to A, \qquad T_i \mapsto f_i,
\]
that induces a map of $\kk$-projective schemes
\[
 \phi :\proj (A) \setminus V(\fs)=\bigcup D_+(f_i) \to \bigcup D_+(T_i)=\PP^{n-1}_\kk.
\]
We want to compute the closed image of $\phi$, called, \sti of $\phi$.

From Lemma \ref{sti}, the $\ker(h)$ defines the closure of the image of $\phi$. If $\jjj$ stands for $\ker(h)^\sim$, then $V(\jjj)=V((\ker(h):(\ts)^\infty)^\sim)$.

In this subsection we compute the implicit equation of $V(\jjj)$ with a different point of view respecto to the subsection above. Hence, assume $\kk$ is a field, $A$ is a polynomial ring in the variables $\Xss$. Thus, the maps $h$ and $\phi$ are rewritten: 
\[
 h:\kk[\ts] \to \kk[\Xss], T_i \mapsto f_i, \mbox{and}
\]
\[
\phi :\PP^{n-2}_\kk \setminus V(\fs)=\bigcup D_+(f_i) \to \bigcup D_+(T_i)=\PP^{n-1}_\kk.
\]
We have a rational map 
\begin{equation}
 \phi :\PP^{n-2}_\kk \dto \PP^{n-1}_\kk: (x_1:\hdots:x_{n-1})\mapsto (f_1:\hdots:f_s)(x_1,\hdots,x_{n-1}).
\end{equation}
If $\phi$ is generically finite, then $\im(\phi)$ is a hypersurface in $\pnnk$, and the implicitization problem consists in computing the equation that spans the principal ideal $\ker(h)$.

Denote by $I=(f_1,\hdots,f_n)$, with $f_i$ of degree $d$. The grading on $A$ is the standard grading where $\deg(X_i)=1$. Finally, we write $\Z.$, $\B.$ and $\M.$ for the approximation complexes associated to $I$, defined in \ref{CA}.

\begin{note}
The aim of this section is to show that in the implicitization context we consider, the complex $\Z.$ is acyclic and gives a resolution for $\SIA$. We will see that splitting this complex in its homogeneous parts we can get the implicit equation by taking determinant of an appropriate strand \cite[Appendix A]{GKZ94}.
\end{note}

The relation between this section and  the sections above is given by the following result:

\begin{thm}[{\cite[Prop.\ 4.2]{Buse1}}]\label{ann=kerh}
If $H^0_\mm(A)=0$, then,  
\[
 \ann_{\kk[\ts]}(\RIA_\nu)=\ker(h),\ \mbox{for all }\nu\in \NN.
\]

\end{thm}
Remark that this always happens when $A=\kk[\Xss]$. We get the following result that relates  $\ann_{A}(\RIA_\eta)$ with the local cohomology module $H^0_\mm(\RIA)$.

\begin{lem}[{\cite[Prop.\ 1.2]{Buse2}}]\label{lema aux sobre anuladores} For a ring $R$ and $B=R[\Xss]/I'$, such that $R\cap I'=0$, and let $\eta\in \NN$ be such that $H^0_\mm(B)_\eta=0$. Then 
\[
 \ann_{R}(B_\eta)=\ann_{R}(B_{\eta+\nu})=H^0_\mm(B)_0,\ \mbox{for all }\nu\in \NN.
\]
\end{lem}

In order to get a generator for $\ker(h)$, is necessary to compute a resolution. In spite of this good property of $\RIA$, there are no universal resolutions for $\RIA$. This is one of the key points in our approach. Hence, we will approximate $\RIA$ by $\SIA$, which, as we have seen, in several cases it is a good approximation. Henceforward, we will give conditions in order to compute $\ker(h)$ from $\SIA$.

Recall we have a $\ZZ^2$-grading on $A[\ts]$, which transfers to a $\ZZ^2$-grading in $\SIA$ via the presentation: 
\[
 0\to J' \to A[\ts] \stackrel{\alpha}{\to} \SIA \to 0,
\]
where $J'=\{\sum g_iT_i :\sum g_if_i=0, g_i\in A[\ts]\}$, as has been proven in Section \ref{RIASIA}.

Denote by $\SIA_\nu$ the $\nu$-graded strand of $\SIA$, corresponding to the grading on $A$. Precisely, $\SIA_\nu=\bigoplus_{t\geq 0}A_\nu \sym \ ^t_A(I)$, where $\sym\ ^t_A(I)$ denotes the $t$-graded strand with respect to the grading on the $T_i$'s.  

\begin{prop}[{\cite[Prop.\ 5.1]{BuJo03}}]\label{impl 1} Assume $I$ is of linear type off $V(\mm)$, and set $\eta\in\ZZ$ such that $H^0_\mm(\SIA)_\nu=0$ for all $\nu\geq \eta$. Then 
\[
 \ann_{\kk[\ts]}(\SIA_\nu)=\ker(h),\ \mbox{for all }\nu\geq \eta.
\]
\end{prop}

we conclude the following result:

\begin{cor}
If $H^0_\mm(\SIA_\nu)=0$ then, 
\[
 \ann_{\kk[\ts]}(\SIA_\nu)\subset\ker(h),\ \mbox{for all }\nu\geq \eta.
\]
\end{cor}

We will assume that the map $\phi:\proj (A) \dashrightarrow \pnk$ is generically finite, hence, $\phi$ defines a hypersurface in $\pnk$ and thus, $\ker(h)$ is principal. Denote by $H$ the irreducible implicit equation which defines the closure of $\im(\phi)$.

First, we will assume that $V(I)=V(\mm)$ in $\spec(A)$, namely $\phi$ will have empty base locus. If $V(I)=\emptyset$ in $\proj (A)$, from Proposition \ref{acicl sin pb} we have that the complex $\Z.$ is acyclic since $\depth(\mm:A)=n-1$. Hence, it provides a resolution for $\SIA$. Thus, we can compute $\ker(h)$ as the MacRae invariant $\rae(\SIA_\nu)$ which coincides with the determinant of $(\Z.)_\nu$, for $\nu\geq \eta$. 

\begin{thm}[{\cite[Thm.\ 5.2]{BuJo03}}]\label{implicit sin pb}
Assume that $\rad(I)=\rad(\mm)$. Let $\eta\in \ZZ$ is such that $H^0_\mm(\SIA)_\nu=0$ for all $\nu\geq \eta$. Then, the homogeneous strand of degree $\nu$ of the complex 
\[
 0\to (\ZZZ_{n-1})_\nu\to (\ZZZ_{n-2})_\nu \to \hdots \to (\ZZZ_1)_\nu \to A_\nu[\ts]
\]
is $H^{\deg(\phi)}$, of degree $d^{n-2}$.
\end{thm}

We deduce from Theorem \ref{implicit sin pb} that:

\begin{prop}
Under the hypothesis of Theorem \ref{implicit sin pb}, $H^{\deg(\phi)}$ can be computed as the gcd of the maximal minors of the map of $\kk[\ts]$-modules 
\[
 (\ZZZ_1)_\nu \stackrel{d_T}{\to} A_\nu[\ts],\ \mbox{for all }\nu \geq \eta.
\]
\end{prop}

We can give an specific bound for $\eta$. Recall that in the case of ``moving curves"{} and ``moving surfaces"{} the sizes of matrices could be computed \textsl{a priori} and were related to the \textsl{the regularity} of the ring. In the same way, $\eta$ depends on intrinsic characteristic $I$.

\begin{prop}[{\cite[Prop.\ 5.5]{BuJo03}}]\label{eta sin pb}
Let $n\geq 3$ and assume that $\rad(I)=\mm$. Then, $H^0_{\mm}(\SIA)_\nu=0$ for all $\nu\geq (n-2)(d-1)$.
\end{prop}

\medskip

We will now overview the case where $\phi$ admits ``good'' base points. It is no know how to trear this case in great generality, hence, we will assume that the base locus $V(I)=V(\fs)$, is a locally complete intersection  (LCI) in $\proj (A)$ of codimension $n-2$. Thus, we have that $V(I)\subset \PP^{n-2}_\kk$ is locally given by a regular sequence and $\depth(I:A)=n-2<\depth(\mm:A)=n-1$. From Proposition \ref{acicl con pb} we get that the complex $\Z.$ is acyclic. We conclude that $\Z.$ is a resolution of $\SIA$. We have the following result on implicitization:

\begin{thm}[{\cite[Thm.\ 5.7]{BuJo03}}]\label{implicit con pb}
Let $I=(\fs)$ be a LCI in $\proj (A)$ of codimension $n-2$, and $\phi$ is generically finite. Let $\eta\in \ZZ$ be such that $H^0_\mm(\SIA)\nu=0$ for all $\nu\geq \eta$. Then, the determinant of the strand of degree $\nu$ of the complex 
\[
 0\to (\ZZZ_{n-1})_\nu\to (\ZZZ_{n-2})_\nu \to \hdots \to (\ZZZ_1)_\nu \to A_\nu[\ts]
\]
is $H^{\deg(\phi)}$, of degree $d^{n-2}-\dim_\kk\Gamma(\proj (A)/I, \OO_{\proj (A)/I})$.
\end{thm}

We obtain that:

\begin{prop}
Under the hypothesis of Theorem \ref{implicit con pb}, $H^{\deg(\phi)}$ can be computed as the gcd of the maximal minors of the map of $\kk[\ts]$-modules 
\[
 (\ZZZ_1)_\nu \stackrel{d_T}{\to} A_\nu[\ts],\ \mbox{for all }\nu \geq \eta.
\]
\end{prop}

Similar to Proposition \ref{eta sin pb}, it is possible to give a bound for $\eta$ as is shown in the next result:

\begin{prop}[{\cite[Prop.\ 5.10]{BuJo03}}]
Let $n\geq 3$ and assume $I=(\fs)$ is a LCI in $\proj (A)$ of codimension $n-2$. Then, $H^0_{\mm}(\SIA)_\nu=0$ for all $\nu\geq (n-2)(d-1)$.
\end{prop}

Next, we present several results that extent the previous work, and that precede the work in this thesis. For an ideal $I$ of a $\ZZ$-graded $\kk$-algebra $A$, we denote 
\[
 \epsilon_I:=\indeg(I)=\inf\{\nu\in\ZZ : I_\nu\neq 0\}.
\]

\begin{thm}\label{ThmFinalIntro}
Let $I=(\fs)$ be an ideal of $A$ of codimension $n-2$ in $\proj (A)$. Let $\eta:=(n-1)(d-1)-\epsilon_I$.
\begin{enumerate}
\item The following statements are equivalent:
	\begin{enumerate}
	\item $V(I)$ is locally defined by at most $n-1$ equations;
	\item $\Z.$ is acyclic;
	\item $(\Z.)_\nu$ is acyclic for $\nu\gg 0$.
	\end{enumerate}
\item If $\Z.$ is acyclic, then: 
\[
 \det((\Z.)_\nu)=\rae(\SIA_\nu)=H^{(\deg(\phi))}G,\ \mbox{for all }\nu\geq \eta.
\]
where $G\neq 0$ is a constant polynomial iff $V(I)$ is LCI in $\proj(A)$.
\item Moreover, following statements are equivalent:
	\begin{enumerate}
	\item $V(I)$ is locally of linear type;
	\item $V(I)$ is locally a complete intersection;
	\item $\proj(\SIA)=\proj(\RIA)$;
	\item $G=1$, that is, $\det((\Z.)_\nu)=\rae(\SIA_\nu)=H^{(\deg(\phi))}$ for all $\nu\geq \eta$.
	\end{enumerate}
\end{enumerate}
\end{thm}

\begin{note}
Recall that we have that $\alpha:A[\ts]\to\SIA$ is surjective, as $A=\kk[\Xss]$, we have that there exists an injective map $\proj(\SIA)\hookrightarrow \pnnk\times \pnk$. With the notation of Theorem \ref{ThmFinalIntro} we have that
\[
 \rae(\SIA_\nu)=(\pi_2)_*(\proj(\SIA))\cong \kk[\ts](-d^{n-2}+\sum_{x\in V(I)}d_x),
\]
for all $\nu\geq (n-1)(d-1)-\epsilon_I$. This says that $\deg(G)$ is a sum of number that measure how far is $V(I)$ from being LCI. Precisely, 
\[
 \deg(G)=\sum_{x\in V(I)} (e_x-d_x),
\]
where $e_x:=e(J_x, R_x)$ is the multiplicity in $x$ and $d_x:=\dim_{A_x/x\cdot A_x}(A_x/I_x)$.
\end{note}

\chapter[Preliminaries on toric varieties]{Preliminaries on toric varieties}
\label{ch:toric-varieties}

All along this chapter we will follow \cite{Fu93} and \cite{CoxTV}. We assume
that the reader is familiar with the definition of (normal) toric varieties
in terms of a rational polyhedral fan.

As usual, $\NNbm$ and $\MMbm$ denote dual lattices of rank $n-1$, which
correspond respectively to the one parameter subgroups and characters of
the associated torus $T = T_{\NNbm} = \Spec(\kk[\MMbm])$. Here $\kk$ denotes a
fixed field.
We denote  by $\gen{-,-}: \MMbm\times\NNbm \to \ZZ$  the natural pairing.

\section{Divisors on toric varieties}

A divisor on a toric variety which is invariant under the action of the torus
admits an explicit characterization in terms of lattice objects. The aim of the
present section is to summarize such powerful description. 

Let $\Delta$ be a rational polyhedral fan in the lattice $\NNbm\cong \ZZ^{n-1}$
and let $\Tc_\Delta$ be the corresponding toric variety with torus
$T=\Spec(\kk[\MMbm])$. 

If we denote by $\Delta(1)$ the set of rays of the fan,
then each orbit $\OO_\rho$ (of the action of $T$ on $\Tc_\Delta$) corresponding to a ray $\rho$ in $\Delta(1)$ is a
torus of dimension $n-2$.  The orbit closure $D_\rho=\overline{\OO_\rho}$ has
then the same dimension $n-2$. It follows that to each ray $\rho$ corresponds an
irreducible subvariety of $\Tc_\Delta$ of codimension $1$, i.e.\ a prime divisor
on $\Tc_\Delta$.
\begin{defn}
 A Weil divisor $D =\sum a_i D_i$ on the toric variety $\Tc_\Delta$ is
said to be $T$-invariant if every prime divisor $D_i$ is invariant under the action
of the torus $T$ on $\Tc_\Delta$.
\end{defn}
\begin{prop}
The $T$-invariant Weil divisors are exactly the divisors of the form $\sum_{\rho\in\Delta(1)} a_\rho D_\rho$, $a_i\in \kk$.
\end{prop}

We turn now our attention to $T$-invariant Cartier divisors.

\begin{defn}
 A Cartier divisor $D$ on a toric variety $\Tc_\Delta$ is said to be
$T$-invariant if it corresponds to a $T$-invariant Weil divisor.
\end{defn}

We start by giving a description of the Cartier divisor corresponding to a
character of the torus $T$. Since a character $\chi^u$ defines a non-zero rational function
on the toric variety $\Tc_\Delta$, then $\set{(\Tc_\Delta, \chi^u )}$ is a
Cartier divisor which we denote
by $\div(\chi^u )$. For each ray $\rho \in \Delta(1)$, denote by $n_\rho$ the corresponding minimal
generator (i.e.\ the first lattice point along the ray, starting from the vertex).
The proof of the next three statements can be found in \cite[page 61]{Fu93}.

\begin{lem}
 Let $\Tc_\Delta$ be a toric variety.
Let $u$ be an element of $\MMbm$ and $\chi^u$ its corresponding character, then
$ \ord_{D_{\rho}}(\chi^u) = \gen{u, n_\rho}$ for every $\rho \in \Delta(1)$.
\end{lem}

We deduce that the Weil divisor associated to the principal Cartier divisor
$\set{(\Tc_\Delta, \chi^u )}$ is
$\sum_{\rho \in \Delta(1)}\gen{u, n_\rho} D\rho$.

For affine toric varieties, a very strong result holds.

\begin{thm}\label{thm-cone-div}
 Let $U_\sigma$ be the affine toric variety of a cone $\sigma$ in $\ZZ^{n-1}$,
then every $T$-invariant Cartier divisor on $U_\sigma$ is of the form
${(U_\sigma , \chi^u )}$ for some
character $\chi^u$ of the torus $T$. In particular, every $T$-invariant Cartier
divisor on $U_\sigma$ is principal.
\end{thm}

Theorem \ref{thm-cone-div} can be used to describe a $T$-invariant Cartier
divisor $D$ on a general toric variety $\Tc_\Delta$. Indeed, consider the open
cover of
$\Tc_\Delta$ given by the affine toric varieties $U_\sigma$, as $\sigma$ varies
in $\Delta$. By the
above theorem, for each $\sigma$ we can find an element $u(\sigma)$ such that the
local equation of $D$ on $U_\sigma$ is $\chi^{-u(\sigma)}$, so that $D = \set{(U_\sigma , \chi^{-u(\sigma)} )}_{\sigma\in\Delta}$
is the description of the $T$-invariant Cartier divisor $D$.

We can make use of Theorem \ref{thm-cone-div} to determine when two $T$-invariant Cartier divisors are the same. Since the group of Cartier
divisors is embedded in the group of Weil divisors, two Cartier divisors are
identical if and only if their associated Weil divisors are so. In particular
two $T$-invariant Cartier divisors 
$D = \set{(U_\sigma , \chi^u )}$ and $D' = \set{(U_\sigma , \chi^{u'})}$ (for
$u$ and $u'$ in $\MMbm$) on an affine toric variety $U_\sigma$ are identical if
and only if
$[D] = \sum_{\rho \in\Delta(1)}\gen{u,n_\rho} D_\rho$ and $[D'] = \sum_{\rho \in\Delta(1)} \gen{u',n_\rho} D_\rho$ are identical. This
happens if and only if $\sum_{\rho \in\Delta(1)}\gen{u-u',n_\rho} D_\rho=0$. 
This last statement is equivalent to saying that $u-u'$ lies in
$\sigma^\perp\cap \MMbm$, which
is a sublattice of $\MMbm$. Therefore we have the following:

\begin{prop}\label{prop-caract-T-inv-Cart-div}
 There is a bijection between the set of $T$-invariant Cartier
divisors on an affine toric variety $U_\sigma$ and the quotient lattice $\MMbm/\sigma^\perp\cap \MMbm$.
\end{prop}

\section{Ample sheaves and support functions}

In this section we give a characterization of the sheaf associated to a $T$-
invariant divisor. This allows us to state two criteria for such a sheaf to be
ample or very ample.

Recall that the support $\supp(\Delta)$ of a fan $\Delta$ is
defined to be the union of all its cones.

\begin{defn}
 A function $\psi : \supp(\Delta) \to \RR$ is said to be a $\Delta$-linear
support function if it is linear on each cone $\sigma$ of $\Delta$, that is, on each cone it
is determined by a linear function, and assumes integer values at lattice
vectors, i.e.\ $\psi(\supp(\Delta) \cap \NNbm ) \subset \ZZ$.
If there is no possibility of confusion, we call $\psi$ just a support function. A $\Delta$-linear support function $\psi$ is said to be strictly convex
if it is convex and the linear functions determined by different cones are
different.
\end{defn}

Let now $\Delta$ be a rational polyhedral fan and $\Tc_\Delta$ the associated toric variety.
Combining Theorem \ref{thm-cone-div} and
Proposition \ref{prop-caract-T-inv-Cart-div}, we see that a Cartier divisor is specified by  $\set{u(\sigma) \in \MMbm/\sigma^\perp\cap \MMbm}_{\sigma\in\Delta}$.

\begin{prop} \label{prop-C}
 There is a bijective correspondence between $T$-invariant
Cartier divisors on a toric variety $\Tc_\Delta$ and $\Delta$-linear support
functions.
\end{prop}

We also have the following general result that will be used.

\begin{lem}
 Let $\Tc_\Delta$ be a toric variety and $T = \Spec(\kk[\MMbm])$ its torus. Let
$D$ be a $T$-invariant Cartier divisor and $\OO(D)$ its associated sheaf. If we
denote by $\OO$ the structure sheaf of $\Tc_\Delta$, then we have $\Gamma(T,
\OO(D)) = \Gamma(T,\OO)$.
\end{lem}

Note that $\kk[\MMbm]$ can be expressed as a direct sum $\kk[\MMbm]=\bigoplus_{u\in \MMbm}\kk\chi^u$,
so the previous lemma says that $\Gamma(T, \OO(D))=\bigoplus_{u\in \MMbm}\kk\chi^u$.
 
Assume the fan is complete, that is $\supp(\Delta) = \NNbm_\RR$. Using the
description in Proposition~\ref{prop-C}, we can see that any  $T$-invariant Cartier divisor $D$ defines a
polytope $\Nc_D$. Let $\psi_D$ be the support function defined by $D$, then, identifying vectors $u$ of
$\MMbm_\RR$ with linear functions from $\NNbm_\RR$ to $\RR$, we define $\Nc_D$ to be
\begin{equation}\label{eq-def-NcD}
 \Nc_D = \set{u \in \MMbm_\RR : u \geq \psi_D \textnormal{ on } \supp(\Delta)}.
\end{equation}
Now, identifying $D$ with its corresponding Weil divisor $[D] =  \sum a_\rho D_{\rho}$, we can rewrite \eqref{eq-def-NcD} as
\begin{equation}\label{eq-def-NcD-2}
 \Nc_D = \set{u \in \MMbm_\RR : \gen{u,n_\rho} \geq -a_\rho\ \forall \rho\in\Delta(1)}    
\end{equation}
A priori, \eqref{eq-def-NcD-2} only says that $\Nc_D$ is a polyhedron (an intersection of closed
half spaces), but it is shown in \cite[pp.\ 67]{Fu93}, that $\Nc_D$ is in fact bounded
and therefore a polytope under our assumption that $\supp(\Delta) = \NNbm_\RR$.

Reciprocally, let $\Nc$ be a full dimensional lattice polytope in $\MMbm_\RR$,
and let $\Delta(\Nc)$ be its normal fan. Two vectors $v$ and $v'$ belong to the interior of the same cone $\Delta(\Nc)$ if and only if the linear functions $\gen{v,-}$ and $\gen{v',-}$ attain their minimum over $\Nc$ at the same face of $\Nc$. The cones in this
fan are in bijection with the domains of linearity of the associated support
function (see \ref{eq-supp-fuct}), which is strictly convex. Let $D = D_\psi$ be the $T$-invariant
Cartier divisor  corresponding to
a support function $\psi$ on the associated toric variety
$\Tc_{\Delta(\Nc)}$, and let $\OO(D)$ be its associated sheaf.

\begin{thm}
 With notation as above we have
\[
 \Gamma(\Tc_\Delta, \OO(D))=\bigoplus_{u\in \Nc_D\cap \MMbm}\kk\chi^u
\]
where $\Nc_D$ is the polytope of \eqref{eq-def-NcD-2}.
\end{thm}

Let $\psi$ be a $\Delta$-linear support function, with $\Delta(\Nc)$ the fan of
a polytope $\Nc$ in $\MMbm_\RR$, and let $u(\sigma) \in \MMbm_\RR$ such that
$\psi(v) =  \gen{u(\sigma),v}$ for any $v$ in $\sigma$. In this case it is
straightforward to check that $\psi$ is
convex if and only if for every maximal cone $\sigma$ of $\Delta(\Nc)$ and $v$
in $\supp(\Delta(\Nc))$ we
have $\gen{u(\sigma), v} \geq \psi(v)$.
Theorems \ref{thm-O(D)ample-strconvex} and \ref{thm-genbysection-convex} give a very explicit criterion in terms of the support
function $\psi_D$ to determine when $\OO(D)$ is ample or very ample.

\begin{thm}\label{thm-genbysection-convex}
 Let $\Tc_{\Delta(\Nc)}$ be the toric variety of a polytope $\Nc$.
Let $D$ be the divisor associated to a support function $\psi$, then $\OO(D)$ is
generated by its sections if and only if $\psi$ is convex.
\end{thm}

Denote by $\PP^N$ the $N$-dimensional projective space over $\kk$. Let $D$ be a $T$-invariant Cartier
divisor on a toric variety $\Tc_\Delta$ such that $\OO(D)$ is generated by its
sections.
Choosing and ordering a basis $\set{\chi^{u_i} : u_i \in \Nc_D \cap \MMbm }$ gives a morphism
\begin{equation}
f_D : \Tc_\Delta \to \PP^N: x \to (\chi^{u_0}(x),\hdots,\chi^{u_N}(x))
\end{equation}
where $N + 1 = \#(\Nc_D \cap \MMbm)$. Such a mapping is a closed embedding if
and only
if the sheaf $\OO(D)$ is very ample. As in the previous theorem, we can give a
characterization of this condition in terms of the support function $\psi$ of $D$.

\begin{thm}\label{thm-O(D)ample-strconvex}. Let $D$ be a $T$-invariant Cartier divisor on a toric variety
$\Tc_\Delta$, then $\OO(D)$ is ample if and only if $\psi_D$ is strictly convex.
Moreover, $\OO(D)$ is very ample if and only if $\psi_D$ is strictly convex and
for every
maximal cone $\sigma$ of $\Delta$, the lattice points of the dual cone $\sigma^\vee\cap \MMbm$ are generated by $\set{u - u(\sigma) : u \in \Nc_D \cap \MMbm }$.
\end{thm}

We show now that every toric variety arising from a polytope is projective. This fact makes it possible to compare the two different constructions of a toric variety we have studied, and show that they are indeed equivalent.

Let $\Nc$ be a full dimensional lattice polytope in $\MMbm_\RR$ and
$\Delta(\Nc)$ its normal fan in $\NNbm_\RR$. Recall that, the support of
$\Delta(\Nc)$ is such that $\supp(\Delta(\Nc))=\NNbm_\RR$. We define a function
$\psi_\Nc:\supp(\Delta(\Nc)) \to \RR$ as
\begin{equation}\label{eq-supp-fuct}
 \psi_\Nc (v) = \inf\set{ \gen{u,v} : u \in \Nc}.
\end{equation}
We call this function the support function of $\Nc$. This name makes
sense since the support function $\psi_\Nc$ of a lattice polytope $\Nc$ is a
$\Delta(\Nc)$-linear support function. Moreover, the support function $\psi_\Nc$ of a lattice polytope $\Nc$ is strictly convex. Indeed, convexity follows from the definition, since
$\inf\set{a + b : a \in A, b \in B} = \inf\set{a : a \in A} + \inf\set{b : b \in B}$
for arbitrary sets $A$ and $B$ of real numbers.

\begin{prop}\label{prop-toric-is-projective}
 The toric variety of a polytope is projective.
\end{prop}
\begin{proof}
 Let $\Nc$ be a polytope in $M_\RR$ and $\Tc=\Tc_\Nc$ the associated toric variety.
The previous remark shows that the support function $\psi_\Nc$ is strictly convex.
Then, by Theorem \ref{thm-O(D)ample-strconvex}, $\psi_\Nc$ determines a divisor $D$ on $\Tc$ whose associated
sheaf $\OO(D)$ is ample. By \cite[Sec.\ II]{Hart}, there exists an integer such that the
sheaf $\OO(D)^{\otimes m}$ on $\Tc$ is very ample. Since $\Tc$ is complete, in particular it
is proper, so $\Tc$ is a proper algebraic variety admitting a very ample sheaf.
This shows that $\Tc$ is projective.
\end{proof}

\section{Projective toric varieties from a polytope}

 In this section we review the construction of a projective toric variety
associated to a lattice polytope $\Nc \subset \MMbm_\RR\cong \RR^{n-1}$ (see also \cite{GKZ94}).

Let $\aa = \Nc\cap \MMbm$ be the set of lattice points of $\Nc$.
Let $\kk$ be a field and $\PP^N$ the projective $N$-space over $\kk$, where $N +
1$ is the cardinality of $\aa$. Write $\aa = \set{\alpha_0,\hdots,\alpha_N }$,
where $\alpha_i = (\alpha_{i,1},\hdots,\alpha_{i,n-1})$ for $i=0,\hdots,N$.
We have a map
\begin{equation}
 \rho_\aa:(\kk^\ast)^{n-1}  \hookrightarrow \PP^N, 
\end{equation}
defined by
$ \rho_\aa(t_1, \hdots, t_{n-1}) = (t_1^{\alpha_{0,1}}\cdots
t_{n-1}^{\alpha_{0,n-1}}:\hdots:t_1^{\alpha_{N,1}}\cdots
t_{n-1}^{\alpha_{N,n-1}})$.

For simplicity, we set $t = (t_1,\hdots,t_{n-1})$ and
$t^{\alpha_i}=t_1^{\alpha_{i,1}}\cdots t_{n-1}^{\alpha_{i,n-1}}$, hence
\[
 \rho_\aa:(\kk^\ast)^{n-1}  \hookrightarrow \PP^N: t\mapsto
(t^{\alpha_0}:\hdots:t^{\alpha_N}).
\]
The Zariski closure of the image of $\rho_\aa$ in $\PP^N$ is called the
projective toric variety $\Tc_\Nc$ associated to $\Nc$, and we will write $\Tc$
instead of $\Tc_\Nc$ when $\Nc$ is understood:
\begin{equation}
 \Tc_\Nc := \overline{\im(\rho_\aa)}.
\end{equation}

A general affine variety $V = \Spec(R)$ is said to be normal if it is
irreducible and its local rings $\OO_{V,p}$ at each $p$ of $V$ are integrally
closed (cf.\ \cite[Prop.\ 3.0.11]{CoxTV}). This last condition is equivalent to the $\kk$-algebra $R$ being
integrally closed. In particular, the affine toric variety
$U_\sigma = \Spec(\kk[\sigma^\vee\cap \MMbm])$ associated to a rational polyhedral cone $\sigma$ in $\NNbm_\RR$ is always
irreducible. Moreover, $U_\sigma$ is normal because the corresponding monoid algebra
$\kk[\sigma^\vee\cap \MMbm]$ is an integrally closed ring.

We will give some important results about the normality. 

\begin{defn}A full dimensional lattice polytope $\Nc \subset \MMbm_\RR$ is very ample if for
every vertex
$m \in \Nc$, the semigroup generated by the set $\Nc \cap \MMbm -m= \set{m'-m :
m' \in \Nc \cap \MMbm}$ is saturated in $\MMbm$.
\end{defn}

\begin{thm}[{\cite[Thm.\ 2.2.11, Prop.\ 2.2.17 and Cor.\ 2.2.18]{CoxTV}}]
Let $\Nc \subseteq \MMbm_\RR$ be a full dimensional lattice polytope of
dimension $n \geq 3$, then $k\cdot \Nc$ is normal for all $k \geq n-2$.
Moreover, a normal lattice polytope $\Nc$ is very ample. Hence, if $\dim(\Nc) \geq 2$, then $k\cdot \Nc$ is very ample for all $k \geq
n-2$. And if $\dim(\Nc) = 2$, then $\Nc$ is very ample.
\end{thm}

Thus, we have that every full dimensional lattice polygon $\Nc \subseteq \RR^2$
is normal. 

Having established more than one definition of toric varieties, it makes sense to compare both of them.

\begin{thm}[{\cite[Prop.\ 3.1.6]{CoxTV}}]
Let $\Nc \subseteq \MMbm_\RR$ be a full dimensional lattice polytope. Let $k\in \ZZ$ be such that $k\cdot \Nc$ is very ample. Then $\Tc_{k\cdot \Nc} \cong \Tc_{\Delta(\Nc)}$, where $\Delta(\Nc)$ is the normal fan of $\Nc$.
\end{thm}

We will give yet another approach to toric varieties in the following section.

\section{The Cox ring of a toric variety}\label{CoxRing}

Our main motivation in Chapter \ref{ch:CastelMum} for considering regularity in general $\G$-gradings comes from toric geometry. Among $\G$-graded rings, homogeneous coordinate rings of a toric varieties are of particular interest in geometry. When $\Tc$ is a toric variety, $\G:=\DT$ is the (torus-invariant) divisor class group of $\Tc$, also called the Chow group of $\Tc$. In this case, the grading can be related geometrically with the action of this group on the toric variety, and hence, the graded structure on the ring can be interpreted in terms of global sections of the structural sheaf of $\Tc$ and in terms of character and valuations.

Henceforward, let $\Delta$ be a non-degenerate fan in the lattice $\NNbm\cong \ZZ^{n-1}$, and let $\Tc$ be a toric variety associated to $\Delta$. Write $\Delta(i)$ for the set of $i$-dimensional cones in $\Delta$. As we recalled, there is a bijection between the set $\Delta(i)$ and the set of closed torus-invariant $i$-dimensional subvarieties of $\Tc$. In particular, each $\rho\in \Delta(1)$ corresponds to the torus-invariant Weil divisor $D_\rho\in \ZZ^{\Delta(1)}\cong \ZZ^{n-1}$.

Suppose that $\rho_1,\hdots,\rho_{s}\in \Delta(1)$ are one-dimensional cones of $\Delta$ and assume $\Delta(1)$ spans $\RR^{n-1}$. As before, $n_{\rho_i}$ denotes the primitive generator of $\rho_i$. There is a map $\MMbm\nto{\bfrho}{\to} \ZZ^{\Delta(1)}: m\mapsto \sum_{i=1}^s\gen{m,n_{\rho_i}}D_{\rho_i}$. We will identify $[D_{\rho_i}]$ with a variable $x_i$.

The torus-invariant divisor classes correspond to the elements of the cokernel $\DT$ of this map $\bfrho$, getting an exact sequence
\[
 0\to \ZZ^{n-1}\cong \MMbm \nto{\bfrho}{\lto} \ZZ^{s} \nto{\pi}{\lto} \DT \to 0.
\]

Set $S:=k[x_1\hdots,x_{s}]$. From the sequence above we introduce in $S$ a $\DT$-grading, which is coarser than the standard $\ZZ^{n-1}$-grading.

\medskip

To any non-degenerate toric variety $\Tc$, we associate an homogeneous coordinate ring, called the \textsl{Cox ring of $\Tc$} (cf.\ \cite{Cox95}). D.\ A.\ Cox defines (loc.\ cit.) the homogeneous coordinate ring of $\Tc$ to be the polynomial ring $S$ together with the given $\DT$-grading. We next discuss briefly this grading. A monomial $\prod x_i^{a_i}$ determines a divisor $D=\sum_i a_iD_{\rho_i}$ which will be denoted by $\x^D$. For a monomial $\x^D\in S$ we define its degree as $\deg(\x^D)=[D]\in \DT$.

Cox remarks loc.\ cit.\ that the set $\Delta(1)$ is enough for defining the graded structure of $S$, but the ring $S$ and its graded structure does not suffice for reconstructing the fan. In order to not to lose the fan information, we consider the irrelevant ideal
\[
 B:= \spann{\prod_{n_{\rho_i}\notin \sigma} x_i : \sigma \in \Delta},
\]
where the product is taken over all the $n_{\rho_i}$ such that the ray $\RR_{\geq 0}n_{\rho_i}$ is not contained as an edge in any cone $\sigma\in \Delta$. Finally, the Cox ring of $\Tc$ will be the $\DT$-graded polynomial ring $S$, with the irrelevant ideal $B$.

\medskip

Given a $\DT$-graded $S$-module $P$, Cox constructs a quasi-coherent sheaf $P^\sim$ on $\Tc$ by localizing just as in the case of projective space, and he shows that finitely generated modules give rise to coherent sheaves. It was shown by Cox (cf.\ \cite{Cox95}) for simplicial toric varieties, and by Mustata in general (cf.\ \cite{Mus02}), that every coherent $\OO_{\Tc}$-module may be written as $P^\sim$, for a finitely generated $\DT$-graded $S$-module $P$.

For any $\DT$-graded $S$-module $P$ and any $\delta\in \DT$ we may define $P(\delta)$ to be the graded module with components $P(\delta)_\epsilon = P_{\delta+\epsilon}$ and we set
\[
 H^i_{\ast}(\Tc,P^\sim):=\bigoplus_{\delta\in \DT}H^i(\Tc,P(\delta)^\sim).
\]
We have $H^0(\Tc,\OO_\Tc(\delta)) = S_\delta$, the homogeneous piece of $S$ of degree $\delta$, for each $\delta \in \DT$. In fact each $H^i_\ast(\Tc,\OO_\Tc)$ is a $\ZZ^{n-1}$-graded $S$-module. We can compute (cf.\ \cite[Prop. 1.3]{Mus02}), for $i>0$,
\begin{equation}\label{equLC-SCforS}
 H^i_{\ast}(\Tc,P^\sim)\cong H^{i+1}_B(P):= \colimit{j} \ext^i_S(S/B^j,S).
\end{equation}
and an exact sequence $0\lto H^0_B(P)\lto P \lto H^0_{\ast}(\Tc,P^\sim)\lto H^{1}_B(P)\lto 0$.

We will use these results in Chapter \ref{ch:toric-pn}, applied to the
computation of implicit equations of images of rational maps of toric
hypersurfaces.

\chapter[Implicit equations of Toric hypersurfaces in projective space by means of an embedding]{Implicit equations of Toric hypersurfaces in projective space by means of an embedding}
\label{ch:toric-emb-pn}

\section{Introduction.}\label{sec:intro}
In this chapter we extend the method of computing an implicit equation of a parametrized hypersurface in $\PP^n$ focusing on different compactifications of the domain $\Tc$, following the ideas of \cite{Bot09}. Hereafter in this chapter we will always assume that $\Tc$ is embedded in $\PP^N$, and its coordinate ring $A$ is $n$-dimensional, graded and Cohen-Macaulay.

\medskip

In Section \ref{sec2setting} we give a fast overview on the general implicitization setting in codimension one, following the spirit of many papers in this subject revised in Chapter \ref{ch:elimination} /cf.\ also \cite{BuJo03}, \cite{BCJ06}, \cite{BD07}), as well as in \cite{BDD08} and \cite{Bot09}. We begin by considering the affine setting and we continue by considering the mentioned compactifications. We show in Section \ref{sec2setting} one important application which motivated our study: $\Tc$ is the toric compactification defined from the Newton polytope of the polynomials defining the rational map.

\medskip

In Section \ref{sec3Pn} we focus on the implicitization problem for a rational map $\varphi:\Tc \dto \PP^n$ defined by $n+1$ polynomials of degree $d$. We extend the method for projective $2$-dimensional toric varieties developed in \cite{BDD08} to a map defined over an $(n-1)$-dimensional Cohen-Macaulay closed scheme $\Tc$ embedded in $\PP^N$. We show that we can relax the hypotheses of \cite{BDD08} on the base locus by admitting it to be a zero-dimensional almost locally complete intersection scheme.

Precisely, as we have seen in Chapter \ref{ch:elimination}, we associate a complex $(\Zc_{\bullet})_\bullet$ to the map $\varphi$. Recall from Chapter \ref{ch:elimination}, that the determinant $D$ of $\Zc_{\bullet}$ in degree $\nu$ can be computed either as an alternating sum of subdeterminants of the differentials in $\Zc_{\nu}$ or as the greatest common divisor of the maximal-size minors of the matrix $M_\nu:(\Zc_1)_{\nu} \rightarrow (\Zc_0)_{\nu}$ associated to the right-most map of $(\Zc_{\bullet})_{\nu}$. Theorem \ref{mainthT}, which can be considered the main result of this chapter, states that this gcd computes a power of the implicit equation (with some extraneous factor), in a good degree $\nu$.


\section{General setting} \label{sec2setting}\label{ImageCodim1}

Throughout this section we will give a general setting for the implicitization problem of hypersurfaces. Our aim is to analyze how far these techniques from homological commutative algebra (syzygies and graded resolutions) can be applied.

\medskip

Write $\AA^k:= \Spec (\kk[T_1,\hdots,T_k])$ for the $k$-dimensional affine space over $\kk$. Assume we are given a rational map
\begin{equation}\label{initsettingPn}
f: \AA^{n-1} \dto \AA^n : \s:=(s_1,\hdots,s_{n-1})  \mapsto \paren{\frac{f_1}{g_1},\hdots,\frac{f_n}{g_n}}(\s)
\end{equation}
where $\deg(f_i)=d_i$ and $\deg(g_i)=e_i$, and $f_i,\ g_i$ without common factors for all $i=1,\hdots n$. Observe that this setting is general enough to include all classical implicitization problems. Typically all $g_i$ are assumed to be equal and a few conditions on the degrees are needed, depending on the context. 

We consider a rational map $\varphi:\Tc \dto \PP^n$, where $\Tc$ is a suitable compactification of a suitable dense open subset of $\AA^{n-1}$, in such a way that the map $f$ extends from $\Tc$ to $\PP^n$ via $\varphi$ and that the closed image of $f$ can be recovered from the closed image of $\varphi$. 

Assume $\Tc$ can be embedded into some $\PP^N$, and set $A$ for the homogeneous coordinate ring of $\Tc$. Since $\AA^{n-1}$ is irreducible, so is $\Tc$, hence $A$ is a domain. Assume also that the closure of the image of $\varphi$ is a hypersurface in $\PP^n$, hence, $\ker(\varphi^*)$ is a principal ideal, generated by the implicit equation. 

Most of our results are stated for a general arithmetically Cohen-Macaulay scheme as domain. Nevertheless, the map \eqref{initsettingPn} gives rise, naturally, to a toric variety $\Tc$ on the domain (cf.\ \cite[Sect.\ 2]{KD06}, \cite{Co03b}, and \cite[Ch.\ 5 \& 6]{GKZ94}) associated to the following polytope $\Nc(f)$.

\begin{defn}\label{defNf}
Given a polynomial $h=\sum_{\alpha \in \ZZ^{n-1}}a_\alpha t^\alpha$ we define its Newton polytope, $\Nc(h)$, as the convex hull of the finite set $\set{\alpha : a_\alpha\neq 0}\subset \ZZ^{n-1}$. Now, let $f$ denote a map as in equation \eqref{initsettingPn}. We will write 
\[
 \Nc(f):=\conv\paren{\bigcup_{i=1}^n \paren{\Nc(f_i)\cup \Nc(g_i)}}
\]
the convex hull of the union of the Newton polytopes of all the polynomials defining the map $f$.
\end{defn}

There is a standard way of associating a semigroup $S_\Nc$ to a polytope $\Nc\subset \RR^{n-1}$. Indeed, take $\iota: \RR^{n-1}\hto \RR^n: x\mapsto (x,1)$, and define $S_\Nc$ as the semigroup generated by the lattice points in $\iota(\Nc)$. Due to a theorem of Hochster, if $S_\Nc$ is normal then the semigroup algebra $\kk[S_\Nc]$ is Cohen-Macaulay. Unluckily, it turns out that $S_\Nc$ is in general not always normal. A geometric or combinatorial characterization of the normality of $\kk[S_\Nc]$ is one of the most important open problem in combinatorial algebra (cf \cite{BGN97}). 

Note that $m\Nc\times \{m\}=\{(p_1+\cdots+p_m,m)\ :\ p_i\in\Nc\}\subset S_\Nc\cap (\ZZ^{n-1}\times \{m\})$ for any $m\in \NN$, but in general these two sets are not equal.
When this happens for all $m\in \NN$, we say that the polytope $\Nc$ is normal, equivalently $(m\cdot \Nc) \cap \ZZ^{n-1} = m \cdot (\Nc \cap  \ZZ^{n-1})$ for all $m \in \NN$, and in this case it follows that $\kk[S_\Nc]$ is Cohen-Macaulay. 

\begin{thm}
Let $\Nc \subseteq \MMbm_\RR$ be a full dimensional lattice polytope of
dimension $n-1 \geq 2$. Then $m\cdot \Nc$ is normal for all $m \geq n-2$.
\end{thm}
We refer the reader to  \cite[Thm.\ 2.2.11.]{CoxTV} for a proof. We deduce that every full dimensional lattice polygon $\Nc \subseteq \RR^2$
is normal.

In this chapter we focus on the study of toric varieties by fixing an embedding. Changing $\Nc$ by a multiple $l\cdot \Nc$ changes the embedding, hence, we will fix the polytope. Since we also need Cohen-Macaulayness of the quotient ring by the corresponding toric ideal in several results, we will assume throughout that $\Nc$ is normal.

\begin{rem}\label{NfAlwaysNormal}
Given a map $f$ as in Equation \eqref{initsettingPn}, we will always assume that $\Nc:=\Nc(f)$ is normal. Therefore, the coordinate ring $A$ of $\Tc=\Tc(\Nc)$ will be always Cohen-Macaulay, hence $\Tc\subset \PP^N$ will be \emph{arithmetically Cohen-Macaulay (aCM)}. This is automatic when $n-1=2$.
\end{rem}

As we recalled in Chapter \ref{ch:toric-varieties}, the polytope $\Nc(f)$ defines an $(n-1)$-dimensional projective toric variety $\Tc$ provided with an ample line bundle which defines an embedding: if $N=\#(\Nc(f) \cap \ZZ^{n-1})-1$ we have $\Tc \subseteq \PP^N$. Write $\rho$ for the embedding determined by this ample line bundle. We get that the map
\begin{equation}
 (\AA^*)^{n-1}  \stackrel{\rho}{\hookrightarrow} \PP^N : (\s) \mapsto (\ldots : \s^\alpha  : \ldots),
\end{equation}
where $\alpha \in \Nc(f) \cap \ZZ^{n-1}$, factorizes $f$ through a rational map $\varphi$ with domain $\Tc$, that is $f=\varphi \circ \rho$. We will show later in this chapter that by taking $\Nc'(f)$ as the smallest lattice contraction of $\Nc(f)$ (that is $\Nc'(f)$ is a lattice polytope such that $\Nc(f)=d\Nc'(f)$ and $d\in \ZZ$ is as big as possible) the computation becomes essentially better.

The main reason for considering projective toric varieties associated to the Newton polytope $\Nc(f)$ of $f$, is based on the following fact.

\begin{rem}
Assume $f$ is as in Equation \eqref{initsettingPn}, with $g_1=\cdots =g_n$. Write $f_0:=g_i$ for all $i$. Assume also that all $f_i$ are generic with Newton polytope $\Nc$, and hence write $\Nc:=\Nc(f_i)$ for all $i$. Set $N:=\#(\Nc \cap \ZZ^{n-1})-1$ and let $\Tc \subset \PP^{N}$ be the toric variety associated to $\Nc$. Write $\varphi: \Tc \dto \PP^n: \T \mapsto (h_0:\cdots:h_n)$ the map induced by $f$. Since the coefficients are generic, the vector of coefficients of $h_0,\hdots,h_n$ is not in $V(\Res_\Nc(h_0,\hdots,h_n))$; where $V(\Res_\Nc(h_0,\hdots,h_n))$ stands for the zero locus of the sparse resultant $\Res_\Nc(h_0,\hdots,h_n)$ associated to $h_0,\hdots,h_n$ and $\Nc$. Hence, they have no common root in $\Tc$. Thus, $\varphi$ has empty base locus in $\Tc$.

If we take instead another lattice polytope $\tilde{\Nc}$ strictly containing $\Nc$, the $f_i$ will not be generic relative to $\tilde{\Nc}$, and typically the associated map $\tilde{\varphi}$ will have a non-empty base locus in the toric variety $\tilde{\Tc}$ associated to $\tilde{\Nc}$.
\end{rem}


\section{The implicitization problem}\label{sec3Pn}

In this section we focus on the computation of the implicit equation of a hypersurface in $\PP^n$, parametrized by an $(n-1)$-dimensional arithmetically Cohen Macaulay (aCM) subscheme of some projective space $\PP^N$. We generalize what we have seen in Chapter \ref{ch:elimination} following the ideas of \cite{BDD08} and \cite{Bot09}, etc., and we give a more general result on the acyclicity of the approximation complex of cycles, by relaxing conditions on the base ring and on the base locus. 

\medskip
Henceforward in this section, let $\Tc$ be a $(n-1)$-dimensional projective aCM closed scheme over a field $\kk$, embedded in $\PP^N_{\kk}$, for some $N\in \NN$. Write $J$ the homogeneous defining ideal of $\Tc$ and $A=\kk[T_0,\hdots,T_N]/J$ for its CM coordinate ring. Set $\T:= T_0,\hdots,T_N$ the variables in $\PP^N$, and $\X$ the sequence $T_0,\hdots,T_n$ of variables in $\PP^n$.

We denote $\mm:=A_+=(\T)\subset A$, the maximal homogeneous irrelevant ideal of $A$.

Let $\varphi$ be a finite map defined over a relative open set $U$ in $\Tc$ defining a hypersurface in $\PP^n$, e.g.\ $U=\Omega$:
\begin{equation}\label{eqSettingPn}
 \PP^N \supset \Tc \nto{\varphi}{\dto} \PP^n : \T \mapsto (h_0:\cdots:h_n)(\T),
\end{equation}
where $h_0,\hdots,h_n$ are homogeneous elements of $A$ of degree $d$. Set $\h:=h_0,\hdots,h_n$. The map $\varphi$ gives rise to a morphism of graded $\kk$-algebras in the opposite sense
\begin{equation}
 \kk[T_0,\hdots,T_n] \nto{\varphi^\ast}{\lto} A: T_i \mapsto h_i(\T).
\end{equation}
Since $\ker(\varphi^\ast)$ is a principal ideal in $\kk[\X]$, write $H$ for a generator. We proceed as in Chapter \ref{ch:elimination}, \cite{BuJo03} or in \cite{BDD08} to get a matrix (representation matrix) such that the gcd of its maximal minors gives $H^{\deg(\varphi)}$, or possibly, a multiple of it.

\begin{defn}\label{defRepresentationMatrix}
Let $\Sc \subset \PP^n$ be a hypersurface. A matrix $M$ with entries in the polynomial ring $\kk[X_0,\ldots,X_n]$ is called a \textit{representation matrix} of $\Sc$ if it is generically of full rank and if the rank of $M$ evaluated in a point $p$ of $\PP^n$ drops if and only if the point $p$ lies on $\Sc$.
\end{defn}

\begin{rem}\label{RemToricCasePn}
Observe that if we start with an affine setting as in \eqref{initsettingPn}, $\Tc\subset \PP^N$ can be taken as the embedded toric variety associated to $\Nc'(f)$. In the classical implicitization problem it is common to suppose that $g_i=g_j$ for all $i$ and $j$, and $\deg(f_i)=\deg(g_i)=d$ for all $i$. Hence, write $f_0$ for any of the $g_i$.  This setting gives naturally rise to a homogeneous compactification of the codomain, defined by the embedding
\begin{equation}
\AA^n \nto{j}{\hto} \PP^n: \x \mapsto (1:\x).
\end{equation}
It is clear that for $f_0,\hdots,f_n$ taken as above, the map $f: \AA^{n-1}\dto \AA^n$ of equation \eqref{initsettingPn} compactifies via $\rho$ and $j$ to $\varphi:\Tc \dto \PP^n$. It is important to note that $\overline{\im(f)}$ can be obtained from $\overline{\im(\varphi)}$ and vice-versa, via the classical (first variable) dehomogenization and homogenization respectively. Finally, we want to give a matrix representation for a toric hypersurface of $\PP^n$ given as the image of the toric rational map $\varphi: \Tc \dto \PP^n : \T \mapsto (h_0:\cdots:h_n)(\T)$.
\end{rem}

Since $\varphi:\Tc \dto \PP^n$ is not, in principle, defined everywhere in $\Tc$, we set $\Omega$ for the open set of definition of $\varphi$. Precisely, we define

\begin{defn}\label{defXyOmegaPn}
 Let $\varphi:\Tc \dto \PP^n$ given by $\s\mapsto (h_0:\cdots:h_{n})(\s)$ The base locus of $\varphi$ is the closed subscheme of $\Tc$
\[
 X:=\Proj\paren{A/(h_0,\hdots,h_{n})}.
\]
We call $\Omega$ the complement of the base locus, namely $\Omega:=\Tc \setminus X$. Let $\Gamma_\Omega$ be the graph of $\varphi$ inside $\Omega \times \PP^n$. 
\end{defn}

Clearly $\Gamma_\Omega \nto{\pi_1}{\lto} \Omega$ is birational, which is in general not the case over $X$. As was shown in \cite{Bot08}, the scheme structure of the base locus when we take $(\P1)^n$ as the codomain, can be fairly complicated and extraneous factors may occur when projecting on $(\P1)^n$ via $\pi_2$ (cf.\ \ref{ch:toric-emb-p1xxp1}). This motivates the need for a splitting of the base locus, giving rise to families of multiprojective bundles over $\Tc$.

Due to this important difference between the projective and multiprojective case, we need to separate the study of the two settings. In the next section, we treat the case of $\PP^n$, and in Chapter \ref{ch:toric-emb-p1xxp1} the case of $(\P1)^n$. In both situations, we find a matrix representation of the closed image of the rational map $\varphi$, and we compute the implicit equation and extraneous factors that occur.

Next, we introduce the homological machinery needed to deal with the computations of the implicit equations and the representation matrix of the hypersurface.

\subsection{Homological algebra tools}

In this section we will study some properties of approximation complexes, introduced in Chapter \ref{ch:elimination}. Our aim is to get similar results in a new context: the ring $A$ is the coordinate ring of a toric variety, which is CM, but in general not Gorenstein. The non-Gorensteinness makes things more complicated since, for example, we cannot identify $H^{\dim(A)}_\mm(A)$ with $\omega_A^\vee$. We will first brefly recall the definition of these complexes, just in order to fix a notation, and later prove that if the ideal $I$ is LACI then the associated $\Zc$-complex is acyclic. Finally we give a bound for the regularity of the symmetric algebra of $I$ over $A$.

For simplicity, we denote by $T_i$ the classes of each variable in the quotient ring $A=\kk[\T]/J$. Recall that $A$ is canonically graded, each variable having weight $1$. Let $I=(h_0,\hdots,h_{n}) \subset A$ be the ideal generated by the $h_i$'s.

More precisely, we will see that the implicit equation of $\Sc$ can be recovered as the determinant of certain graded parts of the $\Zc$-complex we define below. We denote by ${\Zc}_{\bullet}$ the approximation complex of cycles associated to the sequence $h_0,\hdots,h_{n}$ of homogeneous elements of degree $d$ over $A$ (cf.\ \cite{Va94}), as in the Definition \ref{defZcomplex}. 

\medskip
 Consider the Koszul complex $(K_\bullet(\h,A),\delta_\bullet)$ associated to $h_0,\ldots,h_{n}$ over $A$ and denote $Z_i=\ker(\delta_i)$, $B_i=\im(\delta_{i+1})$. It is of the form
\[
K_\bullet(\h,A): \quad  A[-(n+1)d] \nto{\delta_{n+1}}{\lto} A[-nd]^{n+1} \nto{\delta_n}{\lto} \cdots \nto{\delta_2}{\lto} A[-d]^{n+1} \nto{\delta_1}{\lto} A
\]
where the differentials are matrices such that every non-zero entry is $\pm h_i$ for some $i$. 

Write $K_i:= \bigwedge^iA^{n+1}[-i\cdot d]$. Since $Z_i\subset K_i$, it keeps the shift in the degree. Note that with this notation the sequence 
\begin{equation}\label{sesZKB}
 0\to Z_i\to K_i\to B_{i-1}\to 0
\end{equation}
is exact graded, and no degree shift is needed.

We introduce new variables $T_0,\ldots,T_{n}$ with $\deg(T_i)=1$. Since $A$ is $\NN$-graded, $A[\X]$ inherits a bigrading. 

\begin{defn}\label{defZcomplex}
Denote by $\Zc_i= Z_i[i \cdot d] \otimes_A A[\X]$ the ideal of cycles in $A[\X]$, and write $[-]$ for the degree shift in the variables $T_i$ and $(-)$ the one in the $T_i$. The approximation complex of cycles $(\Zc_\bullet(\h,A),\epsilon_\bullet)$, or simply $\Zc_\bullet$, is the complex 
\begin{equation}\label{CompAppZ}
\Zc_\bullet(\h,A): \quad 0 \lto \Zc_n(-n) \nto{\epsilon_n}{\lto} \Zc_{n-1}(-(n-1)) \nto{\epsilon_{n-1}}{\lto} \cdots \nto{\epsilon_2}{\lto} \Zc_1(-1) \nto{\epsilon_1}{\lto} \Zc_0
\end{equation}
where the differentials $\epsilon_\bullet$ are obtained by replacing $h_i$ by $T_i$ for all $i$ in the matrices of $\delta_\bullet$. 
\end{defn}

Recall that $H_0(\Zc_\bullet) = A[\X]/\im(\epsilon_1) \cong \SIA$. Note that the degree shifts are with respect to the grading $(-)$ given by the $T_i$'s, while the degree shifts with respect to the grading of $A$ are already contained in our definition of the $\Zc_i$'s. From now on, when we take the degree $\nu$ part of the approximation complex, denoted $(\Zc_\bullet)_\nu$, it should always be understood to be taken with respect to the grading $[-]$ induced by $A$.

\medskip

Under certain conditions on the base locus of the map, this complex is acyclic and provides a free $\kk[\X]$-resolutions of $(\SIA)_\nu$ for all $\nu$. Hence, we focus on finding acyclicity conditions for the complex ${\Zc}_{\bullet}$. In this direction we have

\begin{lem}\label{ZacycALCI} 
Let $m\geq n$ be non-negative integers, $A$ an $m$-dimensional graded Cohen-Macaulay ring and $I=(h_0,\ldots,h_{n}) \subset A$ is of codimension (hence depth) at least $n-1$ with $\deg(h_i)=d$ for all $i$. Assume that $X:=\Proj (A/I) \subset \Sc$ is locally defined by $n$ equations (i.e.\ locally an almost complete intersection). Then ${\Zc}_{\bullet}$ is acyclic.
\end{lem}
\begin{proof}
The proof follows ideas of \cite[Lemma 2]{BC05} and \cite[Lemma 1]{BD07}. Observe that the lemma is unaffected by an extension of the base field, so one may assume that $\kk$ is infinite.

By \cite[{Theorem} 12.9]{HSV}, we know that $\Zc_\bullet$ is acyclic (resp.\ acyclic outside $V(\mm)$) if and only if $I$ is generated by a proper sequence (resp.\ $X$ is locally defined by a proper sequence), see Theorem \ref{thmProperZacyclic}. Recall that a sequence $a_1,\ldots,a_n$ of elements in a commutative ring $B$ is a \emph{proper sequence} if $a_{i+1}H_{j}(a_1,\ldots,a_i;B)=0$ for $i=0,\ldots,n-1$ and $j>0$, where the $H_j$'s denote the homology groups of the corresponding Koszul complex (cf.\ Definition \ref{sucesiones}).

By following the same argument of \cite[{Lemma} 2]{BC05} and since $X$ is locally defined by $n$ equations, one can choose $\tilde h_0,\ldots,\tilde h_{n}$ to be sufficiently generic linear combinations of the $h_i$'s such that 
\begin{enumerate}
 \item $(\tilde h_0,\ldots,\tilde h_{n})=(h_0,\ldots,h_{n}) \subset A$,
 \item $\tilde h_0,\ldots,\tilde h_{n-2}$ is an $A$-regular sequence, hence $\tilde h_0\ldots,\tilde h_{n-1}$ is a proper sequence in $A$,
 \item $\tilde h_0,\ldots,\tilde h_{n-1}$ define $X$ in codimension $n-1$.
\end{enumerate}

Note that this last condition is slightly more general (and coincides when $m=n$) than the one in \cite[{Lemma} 2]{BC05}. Set $J:=(\tilde h_0,\dots,\tilde h_{n-1})$ and write $J^{um}$ for the unmixed part of $J$ of codimension $n-1$. Hence, observe that we obtain $\tilde h_{n} \in J^{um}$. 

\medskip

Since $\tilde h_{n}\in J^{um}$, we show that $\tilde h_{n}H_1(\tilde h_0,\dots,\tilde h_{n-1};A)=0$. Applying \cite[Thm.\ 1.6.16]{BH} to the sequence $\tilde h_0,\hdots, \tilde h_{n-1}$, we obtain that $H_1(\tilde h_0,\dots,\tilde h_{n-1};A)\cong \ext^{n-1}_A(A/J,A)$. Taking the long exact sequence of $\ext^\bullet_A(-,A)$ coming from the short exact sequence $0\to J^{um}/J\to A/J\to A/J^{um}\to 0$, we get that
\begin{equation*}
 \xymatrix@R-16pt{
  \cdots \ar[r]
   & \ext^{n-2}_A(J^{um}/J,A) \ar[r] 
   & \ext^{n-1}_A(A/J,A)\ar `[r] `[d] '[dll] *{} `[ddll] `[ddll] [ddl]
   & \\
 &&& \\
 {}
   & \ext^{n-1}_A(A/J^{um},A) \ar[r]
   & \ext^{n-1}_A(J^{um}/J,A) \ar[r]
   & \cdots
}
\end{equation*}
is exact. Since $A$ is a Cohen-Macaulay noetherian graded ring, and $J^{um}/J$ is a $m-(n-1)$-dimensional $A$-module, $\ext^{n-1}_A(J^{um}/J,A)$ and $\ext^{n-2}_A(J^{um}/J,A)$ vanish (cf.\ \cite[Thm.\ 17.1]{Mats}). Hence
\[
 \ext^{n-1}_A(A/J,A)\cong \ext^{n-1}_A(A/J^{um},A),
\]
thus, since $\tilde h_{n}\in J^{um}$, $\tilde h_{n}$ annihilates $\ext^{n-1}_A(A/J^{um},A)$, hence also $\tilde h_n$ annihilates $H_1(\tilde h_0,\hdots,\tilde h_{n-1};A)$ which finishes the proof. 
\end{proof}

We stress in the following remark one useful application of the previous Lemma \ref{ZacycALCI}.

\begin{rem}\label{remAcycZPmPn}
Let $m\geq n$ be non-negative integers. Set $\Tc$ an arithmetically Cohen-Macaulay scheme over $\kk$ embedded in some $\PP^N$ with coordinate ring $A$ of affine dimension $m$. Assume we are given a rational map $\varphi:\Tc \dto \PP^n$ given by $n+1$ homogeneous polynomials $h_0,\hdots,h_n\in A:=\kk[T_0,\hdots,T_N]/I(\Tc)$. Write $\Zc_\bullet$ for the approximation complex of cycles associated to the sequence $h_0,\hdots,h_n$. If the base locus of $\varphi$, $X\subset \Tc$, is locally defined by $n$ equations, then $\Zc_\bullet$ is acyclic, independent of $m$ and $N$.
\end{rem}

We translate Lemma \ref{ZacycALCI} geometrically.

\begin{cor}\label{corAcycZToricPn}
Assume $m = n$ is a non-negative integer. Let $\Tc$ be an $(n-1)$-dimensional arithmetically CM closed subscheme of $\PP^N$ defined by a homogeneous ideal $J$, and coordinate ring $A=\kk[\T]/J$. Assume we are given a rational map $\varphi:\Tc \dto \PP^n$ given by $n+1$ homogeneous polynomials $h_0,\hdots,h_n\in A$ of degree $d$. Write $\Zc_\bullet$ for the approximation complex of cycles associated to the sequence $h_0,\hdots,h_n$. If the base locus of $\varphi$, $X\subset \Tc$, is finite, and locally an almost complete intersection (defined by $n$ equations), then $\Zc_\bullet$ is acyclic.
\end{cor}

The following result establishes a vanishing criterion on the graded strands of the local cohomology of $\Sym_A(I)$, which ensures that the implicit equation can be obtained as a generator of the annihilator of the symmetric algebra in that degree.

Since $A$ is a finitely generated graded Cohen Macaulay $A$-module of dimension $n$, $H^i_\mm(A)=0$ for all $i \neq n$ and $H^n_\mm(A) =\omega_A^\vee$, where $(-)^\vee:=\ ^*\hom_A(-,\kk)$ stands for the Matlis dualizing functor (cf.\ \cite{BH}). Write
\begin{equation}\label{eqai}
 a_i(M):=\inf \{\mu\ : (H^i_\mm(M))_{>\mu}=0\}.
\end{equation}
Hence, we set
\begin{equation}\label{eqgamma}
 \gamma:= a_n(A)=\inf\{\mu\ :\ (\omega_A^\vee)_\mu=0\},
\end{equation}
and we conclude the following result.

\begin{thm}\label{annih} 
Let $A=\kk[\T]/J$ be a CM graded ring of dimension $n$. Let $I=(h_0,\hdots,h_n)$ be a homogeneous ideal of $A$, with $\deg(h_i)=d$ for all $i$. Let $X:=\Proj(A/I) \subset \Tc$ be finite and locally an almost complete intersection. Set
\begin{equation}\label{eqnu0}
 \nu_0 := \max\{(n-2)d, (n-1)d-\gamma\},
\end{equation}
then $H^0_\mathfrak{m} ( \Sym_A(I) )_\nu =0$ for all $\nu\geq \nu_0$.
\end{thm}
\begin{proof}
For the bound on $\nu$, consider the two spectral sequences associated to the double complex $C^\bullet_\mm(\Zc_\bullet)$, both converging to the hypercohomology of $\Zc_\bullet$. The first spectral sequence stabilizes at step two with
\[
_\infty'E^p_q =\ _2'E^p_q = H^p_\mm(H_q(\Zc_\bullet)) = \left\lbrace\begin{array}{ll}H^p_\mm(\SIA) & \mbox{for }q=0, \\
0 & \mbox{otherwise.} \end{array}\right.
\]
The second has first terms $_1''E^p_q =\ H^p_\mm(Z_q)[qd]\otimes_A A[\X](-q)$. The comparison of the two spectral sequences shows that $H^0_\mm(\Sym_A(I))_\nu$ vanishes as soon as $(_1{''}E^{p}_p)_\nu$ vanishes for all $p$, in fact we have that 
\[
 \endd(H^0_\mm(\Sym_A(I)))\leq \max_{p\geq 0}\{\endd(_1{''}E^{p}_p)\}=\max_{p\geq 0}\{\endd(H^p_\mm(Z_p))-pd\},
\]
where we denote, for an $A$-module $M$, $\endd(M)= \max \{ \nu \ | \ M_\nu \neq 0 \}$.
Since $Z_0\cong A$ we get $H^0_\mm(Z_0)=0$. The sequence $ 0 \to Z_{i+1} \to K_{i+1} \to B_i \to 0$ is graded exact (cf.\ Equation \eqref{sesZKB}), hence, from the long exact sequence of local cohomology for $i=0$ (writing $B_0=I$) we obtain
\[
  \cdots \to H^0_\mm(I) \to H^1_\mm(Z_{1}) \to H^1_\mm(K_{1}) \to \cdots .
\]
As $I$ is an ideal of an integral domain, $H^0_\mm(I)=0$, it follows from the local cohomology of $A$ that $H^1_\mm(K_{1})=0$, hence $H^1_\mm(Z_{1})$ vanishes.
By construction, $Z_{n+1}=0$ and $B_{n}=\im(d_{n})\simeq A[-d]$. Using the fact that $H^{n}_\mathfrak{m}(A)_\nu =0$ for $\nu \geq -1$ (resp.\ $\nu \geq 0$), we can deduce that $H^{n}_\mathfrak{m}(Z_{n})_\nu=H^{n}_\mathfrak{m}(B_{n})_\nu = (\omega_A^\vee)[d]= 0$ if $\nu \geq d-\gamma$.
Write 
\[
 \epsilon_p:= \endd(_1{''}E^{p}_p)=\endd(H^p_\mm(Z_p))-pd 
\]
By \cite[Cor.\ 6.2.v]{Ch04} $\endd(H^p_\mm(Z_p))\leq  \max_{0\leq i\leq n-p} \{a_{p+i}(A) + (p+i+1)d\}=\max\{nd, (n+1)d-\gamma\}$, where $\gamma:=-a_{n}(A)$ as above. Hence, $ \epsilon_p:= \max\{(n-p)d, (n+1-p)d-\gamma\}$. As $\epsilon_p$ decreases when $p$ increases, $ \epsilon_p\leq \epsilon_2= \max\{(n-2)d, (n-1)d-\gamma\}$ which completes the proof.
\end{proof} 

This generalizes what we sketched in Chapter \ref{ch:elimination} according to \cite{BuJo03} and \cite{BCJ06} and also, we generalize \cite{BDD08} to general $(n-1)$-dimensional arithmetically Cohen-Macaulay schemes with almost locally complete intersection base locus. Next, we recall how the homological tools developed in this part are applied for computing the implicit equation of the closed image of a rational map.

\subsection{The representation matrix, the implicit equation, and the extraneous factor}

It is well known that the annihilator above can be computed as the determinant (or MacRae invariant) of the complex $(\Zc_\bullet)_{\nu_0}$ (cf.\ Chapter \ref{ch:elimination} and for example, \cite{BuJo03}, \cite{BCJ06}, \cite{Bot08}, \cite{BDD08}). Hence, the determinant of the complex $(\Zc_\bullet)_{\nu_0}$ is a multiple of a power of the implicit equation of $\Sc$. Indeed, we conclude the following result.

\begin{lem}\label{lemAnnih} 
Let $\Tc$ be an $(n-1)$-dimensional arithmetically CM closed subscheme of $\PP^N$ defined by a homogeneous ideal $J$, and coordinate ring $A=\kk[\T]/J$. Let $I=(h_0,\hdots,h_n)$ be a homogeneous ideal of $A$, with $\deg(h_i)=d$ for all $i$. Take $\varphi$ as in \eqref{eqSettingPn}, and let $X:=\Proj(A/I) \subset \Tc$, the base locus of $\varphi$, be finite and locally an almost complete intersection. Set $\nu_0 := \max\{(n-2)d, (n-1)d-\gamma\}$, then $H^0_\mathfrak{m} ( \Sym_A(I) )_\nu =0$ and $\ann_{\kk[\X]} ( \Sym_A(I)_\nu )\subset \ker(\varphi^\ast)$, for all $\nu \geq \nu_0$.
\end{lem}
\begin{proof}
 The first part follows from \ref{annih}. The proof of the second part can be taken verbatim from \cite[Lemma 2]{BD07}.
\end{proof}

\begin{cor}\label{mainthT}
Let $\Tc$ be an $(n-1)$-dimensional arithmetically CM closed subscheme of $\PP^N$ defined by a homogeneous ideal $J$, and coordinate ring $A=\kk[\T]/J$. Let $I=(h_0,\hdots,h_n)$ be an homogeneous ideal of $A$, with $\deg(h_i)=d$ for all $i$. Let $X:=\Proj(A/I) \subset \Tc$ be finite and locally almost a complete intersection. Let $\nu_0$ be as in \eqref{eqnu0}. For any integer $\nu \geq \nu_0$ the determinant $D$ of the complex $(\Zc_\bullet)_\nu$ of $\kk[\X]$-modules defines (up to multiplication with a constant) the same non-zero element in $\kk[\X]$. Moreover, $D=F^{\deg(\varphi)}G$, where $F$ is the implicit equation of $\Sc$.
\end{cor}
\begin{proof}
It follows from Lemma \ref{ZacycALCI}, Lemma \ref{lemAnnih}, and Theorem \ref{locosymalgT}, by following the same lines of the proof of \cite[Thm.\ 5.2]{BuJo03}.
\end{proof}

By \cite[Appendix A]{GKZ94}, the determinant $D$ can be computed either as an alternating product of subdeterminants of the differentials in $(\Zc_\bullet)_{\nu}$ or as the greatest common divisor of the maximal-size minors of the matrix $M$ associated to the
right-most map $(\Zc_1)_{\nu} \rightarrow  (\Zc_0)_{\nu}$ of the $\Zc$-complex (cf.\ Definition \ref{defZcomplex}). Note that this matrix is nothing else than the matrix $M_\nu$ of linear syzygies as described in the introduction; it can be computed with the same algorithm as in \cite{BD07} or \cite{BDD08}. Hence, if $\Tc  \stackrel{\varphi}{\dashrightarrow} \PP^n$ is as in Corollary \ref{mainthT}, the matrix $M_\nu$ of linear syzygies of $h_0,\ldots,h_n$ in degree $\nu \geq \nu_0$ is a representation matrix for the closed image of $\varphi$.

As was done by Bus\'e et al.\ in \cite[Sec.\ 2]{BCJ06}, we conclude that the the extraneous factor $G$ can be described in terms of linear forms.

\begin{prop}\label{extraFactorThm}
If the field $\kk$ is algebraically closed and $X$ is locally generated by at most $n$ elements then, there exist linear forms $L_x\in \kk[\X]$, and integers $e_x$ and $d_x$ such that
\[
 G=\prod_{x\in X} L_x^{e_x-d_x}\in \kk[\X].
\]
Moreover, if we identify $x$ with the prime ideal in $\Spec(A)$ defining the point $x$, $e_x$ is the Hilbert-Samuel multiplicity $e(I_x, A_x)$, and $d_x:= \dim_{A_x/xA_x}(A_x/I_x)$.
\end{prop}

\begin{proof}
 The proof goes along the same lines of \cite[Prop.\ 5]{BCJ06}, just observe that \cite[Lemma 6]{BCJ06} is stated for a Cohen-Macaulay ring as is $A$ in our case.
\end{proof}


\section{The representation matrix for toric surfaces}\label{sec:equation}

We applied here in down to earth terms, the results above for the case of toric surfaces following \cite{BDD08}. It is a natural question how this kind of matrix representation can be used concretely to rational surfaces defined as the image of a map
\begin{eqnarray*}
 \AA^2 & \stackrel{f}{\dashrightarrow} & \AA^3 \\
(s,t) & \mapsto & \left(\frac{f_1(s,t)}{f_0(s,t)},\frac{f_2(s,t)}{f_0(s,t)},\frac{f_3(s,t)}{f_0(s,t)}\right)
\end{eqnarray*}
where $f_i \in \kk[s,t]$ are coprime polynomials of degree $d$. In order to put the problem in the context of graded modules, one first has to consider an associated projective map
\begin{eqnarray*}
 \Tc &  \stackrel{\varphi}{\dashrightarrow}  & \PP^3 \\
 P & \mapsto & (h_0(P):h_1(P):h_2(P):h_3(P))
\end{eqnarray*}
where $\Tc$ is a 2-dimensional projective toric variety (for example $\PP^2$ or $\PP^1 \times \PP^1$) with coordinate ring $A$ and the $h_i \in A$ are homogenized versions of their affine counterparts $f_i$. In other words, as in Section \ref{sec2setting}, $\Tc$ is a suitable compactification of the affine space $(\AA^*)^2$ \cite{Co03,Fu93}. In this case, a linear syzygy (or moving plane) of the parametrization $g$ is a linear relation on the $h_0,\ldots,h_3$, i.e. a linear form $L = a_0X_0+a_1X_1+a_2X_2+a_3X_3$ in the variables $X_0,\ldots,X_3$ with $a_i \in \kk[s,t]$ such that
\begin{equation}\sum_{i=0,\ldots,3} a_i h_i =0 \label{syzygyequation} \end{equation}
Recall that in the same way as for curves, one can set up the matrix
$M_\nu$ of coefficients of the syzygies in a certain degree $\nu$, but unlike the case of curves, it is in general not possible to choose a degree $\nu$ such that $M_\nu$ is a square matrix representation of the surface (cf.\ Chapters \ref{implicitization} and \ref{sec:intro}). 

Since we are looking for a matrix representation, we will assume that the base locus $X:=\Proj(A/I)$ is locally a complete intersection. Thus, we will get a symbolic matrix $m_\nu$, whose rank drops at $p$ if and only if $p$ lies on the surface. 

\begin{thm}[Thm.\ \ref{mainthT}]\label{mainthTdim2}
Suppose that $X:=\Proj(A/I) \subset \Tc$ has at most dimension $0$ and is
locally a complete intersection. Let $\gamma=\inf\{\mu\ :\ (\omega_A^\vee)_\mu=0\},$ be as in \eqref{eqgamma} and $\nu_0= 2d-\gamma$. For any integer $\nu \geq \nu_0$
the determinant $D$ of the complex $(\Zc_\bullet)_\nu$ of
$\kk[\underline{T}]$-modules defines (up to multiplication with a constant) the same non-zero
element in $\kk[\X]$ and
$$D=F^{\deg(\varphi)}$$ where $F$ is the implicit equation
of $\Sc$. 
\end{thm}

By Theorem \ref{annih}, one can replace the bound in this result by the more precise bound $\nu_0=\max\{d-\gamma, 2d+1- \indeg(H^0_\mm(\omega_A/I.\omega_A))\}$ if there is at least one base point. 

By \cite[Appendix A]{GKZ94}, as mentioned in Chapter \ref{ch:elimination}, the determinant $D$ can be computed either as an alternating sum of subdeterminants of the differentials in $\Zc_{\nu}$ or as the greatest common divisor of the maximal-size minors of the matrix $M$ associated to the
first map $(\Zc_1)_{\nu} \rightarrow  (\Zc_0)_{\nu}$. Note that this matrix is nothing else than the matrix $M_\nu$ of linear syzygies as described in the introduction; it can be computed with the  same algorithm as in \cite{BD07} by solving the linear system given by the
degree $\nu_0$ part of \eqref{syzygyequation}, cf.\ Chapter \ref{ch:ch-algor-ToricP2}. As an immediate corollary we deduce the following very simple translation of Theorem \ref{mainthT}, which can be considered the main result of this section.\medskip

\begin{cor}\label{cor:main}
Let $\varphi:\Tc  \dashrightarrow \PP^3$ be a rational parametrization of the surface $\Sc \subset \PP^3$ given by $\varphi=(h_0:h_1:h_2:h_3)$ with $h_i \in A$.
 Let $M_\nu$ be the matrix of linear syzygies of $h_0,\ldots,h_3$ in degree $\nu \geq 2d-\gamma$, i.e. the matrix of coefficients of a $\kk$-basis of $\Syz(\varphi)_\nu$ with respect to a $\kk$-basis of $A_\nu$. If $\varphi$ has only finitely many base points, which are local complete intersections, then $M_\nu$ is a representation matrix for the surface~$\Sc$.
\end{cor}
\medskip
We should also remark that by \cite[Prop. 1]{KD06} (or \cite[Appendix]{Co01})
the degree of the surface $\Sc$ can be expressed in terms of the area of the Newton polytope and
the Hilbert-Samuel multiplicities of the base points:
\begin{equation}\label{deg2T}
\deg(\varphi)\deg(\Sc)=\mathrm{Area}(\Nc(f))-\sum_{\pp \in V(h_0,\ldots,h_3)\subset \Tc} e_\pp
\end{equation}
where $\mathrm{Area}(\Nc(f))$ is twice the Euclidean area of $\Nc(f)$, i.e. the normalized area of the polygon. For locally complete intersections, the multiplicity  $e_\pp$  of the base point~$\pp$
is just the vector space dimension of the local quotient ring at~$\pp$.

\section{The special case of biprojective surfaces}\label{segreT}

 Bihomogeneous parametrizations, i.e. the case $\Tc=\PP^1 \times \PP^1$, are particularly important in practical applications, so we will now make explicit the most important constructions in that case and make some refinements.

In this section, we consider a rational parametrization of a surface~$\Sc$
\begin{eqnarray*}
 \PP^1 \times \PP^1 &  \stackrel{f}{\dashrightarrow}  & \PP^3 \\
(s:u) \times (t:v) & \mapsto & (f_0:f_1:f_2:f_3)(s,u,t,v)
\end{eqnarray*}
where the polynomials $f_0,\ldots,f_3$ are bihomogeneous of
bidegree $(e_1,e_2)$ with respect to the homogeneous variable pairs $(s:u)$ and
$(t:v)$, and $e_1,e_2$ are positive integers. We make the same assumptions as
in the general toric case. Let $d=\gcd(e_1,e_2)$, $e_1'=\frac{e_1}{d}$, and $e_2'=\frac{e_2}{d}$. So we assume that the Newton polytope $\Nc(f)$ is a rectangle of length $e_1$ and width $e_2$ and $\Nc'(f)$ is a rectangle of length $e_1'$ and width $e_2'$.

So $\PP^1\times \PP^1$ can be embedded in $\PP^m$, $m=(e'_1+1)(e'_2+1)-1$ through the  \emph{Segre-Veronese embedding} $\rho=\rho_{e_1,e_2}$
\begin{eqnarray*}
 \PP^1\times \PP^1 &  \stackrel{\rho}{\hookrightarrow} & \PP^m \\
(s:u)\times(t:v) & \mapsto & (\ldots : s^i u^{e'_1-i} t^j v^{e'_2-j} : \ldots)
\end{eqnarray*}
We denote by $\Tc$ its image, which is an irreducible surface in $\PP^m$, whose ideal $J$ is generated by quadratic binomials. We have the following commutative diagram.
\begin{equation}
\xymatrix{
\PP^1\times \PP^1 \ar@{-->}[r]^-{f} \ar@{->}[d]^{\rho} & \PP^3 \\
\Tc \ar@{-->}[ur]_\varphi }       \label{diagram1T}
\end{equation}
with $\varphi =(h_0:\ldots:h_3)$, the $h_i$ being polynomials in the variables $T_0,\ldots,T_m$ of degree $d$. We denote by $A=\kk[T_0,\ldots,T_m]/J$ the homogeneous coordinate ring of $\Tc$.
We can give an alternative construction of the coordinate ring (cf.\ Section \ref{CoxRing}). Consider the $\NN$-graded $\kk$-algebra
\[
 S:=\bigoplus_{n\in \NN} \left( \kk[s,u]_{n e_1'} \otimes_\kk \kk[t,v]_{n e_2'} \right)  \subset \kk[s,u,t,v]
\]
which is finitely generated by $S_1$ as an $S_0$-algebra.  Then $\PP^1\times \PP^1$ is the bihomogeneous spectrum $\mathrm{Biproj}(S)$ of $S$, since $\Proj(\bigoplus_{n\in \NN} \kk[s,u]_{n e_1'})=\Proj(\bigoplus_{n\in \NN} \kk[t,v]_{n e_2'})=\PP^1$. Write $T^{i,j}:=T_{(e'_2+1)i+j}$ for $i=0,\ldots, e'_1$ and $j=0,\ldots,e'_2$. The Segre-Veronese embedding $\rho$ induces an isomorphism of $\NN$-graded $\kk$-algebras
\begin{eqnarray*}
 A & \xrightarrow{\theta} & S \\
T^{i,j} & \mapsto & s^i u^{e_1'-i} t^j v^{e_2'-j} .
\end{eqnarray*}
The implicit equation of $\Sc$ can be obtained by the method of approximation complexes by computing the kernel of the map
\begin{eqnarray*}
 \kk[X_0,\ldots,X_3] & \rightarrow & A \\
 X_i & \mapsto & h_i
\end{eqnarray*}
 The ring $A$ is an affine normal semigroup ring and it is Cohen-Macaulay. It is Gorenstein if and only if $e_1'=e_2'=1$ (or equivalently $e_1=e_2$), which is the case treated in \cite{BD07}. The ideal $J$ is easier to describe than in the general toric case (compare \cite[6.2]{Sul06} for the case
$e'_2=2$).
The generators of $J$ can be described explicitly. Denote
\[
 A_i:= \begin{pmatrix} T^{i,0} & \ldots & T^{i,e'_2-1} \\  T^{i,1} & \ldots & T^{i,e'_2} \end{pmatrix},
\]
then  the ideal $J$ is generated by the $2$-minors of the $ 4 \times e'_1 e'_2$-matrix below built from the matrices $A_i$:
\begin{equation}\label{Jgens}
\begin{pmatrix} A_0 & \ldots & A_{e'_1-1}  \\  A_1  & \ldots & A_{e'_1} \end{pmatrix}.
\end{equation}

The degree formula for this setting, which is a direct corollary of \eqref{deg2T}:
\begin{equation*}
\deg(\varphi)\deg(\Sc)=2e_1e_2-\sum_{\pp \in V(h_0,\ldots,h_3)\subset \Tc} e_\pp
\end{equation*}
where as before $e_\pp$ is the multiplicity of the base point~$\pp$.

We claim that it is better to choose the toric variety defined by $\Nc'(f)$ instead
of $\Nc(f)$. Let us now give some explanations why this is the case. As we have seen, a bihomogeneous parametrization
of bidegree $(e_1,e_2)$ gives rise to the toric variety $\Tc=\PP^1 \times \PP^1$ determined by a rectangle of length $e_1'$ and width $e_2'$, where $e_i'=\frac{e_i}{d}$, $d=\gcd(e_1,e_2)$, and whose coordinate ring can be described as
$$S:=\bigoplus_{n\in \NN} \left( \kk[s,u]_{n e_1'} \otimes_\kk \kk[t,v]_{n e_2'} \right)  \subset \kk[s,u,t,v]$$
Instead of this embedding of $\PP^1 \times \PP^1$ we could equally choose the embedding defined by $\Nc(f)$, i.e. a rectangle
of length $e_1$ and width $e_2$, in which case we obtain the following coordinate ring
$$\hat{S}:=\bigoplus_{n\in \NN} \left( \kk[s,u]_{n e_1} \otimes_\kk \kk[t,v]_{n e_2} \right)  \subset \kk[s,u,t,v]$$
It is clear that this ring also defines $\PP^1 \times \PP^1$ and we obviously have an isomorphism 
$$ \hat{S}_n \simeq S_{d \cdot n}$$
between the graded parts of the two rings, which means that the grading of $\hat S$ is coarser and contains
less information. It is easy to check that the above isomorphism induces an isomorphism between the corresponding
graded parts of the approximation complexes  $\Zc_\bullet$ corresponding to $S$ and $\hat{\Zc}_\bullet$ corresponding to $\hat{S}$,
namely
$$ \hat{\Zc}_\nu \simeq \Zc_{d \cdot \nu}$$
If the optimal bound  in Theorem \ref{mainthT} for the complex $\Zc$ is a multiple of $d$, i.e. $\nu_0=d \cdot \eta$,
then the optimal bound for $\hat \Zc$ is $\hat{\nu}_0=\eta$ and we obtain isomorphic complexes in these degrees and
the matrix sizes will be equal in both cases. If not, the optimal bound $\hat \nu_0$ is the smallest integer bigger than
$\frac{\nu_0}{d}$ and in this case, the vector spaces in $\hat \Zc_{\hat \nu_0}$ will be of higher dimension than their
counterparts in $\Zc_{\nu_0}$ and the matrices of the maps will be bigger. An example of this is given in the next section.

\section{Examples}\label{sec:examples}
\begin{exmp}
 We first treat some examples from \cite{KD06}. Example 10 in the cited paper, which could not be solved in a satisfactory manner in \cite{BD07}, is a surface parametrized by
\begin{eqnarray*}
f_0 & =&(t+t^2)(s-1)^2+(-1-st-s^2t)(t-1)^2 \\
f_1 &=& (t+t^2)(s-1)^2+(1+st-s^2t)(t-1)^2 \\
f_2 &=& (-t-t^2)(s-1)^2+(-1+st+s^2t)(t-1)^2 \\
f_3 &=& (t-t^2)(s-1)^2+(-1-st+s^2t)(t-1)^2
\end{eqnarray*}

 The Newton polytope $\Nc'(f)$ of this parametrization is \medskip
\medskip

\begin{center}
 \includegraphics{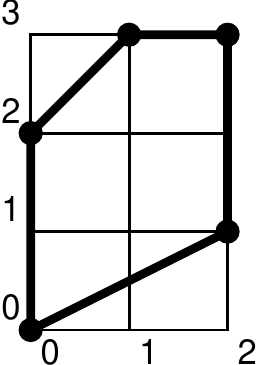}
\end{center}

 We can compute the new parametrization over the associated variety,
which is given by linear forms $h_0,\ldots,h_3$, i.e. $d=1$ (since there is no smaller homothety $\Nc'(f)$
of $\Nc(f)$) and the coordinate ring is $A=\kk[T_0,\ldots,T_8]/J$ where $J$
is generated by $21$ binomials of degrees $2$ and $3$. Recall that the $9$ variables correspond to the $9$ integer points in the Newton polytope. In the optimal degree $\nu_0 =1$ as in Theorem~\ref{annih}, the implicit equation of degree $5$ of the surface $\Sc$ 
is represented by a $9 \times 14$-matrix, compared to a $15 \times 15$-matrix  with the toric resultant method (from which a $11\times 11$-minor has to be computed) and a $5 \times 5$-matrix with the method of moving planes and quadrics. Note also that this is a major improvement of the method in \cite{BD07}, where  a $36\times 42$-matrix representation was computed for the
same example.
\end{exmp}
\medskip
\begin{exmp}
Example $11$ of \cite{KD06} is similar to Example 10 but an additional term is added, which transforms
the point $(1,1)$ into a non-LCI base point. The parametrization is

{\footnotesize \begin{eqnarray*}
f_0 & =&(t+t^2)(s-1)^2+(-1-st-s^2t)(t-1)^2+(t+st+st^2)(s-1)(t-1) \\
f_1 &=& (t+t^2)(s-1)^2+(1+st-s^2t)(t-1)^2 +(t+st+st^2)(s-1)(t-1)\\
f_2 &=& (-t-t^2)(s-1)^2+(-1+st+s^2t)(t-1)^2 +(t+st+st^2)(s-1)(t-1)\\
f_3 &=& (t-t^2)(s-1)^2+(-1-st+s^2t)(t-1)^2+(t+st+st^2)(s-1)(t-1)
\end{eqnarray*}  }

The Newton polytope has not changed, so the embedding as a toric variety and the coordinate ring $A$ are the same as in the previous example. Again the new map is given by $h_0,\ldots,h_3$ of degree $1$.

As in  \cite{KD06}, the method represents (with $\nu_0=1$) the implicit
equation of degree $5$ times a linear extraneous factor caused by the non-LCI base point. While the Chow form method represents this polynomial as a $12\times 12$-minor of a  $15\times 15$-matrix, our representation matrix is $9 \times 13$. Note that in this case, the method of moving lines and quadrics fails. 
\end{exmp}
\medskip
\begin{exmp} \label{indegfalse}
In this example, we will see that if the ring $A$ is not Gorenstein, the correction term for $\nu_0$ is different from  $\indeg(I^\sat)$, unlike in the homogeneous and the unmixed bihomogeneous cases. Consider the parametrization
\begin{eqnarray*}
f_0 &=& (s^2+t^2)t^6s^4+(-1-s^3t^4-s^4t^4)(t-1)^5(s^2-1) \\
f_1 &=& (s^2+t^2)t^6s^4+(1+s^3t^4-s^4t^4)(t-1)^5(s^2-1) \\
f_2 &=& (-s^2-t^2)t^6s^4+(-1+s^3t^4+s^4t^4)(t-1)^5(s^2-1) \\
f_3 &=& (s^2-t^2)t^6s^4+(-1-s^3t^4+s^4t^4)(t-1)^5(s^2-1)
\end{eqnarray*}  
We will consider this as a bihomogeneous parametrization of bidegree $(6,9)$, that is we will choose the embedding $\rho$ corresponding to a rectangle of length 2 and width $3$. The actual Newton polytope $\Nc(f)$ is smaller than the $(6,9)$-rectangle, but does not allow a smaller homothety. One obtains $A=\kk[T_0,\ldots,T_{11}] /J$, where $J$ is generated by $43$ quadratic binomials and the associated $h_i$ are of degree $d=3$. It turns out that $\nu_0=4$ is the lowest degree  such that the implicit equation of degree $46$ is represented as determinant of $\Zc_{\nu_0}$, the matrix of the first map being of size $117 \times 200$. So we cannot compute $\nu_0$ as $2d-\indeg(I^\sat)=6-3=3$, as one might have been tempted to conjecture based on the results of the homogeneous case. This is of course due to $A$ not being Gorenstein, since the rectangle contains two interior points.\medskip

 Let us make a remark on the computation of the 
representation matrix. It turns out that this is highly efficient. Even if we choose the non-optimal bound $\nu=6$ as given in Theorem \ref{mainthT}, the computation of the $247 \times 518$ representation matrix is computed instantaneously in Macaulay2. Just to give an idea of what happens if we take higher degrees: For $\nu=30$ a $5551 \times 15566$-matrix is computed in about 30 seconds, and for $\nu=50$ we need slightly less than 5 minutes to compute a $15251 \times 43946$ matrix. 

In any case, the computation of the matrix is relatively cheap and the main interest in lowering the bound $\nu_0$ as much as possible is the reduction of the size of the matrix, not the time of its computation. This reduction improves the performance of algorithmic applications of our approach,  notably to decide whether a given point lies in the parametrized surface.

\end{exmp}
\medskip
\begin{exmp} \label{interestingexample}
In the previous example, we did not fully exploit the structure of $\Nc(f)$ and chose a bigger polygon for the embedding. Here is an example where this is necessary to represent the implicit equation without extraneous factors. Take $(f_0,f_1,f_2,f_3) = (st^6+2,st^5-3st^3,st^4+5s^2t^6,2+s^2t^6)$. This is a very sparse parametrization and we have $\Nc(f)=\Nc'(f)$. The coordinate ring is $A=\kk[T_0,\ldots,T_5]/J$, where $J=(T_3^2-T_2T_4, T_2T_3-T_1T_4, T_2^2-T_1T_3, T_1^2-T_0T_5)$ and the new base-point-free parametrization $\varphi$ is given by $(h_0,h_1,h_2,h_3)=(2T_0+T_4,-3T_1+T_3, T_2+5T_5, 2T_0+T_5)$. The Newton polytope $\Nc(f)$ looks as follows.
\begin{center}
  \includegraphics[scale=0.95]{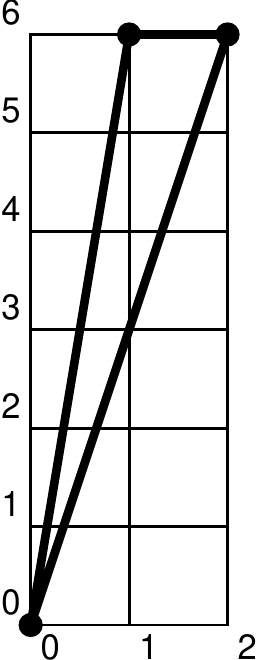}
\end{center}

 For $\nu_0=2d=2$ we can compute the matrix of the first map of $(\Zc_\bullet)_{\nu_0}$, which is a $17 \times 34$-matrix. 
The greatest common divisor of the $17$-minors of this matrix is the homogeneous implicit equation of the surface; it is of degree 6 in the variables $X_0,\ldots,X_3$:
\begin{eqnarray*}
 & & 2809X_0^2X_1^4 + 124002X_1^6 - 5618X_0^3X_1^2X_2 + 66816X_0X_1^4X_2 +
2809X_0^4X_2^2\\
& &- 50580X_0^2X_1^2X_2^2  + 86976X_1^4 X_2^2 + 212X_0^3X_2^3  - 14210X_0X_1^2X_2^3  + 3078X_0^2 X_2^4 \\
& & + 13632X_1^2 X_2^4  + 116X_0X_2^5 + 841X_2^6  + 14045X_0^3 X_1^2 X_3 - 169849X_0X_1^4 X_3 \\
& & -14045X_0^4 X_2X_3 + 261327X_0^2 X_1^2 X_2X_3 - 468288X_1^4 X_2X_3 - 7208X_0^3 X_2^2 X_3 \\
& & + 157155X_0X_1^2 X_2^3 X_3 - 31098X_0^2 X_2^3 X_3 - 129215X_1^2 X_2^3 X_3 - 4528X_0X_2^4 X_3  \\
& & - 12673X_2^5 X_3 - 16695X_0^2 X_1^2 X_3^2  + 169600X_1^4 X_3^2  +
30740X_0^3 X_2X_3^2 \\
& & - 433384X_0X_1^2 X_2X_3^2 + 82434X_0^2 X_2^2 X_3^2  + 269745X_1^2 X_2^2 X_3^2  + 36696X_0X_2^3 X_3^2 \\
& &  + 63946X_2^4 X_3^2  + 2775X_0X_1^2 X_3^3  - 19470X_0^2 X_2X_3^4  + 177675X_1^2 X_2X_3^3  \\ 
& & - 85360X_0X_2^2 X_3^3  - 109490X_2^3 X_3^3  - 125X_1^2 X_3^4  + 2900X_0X_2X_3^4   \\
& & + 7325X_2^2 X_3^4  - 125X_2X_3^5 
\end{eqnarray*}

 As in Example \ref{indegfalse} we could have considered the parametrization as a bihomogeneous map either of bidegree $(2,6)$ or of
bidegree $(1,3)$, i.e. we could have chosen the corresponding rectangles instead
of $\Nc(f)$. This leads to more complicated coordinate rings 
($20$ resp. $7$ variables and $160$ resp. $15$ generators of $J$) and to bigger matrices
(of size $21 \times 34$ in both cases). Even more importantly, the parametrizations will have a non-LCI base point and the matrices do not represent the implicit equation but a multiple of it (of degree $9$). Instead, if we consider the map as a homogeneous map of degree $8$, the results are even worse: For $\nu_0 = 6$, the $28 \times 35$-matrix $M_{\nu_0}$ represents a multiple of the implicit equation of degree $21$.

To sum up, in this example the toric version of the method of approximation complexes works well, whereas it fails over $\PP^1 \times \PP^1$ and $\PP^2$. This shows that the extension of the method to toric varieties really is a generalization and makes the method applicable to a larger class of parametrizations.\medskip

 Interestingly, we can even do better than with $\Nc(f)$ by choosing a smaller polytope. The philosophy is that the choice of the optimal polytope is a  compromise between two criteria:
\begin{itemize}
 \item The polytope should be as simple as possible in order to avoid that the ring $A$ becomes too complicated.
\item The polytope should respect the sparseness of the parametrization (i.e. be close to the Newton polytope) so that no base points appear which are not local complete intersections.
\end{itemize}

 So let us repeat the same example with another polytope $Q$, which is small enough to reduce the size of the matrix but which only adds well-behaved (i.e. local complete intersection) base points:
\begin{center}
\includegraphics{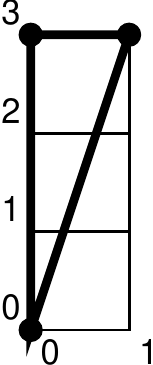}
\end{center}
 The Newton polytope $\Nc(f)$ is contained in $2 \cdot Q$, so the parametrization
will factor through the toric variety associated to $Q$, more precisely we obtain
a new parametrization defined by
 $$(h_0,h_1,h_2,h_3)=(2T_0^2+T_3T_4,-3T_0T_4+T_2T_4, T_1T_4+5T_4^2,2T_0^2+T_4^2)$$
over the coordinate ring $A=\kk[T_0,\ldots,T_4]/J$ with $J=(T_2^2-T_1T_3, T_1T_2-T_0T_3, T_1^2-T_0T_2)$. The optimal bound is $\nu_0=2$ and in this degree the implicit equation is represented directly without extraneous factors by a $12 \times 19$-matrix, which is smaller than the $17 \times 34$ we had before. 
\end{exmp}

\medskip

\begin{exmp}
As we have seen, the size of the matrix representation depends on the given
parametrization and as a preconditioning step it is often advantageous to choose a simpler parametrization of the same surface, if that is possible. For example, approaches such as \cite{Sc03} can be used to find a simpler reparametrization of the given surface and optimize the presented methods.

 Another important factor to consider is that all the methods we have seen
represent the implicit equation to the power of the degree of the parametrization. On one hand, it can be seen as an advantage that this piece of geometric information is encoded in the matrix representation, but on the other hand, for certain applications one might  be willing to sacrifice the information about the parametric degree in order to obtain smaller matrices. If this is the case, there exist (for certain surface parametrizations)  algorithms to compute a proper reparametrization of the surface, e.g. \cite{PeDi06}, and in these cases it is highly advisable to do so before computing the matrix representation, because this will allow us to represent the implicit equation directly instead of one of its powers, and the matrices will be significantly smaller. Let us illustrate this with Example 2 from \cite{PeDi06}, which treats a parametrization $f$ defined by
\begin{eqnarray*}
f_0 & =& (t^4+2t^2+5)(s^4+1) \\
f_1 &=& (s^4t^4+2s^4t^2+5s^4+2t^4+4t^2+11)(s^4+1) \\
f_2 &=& (s^4t^4+2s^4t^2+5s^4+t^4+2t^2+6)\\
f_3 &=&  -(s^4t^4+2s^4t^2+5s^4+t^4+2t^2+3)(s^4+1)
\end{eqnarray*}

 This is a parametrization of bidegree $(8,4)$ and its Newton polytope is the whole rectangle of length 8 and width 4, so we can apply the method of
approximation complexes for $\PP^1 \times \PP^1$. We obtain a matrix of size $45 \times 59$ representing $F_\Sc^{16}$, where $$F_\Sc = 2X_0X_1-X_1X_2-3X_0X_3-2X_1X_3+3X_3^2$$
is the implicit equation and $\deg(f)=16$. Using the algorithm presented in \cite{PeDi06} one can compute the following proper reparametrization of the surface $\Sc$:
\begin{eqnarray*}
f_0 & =& (t-5)(s-1) \\
f_1 &=&-(11+st-5s-2t)(s-1) \\
f_2 &=& 6-t-5s+st\\
f_3 &=& (-t+st-5s+3)(s-1)
\end{eqnarray*}
This parametrization of bidegree $(2,1)$ represents $F_\Sc$ directly by a $6 \times 11$-matrix.
\end{exmp}

\section{Final remarks} \label{sec:final}
 Representation matrices can be efficiently constructed by solving a linear system of relatively small size (in our case $\dim_\kk(A_{\nu+d})$ equations in $4 \dim_\kk(A_\nu)$ variables). This means that their computation is much faster than the computation of the implicit equation and they are thus an interesting alternative as an implicit representation 
of the surface. 

In this chapter, we have extended the method of matrix representations by linear syzygies to the case of rational surfaces parametrized over toric varieties (and in particular to bihomogeneous parametrizations). This generalization provides a better understanding of the method through the use of combinatorial commutative algebra. From a practical point of view, it is also a major improvement, as it makes the method 
applicable for a much wider range of parametrizations (for example, by avoiding unnecessary base points with bad properties) and
leads to significantly smaller representation matrices.  Let us sum up the advantages and disadvantages compared to other techniques
to compute matrix representations (e.g. the ones introduced in \cite{KD06}). The most important advantages are:
\begin{itemize}
\item The method works in a very general setting and makes only minimal assumptions on the parametrization. In particular, it works well in the presence of base points.
\item Unlike the method of toric resultants, we do not have to extract a maximal minor of unknown size, since the matrices
are generically of full rank.
\item The structure of the Newton polytope of the parametrization is exploited, so one obtains  much better results for sparse
parametrizations, both in terms of computation time and in terms of the size of the representation matrix. Moreover, it subsumes
the known method of approximation complexes in the case of dense homogeneous parametrizations, in which case the methods coincide.
\end{itemize}
 Disadvantages of the method are the following.
\begin{itemize}
\item Unlike with the toric resultant or the method of moving planes and surfaces, the matrix representations are not square.
\item The matrices involved are generally bigger than with the me\-thod of moving planes and surfaces.
\end{itemize}
 It is important to remark that those disadvantages are inherent to the choice of the method: A square matrix built from linear syzygies does
not exist in general and it is an automatic consequence that if one only uses linear syzygies to construct the matrix, it has to be bigger than a matrix which also uses entries of higher degree. The choice of the method to use depends very much on the given parametrization and on what one needs to do with the matrix representation.

\chapter[Implicit equations of toric hypersurfaces in multiprojective space]{Implicit equations of toric hypersurfaces in multiprojective space by means of an embeddings}
\label{ch:toric-emb-p1xxp1}


\section{Introduction}

The aim of this chapter is to compute the implicit equation of a hypersurface in $(\PP^1)^n$, parametrized by a toric variety. Assume we are given by a map
\begin{equation}\label{initsettingAsPn}
f: \AA^{n-1} \dto \AA^n : \s:=(s_1,\hdots,s_{n-1})  \mapsto \paren{\frac{f_1}{g_1},\hdots,\frac{f_n}{g_n}}(\s),
\end{equation}
where $\deg(f_i)=d_i$ and $\deg(g_i)=e_i$, and $f_i,\ g_i$ without common factors for all $i=1,\hdots n$. In Chapter \ref{ch:toric-emb-pn} we studied the case where $g_1=\cdots=g_n$. In all cases, we can reduce in theory all problems to this setting, by taking common denominator. However, there is also a big spectrum of problems that are not well adapted to taking a common denominator. Typically this process enlarges the base locus of $f$. This also increases the number of monomials and increasing the degree of the polynomials which could imply having a ``worse" compactification of the domain, forcing an embedding into a bigger projective space. For these many reasons, taking common denominator could be considerably harmful for the algorithmic approach.

In order to consider more general parametrizations given by rational maps of the form $f=\gens{\frac{f_1}{g_1},\hdots,\frac{f_n}{g_n}}$ with different denominators $g_1,\hdots,g_n$, we develop in this chapter the study of the $(\PP^1)^n$ compactification of the codomain.  With this approach, we generalize, in the spirit of \cite{Bot08}, the method of implicitization of projective hypersurfaces embedded in $(\P1)^n$ to general hypersurfaces parametrized by any $(n-1)$-dimensional arithmetically Cohen-Macaulay closed subscheme of $\PP^N$. As in the mentioned articles, we compute the implicit equation as the determinant of a complex, which coincides with the gcd of the maximal minors of the last matrix of the complex, and we give a deep study of the geometry of the base locus.

Section \ref{sec5algorithmic} is devoted to the algorithmic approach of  both cases studied in Chapter \ref{ch:toric-emb-pn} and \ref{ch:toric-emb-p1xxp1}. We show how to compute the dimension of the representation matrices obtained in both cases by means of the Hilbert functions of the ring $A$ and its Koszul cycles. In the last part of this chapter, we show, for the case of toric  parametrizations given from a polytope $\Nc(f)$ (cf.\ \ref{defNf}), how the interplay between homotheties of $\Nc(f)$ and degree of the maps may lead to have smaller matrices.

We conclude by giving in section \ref{sec6examples} several examples. First, we show in a very sparse setting the advantage of not considering the homogeneous compactification of the domain when denominators are very different. We extend in the second example this idea to the case of a generic affine rational map in dimension $2$ with fixed Newton polytope.
In the last example we give, for a parametrized toric hypersurface of $(\P1)^n$, a detailed analysis of the relation between the nature of the base locus of a map and the extra factors appearing in the computed equation.

\section{General setting}

Throughout this section, as in the previous chapter, we will write $\AA^k:= \Spec (\kk[T_1,\hdots,T_k])$ for the $k$-dimensional affine space over $\kk$. Assume we are given a rational map
\begin{equation}\label{initsetting}
f: \AA^{n-1} \dto \AA^n : \s:=(s_1,\hdots,s_{n-1})  \mapsto \paren{\frac{f_1}{g_1},\hdots,\frac{f_n}{g_n}}(\s)
\end{equation}
where $\deg(f_i)=d_i$ and $\deg(g_i)=e_i$ without common factors. Observe that this setting is general enough to include all classical implicitization problems. We consider in this chapter the same toric compactification $\Tc$  of Chapter \ref{ch:toric-emb-pn}, but a different one for $\AA^n$: $(\PP^1)^n$.

As in Chapter \ref{ch:toric-emb-pn}, we assume $\Tc$ can be embedded into some $\PP^N$, and set $A$ for the homogeneous coordinate ring of $\Tc$. Since $\AA^{n-1}$ is irreducible, so is $\Tc$, hence $A$ is a domain. The map \eqref{initsetting} gives rise to a toric variety $\Tc$ on the domain (cf.\ Charpter \ref{ch:toric-varieties} and \ref{ch:toric-emb-pn}) associated to the following polytope $\Nc(f)$. Recall from Definition \ref{defNf} that we will write 
\[
 \Nc(f):=\conv\paren{\bigcup_{i=1}^n \paren{\Nc(f_i)\cup \Nc(g_i)}}
\]
the convex hull of the union of the Newton polytopes of all the polynomials defining the map $f$.

Recall that the polytope $\Nc(f)$ defines a $(n-1)$-dimensional projective toric variety $\Tc$ provided with an ample line bundle which defines an embedding: for $N=\#(\Nc(f) \cap \ZZ^{n-1})-1$ we can write $\Tc \subseteq \PP^N$  (cf.\ Chapter \ref{ch:toric-varieties} and Chapter \ref{ch:toric-emb-pn}). Write $\rho$ for the embedding determined by this ample line bundle. We get that the map
\begin{equation}
 (\AA^*)^{n-1}  \stackrel{\rho}{\hookrightarrow} \PP^N : (\s) \mapsto (\ldots : \s^\alpha  : \ldots),
\end{equation}
where $\alpha \in \Nc(f) \cap \ZZ^{n-1}$, factorizes $f$ through a rational map with domain $\Tc$. Hence, take $\Tc\subset \PP^N$ the toric embedding obtained from $\Nc(f)$ (according to Definition \ref{defNf}). The multi-projective compactification of $\AA^n$ is given by
\begin{equation}\label{RemToricCaseP1n}
 \AA^n \nto{\iota}{\hto}(\PP^1)^n: (x_1,\hdots,x_n) \mapsto (x_1:1)\times\cdots\times(x_n:1).
\end{equation}
Thus, $f$ compactifies via $\rho$ and $\iota$ through $\Tc$ to $\phi:\Tc \dto(\PP^1)^n$ making the following diagram commute:
\begin{equation}\label{diagramphiP1xxp1}
 \xymatrix{
 (\AA^\ast)^{n-1}\ar@{^{(}->}[d]^\rho\ar[r]^f	& \AA^{n}\ar@{^{(}->}[d]^\iota\\
 \Tc \ar@{-->}[r]^\phi 				& (\PP^1)^n
}
\end{equation}
That is, $\iota\circ f=\phi\circ\rho$. We will consider henceforward in this chapter rational maps $\phi:\Tc \dto (\PP^1)^n$, as defined in \eqref{diagramphiP1xxp1}.

\section{Tools from homological algebra}

We present here some basic tools of commutative algebra we will need for our purpose. Recall that in this chapter $A=\kk[T_0,\hdots,T_N]/J$ is the CM graded coordinate ring of an $(n-1)$-dimensional projective arithmetically Cohen Macaulay closed scheme $\Tc$ defined by $J$ in $\PP^N$. Set $\T:= T_0,\hdots,T_N$ the variables in $\PP^N$, and $\X$ the sequence $X_1,Y_1,\hdots,X_n,Y_n$, of variables in $(\PP^1)^n$. Write $\mm:=A_+=(\T)\subset A$ for the maximal irrelevant homogeneous ideal of $A$. Denote $R=A\otimes_\kk \kk[X_1,Y_1,\hdots,X_n,Y_n]$. Assume we are given $f_i$, $g_i$, for $i=1,\hdots,n$, $n$ pairs of homogeneous polynomials in $A$ without common factors, satisfying $\deg(f_i)=\deg(g_i)=d_i$ for all $i$.

We associate to each pair of homogeneous polynomials $f_i$, $g_i$ a linear form $L_i:=Y_i f_i - X_i g_i$ in the ring $R:=A[\X]$ of bidegree $(d_i,1)$. 
Write $\k.$ for the Koszul complex $\k.(L_1,\hdots,L_n; R)$, associated to the sequence $L_1,\hdots,L_n$ and coefficients in $R$. The $\NN^n$-graded $\kk$-algebra $\BB:= \coker (\bigoplus_i R(-d_i,-1) \to R)$ is the multihomogeneous coordinate ring of the incidence scheme $\Gamma=\overline{\Gamma_\Omega}$. It can be easily observed that $\BB\cong \bigotimes_A \Sym_A(I^{(i)})\cong R/(L_1,\hdots,L_n)$. 

We defined in Section \ref{CA} approximation complexes. We will remark here the relation between approximation complexes and Koszul complex. Precisely, take $f$ and $g$ two homogeneous elements in $A$ of degree $d$, and take $A[X,Y]$ the polynomial rings in two variables and coefficients in $A$. According to the notation above, define  $L:=Y\cdot f-X\cdot g\in A_d[X,Y]_1$.

\begin{prop}\label{iso long 1}
If the sequence $\{f,g\}$ is regular in $A$, then there exists a bigraded isomorphism of complexes $\Z.(f,g) \cong \k.(L;A[X,Y])$.
\end{prop}

\begin{proof}
Given the sequence $\{f, g\}$ the approximation complex is:
\[
\Z.(f,g): \ 0\to Z_1[d]\otimes_A A[X,Y](-1) \stackrel{(x, y)}{\longrightarrow} Z_0\otimes_A A[X,Y] \to 0. 
\]
As the sequence $\{f,g\}$ is regular, $H_1(\k.^A(f,g))=0$. Hence $Z_1=(-g,f)A\cong A$ by the isomorphism $a\in A\mapsto (-g\cdot a,f\cdot a)\in Z_1$, given by the left-most map of $\k.^A(f,g)$. Tensoring with $A[X,Y]$ we get an isomorphism of $A$-modules, $\ZZZ_1\cong A[X,Y]$. 
Write $\k.$ for $\k.(L;A[X,Y])$, $[-]$ for the degree shift on the grading on $A$ and $(-)$ the shift on $X,Y$. The commutativity of the diagram 
\[
\xymatrix@1{ 
\Z. :\  0\ar[r] & Z_1[d]\otimes_A A[X,Y](-1)\ar[r]^(.55){(x,y)} &Z_0[d]\otimes_A A[X,Y] \ar[r]& 0\ \\
\k. :\  0\ar[r] & A[X,Y][-d](-1) \ar[r]^(.6){L}\ar[u]^{\psi\otimes_A 1_{A[X,Y]}} &A[X,Y] \ar[r]\ar[u]^{=}& 0,} 
\]
shows that $\Z.(f,g) \cong \k.(L;A[X,Y])$
\end{proof}

Keeping the same notation, we conclude the following result:

\begin{cor}\label{iso p1x...xp1}
If the sequence $\{f_i,g_i\}$ is regular for all $i=1,\hdots, n$. Then, there is an isomorphism of $A$-complexes 
\[
 \bigotimes_{i=1}^{n}\Zi \cong \k.(L_1,\hdots,L_n;A[\X]).
\]
\end{cor}

This fact corresponds to the idea that a map $\phi: \Tc\dto (\PP^1)^n$ is like having $n$ maps $\phi_i: \Tc\dto \PP^1$ given by each pair $\phi_i=(f_i:g_i)$ whose product gives $\phi$. Each $\phi_i$ gives a map of rings $\phi_i^\ast:\kk[X_i,Y_i]\to A$ whose tensor product gives $\phi^\ast$. 

\begin{rem}\label{KresB}
 Observe also that if the sequence $L_1,\hdots,L_n$ is regular in $A[\X]$, then $\k.(L_1,\hdots,L_n;A[\X])$ provides a resolution of $\BB$, that is $H_0(\k.)=\BB$.
\end{rem}

As a consequence of Remark \ref{KresB} and of Corollary \ref{iso p1x...xp1}, we can forget about approximation complexes all along this chapter, and focus on Koszul complexes.

In order to compute the representation matrix $M_\nu$ and the implicit equation of $\phi$, we need to be able to get acyclicity conditions for $\k.$. Indeed, consider the following matrix
\begin{equation}\label{matrizota}
 \Xi=\left(\begin{array}{ccccc}-g_1&0&\cdots&0\\ f_1&0&\cdots&0\\ 
\vdots&\vdots&\ddots&\vdots\\ 0&0&\cdots&-g_n\\ 0&0&\cdots&f_n  \end{array}\right)\in Mat_{2n,n}(A).
\end{equation}
Henceforward, we will write $I_r:=I_r(\Xi)$ for the ideal of $A$ generated by the $r\times r$ minors of $\Xi$, for $0\leq r \leq r_0:=\min\{n+1,m\}$, and define $I_0:=A$ and $I_r:=0$ for $r>r_0$. 

A theorem due to L. Avramov gives necessary and sufficient conditions for $(L_1,\hdots,L_n)$ to be a regular sequence in $R$ in terms of the depth the ideals of minors $I_r$. Precisely:
\begin{thm}[{\cite[Prop.\ 1]{avr}}]\label{avramov} The ideal $(L_1,\hdots,L_n)$ is a complete intersection in $R$ if and only if for all $r=1,\hdots,n$, $\codim_A(I_r)\geq n-r+1$.
\end{thm}

The matrix \eqref{matrizota} defines a map of $A$-modules $\psi: A^n \to A^{2n} \cong \bigoplus_{i=1}^n A[x_i,y_i]_1$, we verify that the symmetric algebra $\Sym_A(\coker(\psi))\cong A[\X]/(L_1,\hdots,L_n)$. 
Since $\Sym_A(\coker(\psi))=\BB$ is naturally multigraded, it can be seen as a subscheme of $\Tc\times(\P1)^{n}$. This embedding is determined by the natural projection $A[\X]\to A[\X]/(L_1,\hdots,L_n)$. In fact, the graph of $\phi$ is an $(n-1)$-dimensional irreducible component of $\Proj(\Sym_A(\coker(\psi)))\subset \Tc\times(\P1)^{n}$ which is a projective fiber bundle outside the base locus of $\phi$ in $\Tc$. 

Our aim is to show that under certain conditions on the $L_i$ and on the ideals $I_i$, there exist an element in $\kk[\X]$ that vanishes whenever $L_1,\hdots,L_n$ have a common root in $\Tc$ (cf.\ Theorems \ref{teoRes} and \ref{teoResGral}). This polynomial coincides with the sparse resultant $\Res_{\Tc}(L_1,\hdots,L_n)$. We will see that it is not irreducible in general, in fact, it is not only a power of the implicit equation, it can also have some extraneous factors, while the generic sparse resultant is always irreducible. Those factors come from some components of the base locus of $\phi$ which are not necessarily a common root of all $L_i$: it is enough that one of them vanishes at some point $p$ of $\Tc$ to obtain a base point of $\phi$. We will give sufficient conditions for avoiding extraneous factors.

We compute the implicit equation of the closed image of $\phi$ as a factor of the determinant of $(\k.)_{(\nu,*)}$, for certain degree $\nu$ in the grading of $A$. As in \cite{BDD08} the last map of this complex of vector spaces is a matrix $M_{\nu}$ that represents the closed image of $\phi$. Thus, we focus on the computation of the regularity of $\BB$ in order to bound $\nu$. Recall from Equation \eqref{eqgamma} that $\gamma:= \inf\{\mu\ :\ (\omega_A^\vee)_\mu=0\}$.

\begin{thm} \label{locosymalgT}
Suppose that $A$ is Cohen-Macaulay and $\k.$ is acyclic. Then 
\[
 H^0_\mm(\BB)_\nu=0 \textnormal{ for all }\nu \geq \nu_0 = \paren{\sum_i d_i} -\gamma.
\]
\end{thm}
\begin{proof} 
Write $\KK_q$ for the $q$-th object in $\k.$. Consider the two spectral sequences associated to the double complex $C^\bullet_\mm(\k.)$, both converging to the hypercohomology of $\k.$. As $\k.$ is acyclic the first spectral sequence stabilizes at the $E_2$-term. The second one has as $E_1$-term $_1''E^p_q =\ H^p_\mm(\KK_q)$. 

Since $H_0(\k.)=\BB$ (cf.\ Remark \ref{KresB}), the comparison of the two spectral sequences shows that $H^0_\mm(\BB)_\nu$ vanishes as soon as $(_1{''}E^{p}_p)_\nu$
vanishes for all $p$. In fact we have 
\[
 \endd(H^0_\mm(\BB))\leq \max_{p\geq 0}\{\endd(_1{''}E^{p}_p)\}=\max_{p\geq 0}\{\endd(H^p_\mm(\KK_p))\}.
\]
It remains to observe that, since $\KK_p=\bigoplus_{i_1,\hdots,i_p} A(-\sum_{j=1}^p d_{i_j})\otimes_\kk \kk[\X](-p)$ and $\kk[\X]$ is flat over $\kk$, 
\[
 \max_{p\geq 0}\{\endd(H^p_\mm(\KK_p))\}=\max_{p\geq 0} \{\max_{i_1,\hdots,i_p}\{\endd(H^p_\mm(A(-\sum_{j=1}^p d_{i_j})))\}\}.
\]
Hence, as $A$ is CM, we have
\[
 \endd(H^p_\mm(\KK_p))= \left\lbrace\begin{array}{ll}\endd(H^n_\mm(\omega_A^\vee(-\sum_i d_i)))& \mbox{for }p=n, \\
0 & \mbox{otherwise.} \end{array}\right.
\]
Finally, since $(H^n_\mm(\omega_A^\vee))_\nu=0$ for all $\nu\geq -\gamma$, we get 
\[
  \endd(H^n_\mm(\omega_A^\vee(-\sum_i d_i)))=\endd(H^n_\mm(\omega_A^\vee)))+\sum_i d_i<\sum_i d_i -\gamma. \qedhere
\]
\end{proof}

In order to compute the representation matrix $\Xi_\nu$ and the implicit equation of $\phi$, we will get acyclicity conditions for $\k.$ from L. Avramov's Theorem (cf.\ \ref{avramov}). As we mentioned above, this theorem gives necessary and sufficient conditions for $(L_1,\hdots,L_n)$ to be a regular sequence in $R$ in terms of the depth of certain ideals of minors of the matrix $\Xi:=(m_{ij})_{i,j}\in \mat_{2n,n}(A)$ defined in \eqref{matrizota}. 

Recall that $I_r:=I_r(\Xi)$ for the ideal of $A$ generated by the $r\times r$ minors of $\Xi$, for $0\leq r \leq r_0:=\min\{n+1,m\}$, and that $I_0:=A$ and $I_r:=0$ for $r>r_0$. The following result relates both algebraic and geometric aspects. It gives conditions in terms of the ideals of minors $I_r$, for the complex to being acyclic, and on the equation given by the determinant of a graded branch for describing the closed image of $\phi$.


\section{The implicitization problem}

Here we generalize the work in \cite{Bot08}. Hereafter in this chapter, let $\Tc$ be a $(n-1)$-dimensional projective arithmetically Cohen Macaulay closed scheme over a field $\kk$, embedded in $\PP^N_{\kk}$, for some $N\in \NN$. Write $A=\kk[T_0,\hdots,T_N]/J$ for its CM graded coordinate ring, and let $J$ denote the homogeneous defining ideal of $\Tc$. Set $\T:= T_0,\hdots,T_N$ the variables in $\PP^N$, and $\X$ the sequence $X_1,Y_1,\hdots,X_n,Y_n$, of variables in $(\PP^1)^n$. Write $\mm:=A_+=(\T)\subset A$ for the maximal irrelevant homogeneous ideal of $A$.

Let $\phi$ be a finite map over a relative open set $U$ of $\Tc$ defining a hypersurface in $\PP^n$:
\begin{equation}\label{eqSettingP1n}
 \PP^N \supset \Tc \nto{\phi}{\dto} (\PP^1)^n : \T \mapsto (f_1:g_1)\times\cdots\times(f_n:g_n)(\T),
\end{equation}
where $f_i$ and $g_i$ are homogeneous elements of $A$ of degree $d_i$, for $i=1,\hdots,n$. As in the section before, this map $\phi$ gives rise to a morphism of graded $\kk$-algebras in the opposite sense
\begin{equation}
 \kk[\X] \nto{\phi^\ast}{\lto} A: X_i \mapsto f_i(\T), Y_i \mapsto g_i(\T).
\end{equation}
Since $\ker(\phi^\ast)$ is a principal ideal in $\kk[\X]$, write $H$ for a generator. We proceed as in \cite{Bot08} to get a matrix such that the gcd of its maximal minors equals $H^{\deg(\phi)}$, or a multiple of it.

\medskip

Assume that we are given a rational map like the one in \eqref{initsetting} with $\deg(f_i)=\deg(g_i)=d_i$, $i=1,\hdots,n$. Take $\Tc\subset \PP^N$ the toric embedding obtained from $\Nc(f)$ cf.\ Definition \ref{defNf}. Recall from \ref{RemToricCaseP1n} that the multi-projective compactification is given by
\begin{equation}
 \AA^n \nto{\iota}{\hto}(\PP^1)^n: (x_1,\hdots,x_n) \mapsto (x_1:1)\times\cdots\times(x_n:1).
\end{equation}
As before, $f$ compactifies via $\rho$ and $\iota$ through $\Tc$ to $\phi:\Tc \dto(\PP^1)^n$ as defined in \eqref{eqSettingP1n}, that is $\iota\circ f=\phi\circ\rho$. 

\medskip

\subsection{The implicit equation}

In this part we gather together the facts about acyclicity of the complex $\k.$, and the geometric interpretation of the zeroes of the ideals of minors $I_r$. We show that under suitable hypotheses no extraneous factor occurs. One very important difference from Chapter \ref{ch:toric-emb-pn}, is that the base locus of $\phi$ has always codimension $2$, instead of being zero-dimensional. This makes slightly more complicated the well understanding of the geometry of the base locus, and hence, the nature of the extraneous factor. In order to do this, we introduce some previous notation, following that of \cite{Bot08}.

Denote by $W$ the closed subscheme of $\Tc \subset \PP^N$ given by the common zeroes of all $2n$ polynomials $f_i,g_i$, write $I^{(i)}$ for the ideal $(f_i,g_i)$ of $A$, and $X$ the base locus of $\phi$ defined in \ref{WX}, namely
\begin{equation}\label{WX}
 W:=\Proj\paren{A/\sum_{i}I^{(i)}}, \textnormal{ and } \qquad X:=\Proj\paren{A/\prod_{i}I^{(i)}}.
\end{equation}
\begin{defn}\label{defXyOmegaP1xxP1}
 We call $\Omega$ the complement of the base locus, namely $\Omega:=\Tc \setminus X$. Let $\Gamma_\Omega$ be the graph of $\phi$ or $\phi$ inside $\Omega \times (\PP^1)^n$. 
\end{defn}

Set $\alpha\subset [1,n]$, write $I^{(\alpha)}:=\sum_{j\in \alpha} I^{(j)}$, and set $X_{\alpha}:=\Proj(A/I^{(\alpha)})$ and $U_{\alpha}:=X_{\alpha}\setminus\bigcup_{j\notin \alpha}X_{\{j\}}$. If $U_{\alpha}$ is non-empty, consider $p\in U_{\alpha}$, then $\dim(\pi_1^{-1}(p))=|\alpha|$. As the fiber over $U_\alpha$ is equidimensional by construction, write
\begin{equation}\label{VBunbleAlpha}
 \EEE_{\alpha}:=\pi_1^{-1}(U_{\alpha})\subset \Tc\times (\P1)^{n}
\end{equation}
for the fiber over $U_{\alpha}$, which defines a multiprojective bundle of rank $|\alpha|$. Consequently, 
\[
 \codim(\EEE_{\alpha})=n-|\alpha|+(\codim_{\Tc}(U_{\alpha})).
\]
Recall from Definition \ref{defXyOmegaP1xxP1} that $\Gamma_\Omega$ is the graph of $\phi$, and set $\Gamma:=\Biproj(\BB)$, the incidence scheme of the linear forms $L_i$. We show in the following theorem that under suitable hypothesis $\Gamma=\overline{\Gamma_\Omega}$, and that $\pi_2(\Gamma)=\Hc$ the implicit equation of the closed image of $\phi$.

\begin{thm}\label{teoRes} Let $\phi: \Tc\dashrightarrow (\P1)^{n}$ be defined by the pairs $(f_i:g_i)$, not both being zero, as in \eqref{eqSettingP1n}. Write for $i=1,\dots,n$, $L_i:=\fxi$ and $\BB:=A[\X]/(L_1,\hdots,L_n)$. Take $\nu_0 = (\sum_i d_i) -\gamma$ as in Theorem \ref{locosymalgT}.
\begin{enumerate}
\item The following statements are equivalent:
 \begin{enumerate}
  \item \label{teo-kos} $\k.$ is a free resolution of $\BB$;
  \item \label{teo-dep} $\codim_A(I_r)\geq n-r+1$ for all $r=1,\hdots,n$;
  \item \label{teo-cod} $\dim \paren{\underset{\alpha\subset [1,n], |\alpha|=r}{\bigcap} V \paren{\prod_{j\in \alpha}I^{(j)}}}\leq r-2$ for all $r=1,\hdots,n$. 
 \end{enumerate}

\item\label{teo-equ} If any (all) of the items above are satisfied, then $M_\nu$ has generically maximal rank, namely $\binom{n-1+\nu}{\nu}$. Moreover, if for all $\alpha\subset [1,n]$, $\codim_A(I^{(\alpha)})> |\alpha|$, then, 
\[
 \det((\k.)_\nu)=\det(M_\nu)=H^{\deg(\phi)}, \quad \textnormal{for }\nu\geq\nu_0,
\]
where $\det(M_\nu)$ and $H$ is the irreducible implicit equation of the closed image of $\phi$.
\end{enumerate}
\end{thm}

\begin{proof}
\eqref{teo-kos} $\Leftrightarrow$ \eqref{teo-dep} follows from Avramov's Theorem \ref{avramov}.

\eqref{teo-dep} $\Leftrightarrow$ \eqref{teo-cod} Note that each $r\times r$-minor of $M$ can be expressed as a product of $r$ polynomials, where for each column we choose either $f$ or $g$. Then, the ideal of minors involving the columns $i_1,\hdots,i_{r}$ coincides with the ideal $I^{(i_1)}\cdots I^{(i_{r})}$. Since we have assumed that for any $i$ $f_i\neq 0$ or $g_i\neq 0$, the condition $\dim( V(I^{(1)}\cdots I^{(n)}))\leq n-2$ is automatically satisfied.

\eqref{teo-kos} $\Rightarrow$ \eqref{teo-equ} is a classical result, first studied by J.-P. Jouanolou in \cite[$\S$3.5]{Jou2}, reviewed in \cite{GKZ94}, and also used by L. Bus\'e, M. Chardin and J-P. Jouanolou, in their previous articles in the area.

\medskip
For proving the second part of point \ref{teo-equ}, the hypotheses have been taken in such a way that $\codim_A(\sum_{j\in \alpha}I^{(j)})> |\alpha|$, for all $\alpha\subset [1,n]$, which implies that $\codim_{\Tc}(U_{\alpha})>|\alpha|$, thus 
\[
\codim(\EEE_{\alpha})>n=\codim(\Gamma_\Omega). 
\]
Set $\Gamma_U:=\coprod_{\alpha}\EEE_{\alpha}$, and observe that $\Gamma\setminus\Gamma_U=\Gamma_\Omega$. Clearly, $\codim(\Gamma_U)>n=\codim(\Gamma_\Omega)=\codim(\overline{\Gamma_\Omega})$.

Since $\Spec (\BB)$ is a complete intersection in $\AA^{2n}$, it is unmixed and purely of codimension $n$. As a consequence, $\Gamma\neq \emptyset$ is also purely of codimension $n$. This and the fact that $\codim (\Gamma_U)>n$ implies that $\Gamma=\overline{\Gamma_\Omega}$. The graph $\Gamma_\Omega$ is irreducible hence $\Gamma$ as well, and its projection (the closure of the image of $\phi$) is of codimension-one. 

It remains to observe that $\k.$ is acyclic, and $H_0(\k.)\cong\BB$ (cf.\ Remark \ref{KresB}). Considering the homogeneous strand of degree $\nu>\eta$ we get the following chain of identities (cf.\ \cite{KMun}):
\[
 \begin{array}{rl}
[\det((\k.)_\nu)]&=\div_{\kk[\X]}(H_0(\k.)_\nu)\\
&=\div_{\kk[\X]}(\BB_\nu)\\
&=\sum_{\pp \textnormal{ prime, }\codim_{\kk[\X]}(\pp)=1}\length_{\kk[\X]_\pp} ((\BB_\nu)_\pp)[\pp]. 
\end{array}
\]
Our hypothesis were taken in such a way that only one prime occurs. Also since 
\[
 [\det((\k.)_\nu)]=\div_{\kk[\X]}(\res)= e \cdot [\qq],
\]
for some integer $e$ and $\qq:=(H)\subset \kk[\X]$, we have that 
\[
 \sum_{\pp \textnormal{ prime, }\codim(\pp)=1}\length_{\kk[\X]_\pp} ((\BB_\nu)_\pp)[\pp]=e\cdot [\qq],
\]
and so $[\det((\k.)_\nu)]=\length_{\kk[\X]_\qq} ((\BB_\nu)_\qq)[\qq]$. Denote $\kkk(\qq):=\kk[\X]_\qq/\qq \cdot \kk[\X]_\qq$. Since $\Gamma$ is irreducible, we have
\[
 \length_{\kk[\X]_\qq} ((\BB_\nu)_\qq)=\dim_{\kkk(\qq)}{(\BB_\nu\otimes_{\kk[\X]_\qq} \kkk(\qq))}=\deg(\phi),
\]
 which completes the proof.
\end{proof}

\begin{rem}\label{remHdivRes}
We showed that the scheme $\pi_2(\Gamma)$ is defined by the polynomial $\det(M_\nu)$, while the closed image of $\phi$ coincides with $\pi_2(\overline{\Gamma_\Omega})$, hence the polynomial $H$ divides $\det(M_\nu)$. Moreover, from the proof above we conclude that $H^{\deg(\phi)}$ also divides $\det(M_\nu)$. And if $[\overline{\EEE_{\alpha}}]$ is an algebraic cycle of $\Tc\times (\P1)^{n+1}$ of codimension $n+1$, then $[\overline{\pi_2(\EEE_{\alpha})}]$ is not a divisor in $(\P1)^{n+1}$, and consequently $\det(M_\nu)$ has no other factor than $H^{\deg(\phi)}$.
\end{rem}

\medskip

\begin{rem}
 With the hypotheses of Theorem \ref{teoRes} part \ref{teo-equ}, assuming $\Tc=\PP^{n-1}$, denoting by $\deg_i$ the degree on the variables $x_i,y_i$ and by $\deg_{tot}$ the total one, we have:
\begin{enumerate}
 \item $\deg_i(H)\deg(\phi)=\prod_{j\neq i}d_j$;
 \item $\deg_{tot}(H)\deg(\phi)=\sum_i\prod_{j\neq i}d_j$.
\end{enumerate}
\end{rem}

\subsection{Analysis of the extraneous factors}\label{sec:extraneousFactorP1xxP1}

Theorem \ref{teoRes} can be generalized (in the sense of \cite[Sec.\ 4.2]{Bot08}) taking into account the fibers in $\Tc\times (\PP^1)^n$ that give rise to extraneous factors, by relaxing the conditions on the ideals $I_r$ stated in Theorem \ref{teoRes}. Recall from \eqref{WX} that $W:=\Proj(A/\sum_{i}I^{(i)})$ and $X:=\Proj(A/\prod_{i}I^{(i)})$, and that for each $\alpha\subset [1,n]$, $I^{(\alpha)}:=\sum_{j\in \alpha} I^{(j)}$, $X_{\alpha}:=\Proj(A/I^{(\alpha)})$ and $U_{\alpha}:=X_{\alpha}\setminus\bigcup_{j\notin \alpha}X_{\{j\}}$. As was defined in \eqref{VBunbleAlpha}, $ \EEE_{\alpha}:=\pi_1^{-1}(U_{\alpha})\subset \Tc\times (\P1)^{n}$ is a multiprojective bundle of rank $|\alpha|$ over $U_{\alpha}$, such that $\codim(\EEE_{\alpha})=n-|\alpha|+(\codim_{\Tc}(U_{\alpha}))$.

In order to understand this, we will first analyze some simple cases, namely, where this phenomenon occurs over a finite set of points of the base locus; and later, we will deduce the general implicitization result.

\begin{exmp}
 Assume we are given a rational map $\phi:\PP^2\dashrightarrow \P1 \times \P1 \times \P1$, where $\phi(u:v:w)=(f_1(u,v,w):g_1(u,v,w)) \times (f_2(u,v,w):g_2(u,v,w)) \times (f_3(u,v,w):g_3(u,v,w))$, of degrees $d, d'$ and $d''$ respectively. 

We may suppose that each of the pairs of polynomials $\{f_1,g_1\}$, $\{f_2,g_2\}$ and $\{f_3,g_3\}$ have no common factors. Then, the condition $\codim_A(I^{(i)})\geq 2$ is automatically satisfied. Assume also that $W=\emptyset$, this is, there are no common roots to all $6$ polynomials.

We will show here that, if we don't ask for the ``correct" codimension conditions, we could be implicitizing some extraneous geometric objects. For instance, suppose that we take a simple point $p\in V(I^{(1)}+I^{(2)})\neq \emptyset$. 
Consequently $L_1(u,v,w,\X)=L_2(u,v,w,\X)=0$ for all choices of $\X$. Nevertheless, $L_3(u,v,w,\X)=g_3(u,v,w)x_3-f_3(u,v,w)y_3=0$ imposes the nontrivial condition $g_3(p)x_3-f_3(p)y_3=0$ on $(\zz)$, hence there is one point $q=(f_3(p):g_3(p))\in \P1$ which is the solution of this equation. We get $\pi_1^{-1}(p)=\{p\}\times\P1\times \P1\times\{q\}$. As we do not want the reader to focus on the precise computation of this point $q$, we will usually write $\{*\}$ for the point $\{q\}$ obtained as the solution of the only nontrivial equation.

Suppose also that, for simplicity, $V(I^{(1)}+I^{(2)})=\{p\}$, $V(I^{(1)}+I^{(3)})=\emptyset$, and $V(I^{(2)}+I^{(3)})=\emptyset$. This says that if we compute $\pi_2(\Gamma)$, then we get
\[
 \begin{array}{ll}
	\pi_2(\Gamma)\ & = \pi_2(\pi_1^{-1}(\overline{\Omega}\cup X))=\pi_2(\overline{\pi_1^{-1}(\Omega))}\cup \pi_2(\pi_1^{-1}(X))=\\
	& = \pi_2(\overline{\Gamma_\Omega})\cup (\pi_2( \{p\}\times \P1\times \P1\times\{*\})=\\
	& = \overline{\im(\phi)}\cup(\P1\times \P1\times\{*\}),
 \end{array}
\]
where $X =\Proj(A/\prod_{i}I^{(i)})$ is the base locus of $\phi$ as in \eqref{WX}, and $\Omega=\PP^n\setminus X$ its domain.
 
Hence, $\det(\k.(L_1,L_2,L_3)_\nu)=H^{\deg(\phi)}\cdot G$, where $G=L_3(p)$. Indeed, observe that each time there is only one extraneous hyperplane appearing (over a point $p$ with multiplicity one), which corresponds to $\pi_2(\pi_1^{-1}(p))$, then $\pi_1^{-1}(p)$ is a closed subscheme of $\Gamma$,  defined by the equation $L_3(p)=0$. Hence, we get that 
\[
 \det(\k.(L_1,L_2,L_3)_\nu)=\HH^{\deg(\phi)}\cdot L_3(p). 
\]
\end{exmp}

We will now generalize Theorem \ref{teoRes} in the spirit of the example above. For each $i\in \{0,\hdots,n\}$ take $X_{\hat{i}}:=\Proj(A/\sum_{j\neq i} I^{(j)})$. 

\medskip

\begin{prop}\label{casoFibraNdim}
Let $\phi: \Tc\dashrightarrow (\P1)^{n}$ be a rational map that satisfies conditions \eqref{teo-kos}-\eqref{teo-cod} of Theorem \ref{teoRes}. Assume further that for all $\alpha:=\{i_0,\hdots, i_k\}\subset [1,n]$, with $k<n-1$, $\codim_A(I^{(\alpha)})> |\alpha|$.
 Then, there exist non-negative integers $\mu_p$ such that: 
\[
\res_\Tc(L_1,\dots,L_n)=H^{\deg(\phi)}\cdot \prod_{i=1}^n\prod_{p\in X_{\hat{i}}}L_i(p)^{\mu_{p}}.
\]
\end{prop}

\begin{proof}
 Denote by $\Gamma_0:=\overline{\Gamma_\Omega}$ the closure of the graph of $\phi$, $\Gamma$ as before. From Remark \ref{remHdivRes}, we can write 
\[
 G:=\frac{\res_A(L_1,\dots,L_n)}{H^{\deg(\phi)}},
\]
 the extra factor. It is clear that $G$ defines a divisor in $(\P1)^{n}$ with support on $\pi_2(\Gamma\setminus\Gamma_0)$. From the proof of Theorem \ref{teoRes}, we have that $\Gamma$ and $\Gamma_0$ coincide outside $X\times (\P1)^{n}$. As $\Gamma$ is defined by linear equations in the second group of variables, then $\Gamma\setminus\Gamma_0$ is supported on a union of linear spaces over the points of $X$, and so, its closure is supported on the union of the linear spaces $(\pi_1)^{-1}(p)\cong \{p\}\times ((\P1)^{n-1}\times \{*\})$, where $\{*\}$ is the point $(x:y)\in \P1$ such that $L_i(p,x,y)=0$ for suitable $i$. It follows that $\pi_2((\pi_1)^{-1}(p))\subset V(L_i)\subset (\P1)^{n}$, and consequently 
\[
 G=\prod_{p\in X}L_i(p)^{\mu_{p}},
\]
for some non-negative integers $\mu_p$.
\end{proof} 

\begin{lem}\label{lemCodimAvramov}
 Let $\phi: \Tc \dashrightarrow (\P1)^{n}$, be a rational map satisfying condition 1 in Theorem \ref{teoRes}. Then, for all $\alpha\subset [1,n]$, $\codim_A(I^{(\alpha)})\geq |\alpha|$.
\end{lem}

\begin{proof}
To show this we will use Avramov's Theorem \ref{avramov}. Take $\alpha:=\{i_1,\hdots,i_k \}\subset [1,n]$ for $1\leq k\leq n$. Denote by $I$ the ideal $I^{(i_1)}+\cdots+I^{(i_n)}$, $I^{(\alpha)}=\sum_{j=1}^k I^{(i_j)}$ and $I^{(\complement\alpha)}=\sum_{l=k+1}^n I^{(i_l)}$, hence $I=I^{(\alpha)}+I^{(\complement\alpha)}$. As $(L_1,\hdots,L_n)$ is a complete intersection in $R$, also is $(L_{i_1},\hdots,L_{i_k})$ in $A[x_{i_1},y_{i_1},\hdots,x_{i_k},y_{i_k}]$. Applying Avramov's Theorem \ref{avramov} to the ideal $(L_{I_1},\hdots,L_{i_k})$, for $r=1$ we have that $\codim_A(I^{(\alpha)})\geq k=|\alpha|$.

Observe that as $I^{(\alpha)}$ is generated by a subset of the set of generators of $I$ then $I^{(\alpha)}$ is also a complete intersection in $R$. Now, as it is generated by elements only depending on the variables $x_{i_j},y_{i_j}$ for $j=1,\hdots,k$, we have that it is also a complete intersection in $A[x_{i_1},y_{i_1},\hdots,x_{i_k},y_{i_k}]$.
\end{proof}

We define the basic language needed to describe the geometry of the base locus of $\phi$.

\begin{defn}\label{defUOmegaXalpha}
For each $\alpha\subset [1,n]$, denote by $\Theta:=\{\alpha\subset [1,n]\ :\ \codim(I^{(\alpha)})=|\alpha|\}$. Hence, let $I^{(\alpha)}=(\cap_{\qq_i\in \Lambda_\alpha} \qq_i)\cap \qq'$ be a primary decomposition, where $\Lambda_\alpha$ is the set of primary ideals of codimension $|\alpha|$, and $\codim_A(\qq')>|\alpha|$. Write $X_{\alpha,i}:=\Proj(A/\qq_i)$ with $\qq_i\in \Lambda_\alpha$, and let $X_{\alpha,i}^{red}$ be the associated reduced variety. 

Write $\alpha:=\{i_1,\hdots,i_k\}\subset [1,n]$, and denote by $\pi_\alpha: (\P1)^{n}\to (\P1)^{n-|\alpha|}$ the projection given by 
\[
 \pi_\alpha: (x_1:y_1)\times \cdots \times(x_n:y_n)\mapsto (x_{i_{k+1}}:y_{i_{k+1}})\times \cdots \times(x_{i_{n}}:y_{i_{n}}).
\]

Set $\PP_\alpha:= \pi_{\alpha}((\P1)^{n})$, and define $\phi_{\alpha}:=\pi_{\alpha}\circ \phi:\Tc\dto \PP_\alpha$. 

 Denote by $W_\alpha$ the base locus of $\phi_{\alpha}$. Clearly $W\subset W_{\alpha}\subset X$ (cf.\ equation \eqref{WX}). Denote $\UU_{\alpha}:=\Tc\setminus W_{\alpha}$, the open set where $\phi_{\alpha}$ is well defined. Write $\Omega_{\alpha}:= X_{\alpha}\cap \UU_{\alpha}$ and $\Omega_{\alpha,i}:= X_{\alpha,i}\cap \UU_{\alpha}$. If $\alpha$ is empty, we set $\pi_\alpha=Id_{(\P1)^{n}}$, $\phi_\alpha=\phi$, $W_\alpha=W$ and $\UU_{\alpha}=\Omega_{\alpha}=\Omega$.

We get a commutative diagram as follows
\[
 \xymatrix@1{
 \Omega_\alpha\ \ar@{^(->}[r]\ar[rrrd]_{\phi_{\alpha}|_{\Omega_\alpha}}& \ X_\alpha\ \ar@{^(->}[r]& \ \Tc\ \ar@{-->}[r]^{\phi}\ar@{-->}[rd]^{\phi_{\alpha}}& \ (\P1)^{n}\ar@{->>}^{\pi_{\alpha}}[d]&\\
 &&&  **[r]\ \PP_\alpha:= \pi_{\alpha}((\P1)^{n}) &\ \ \ \cong (\P1)^{n-|\alpha|}. }
\]
\end{defn}

\begin{rem}
 Let $p\in \Tc$ be a point, then there exist a unique pair $(\alpha,i)$ such that $p\in \Omega_{\alpha,i}$. If $p\in W$, then $\alpha=\emptyset$ and no $i$ is considered.
\end{rem}
\begin{proof}
 It is clear by definition of $\Omega_\alpha$ that if $p\in W$, then $\alpha=\emptyset$ and no $i$ needs to be considered. Hence, assume that $p\in \Tc\setminus W$. Thus, we define $\alpha:=\{i\in [1,n]\ :\ f_i(p)=g_i(p)= 0\}$ which is a non-empty subset of $[1,n]$. For this set $\alpha$, define $\phi_\alpha$ according to Definition \ref{defUOmegaXalpha}, set $W_\alpha$ the base locus of $\phi_\alpha$ and $X_\alpha:=\Proj(A/I^{(\alpha)})$. By definition, $p\in \Omega_\alpha:=X_\alpha\setminus W_\alpha$. Since, in particular, $p\in X_\alpha$, it is one of its irreducible components that we denote by $X_{\alpha,i}$ following the notation of Definition \ref{defUOmegaXalpha}. We conclude that $p\in \Omega_{\alpha,i}:=X_{\alpha,i}\setminus W_\alpha$, from which we obtain the $(\alpha,i)$ of the statement.
\end{proof}

In the following lemma we define a multiprojective bundle of rank $|\alpha|$ over $\Omega_{\alpha,i}$.

\begin{lem}\label{bundle} 
 For $\phi$ as in Theorem \ref{teoRes}, and for each $\alpha \in \Theta$ and each $\qq_i\in \Lambda_\alpha$, the following statements are satisfied:
\begin{enumerate}
 \item $\Omega_{\alpha,i}$ is non-empty
 \item for all $p\in \Omega_{\alpha,i}$, $\dim(\pi_1^{-1}(p))=|\alpha|$
 \item the restriction $ \phi_{\alpha,i}$ of $\phi$ to $\Omega_{\alpha,i}$, defines a rational map 
\begin{equation}\label{phiai}
 \phi_{\alpha,i}:X_{\alpha,i}\dto \PP_\alpha\cong (\P1)^{n-|\alpha|}.
\end{equation}
\item $Z_{\alpha,i}:=\pi_1^{-1}(\Omega_{\alpha,i})\nto{\pi_1}{\lto}\Omega_{\alpha,i}$ defines a multiprojective bundle $\EEE_{\alpha,i}$ of rank $|\alpha|$ over $\Omega_{\alpha,i}$.
\end{enumerate}
\end{lem}

\begin{proof}
Fix $X_{\alpha,i}\subset X_{\alpha}$ and write $\alpha:=i_1,\hdots,i_k$. As $\Omega_{\alpha,i}=X_{\alpha,i}\setminus \bigcup_{j\notin \alpha}X_{\{j\}}$ it is an open subset of $X_{\alpha,i}$. If $\Omega_{\alpha,i}=\emptyset$ then $X_{\alpha,i}\subset \bigcup_{j\notin \alpha}X_{\{j\}}$, and as it is irreducible, there exists $j$ such that $X_{\alpha,i}\subset X_{\{j\}}$, hence $X_{\alpha,i}\subset X_{\{j\}}\cap X_\alpha= X_{\alpha\cup \{j\}}$. Denote by $\alpha':= \alpha\cup \{j\}$, it follows that $\dim(X_{\alpha'})\geq \dim(X_{\alpha,i})=n-|\alpha|>n-|\alpha'|$, which contradicts the hypothesis.

Let $p\in \Omega_{\alpha,i}$, $\pi_1^{-1}(p) = \{p\}\times\{q_{i_{k+1}}\}\times\cdots\times\{q_{i_{n}}\}\times(\P1)^{|\alpha|}$, where the point $q_{i_{j}}\in \P1$ is the only solution to the nontrivial equation $L_{i_{j}}(p,x_{i_{j}},y_{i_{j}})=y_{i_{j}}f_{i_{j}}(p)-x_{i_{j}}g_{i_{j}}(p)=0$. Then we deduce that $\dim(\pi_1^{-1}(p))=|\alpha|$, and that $\phi_{\alpha,i}:\Omega_{\alpha,i}\to \PP_\alpha:=\pi_\alpha((\P1)^{n})\cong (\P1)^{n-|\alpha|}$ given by $p\in \Omega_{\alpha,i}\mapsto \{q_{i_{k+1}}\}\times\cdots\times\{q_{i_{n}}\}\in \PP_\alpha$, is well defined.

The last statement follows immediately from the previous ones. 
\end{proof}

We get the following result which generalizes Proposition \ref{casoFibraNdim}.

\begin{thm}\label{teoResGral}  Let $\phi: \Tc\dashrightarrow (\P1)^{n}$ be defined by the pairs $(f_i:g_i)$, not both being zero, as in equation \eqref{eqSettingP1n}. Assume that $\codim_A(I_r)\geq n-r+1$ for all $r=1,\hdots,n$. Denote by $H$ the irreducible implicit equation of the closure of its image. Then, there exist relative open subsets, $\Omega_{\alpha,i}$, of $\Tc$ such that the restriction $\phi_{\alpha,i}$ of $\phi$ to $\Omega_{\alpha,i}$ defines a rational map $\phi_{\alpha,i}:\Omega_{\alpha,i}\to \PP_\alpha\cong (\P1)^{n-|\alpha|}$ and positive integers $\mu_{\alpha,i}$ such that:
\[
 \res_{\Tc}(L_0,\hdots,L_n)=H^{\deg(\phi)}\cdot \prod_{\alpha,i} (H_{\alpha,i})^{\mu_{\alpha,i}\cdot\deg(\phi_{\alpha,i})}.
\]
\end{thm}
\begin{proof}
 The proof of this result follows similar lines of that of \cite[Thm.\ 22]{Bot08}. Recall $\Gamma:=\Biproj(\BB)$, and set $\Gamma_0:=\overline{\Gamma_\Omega}$, the closure of the graph of $\phi$. Applying $\pi_2$ to the decomposition $\Gamma\setminus \Gamma_U=\Gamma_0$ we see that $[\pi_2(\Gamma_U)]=[\res_{\Tc}(L_0,\dots,L_n)]-[\pi_2(\Gamma_0)]$ is the divisor associated to the extraneous factors. It is clear that $[\pi_2(\Gamma_U)]$ defines a principal divisor in $(\P1)^{n}$ denote by  $G=\frac{\res_{\Tc}(L_0,\dots,L_n)}{H^{\deg(\phi)}}$, with support on $\pi_2(\Gamma\setminus \Gamma_0)$, and that $\Gamma$ and $\Gamma_0$ coincide outside $X\times (\P1)^{n}$. 

By Lemma \ref{bundle}, for each $\alpha$ and each $\qq_i\in \Delta_\alpha\subset\Lambda_\alpha$, $\phi_{\alpha,i}$ defines a multiprojective bundle $\EEE_{\alpha,i}$ of rank $|\alpha|$ over $\Omega_{\alpha,i}$. 

By definition of $\Delta_\alpha$, $\overline{\pi_2(\EEE_{\alpha,i})}$ is a closed subscheme of $(\P1)^{n}$ of codimension-one. Denoting by $[\overline{\EEE_{\alpha,i}}]={\mu_{\alpha,i}}\cdot[\overline{\EEE_{\alpha,i}^{red}}]$ the class of $\overline{\EEE_{\alpha,i}}$ as an algebraic cycle of codimension $n$ in $\PP^{n-1}\times(\P1)^n$, we have $(\pi_2)_*[\overline{\EEE_{\alpha,i}}]= {\mu_{\alpha,i}}\cdot(\pi_2)_*[\overline{\EEE_{\alpha,i}^{\ red}}]= {\mu_{\alpha,i}}\cdot\deg(\phi_{\alpha,i})\cdot[\pp_{\alpha,i}]$, where $\pp_{\alpha,i}:=(H_{\alpha,i})$. 

As in Theorem \ref{teoRes}, one has for $\nu>\eta$:
\[
 \begin{array}{rl}
[\det((\k.)_\nu)]=&\div_{k[\X]}(H_0(\k.)_\nu)\\
&=\div_{k[\X]}(\BB_\nu)\\
&=\sum_{\pp \textnormal{ prime, }\codim_{k[\X]}(\pp)=1}\length_{k[\X]_\pp} ((\BB_\nu)_\pp)[\pp].
\end{array}
\]
We obtain that
\[
 [\det((\k.)_\nu)]= \sum_{\alpha\in \Theta} \sum_{\pp_{\alpha,i}} \length_{k[\X]_{\pp_{\alpha,i}}} ((\BB_\nu)_{\pp_{\alpha,i}})[\pp_{\alpha,i}]+\length_{k[\X]_{(H)}} ((\BB_\nu)_{(H)})[(H)].
\]

In the formula above, for each $\pp_{\alpha,i}$ we have
\[ 
\length_{k[\X]_{\pp_{\alpha,i}}} ((\BB_\nu)_{\pp_{\alpha,i}})=\dim_{\kk(\pp_{\alpha,i})}{(\BB_\nu\otimes_{k[\X]_{\pp_{\alpha,i}}} \kk(\pp_{\alpha,i}))}={\mu_{\alpha,i}}\cdot\deg(\phi_{\alpha,i}),
\]
where $\kk(\pp_{\alpha,i}):=k[\X]_{\pp_{\alpha,i}}/\pp_{\alpha,i}\cdot k[\X]_{\pp_{\alpha,i}}$.

Consequently we get that for each $\alpha\in \Theta$, there is a factor of $G$, denoted by $H_{\alpha,i}$, that corresponds to the irreducible implicit equation of the scheme theoretic image of $\phi_{\alpha,i}$, raised to a certain power $\mu_{\alpha,i}\cdot\deg(\phi_{\alpha,i})$.
\end{proof} 

\begin{rem}\label{rem-notHypersurface}
 Observe that if $\im(\phi_{\alpha,i})$ is not a hypersurface in $\PP_\alpha$ then $\deg(\phi_{\alpha,i})$ is $0$, hence $(H_{\alpha,i})^{\mu_{\alpha,i}\cdot\deg(\phi_{\alpha,i})}=1$. Thus $\phi_{\alpha,i}$ does not give an extraneous factor.
\end{rem}

\chapter{The algorithmic approach}
\label{ch:algorithmic-approach}

\section{Hilbert and Ehrhart functions}\label{sec5algorithmic}

In this section we focus on the study of the size of the matrices $M_\nu$ obtained in the two cases: $\PP^n$ and $(\PP^1)^n$ developed in Chapters \ref{ch:toric-emb-pn} and \ref{ch:toric-emb-p1xxp1} respectively. Let us analyze first the case of $\PP^n$, thus, where we get a map $\varphi:\Tc \dto \PP^n$ as defined in \eqref{eqSettingPn}. Assume also that the base locus of $\varphi$ is a zero-dimensional almost locally complete intersection scheme. Hence, the associated $\Zc$-complex is acyclic. We have shown in Section \ref{sec3Pn} that the matrix $M_\nu$ is obtained as the right-most map of the $(\nu,*)$-graded strand of the approximation complex of cycles $\Zc_\bullet(\h,A)_{(\nu,*)}$:
\begin{equation*}
\ 0 \to (\Zc_n)_{(\nu,*)}(-n) \to (\Zc_{n-1})_{(\nu,*)}(-(n-1)) \to \cdots \to (\Zc_1)_{(\nu,*)}(-1) \nto{M_\nu}{\to} (\Zc_0)_{(\nu,*)} .
\end{equation*}
Given a graded $A$-module $B$, write $h_B(\mu):=\dim_\kk(B_\mu)$ for the Hilbert function of $B$ at $\mu$. Since $\Zc_i= Z_i[i \cdot d] \otimes_A A[\X]=Z_i[i \cdot d] \otimes_\kk \kk[\X]$, $(\Zc_i)_{(\nu,*)}=(Z_i[i \cdot d])_\nu \otimes_\kk \kk[\X]$, we have $M_\nu \in \mat_{h_{A}(\nu),h_{Z_1}(\nu+d)}(\kk[\X])$.

\medskip

Consider the $(\P1)^n$ compactification of the codomain, and assume we are given a map $\phi: \Tc \dto (\P1)^n$ as the one considered in \eqref{eqSettingP1n}, satisfying the conditions of Theorem \ref{teoResGral}. We obtain the matrix $M_\eta$ computed from the Koszul complex $(\k.)_{(\eta,*)}$. Hence, the matrix $M_\eta$ belongs to $\mat_{h_A(\eta),n h_A(\eta-d)}(\kk[\X])$.

Both numbers $h_{A}(\nu)$ and $h_{Z_1}(\nu+d)$, in the projective and multiprojective setting, can be computed easily in Macaulay2. The cost of computation depends on the ring structure of $A$. When $A$ is just any finitely generated $\NN$-graded Cohen-Macaulay $\kk$-algebra, finding a precise theoretical estimate of these numbers would be very difficult. Also, the module structure of $Z_1$ can also be very intricate. Since it is a $\NN$-graded sub-$A$-module of $A^{n+1}$, we have $h_{Z_1}(\nu+d)\leq (n+1)h_{A}(\nu+d)$.

\medskip

Assume now that the ring $A$ is the coordinate ring of a normal toric variety $\Tc$ defined from a polytope $\Nc$, as mentioned in Section \ref{ImageCodim1}, and later in Remarks \ref{RemToricCasePn} and \eqref{RemToricCaseP1n}. In this setting, the situation above can be rephrased in a more combinatorial fashion. Let $\Nc$ be a $(n-1)$-dimensional normal lattice polytope, that is a full-dimensional normal convex polytope in $\RR^{n-1}$ with vertices lying in $\ZZ^{n-1}$ . For any integer $k \geq 0$, the multiple $k\Nc=\{p_1+\cdots+p_k\ :\ p_i\in \Nc\}$ is also a lattice polytope, and we can count its lattice points. The function taking each integer $k\in \NN$ to the number $E_\Nc(k) = \#((k\Nc) \cap \ZZ^{n-1})$ of lattice points in the polytope $k\Nc$ is the \textit{Ehrhart function} of $\Nc$ (cf.\ \cite{MS}). Write $E^+_\Nc(k)= \#\relint((k\Nc) \cap \ZZ^{n-1})$, the number of integer points in the interior of $k\Nc$ (cf.\ \cite{Latte} for a software for computing those numbers). It is known that there is an identification between $\kk[\relint(C)]$ and $\omega_A$, hence, this can be understood as $E^+_\Nc(k)= h_{\omega_A}(k)$.

Let $C$ be the cone in $\RR^{n-1} \times \RR$ spanned in degree $1$ by the lattice points in the polytope $\Nc$, which is normal by assumption, hence $A$ is Cohen-Macaulay (cf.\ \ref{NfAlwaysNormal}). Assume $\Nc'$ stands for some integer contraction of $\Nc$ which is also normal and take $d\in \NN$ such that $d\Nc'=\Nc$. Then $A'=\kk[\Nc']$ its Cohen-Macaulay semigroup ring. As $d\Nc'=\Nc$, we have that $E_{\Nc'}(d \mu)=E_{\Nc}(\mu)$ for all $\mu$. Set $\gamma:= a_n(A)=\inf\{\mu\ :\ (\omega_A^\vee)_\mu=0\}$ and $\gamma':= a_n(A')=\inf\{\mu\ :\ (\omega_{A'}^\vee)_\mu=0\}$. As $(\omega_A^\vee)_{\mu} = \Hom_\kk(M_{-\mu},\kk)$, we have that $\gamma= \max \{ i \ : \ C_i \ \textnormal{contains } \textnormal{no } \textnormal{interior } \textnormal{points}\}$, where $C_i:=C\cap \ZZ^{n-1}\times \{i\}$, and similarly for $\gamma'$. For a deeper understanding we refer the reader to \cite[Sec.\ 5]{BH}.

Both $A$ and $A'$ give rise to two different -but related- implicitization problems, the following result gives a condition on the rings $A$ and $A'$ to decide when it is algorithmically better to choose one situation or the other.
 
\begin{lem} Take $\Nc$, $\Nc'$, $d$, $\gamma$ and $\gamma'$ as above. Then
 \begin{enumerate}
  \item $\gamma\geq \gamma'$;
  \item $d (\gamma'+1)\geq \gamma+1$;
 \end{enumerate}
\end{lem}
\begin{proof}
 As $d\geq 1$, we can assume $\Nc'\subset \Nc$, hence, the first item follows. For the second item, we just need to observe that if $\mu \Nc \cap \ZZ^n$ is nonempty, then $\mu d \Nc' \cap \ZZ^n$ neither it is. Taking $\mu$ the smallest positive integer with this property, and writing $\gamma=\mu+1$, the second item follows.
\end{proof}

\begin{rem}
 Is not true in general that $d (\gamma+1)> \gamma'+1$: take $\Nc$ as the triangle with vertices $(3,0)$, $(0,3)$, $(0,0)$ and $\Nc'$ the triangle with vertices $(1,0)$, $(0,1)$, $(0,0)$; hence $d=3$, $\gamma=0, \gamma'=2$. We obtain $d (\gamma +1)=3=\gamma' +1$, which shows also that $d \gamma$ need not be bigger than $\gamma'$. It is neither true that $d(\gamma+1)= \gamma'+1$, for instance, take $\Nc$ as the triangle with vertices $(4,0)$, $(0,4)$, $(0,0)$ and $\Nc'$ as before. Observe that $d (\gamma +1)=4(0+1)=4>\gamma' +1=2+1=3$.
\end{rem}
\medskip

\begin{lem}
 Take $\Nc$ be a normal polytope, let $\Nc'$ and $d$ be such that $d \Nc'=\Nc$. Set $\nu_0:=(n-1)-\gamma$ (the bound established in \ref{annih}), and $\nu_0'=d(n-1)-\gamma'$. Write $\delta :=d (\gamma+1)- (\gamma'+1)$. Then $E_\Nc(\nu_0)>E_{\Nc'}(\nu_0')$ if and only if $\delta>d-1$. 
\end{lem}
\begin{proof}
 We have seen that $E_{\Nc'}(d \nu_0)=E_{\Nc}(\nu_0)$, hence, it is enough to compare $E_{\Nc'}(d \nu_0)$ and $E_{\Nc'}(\nu_0')$. Writing $d\gamma= \gamma'+\delta-(d-1)$, we have
\begin{equation*}
 E_{\Nc'}(d \nu_0)=E_{\Nc'}(d(n-1)-d\gamma)=E_{\Nc'}(d(n-1)-\gamma'+\delta-(d-1)),
\end{equation*}
from where we deduce that $E_\Nc(\nu_0)>E_{\Nc'}(\nu_0')$ if and only if $\delta>d-1$.
\end{proof}

\medskip

\begin{cor}
 Let $f:\AA^{n-1}\dto \AA^n$ be a rational map as in \eqref{eqSettingPn} with normal polytope $\Nc:=\Nc(f)$. Let $\Nc'$ be a normal polytope and $d$ such that $d \Nc'=\Nc$. Let $\Tc$ and $\Tc'$ be the arithmetically Cohen-Macaulay toric varieties defined from $\Nc$ and $\Nc'$ respectively, and $\varphi:\Tc\subset \PP^{E_\Nc(1)}\dto \PP^n$ and $\varphi':\Tc'\subset \PP^{E_{\Nc'}(1)}\dto \PP^n$.  Take $\nu_0$, $\nu_0'$ and $\delta$ as above. And write $M_{\nu_0}$ and $M'_{\nu_0'}$ the representation matrices of $\im(\varphi)$ and $\im(\varphi')$ respectively. Then $\#\textnormal{rows}(M_{\nu_0})>\#\textnormal{rows}(M'_{\nu_0'})$ if and only if $\delta>d-1$.
 
\end{cor}

\medskip
In the second case, given a map $\phi:\Tc \dto (\P1)^n$ as in Theorem \ref{teoResGral}, we obtain the matrix $M_\nu$ as the right-most matrix from the Koszul complex $(\k.)_{(\nu,*)}:$
\[
 \ 0 \to A_{\nu-n d}\otimes_\kk \kk[\X](-n) \to \cdots \to (A_{\nu-d})^n\otimes_\kk \kk[\X](-1)  \nto{M_\nu}{\lto} A_\nu\otimes_\kk \kk[\X] \to 0, 
\]
It is clear that $M_\nu$ is a $\dim_\kk(A_{\nu})$ by $\dim_\kk((A_{\nu-d})^n)$ matrix. As  $\bigoplus_{k\geq 0}\gen{C_k}_\kk=\kk[C]$ which is canonically isomorphic to $A$, and also $\dim_\kk(A_{\nu})=E_\Nc(\nu)$ and $\dim_\kk((A_{\nu-d})^n)=n E_\Nc(\nu-d)$, hence
\begin{equation}
 M_\nu \in \mat_{E_\Nc(\nu),n E_\Nc(\nu-d)}(\kk[\X]).
\end{equation}


\section{Examples}\label{sec6examples}

In this section we show, in a few examples, how the theory developed in earlier sections works. We first analyze two concrete examples of parametrized surfaces, given as the image of a rational map defined by rational functions with different denominators. There we show how better is not to take common denominator, and regard their images in $(\P1)^3$ and $(\P1)^4$. Later we show how the method is well adapted for generic rational affine maps.


In the later part of this section we invoke a few examples treated by Bus\'e and Chardin in \cite{BC05}. The main idea of this part is showing that the method generalizes the techniques developed loc. cit. and that in this more general setting we find no better contexts. This complements the argumentation of the authors that no better degrees can be found in these cases, by saying that no better domain or codomain compactifications can be found in general in these particular cases.

\subsection{Implicit equations of dimension $2$ and $3$}

\begin{exmp}
 We consider here an example of a very sparse parametrization where the multihomogeneous compactification of the codomain is fairly better than the homogeneous compactification. We have seen this sme example as Example \ref{interestingexample} focusing on the projective compactification of $\AA^3$. Take $n=3$, and consider the affine map
\[
 f: \AA^2 \dto \AA^3: (s,t) \mapsto \paren{\frac{st^6+2}{st^5-3st^3},\frac{st^6+3}{st^4+5s^2t^6},\frac{st^6+4}{2+s^2t^6}}.
\]
 Observe that in this case there is no smallest multiple of the Newton polytope $\Nc(f)$ with integer vertices, hence, $\Nc(f)=\Nc'(f)$ as can be seen in the picture below.
\begin{center}
 \includegraphics[scale=0.9]{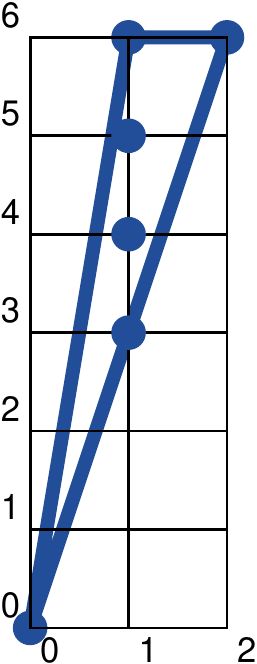}
\end{center}

Computing in Macaulay2 we get that the homogeneous coordinate ring is 
\[
 A=\frac{k[T_0,\ldots,T_5]}{(T_3^2-T_2T_4,T_2T_3-T_1T_4, T_2^2-T_1T_3,T_1^2-T_0T_5)}.
\]
When $\AA^3$ is compactified into $\PP^3$ we obtain from $f$ a new map $\varphi: \Tc\dto \PP^3$ by replacing $(s,t)$ by $T_0,\ldots,T_5$, and taking a common denominator. We can easily see that taking common denominator leads to polynomials of degree up to $23$ and the Newton polytope of the four new polynomials contains $26$ integer points instead of $6$. Again computing in Macaulay2, for $\nu_0=2$, the matrix $M_\nu$ has $351$ rows and about $500$ columns. It can be verified that this compactification gives a base point which is not locally a complete intersection, but locally an almost complete intersection, giving rise to extraneous factors. For more details, see Example \ref{interestingexample}.

On the other hand, compactifying $\AA^3$ into $(\P1)^3$ we get the map
{\small\[
\begin{array}{rcl}
 \phi: \Tc &\dashrightarrow & \PP^1\times \PP^1\times \PP^1\\
(T_0,\ldots,T_5) & \mapsto & (2T_0+T_4:-3T_1+T_3)(3T_0+T_4:T_2+5T_5)(4T_0+T_4:2T_0+T_5)
\end{array}
\]}

Computations in Macaulay2 give that for $\nu_0=3$ the matrix $M_{\nu_0}$ is of size $34 \times 51$. Since there are no base points with two-dimensional fibers, we get no extraneous factors and hence, $H^{\deg (\phi)}$ can be computed as $\frac{\det((34\times 34)\textnormal{-matrix})\cdot \det((1\times 1)\textnormal{-matrix})}{\det((17\times 17)\textnormal{-matrix})}$, getting an equation of degree $(6,6,6)$. For computing the multidegree of the equation, it suffices to observe that the total degree is $34+1-17=18$, since the coefficients on the matrices are all linear. Moreover, just by looking at $\phi$ we see that the degree on each pair of variables must coincide, hence, it has to be $(6,6,6)$.  
\end{exmp}

\medskip

\medskip

\begin{exmp}
 Assume we are given four tuples of polynomials $f_i$, $g_i$, for $i\in [1,4]$, in three variables $s,t,u$. Let them be $f_1=s+tu^2$, $g_1=u^2$, $f_2=st$, $g_2=u^2$, $f_3=su^2$, $g_3=t$, $f_4=stu^2$, $g_4=1$. They define a rational map $f: \AA^3\dto \AA^4$ given by $(s,t,u)\mapsto (f_1/g_1,f_2/g_2,f_3/g_3,f_4/g_4)$.

We compactify $\AA^3$ into the toric variety associated to the smallest multiple of the Newton polytope the input polynomials define. It is easy to see that this polytope $\Nc$ is a $(1\times 1\times 2)$-parallelepiped, and $\Tc\cong (\P1)^3\subset \PP^{11}$.

In order to detect the extraneous factor that occurs, consider the rational map
\[
 \begin{array}{rcl}
  \tilde{\phi}:(\P1)^3&\dto& (\P1)^4\\
(s:s')\times(t:t')\times(u:u')&\mapsto &(\tilde f_1: \tilde g_1)\times(\tilde f_2:\tilde g_2)\times (\tilde f_3:\tilde g_3)\times (\tilde f_4:\tilde g_4),
 \end{array}
\]
where $( \tilde - )$ means homogenizing with respect to the degree $(1,1,2)$ with new variables $s'$, $t'$ and $u'$.

We easily observe that the base locus has codimension $2$, in fact many lines occur in the base locus: There are 
\begin{enumerate}
 \item four lines $\LL_{1}=(1:0)\times(t:t')\times(1:0)$, $\LL_{2}=(1:0)\times(t:t')\times(0:1)$, $\LL_{3}=(0:1)\times(t:t')\times(1:0)$, $\LL_{4}=(0:1)\times(t:t')\times(0:1)$; 
 \item three lines $\LL_{5}=(1:0)\times(1:0)\times(u:u')$, $\LL_{6}=(1:0)\times(0:1)\times(u:u')$, $\LL_{7}=(0:1)\times(1:0)\times(u:u')$; and
 \item three lines $\LL_{8}=(s:s')\times(1:0)\times(1:0)$, $\LL_{9}=(s:s')\times(1:0)\times(0:1)$, $\LL_{10}=(s:s')\times(0:1)\times(0:1)$; 
 \item $7$ points of intersection of the previous lines: $\LL_1\cap \LL_5\cap \LL_8=\{(1:0)\times(1:0)\times(1:0)\}$, $\LL_1\cap \LL_6=\{(1:0)\times(0:1)\times(1:0)\}$, $\LL_2\cap \LL_5\cap \LL_9=\{(1:0)\times(1:0)\times(0:1)\}$, $\LL_2\cap \LL_6\cap \LL_{10}=\{(1:0)\times(0:1)\times(0:1)\}$, $\LL_3\cap \LL_7\cap \LL_8=\{(0:1)\times(1:0)\times(1:0)\}$, $\LL_4\cap \LL_7\cap \LL_9=\{(0:1)\times(1:0)\times(0:1)\}$ and $\LL_4\cap \LL_{10}=\{(0:1)\times(0:1)\times(0:1)\}$.
\end{enumerate}
Over those lines the fiber is of dimension $2$, except over the points of intersection of them.

In the language of Section $4.2$, we have that $W=\emptyset$. The set $\Theta$ formed by the sets $\alpha\subset [1,4]$ giving fibers of dimension $|\alpha|$, is 
\[
 \Theta=\{\{1,2\},\{1,3\},\{1,4\},\{2,3\},\{2,4\},\{3,4\},\{1,2,3\},\{1,3,4\},\{1,2,4\},\{2,3,4\}\}.
\]
Recall that this does not imply that every $\alpha\in \Theta$ will give an extraneous factor (cf.\ Remark \ref{rem-notHypersurface}). We clarify this:

As we have mentioned, the base locus is a union of lines with non-trivial intersection. Take $\alpha= \{1,2\}$. Set-theoretically $X_\alpha = \LL_1 \sqcup \LL_4$, and hence there are two irreducible components of $X_\alpha$, namely $X_{\alpha,1} = \LL_1$ and $X_{\alpha,2} = \LL_4$. The line $X_{\alpha,1} = \LL_1$ only intersects $\LL_5$, $\LL_6$ and $\LL_8$, hence 
\[
 \Omega_{\alpha,1} = \LL_1 \setminus (\LL_5\cap \LL_6\cap \LL_8)= \{(1:0)\times(t:t')\times(1:0)\ :\ t\neq 0 \textnormal{ and } t'\neq 0\}.
\]
\[
 \Omega_{\alpha,2} = \LL_4 \setminus (\LL_7\cap \LL_9\cap \LL_{10})= \{(0:1)\times(t:t')\times(0:1)\ :\ t\neq 0 \textnormal{ and } t'\neq 0\}.
\]
Since $\alpha= \{1,2\}$, the linear forms $L_1(p,\X)$ and $L_2(p,\X)$ vanish identically for all $p\in X_{\alpha}$, while $L_3(p,\X)=f_3(p)Y_3-g_3(p)X_3=t'Y_3$ and $L_4(p,\X)=tY_4$ for $p\in X_{\alpha,1}$. It is easy to note that none of them vanish if and only if $p\in \Omega_{\alpha,1}$. We get that $L_3(p,\X)=tX_3$ and $L_4(p,\X)=t'X_4$ for $p\in X_{\alpha,2}$.

Finally, for $\alpha= \{1,2\}$, we obtain two multiprojective bundles  $\EEE_{\alpha,i}$ over $ \Omega_{\alpha,i}$, for $i=1,2$,
\[
 \EEE_{\alpha,1} :\ \{(1:0)\times(t:t')\times(1:0)\times(\P1)^2\times(t':0)\times(t:0)\ :\ t\neq 0,\ t'\neq 0\} \nto{\pi_1}{\lto} \Omega_{\alpha,1},
\]
\[
 \EEE_{\alpha,2} :\ \{(0:1)\times(t:t')\times(0:1)\times(\P1)^2\times(0:t)\times(0:t')\ :\ t\neq 0,\ t'\neq 0\} \nto{\pi_1}{\lto} \Omega_{\alpha,2}.
\]
Observe that $\im(\phi_{\alpha, 1})=\PP^1\times\PP^1\times(1:0)\times(1:0)$, hence it does not define a hypersurface. Thus, $\phi_{\alpha, 1}$ does not contribute with an extraneous factor. The same for $\phi_{\alpha, 2}$. 

The situation is similar when $\alpha \in \{\{1,3\},\{1,4\},\{2,3\},\{2,4\}\}$, but quite different for $\alpha=\{3,4\}$. Take $\alpha=\{3,4\}$, the linear forms $L_3(p,\X)$ and $L_4(p,\X)$ vanish identically for all $p\in X_{\alpha}$. Take $X_{\alpha,1} = \LL_2$ and $X_{\alpha,2} = \LL_3$. Define $\Omega_{\alpha,1}:=\LL_3\setminus \{(0:1)\times(0:1)\times(1:0), (0:1)\times(1:0)\times(1:0)\}$, and observe that $\phi_{\alpha,1}:\Omega_{\alpha,1}\dto \PP_\alpha$ defines a hypersurface given by the equation $(X_2=0)$. Hence, when $\alpha=\{3,4\}$, $\phi_{\alpha,1}$ does give an extraneous factor.

Now, let us take $\alpha = \{1,2,3\}$ in order to illustrate a different situation. Verifying with the $7$ points listed above, we see that $X_\alpha= \{(1:0)\times(0:1)\times(1:0)\}\cup\{(0:1)\times(0:1)\times(0:1)\}$. Hence, there are two irreducible components $X_{\alpha,1}= \{(1:0)\times(0:1)\times(1:0)\}$ and $X_{\alpha,2}=\{(0:1)\times(0:1)\times(0:1)\}$, and clearly $\Omega_{\alpha,i}= X_{\alpha,i}$ for $i=1,2$. Thus, we get the trivial bundles 
\[
 \EEE_{\alpha,1} :\ \{(1:0)\times(1:0)\times(1:0)\times(\P1)^3\times(1:0)\ :\ t\neq 0 \textnormal{ and } t'\neq 0\} \nto{\pi_1}{\lto} \Omega_{\alpha,1},
\]
\[
 \EEE_{\alpha,2} :\ \{(0:1)\times(0:1)\times(0:1)\times(\P1)^3\times(0:1)\ :\ t\neq 0 \textnormal{ and } t'\neq 0\} \nto{\pi_1}{\lto} \Omega_{\alpha,2}.
\]
These two bundles give rise to the factors $Y_4$ and $X_4$. We conclude with similar argumentation that the extraneous factor is
\[
 G= Y_1^2 X_2 Y_2 Y_3^2 X_4 Y_4.
\]
The degree of the multihomogeneous resultant $\Res_\Nc(L_1,L_2,L_3,L_4)$ in the coefficients of each $L_i$, as polynomials in $s,s',t,t',u$ and $u'$, is equal to $3\cdot 1\cdot 1\cdot 2=6$ for all $i=1,\hdots,4$ by \cite[Prop.\ 2.1, Ch.\ 13]{GKZ94}. So, the total degree of $\det((\k.)_\nu)$ is $24=4\cdot 6$. Indeed, t	he irreducible implicit equation is 
\[
 H= X_4^2 Y_1^2 Y_2^2 Y_3^2+2 X_4 X_2 X_3 Y_1^2 Y_2 Y_3 Y_4-X_4 X_1^2 X_3 Y_2^2 Y_3 Y_4+X_2^2 X_3^2 Y_1^2 Y_4^2,
\]
and $\deg(\phi)=2$. Thus, $\det((\k.)_\nu)=H^2\cdot G$ for $\nu \gg 0$.

\medskip

Let us change now our analysis, and consider the (smallest multiple of) the Newton polytope $\Nc$ of $f_i$ and $g_i$ for $i=1,2,3,4$. We easily see that $\Nc$ is a parallelepiped with opposite extremes in the points $(0,0,0)$ and $(1,1,2)$. For a suitable labeling of the points in $\Nc\cap \ZZ^3$ by $\{T_i\}_{i=0,\hdots,11}$, we have that the toric ideal that defines the toric embedding of $(\AA^\ast)^3 \nto{\iota}{\hto} \PP^{11}$ is

\noindent $J:= I(\Tc) = (T_9 T_{10}-T_8 T_{11}, T_7 T_{10}-T_6 T_{11}, T_5 T_{10}-T_{4} T_{11}, T_3 T_{10}-T_2 T_{11}, T_1 T_{10}-T_0 T_{11}, T^2_9-T_7 T_{11}, T_8 T_9-T_6 T_{11}, T_5 T_9- T_3 T_{11}, T_4 T_9  - T_2 T_{11}  , T_3 T_9  - T_1 T_{11}  , T_2 T_9  - T_0 T_{11}  , T^2_8  - T_6 T_{10}  , T_7 T_8  - T_6 T_9 , T_5 T_8  - T_2 T_{11}  , T_4 T_8  - T_2 T_{10}  , T_3 T_8  - T_0 T_{11}  , T_2 T_8  - T_0 T_{10}  , T_1 T_8 - T_0 T_9 , T_5 T_7  - T_1 T_{11}  , T_4 T_7  - T_0 T_{11}  , T_3 T_7  - T_1 T_9 , T_2 T_7  - T_0 T_9 , T_5 T_6  - T_0 T_{11}  , T_4 T_6  - T_0 T_{10}  , T_3 T_6  - T_0 T_9 , T_2 T_6  - T_0 T_8 , T_1 T_6  - T_0 T_7 , T_3 T_4  - T_2 T_5 , T_1 T_4  - T_0 T_5 , T^2_3  - T_1 T_5 , T_2 T_3  - T_0 T_5 , T^2_2  - T_0 T_4 , T_1 T_2  - T_0 T_3 )$.

\noindent This computation has been done in Macaulay2 using the code in Section \cite{BotAlgo3D}.

The inclusion $\iota:(\AA^\ast)^3 \hto \PP^{11}$ defines a graded morphism of graded rings $\iota^\ast: \kk[T_0,\hdots,T_{11}]/J \to \kk[s,t,u]$. This morphism maps $T_1+T_{10}\mapsto f_1$, $T_7\mapsto g_1$, $T_4\mapsto f_2$, $T_7\mapsto g_2$, $T_6\mapsto f_3$, $T_5\mapsto g_3$, $T_0\mapsto f_4$, and $T_{11}\mapsto g_4$.

Hence, for $\alpha=\{1,2\}$, we have that
\[
 X_{\alpha}=\Proj (\kk[T_0,\hdots,T_{11}]/ (J + (T_1+T_{10},T_4,T_7))).
\]
Using Macaulay2, we can compute the primary decomposition of the radical ideal of $(T_1+T_{10},T_4,T_7)$ in $A:=\kk[T_0,\hdots,T_{11}]/ J$, obtaining the two irreducible components $X_{\alpha,1}$ and $X_{\alpha,2}$. Precisely,
\[
 X_{\alpha,1}=\Proj (\kk[T_0,\hdots,T_{11}]/ (J + (T_{10}, T_8, T_7, T_6, T_4, T_2, T_1, T_0))), \textnormal{ and}
\]
\[
 X_{\alpha,2}=\Proj (\kk[T_0,\hdots,T_{11}]/ (J + (T_{11}, T_7, T_6, T_5, T_4, T_1+T_{10}, T_0))).
\]
After embedding $(\PP^1)^3$ in $\PP^{11}$ via $\iota$, we get that $X_{\alpha,1}=\iota_\ast(L_1)$ and $X_{\alpha,2}=\iota_\ast(L_2)$ which coincides with the situation described above for $\Tc=\P1\times \P1\times \P1$.
\end{exmp}

\subsection{The generic case}

It was shown in \cite{Bot08} and \cite{BDD08}, that suitable compactifications of the source and target of $f$ can really improve the computation time. 

We give here a few examples of affine maps given by rational fractions with very different denominators and as quotients of polynomials of different degree. In this case we see how the different compactifications of the target can vary drastically the size of the matrices we obtain. This example is, in some sense nearer the generic case, where different denominators occur and the polynomials are not of the same degree. Hence, it is easy to construct a big family of examples just by modifying the one below.

\begin{exmp}
 Take $f:\AA^2\dto \AA^3$ given by $(s,t)\mapsto(\frac{s^2+t^2}{st^2},\frac{s^2t^2}{s^2+t^2},\frac{s^2+t^2}{s^2})$. In order to be able to compactify the target in $\PP^3$, we take common denominator. This process increases the degrees of the maps by $3$ and $4$. This shows how ''fictitious'' can be in some cases to take common denominator. The consequences of this phenomena is that the Newton polytope $N$ one obtains from the new $4$ polynomials is really big, in fact, it has $14$ integer points. Hence $\Tc$ embeds in $\PP^{13}$. 

It is easy to see that $N$ has no smaller contraction with integer vertices, hence the map $\varphi$ one gets factorizing through $\Tc$, is given by polynomials of degree $1$ in $14$ variables.  
\[
 \varphi: \Tc \dto \PP^3:(T_0,\hdots,T_{13})\mapsto (T_1+2T_6+T_{13}:T_{12}:T_0+2T_4+T_{10}:T_4+T_{10}).
\]
After some computations one obtains that for $\nu_0=2$, the matrix $M_{\nu_0}\in \mat_{45,90}(\kk[\X])$ is a matrix representation for the closed image of $\varphi$. Hence, the gcd of the maximal minors gives the irreducible implicit equation of degree $7$ up to a power of $2$. Using the complex, this polynomial can be computed as $\frac{\det(45\times45-\textnormal{matrix}).\det(14\times 14-\textnormal{matrix})}{\det(45\times 45-\textnormal{matrix})}$.

\medskip
As we mentioned above, it is more natural in this case not to take common denominator. Thus, consider the map $\phi$ that one obtains by factorizing $f$ through $\Tc$ and then embedding $\AA^3$ in $(\P1)^3$. It can be easily seen that the Newton polytope one gets has $6$ integer points, hence, $\Tc$ embeds in $\PP^{5}$. Finally, one sees that the rational map $\phi$ is given by
\[
 \phi: \PP^{5}\supset\Tc \dto \P1\times \P1\times \P1: (T_0,\hdots,T_5)\mapsto(T_0+T_3:T_2)\times(T_5:T_0+T_3)\times(T_0+T_3:T_3).
\]
It can be seen that in degree $\eta_0:=2$ the complex $(\k.(L_1,L_2,L_3;\kk[s,t,u][\X]))_{(2,*)}$ permits to compute $M_{\eta_0}\in \mat_{15,18}(\kk[\X])$, the matrix representation. Then, in this case, the square of the implicit equation can be computed as the gcd of its maximal minors or as $\frac{\det(15\times15-\textnormal{matrix})}{\det(3\times 3-\textnormal{matrix})}$.

We conclude that in a case where denominators are fairly different, it is notably better to compactify the codomain of $f$ into $(\P1)^3$.
\end{exmp}

\begin{exmp}
 This example shows how the methods work in the generic case with a fixed polytope. We begin by taking $\Nc$ a normal lattice polytope in $\RR^{n-1}$. For the sake of clarity we will treat a particular case in small dimension. Hence, set $n=3$, and consider $\Nc$ as in the drawing below. It will remain clear that this example can be generalized to any dimension and any normal polytope with integer vertices. 

\begin{center}
 \includegraphics[scale=1]{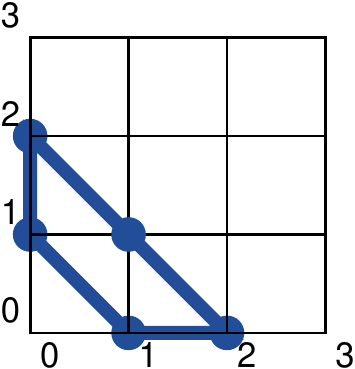}
\end{center}

Assume we are given six generic polynomials $f_1,f_2,f_3,g_1,g_2,g_3$ with support in $\Nc$, hence we get an affine rational map $f:\AA^2\dto \AA^3$ given by $(s,t)\mapsto(\frac{f_1}{g_1},\frac{f_2}{g_2},\frac{f_3}{g_3})$. We write $f_i=\sum_{(a,b)\in \Nc} U_{(a,b),i}\cdot s^at^b$, and $g_i=\sum_{(a,b)\in \Nc} V_{(a,b),i}\cdot s^at^b$. Set $\UU:=\{U_{(a,b),i}, V_{(a,b),i}:\ \textnormal{ for all }(a,b)\in \Nc, \textnormal{ and } i=1,2,3\}$, the set of coefficients, and define $\kk:=\ZZ[\UU]$.

\medskip

Now we focus on computing the implicit equation of a convenient compactification for the map. Let $\Tc$ be the toric variety associated to the Newton polytope $\Nc$, embedded in $\PP^4$. We will compare how the method works in the $\PP^3$ and $(\PP^1)^3$ compactifications of $\AA^3$ with domain $\Tc$. One key point to remark is that these two maps have no base points, since we are taking the toric compactification associated to $\Nc$ and generic coefficients, hence, we will not have any extraneous factors.

In the first case, we take common denominator obtaining four polynomials with generic coefficients in the polytope $3 \Nc$. If we consider the smallest multiple, we recover the polytope $\Nc$, and maps of degree $3$. We obtain in this case that $f$ factorizes through $\Tc\subset \PP^4$  via $\varphi:\Tc \dto \PP^3$, given by $4$ polynomials of degree $3$ in the variables $T_0,\hdots,T_4$. From Lemma \ref{lemAnnih}, we take $\nu_0 := \max\{3, 6-\gamma\}$. Since $2 \Nc$ has integer interior points but $\Nc$ does not, $\gamma=1$, thus $\nu_0=5$. Now, since $X$ is empty in $\Tc$, from Lemma \ref{ZacycALCI}, the complex $\Zc_\bullet$ is acyclic.

From Theorem \ref{mainthT} we see that the implicit equation can be computed as the determinant of the complex $(\Zc_\bullet)_{\nu}$ for $\nu\geq \nu_0$, or as the gcd of the maximal minors of the right-most map $(\Zc_1)_5(-1) \nto{M_\nu}{\lto} (\Zc_0)_5$. We can easily compute the dimension of $A_5$, by the formula $\#(k\cdot \Nc)= (k+1)(k+1+k/2)$. When $k=5$, we get $\#(5\cdot \Nc)=51$, hence $(\Zc_0)_5 = \kk^{51}[\X]$. Since $M_\nu$ gives a surjective map, $(\Zc_1)_5(-1)$ has dimension bigger than or equal to $51$.

\medskip

Instead of taking common denominator, we can proceed by compactifying $\AA^3$ into $(\P1)^3$. In this case we get a map $\phi:\Tc \dto (\P1)^3$ is given by $3$ pairs of linear functions on the variables $T_0,\hdots,T_4$. 

From Theorem \ref{locosymalgT}, we take $\nu \geq \nu_0 = 1+1+1-1=2$. Now, since the polynomials $f_i$ and $g_i$ have generic coefficients, hence $L_i:=Y_i f_i - X_i g_i$ does as well, thus $\k.:=\k.(L_1,L_2,L_3;A[\X])$ is acyclic. From Lemma \ref{teoRes}, the implicit equation can be computed as the determinant of the complex $(\k.)_{\nu}$ for $\nu\geq 2$, or as the gcd of the maximal minors of the right-most map $A_1^{3}[\X](-1)  \nto{M_2}{\lto} A_2[\X]$. Since $\dim(A_0)=1$, $\dim(A_1)=5$ and $\dim(A_2)=12$ we get the complex $\kk^{3}[\X](-2)\to \kk^{15}[\X](-1)  \nto{M_2}{\lto} \kk^{12}[\X]$. Thus, the implicit equation can be computed as the gcd of the maximal minors of a $(12\times 14)$-matrix, or as $\frac{\det(12\times12\textnormal{-matrix})}{\det(3\times 3\textnormal{-matrix})}$.
\end{exmp}

\subsection{A few example with artificial compactifications} We analyze here some small-dimensional examples that have been considered before by other authors, where the method works fairly better with a homogeneous compactification of the codomain. Finally, we illustrate that is much better not to take common denominator in the generic case, by means of an example where denominators are different. 

\begin{exmp}The first one is taken from \cite[Ex. 3.3.1]{BC05} as a base-point-free example. Assume we are given $\varphi:\PP^2 \dto \PP^3: (s:t:u)\mapsto (g_0:g_1:g_2:g_3)$, where $g_0 = s^2t, \ g_1 = t^2u, \ g_2 = su^2, \ g_3 = s^3+t^3+u^3$. In \cite[Ex. 3.3.1]{BC05} it is shown that $\nu_0=4$ and no better bound can be considered. They deduce that $M_{\nu_0}\in \mat_{24,15}(\kk[X_0,X_1,X_2,X_3])$, hence the implicit equation can be computed as the gcd of its maximal minors or as $\frac{\det(15\times 15-\textnormal{matrix}).\det(3\times 3-\textnormal{matrix})}{\det(9\times 9-\textnormal{matrix})}$.

Naturally, this problem can arise from many different affine settings. First, assume that $u$ is the homogenizing variable, and hence, the toric embedding would be $\AA^2 = \{(s:t:1)\}\subset \PP^2$. In any of these cases, the Newton polytope is a triangle with vertices $(0,0)$, $(3,0)$ and $(0,3)$, hence every domain  compactification will be a projective space. If we proceed by taking the embedding corresponding to the smallest homothety, this compactification is $\PP^2$. There are many affine setting for which the projective compactification of the domain gives place to the map we were given.

As a first approach, assume we consider $f_I: \AA^2 \dto \AA^3: (s,t)\mapsto (\frac{f_0}{f_3},\frac{f_1}{f_3},\frac{f_2}{f_3})$.
The projective codomain compactification is the one studied in \cite[Ex. 3.3.1]{BC05}, hence we focus on the rational map 
\[
 \phi_I: \PP^2 \dto \P1\times \P1 \times \P1: (s:t:u)\mapsto (f_0:f_3)\times(f_1:f_3)\times(f_2:f_3).
\] 
It is easy to verify that Avramov's conditions are satisfied, then the implicitization method developed in \cite{Bot08} can be applied. As all $f_i$ are of degree $3$ (nothing gets simplified), $\eta_0:=\sum_{i=1}^3 (d_i-1)+1= 7$. Introduce the variables $\X:=\{X_1,X_2,X_3,Y_1,Y_2,Y_3\}$, and the linear forms $L_1= f_0.Y_1-f_3.X_1, L_2= f_1.Y_2-f_3.X_2, L_3= f_2.Y_3-f_3.X_3$. We have the complex $(\k.(L_1,L_2,L_3;\kk[s,t,u][\X]))_{(7,*)}$
\[
 (\k.)_{(7,*)}: \ 0 \to 0\to (A_{1})^3\otimes_\kk \kk[\X](-2) \to (A_{4})^3\otimes_\kk \kk[\X](-1)  \nto{M_\nu}{\lto} A_7\otimes_\kk \kk[\X] \to 0, 
\]
Since $\dim((A_{1})^3)=3.\dim(A_{1})=3.3=9$, $\dim((A_{4})^3)=3.\dim(A_{4})=3.15=45$ and $\dim(A_{7})=36$, we get
\[
 (\k.)_{(7,*)}: \ 0 \to 0\to \kk^9[\X](-2) \to \kk^{45}[\X](-1)  \nto{M_\nu}{\lto} \kk^{36}[\X] \to 0, 
\]
hence, $M_{\eta_0}\in \mat_{36,45}(\kk[\X])$. Computing the gcd of its maximal minors or even as $\frac{\det(36\times 36-\textnormal{matrix})}{\det(9\times 9-\textnormal{matrix})}$, we get a multihomogeneous non-irreducible equation of multidegree $(9,9,9)$ that gives the irreducible implicit equation of multidegree $(6,6,6)$, and an extra factor $G=Y_1^3Y_2^3Y_3^3$ (cf.\ Theorem \ref{teoResGral}). 

For better understanding the nature of this extra factor, let us analyze the base locus of $\phi_I$, $X$. Observe that $W=\emptyset$ and $X=\{q_1,q_2,q_3\}$, precisely, $q_1=(1:-1:0)$, $q_2=(0:1:-1)$ and $q_3=(1:0:-1)$. In the language of Section \ref{sec:extraneousFactorP1xxP1}, $\Theta:=\{\alpha_1,\alpha_2,\alpha_3\}$, where $\alpha_1=\{1\}\subset \{1,2,3\}$, $\alpha_2=\{2\}$, and $\alpha_3=\{3\}$, hence $X_{\alpha_i}:=\{q_i\}$. Being this three sets irreducible and disjoints, 
$\Omega_{\alpha_i}=X_{\alpha_i}$. We have over each point $q_i$ a trivial multiprojective bundle $\EEE_{\alpha_i}$ of rank $2$ isomorphic to $\P1\times \P1$. Clearly $Y_i$ is the irreducible implicit equation of $\pi_2(\EEE_{\alpha_i})\subset (\P1)^3$, and $3$ the coefficient of the cycle $(\pi_2)_\ast(\EEE_{\alpha_i})$.

\medskip
A different approach consists in considering the following affine map $ f_{II}: \AA^2 \dto \AA^3:(s,t)\mapsto (\frac{f_1}{f_0},\frac{f_2}{f_0},\frac{f_3}{f_0})$. Simplifying, we get the following multiprojective setting
\[
 \phi_{II}: \PP^2 \dto \P1\times \P1 \times \P1:(s:t:u)\mapsto (tu:s^2)\times(u^2:st)\times(s^3+t^3+u^3:s^2t).
\]
Also here, it is easy to verify that Avramov's hypotheses are verified, hence the implicitization method of \cite{Bot08} can be applied. We introduce the variables $\X:=\{X_1,X_2,X_3,Y_1,Y_2,Y_3\}$, and the linear forms $L_1= tu.Y_1-s^2.X_1, L_2= u^2.Y_2-st.X_2, L_3= (s^3+t^3+u^3).Y_3-s^2t.X_3$. We get that $\deg(L_1)=\deg(L_2)=2$ and $\deg(L_3)=3$, hence $\eta_0:=\sum_{i=1}^3 (d_i-1)+1= 5$.  The complex 
$(\k.(L_1,L_2,L_3;\kk[s,t,u][\X]))_{(5,*)}$ is
\[
 (\k.)_{(5,*)}: \ 0 \to 0\to (A_{0}^2\oplus A_{1})\otimes_\kk \kk[\X](-2) \to (A_{3}^2\oplus A_{2})\otimes_\kk \kk[\X](-1)  \nto{M_\nu}{\lto} A_5\otimes_\kk \kk[\X] \to 0, 
\]
and, since $\dim(A_{0}^2\oplus A_{1})=2+3=5$, $\dim(A_{3}^2\oplus A_{2})=2.10+6=26$ and $\dim(A_{5})=21$, it is isomorphic to
\[
 (\k.)_{(5,*)}: \ 0 \to 0\to \kk^5[\X](-2) \to \kk^{26}[\X](-1)  \nto{M_\nu}{\lto} \kk^{21}[\X] \to 0. 
\]
Thus, we get $M_{\eta_0}\in \mat_{21,26}(\kk[\X])$, and a multiple of the implicit equation can be computed as the gcd of its maximal minors or as $\frac{\det(21\times 21-\textnormal{matrix})}{\det(5\times 5-\textnormal{matrix})}$. In this case, we get the irreducible implicit equation of multidegree $(6,6,3)$ and a factor $G=Y_3$. Here, the extra factor occurs due to the presence of a base point $q=(0:1:0)$ that vanishes equations $L_1$ and $L_2$, and giving $L_3(q,\X)=t^3 Y_3$.

\medskip
We can see that the method proposed in \cite{BC05} seems to give smaller matrices, as was predicted for a problem coming from rational maps with the same denominator. In this case, the value $\nu_0=4$ is the best bound for a problem like this, without base points; the advantage it gives is that hence, no extra factors appear. 

On the other hand, the method proposed in Chapter\ref{ch:toric-emb-p1xxp1} gives only two matrices, and it does not involve the computation of the first, second and third syzygies needed for building-up the approximation complex. With this setting we are also computing one extra factor that appears due to the existence of base points with $2$-dimensional fiber. Observe that in this last case, $p=(0:t:0)$ forces $L_1$ and $L_2$ to vanish identically over $p$, and that $L_3(p,\X)=t^3.Y_3$. From Theorem \ref{teoResGral} we have that $\det((\k.)_{(5,*)})=H^{\deg(\phi_{II})}.Y_3^{\mu}$, and $\mu=1$. 
\end{exmp}

\begin{exmp}This second example was taken from \cite[Ex. 3.3.3]{BC05} as a non-base-point-free example. In the chapter the authors analyze the improvement of the bound, and how, as they show in \cite[Thm.\ 4.2]{BC05}, it decreases in presence of base points. Hence, assume we are given $\varphi: \PP^2 \dto \PP^3$, $(s:t:u)\mapsto (g_0:g_1:g_2:g_3)$, where $g_0 = su^2, \ g_1 = t^2(s+u), \ g_2 = st(s+u), \ g_3 = tu(s+u)$. In \cite[Ex.\ 3.3.3]{BC05} they show that $\nu_0=4$ can now be lowered, taking as the best bound $\nu_0=2$. They conclude that the implicit equation can be computed as $\frac{\det(6\times 6-\textnormal{matrix})}{\det(3\times 3-\textnormal{matrix})}$.

Also in this example this problem can arrive from many different affine settings, so at first, let us consider a multiprojective setting $\phi_{I}:\PP^2\dto (\P1)^3$. The idea is showing that even if $\PP^2$ is not necessarily the ``best" toric compactification of $\AA^2$, we can apply it in order to be in the setting of Chapter \ref{ch:toric-emb-p1xxp1}. Hence, consider $\phi_{I}$ defined as
\begin{equation*}
 \phi_I: \PP^2 \dto \P1\times \P1 \times \P1: (s:t:u)\mapsto (su:t(s+u))\times(t:u)\times(s:u).
\end{equation*}
We have in degree $\eta_0:=\sum_i(d_i-1)+1=2$ the complex $(\k.(L_1,L_2,L_3;\kk[s,t,u][\X]))_{(2,*)}$ and hence, $M_{\eta_0}\in \mat_{6,7}(\kk[\X])$, giving a multiple of the implicit equation as the gcd of its maximal minors or as $\frac{\det(6\times6-\textnormal{matrix})}{\det(1\times 1-\textnormal{matrix})}$.

With this setting we compute two extra factors that appear because of the presence of two base points, $p=(0:t:0)$ and $q=(s:0:0)$ having $2$-dimensional fibers. Observe that $L_1(p)=L_3(p)=0$, and that $L_2(p,\X)=t.Y_2$; and $L_1(q)=L_2(q)=0$, and $L_3(q,\X)=s.Y_3$. Hence, from Theorem \ref{teoResGral} we have that $\det((\k.)_{(2,*)})=H^{\deg(\phi_{II})}.Y_2^{\mu_1}.Y_3^{\mu_2}$, precisely $\deg(\phi_{II})=1\mu_1=\mu_2=1$, and $H$ has multidegree $(1,1,1)$.

\medskip

We will now choose a better compactification for $\AA^2$. Hence, define $f:\AA^2 \dto \AA^3$, as $f(s,t)=f_{I}(s:t:1)$, the affine map of $f_{I}$ defined above. Considering both codomain compactifications we obtain: First, the projective case, given by $\varphi_{II}: \PP^5 \supset\Tc \dto \PP^3$, given by $(T_0:T_1:T_2:T_3:T_4:T_5)\mapsto (T_0:T_4+T_5:T_2+T_3:T_1+T_2)$, where $\Tc$ is the toric variety associated to the Newton polytope of $g$, $\Nc(g)$. And second, the multiprojective setting $\phi_{II}:  \PP^3  \supset\Tc\dto \P1\times \P1 \times \P1$, given by $(T_0:T_1:T_2:T_3)\mapsto (T_1:T_2+T_3)\times(T_2:T_0)\times(T_1:T_0)$, where $\Tc$ is the toric variety associated to the Newton polytope of $f$, $\Nc(f)$. Hence $\Tc \cong \P1\times \P1$, with it Segre embedding in $\PP^3$. 

From the map $\varphi_{II}$ we obtain the matrix $M_\nu$ from the right-most map of the $\nu$ graded strand of the approximation complex of cycles, for $\nu_0=1$ (cf.\ Theorem \ref{mainthT}). Computing the dimension of each module of cycle $Z_i[i \cdot d]$ in Macaulay2 we get $M_{\eta_0}\in \mat_{6,10}(\kk[\X])$, hence the implicit equation, of degree $3$, can be computed as the gcd of its maximal minors or as $\frac{\det(6\times6-\textnormal{matrix}).\det(1\times 1-\textnormal{matrix})}{\det(4\times 4-\textnormal{matrix})}$. 

Finally let us look at the case $\phi_{II}:\Tc \dto \P1\times \P1\times \P1$.  We verify that in degree $\eta_0:=\sum_i d_i- \gamma+1=3-2+1=2$ (cf.\ Theorem \ref{locosymalgT}), the complex $(\k.(L_1,L_2,L_3;\kk[s,t,u][\X]))_{(2,*)}$ gives $M_{\eta_0}\in \mat_{9,12}(\kk[\X])$, and thus a power of the implicit equation can be computed as the gcd of its maximal minors or even as $\frac{\det(9\times9-\textnormal{matrix})}{\det(3\times 3-\textnormal{matrix})}$.
\end{exmp}

\begin{rem}
In the previous example we can appreciate that from the algorithmic point of view, considering the toric variety associated the Newton polytope of the defining polynomials, is not necessarily the most efficient choice in terms of the size of the matrices. In both cases, it seems to be a better option considering, as a polytope, the smallest contraction of the triangle $(3,0),(0,3),(0,0)$, namely the triangle $(1,0),(0,1),(0,0)$. 
\begin{center}
\includegraphics[scale=1]{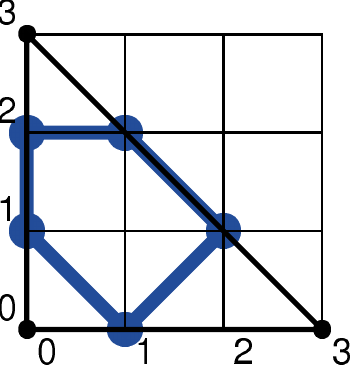}\label{NewPolEj1y2} 
\end{center}
\begin{center} Newton polytope of $g$, $\Nc(g)$. Fig. \eqref{NewPolEj1y2} \end{center}
\begin{center} 
\includegraphics[scale=1]{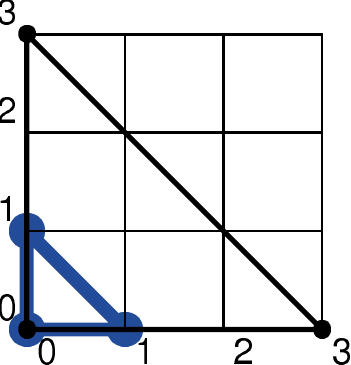}\label{NewPolEj1y2Homot}
\end{center}
\begin{center} Newton polytope of $f$, $\Nc(f)$. Fig. \eqref{NewPolEj1y2Homot} \end{center}

It is clear that the toric variety it defines is $\Tc=\PP^2$; hence, the setting we consider is the map $\varphi: \PP^2 \dto \PP^3$ in the projective case, and $\phi_{I}:\PP^2\dto (\P1)^3$ in the multiprojective one.
\end{rem}


\section{Applications to the computation of sparse discriminants}

The computation of sparse discriminants is equivalent to the implicitization problem for a parametric variety, to which we can apply the techniques developed in the previous sections. In the situation described in \cite{CD}, a rational map $f: \CC^n \dto \CC^n$ given by homogeneous rational functions of total degree zero is associated to an integer matrix $B$ of full rank. This is done in such a way that the corresponding implicit equation is a dehomogenization of a sparse discriminant of generic polynomials with exponents in
a Gale dual of $B$.

\medskip

Suppose for instance that we take the matrix $B$ below:
\[
 B=\left( \begin{array}{rrr}1&0&0\\-2&1&0\\1&-2&1\\0&1&-2\\0&0&1\end{array}\right).
\]
In this case, as the columns of $B$ generate the affine relations of the lattice points $\{0,1,2,3,4\}$. The closed image of the parametrization $f$ is a dehomogenization of the classical discriminant of a generic univariate polynomial of degree $4$. Explicitly, from the matrix we get the linear forms $l_1(u,v,w)=u, \ l_2(u,v,w)=-2u+v, \ l_3(u,v,w)=u-2v+w, \ l_4(u,v,w)=v-2w, \ l_5(u,v,w)=w$ (whose coefficients are read in the rows of $B$), and the polynomials $f_0=l_1\cdot l_3, \ g_0=l_2^2, \ f_1=l_2\cdot l_4, \ g_1=l_3^2, \ f_2=l_3\cdot l_5, \ g_2=l_4^2$ (the exponents of the linear
forms are read from the columns of $B$) . This construction gives rise to the following rational map: 
\[
 \begin{array}{rcl} f: \CC^3 & \dto & \CC^3 \\
 (u,v,w) & \mapsto & (\frac{u(u-2v+w)}{(-2u+v)^2},\frac{(-2u+v)(v-2w)}{(u-2v+w)^2},\frac{(u-2v+w)w}{(v-2w)^2}).
 \end{array}
\]

First, we see that we can get a map from $\PP^2_\CC$ because of the homogeneity of the polynomials. Also, taking common denominator, we can have a map to $\PP^3_\CC$, this is:
\[
 \begin{array}{rcl} f: \PP^2_\CC & \dto & \PP^3_\CC\\
 (u:v:w) & \mapsto & (f_0:f_1 :f_2 :f_3).
 \end{array}
\]
where $f_0= (-2u+v)^2(u-2v+w)^2(v-2w)^2$ is the common denominator, $f_1= u(u-2v+w)^3(v-2w)^2$, $f_2 =(-2u+v)^3(v-2w)^3 $ and $f_3=(u-2v+w)w(-2u+v)^2(u-2v+w)^2$.

The problem with this way of projectivizing is that, in general, we cannot implement the theory developed by L. Bus\'e, M. Chardin, and J-P. Jouanolou, because typically the base locus has unwanted properties, as a consequence of taking common denominator and because of combinatorial reasons.

As a possible way out, we propose in this work to consider the morphism of projective schemes given by:
\[
 \begin{array}{rcl} \phi:\PP^2 & \dto &\P1 \times \P1 \times\P1\\
 (u:v:w) & \mapsto & (f_0:g_0)\times(f_1:g_1)\times(f_2:g_2).
 \end{array}
\]
where $f_0=u(u-2v+w)$, $g_0=(-2u+v)^2$, $f_1=(-2u+v)(v-2w)$, $g_1=(u-2v+w)^2$, $f_2=(u-2v+w)w$ $g_2=(v-2w)^2$.
For this particular example, we get that there are only two base points giving rise to an extra factor, namely $p=(1:2:3)$ and $q=(3:2:1)$. Is easy to see that those points give rise to two linear factors in the equation of the MacRae invariant.

\medskip

First, we observe that this situation is better, because we are not adding common zeroes. Moreover, if a point $(u:v:w)$ is a base point here, it also is in the two settings above: the affine and the projective case $f$.

Remember also that in the $n=2$ case, the condition required on the Koszul complex associated to this map for being acyclic is that the variety $X$, defined as the common zeroes of all the $6$ polynomials, be empty. In general, the conditions we should check are the ones imposed by the Avramov's theorem \ref{avramov}, as was shown in Theorem \ref{teoRes}. 

Note also that if we want to state this situation in the language of approximation complexes, we need only to replace $\k.$ by $\Z.$, because we can assume that $\{\ffi\}$ are regular sequences, due to the fact that $\gcd(\ffi)=1$.

\medskip

\begin{rem}
For a matrix like the $B$ above, it is clear that the closed subvariety $X$ is always empty, due to the fact that all maximal minors of $B$ are not zero, and the polynomials $g_i$'s involve independent conditions. Then, the only common solution to $l_2^2=l_3^2=l_4^2=0$ is $(u,v,w)=(0,0,0)$,  and so $X=\emptyset$ in $\PP^2$.
In this case, it is still better (from an algorithmic approach) to compute the discriminant of a generic polynomial of degree $4$ in a single variable and then dehomogenize, because, in our  setting, the number of variables is bigger than $1$. But when the
number of monomials of a sparse polynomial in many variables is not big, this Gale dual approach for the computation of
sparse discriminants provides a good alternative.
\end{rem}

We will give next an example where we show a more complicated case.

\medskip

\begin{exmp}
Let $C$ be the matrix given by 
\[
 C=\left( \begin{array}{rrr}1&-7&-6\\-1&4&3\\1&0&4\\0&1&-1\\-1&2&0\end{array}\right).
\]
As before, denoting by $b_i$ the $i$-th row of $C$, we get the linear forms 
$l_i(u,v,w)=\gen{b_i,(u,v,w)},$ associated to the row vectors $b_i$ of $B$,
where $\gen{,}$ stands for the inner product in $\CC^3$. Then we define the homogeneous polynomials $f_0=l_1\cdot l_3, \ g_0=l_2\cdot l_5, \ f_1=l_2^4\cdot l_4\cdot l_5^2, \ g_1=l_1^7, \ f_2=l_2^3\cdot l_3^4, \ f_2=l_1^6\cdot l_4$. And we obtain the following rational map:
\[
\begin{array}{rcl}\phi:\PP^2 & \dto & \P1 \times \P1 \times\P1 \\
 (u:v:w) & \mapsto & (f_0:g_0)\times(f_1:g_1)\times(f_2:g_2).
\end{array}
\]

It is easy to see that in this case the variety $X$ is not empty, for instance the point $p=(1:1:-1)$, defined by $l_1=l_2=0$ belongs to $X$.

As was shown by M. A. Cueto and A. Dickenstein in \cite[Lemma 3.1 and Thm. 3.4]{CD}, we can interpret the discriminant computed from the matrix $C$ in terms of the dehomogenized discriminant associated to any matrix of the form $C \cdot M$, where $M$ is a square invertible matrix with integer coefficients. That is, we are allowed to do operations on the columns of the matrix $C$, and still be able to compute the desired discriminant in terms of the matrix obtained from $C$. In \cite{CD} they give an explicit formula for this passage. 

In this particular case, we can multiply $C$ from the right by a determinant $1$ matrix $M$, obtaining
{\small \[
C\cdot M=\left( \begin{array}{rrr}1&-7&-6\\-1&4&3\\1&0&4\\0&1&-1\\-1&2&0\end{array}\right) \cdot \left( \begin{array}{rrr}1&12&-1\\0&6&-1\\0&5&1\end{array}\right)=\left( \begin{array}{rrr}1&0&0\\-1&-3&0\\1&-8&3\\0&11&-2\\-1&0&-1\end{array}\right). 
\]}

Similar to what we have done before, we can see that the closed subvariety $X$ associated to the rational map that we obtain from the matrix $C\cdot M$ is empty. Observe that $\#V(I_2)$ is finite due to the fact that $l_2=l_4=0$ or $l_3=l_4=0$ or $l_3=l_5=0$ should hold.  Moreover it is easy to verify that all maximal minors are nonzero, and this condition implies that any of the previous conditions define a codimension $2$ variety, this is, a finite one. With the notation of Chapter \ref{ch:toric-emb-p1xxp1}, a similar procedure works for seeing see that $\codim_A(I_3)\geq 2$. Finally the first part of Theorem \ref{teoRes} implies that the Koszul complex $\k.$ is acyclic and so we can compute the Macaulay resultant as its determinant. 

Moreover, this property over the minors implies that $\codim_A(I^{(i_0)})=2>k+1=1$ and that $\codim_A(I^{(i_0)}+I^{(i_1)})=3>k+1=2$. So, the second part of Theorem \ref{teoRes} tells us that the determinant of the Koszul complex $\k.$ in degree greater than $(2+8+3)-3=10$ determines exactly the implicit equation of the scheme theoretic image of $\phi$. Observe that, as was shown in \cite[Thm. 2.5]{CD}, for this map, we have that $\deg(\phi)=1$.
\end{exmp}

We remark that the process implemented for triangulating the matrix $C$ via $M$ is not algorithmic for the moment. 

\medskip


\chapter[$\G$-graded Castelnuovo Mumford Regularity]{$\G$-graded Castelnuovo Mumford Regularity}
\label{ch:CastelMum}

\section{Introduction.}\label{sec:introCMReg}

Castelnuovo-Mumford regularity is a fundamental invariant in commutative algebra and algebraic geometry. It is a kind of universal bound for important invariants of graded algebras such as the maximum degree of the syzygies and the maximum non-vanishing degree of the local cohomology modules.

Intuitively, it measures the complexity of a module or sheaf. The regularity of a module approximates the largest degree of the minimal generators and the regularity of a sheaf estimates the smallest twist for which the sheaf is generated by its global sections.  It has been used as a measure for the complexity of computational
problems in algebraic geometry and commutative algebra (see for example \cite{EG} or \cite{BaMum}).

One has often tried to find upper bounds for the Castelnuovo-Mumford regularity in terms of simpler invariants. The simplest invariants which reflect the complexity of a graded algebra are the dimension and the multiplicity. However, the Castelnuovo-Mumford regularity can not be bounded in terms of the multiplicity and the dimension. 

Although the precise definition may seem rather technical. Indeed, the two most popular definitions of Castelnuovo-Mumford regularity are the one in terms of graded
Betti numbers and the one using local cohomology.

For the first one, let $k$ be a field, and let $I$ be an homogeneous ideal in a polynomial ring $R=k[x_0,...,x_n]$ over a field $k$ with characteristic zero. Consider the minimal free resolution of $R/I$ as a graded $R$-module,
\[
 \cdots\to \bigoplus_j R(-d_{i,j})\to\cdots\to \bigoplus_j R(-d_{1,j})\to R\to R/I\to 0.
\]
Then, the Castelnuovo-Mumford regularity of $R/I$ is defined as
\[
 \reg (R/I)=\max_{i,j}\{{d_{i,j}-i}\}.
\]
In general, for a finitely generated graded $R$-module $M$, write $F_i=\bigoplus_j R(-d_{i,j})= \bigoplus_j R[-j]^{\beta_{ij}}$, for a minimal free $R$-resolution of $M$ and set $p := \pd(M) = n-\depth(M)$. Observe that the maps of $F_\bullet\otimes_R k$ are zero, thus, $\tor^R_i(M, k) = H_i(F_\bullet\otimes_R k) = F_i\otimes_R k$ and therefore $\beta_{ij} = \dim_k(\tor^R_i(M, k)_j)$. If $\tor^R_i(M, k) \neq 0$, set 
\[
 b_i(M) := \max\{\mu : \tor^R_i(M, k)\mu \neq 0\},
\]
else, $b_i(M) := -\infty$. Hence $b_i (M)$ is the maximal degree of a minimal generator
of $F_i$, and therefore of the module of $i$-th syzygies of $M$.
The Castelnuovo-Mumford regularity is also a measure of the maximal degrees of generators of the
modules $F_i$: 
\[
 \reg(M) := \max_i\{b_i(M) - i\}. 
\]

Second, we can give two fundamental results that motivated defining Castelnuovo-Mumford regularity in terms of local cohomology: Grothendieck'{}s theorem that asserts that $H^i_\mm(M)=0$ for $i > \dim(M)$ and $i < \depth (M)$, as well as the non vanishing of these modules for $i = \dim(M)$ and $i = \depth (M)$; and Serre'{}s vanishing theorem that implies the vanishing of graded pieces $H^i_\mm (M)_\mu$ for any $i$, and $\mu\gg 0$. The Castelnuovo-Mumford regularity is a measure of this vanishing degree. 

If $H^i_\mm (M) \neq 0$, set
\[
 a_i(M) := \max\{\mu | H^i_\mm (M)_\mu \neq 0\},
\]
else, set $a_i(M) := -\infty$. Then,
\[
 \reg(M) := \max_i\{a_i(M) + i\}.
\]

The maximum over the positive $i$'{}s is also an interesting invariant: 
\[
 \greg(M) := \max_{i>0}\{a_i(M) + i\} = \reg(M/H^0_\mm (M)).
\]

Thus, Castelnuovo-Mumford regularity measures more than the complexity of the ideal $I$ and its syzygies. For more discussion on the regularity, refer to the survey of Bayer and Mumford \cite{BaMum} or \cite{MumRed}.

An interesting question is if one can give bounds for the regularity in terms of the degrees of generators of $I$. It turns out that such bounds are very sensitive to the singularities of the projective scheme defined by $I$, and in general, are very hard to compute. Its value in bounding the degree of syzygies and constructing Hilbert schemes has established that regularity is an indispensable tool in both fields.

The aim of this paper is to develop a multigraded variant of Castelnuovo-Mumford regularity in the spirit of \cite{MlS04} and \cite{HW04}.  We work with modules over a commutative ring $R$ graded by a finitely generated abelian group $\G$. 

One motivation for studying regularity over multigraded polynomial rings comes from toric geometry. For a simplicial toric variety $X$, the homogeneous coordinate ring, introduced in \cite{Cox95}, is a polynomial ring $S$ graded by the divisor class group $G$ of $X$. The dictionary linking the geometry of $X$ with the theory of $G$-graded $S$\nobreakdash-modules leads to geometric interpretations and applications for multigraded regularity. 

In \cite{HW04} Hoffman and Wang define the concept of regularity for bigraded modules over a bigraded polynomial ring motivated by the geometry of $\P1\times \P1$. They prove analogs of some of the classical results on $m$-regularity for graded modules over polynomial algebras. 

In \cite{MlS04} MacLagan and Smith develop a multigraded variant of Castelnuovo-Mumford regularity also motivated by toric geometry. They work with modules over a polynomial ring graded by a finitely generated abelian group, in order to establish the connection with the minimal generators of a module and its behavior in exact sequences.  In this chapter, we extend this work restating some of the results in \cite{MlS04}. 

As in the standard graded case, our definition of multigraded regularity involves the vanishing of graded components of local cohomology, following \cite{HW04}.  

Our notion of Multigraded Castelnuovo-Mumford regularity follows of existing ideas of \cite{HW04} and \cite{MlS04}. In the standard graded case, it reduces to Castelnuovo-Mumford regularity (cf.\ \cite{BaMum}). When $S$ is the homogeneous coordinate ring of a product of projective spaces, multigraded regularity is the weak form of bigraded regularity defined in \cite{HW04}.

One point we are interested in is that Castelnuovo-Mumford regularity establish a relation between the degrees of vanishing of local cohomology modules and the degrees where $\tor$ modules vanish. This provides a powerful tool for computing one region of $\ZZ$ in terms of the other. 

In this chapter, we deal with $\G$-graded polynomial rings, where $\G$ is a finitely generated abelian group. We exploit some of the similarities we get in multigraded regularity with standard regularity, being able to compute the regions of $\G$ where local cohomology modules vanish in terms of the supports of $\tor$ modules, and vice-versa.

Let $S$ be a commutative ring, $\G$ an abelian group and $R:=S[X_1,\hdots,X_n]$, with $\deg(X_i)=\gamma_i$ and $\deg (s)=0$ for $s\in S$. Consider $B\subseteq (X_1,\hdots,X_n)$ a finitely generated graded $R$-ideal and $\Cc$  the monoid generated by $\{\gamma_1,\hdots, \gamma_n\}$, we propose in Definition \ref{defRegLC} that:

For $\gamma\in \G$, and for a homogeneous $R$-module $M$ is  \textsl{weakly $\gamma$-regular} if 
\[
 \gamma \not\in \bigcup_{i} \Supp_\G(H^i_B(M))+\EE_{i}.
\] 
We also set that if further, $M$ is weakly $\gamma'$-regular for any $\gamma'\in \gamma +\Cc$, then $M$ is
\textsl{$\gamma$-regular} and 
\[
 \reg (M):=\{ \gamma\in \G\ \vert\ M\ {\rm is}\ \gamma {\rm -regular}\} .
\]

We deduce from the definition that $\reg (M)$ is the maximal set $S$ of elements in $\G$ such that $S+\Cc =S$ and $M$ is $\gamma$-regular for any $\gamma\in S$.


\medskip

\section{Local Cohomology and graded Betti numbers}
In this chapter we develop a regularity theory for graded rings. Our aim is to give a more general setting to that in \cite{MlS04} and \cite{HW04}, and to establish a clear relation between supports of local cohomology modules with $\tor$ modules and Betti numbers.

Throughout this chapter let $\G$ be a finitely generated abelian group, and let $R$ be a commutative $\G$-graded ring with unit. Let $B$ be a homogeneous ideal of $R$. 

\begin{rem}\label{remGrading}
 Is of particular interest the case where $R$ is a polynomial ring in $n$ variables and $\G=\ZZ^n/K$, is a quotient of $\ZZ^n$ by some subgroup $K$. Note that, if $M$ is a $\ZZ^n$-graded module over a $\ZZ^n$-graded ring, and $\G=\ZZ^n/K$, we can give to $M$ a $\G$-grading coarser than its $\ZZ^n$-grading. For this, define the $\G$-grading on $M$ by setting, for each $\gamma\in \G$, $M_\gamma:=\bigoplus_{d\in \pi^{-1}(\gamma)}M_d$.
\end{rem}

\medskip
In order to fix the notation, we state the following definitions concerning local cohomology of graded modules, and support of a graded modules $M$ on $\G$. Recall that the cohomological dimension $\cd_B(M)$ of a module $M$ is $-\infty$ if $M=0$ and $\max\{i\in \ZZ:\ H^i_B(M)= 0\}$ otherwise.

\begin{defn}\label{defSuppGP}
Let $M$ be a graded $R$-module, the support of the module $M$ is $\Supp_\G(M):=\{\gamma \in \G:\ M_\gamma \neq 0\}$.
\end{defn}

Observe that if $\F.$ is a free resolution of a graded module $M$, much information on the module can be read from the one of the resolution. Next we present a result that permits describing the support of a graded module $M$ in terms of some homological information of a complex which need not be a resolution of $M$, but $M$ is its first non-vanishing homology.

\begin{defn}\label{defDij}
Let $\C.$ be a complex of graded $R$-modules. For all $i,j\in \ZZ$ we define a condition \eqref{eqDij} as above 
\begin{equation}\label{eqDij}\tag{D$_{ij}$}
 H^{i}_B(H_{j}(\C.))\neq 0 \textnormal{ implies } H^{i+\ell +1}_B(H_{j+\ell}(\C.))=H^{i-\ell -1}_B(H_{j-\ell }(\C.))=0 \textnormal{ for all }\ell\geq 1.
\end{equation}
\end{defn}

We have the following result on the support of the local cohomology modules of the homologies of $\C.$.

\begin{thm}\label{ThmRegHGral}
 Let $\C.$ be a complex of graded $R$-modules and $i\in \ZZ$. If \eqref{eqDij} holds, then
\[
 \Supp_\G(H^i_B(H_j(\C.)))\subset \bigcup_{k\in\ZZ}\Supp_\G(H^{i+k}_B(C_{j+k})).
\]
\end{thm}
\begin{proof}
Consider the two spectral sequences that arise from the double complex $\check \Cc^\bullet_B \C.$ of graded $R$-modules. 

The first spectral sequence has as second screen $ _2'E^i_j = H^i_B (H_j(\C.))$. Condition \eqref{eqDij} implies that $_\infty 'E^i_j =\, _2'E^i_j = H^i_B (H_j(\C.))$. The second spectral sequence has as first screen $ _1''E^i_j = H^i_B (C_j)$. 

By comparing both spectral sequences, we deduce that, for $\gamma \in \G$, the vanishing of $(H^{i+k}_B (C_{j+k}))_\gamma$ for all $k$ implies the vanishing of $( _\infty 'E^{i+\ell }_{j+\ell})_\gamma$ for all $\ell$, hence the one of $(H^i_B (H_{j}(\C.)))_\gamma$.
\end{proof}

We next give some cohomological conditions on the complex $\C.$ to imply \eqref{eqDij} of Definition \ref{defDij}. Recall that for an $R$-module $M$ we can compute 
\[
 \cd_B(M):=\min\set{i: H^\ell_B(M)=0 \mbox{ for all }i>\ell},
\]
which is called the cohomological dimension of $M$.

\begin{rem}\label{RemSuppCond}
Let $\C.$ be a complex of graded $R$-modules. Consider the following conditions
\begin{enumerate}
 \item $\C.$ is a right-bounded complex, say $C_j=0$ for $j<0$ and, $\cd_B(H_j(\C.))\leq 1$ for all $j\neq 0$.
 \item For some $q\in \ZZ\cup \{-\infty\}$, $H_j(\C.)=0$ for all $j<q$ and, $\cd_B(H_j(\C.))\leq 1$ for all $j>q$.  
 \item $H_j(\C.)=0$ for $j<0$ and $\cd_B(H_k(\C.))\leq k+i$ for all $k\geq 1$.
\end{enumerate}
Then, 
\begin{enumerate}
 \item[(i)] $(1)\Rightarrow (2)\Rightarrow$ \eqref{eqDij} for all $i,j\in \ZZ$, and
 \item[(ii)] $(3)\Rightarrow$\eqref{eqDij} for  $j=0$.
\end{enumerate}
\end{rem}
\begin{proof}
 For proving item (i), it suffices to show that  $(2)\Rightarrow$ \eqref{eqDij} for all $i,j\in \ZZ$ since $(1)\Rightarrow (2)$ is clear. 
 
 Let $\ell \geq 1$.
 
 Condition (2) implies that $H^i_B (H_j(\C.))=0$ for $j>q$ and $i\not= 0,1$ and for $j<q$.  If $H^i_B (H_j(\C.))\not= 0$, either $j>q$ and $i\in \{ 0,1\}$ in
 which case $j+\ell>q$ and $i+\ell +1\geq 2$ and $i-\ell -1<0$, or $j=q$ in
 which case $j+\ell>q$ and $i+\ell +1\geq 2$ and $j-\ell<0$. In both cases the asserted vanishing holds.

Condition $(3)$ implies that $H^{i+\ell +1}_B(H_{\ell}(\C.))=0$ and  $H_{j-\ell}(\C. )=0$.
\end{proof}

\medskip

\subsection{From Local Cohomology to Betti numbers}\label{LCtoBN}
Assume $R:=S[X_1,\hdots,X_n]$ is a polynomial ring over a commutative ring $S$, 
$\deg(X_i)=\gamma_i\in \G$ for $1\leq i\leq n$ and $\deg (s)=0$ for $s\in S$. Set $\bfgamma:=(\gamma_1,\hdots,\gamma_n)\in \G^n$. 

Let $B\subseteq (X_1,\hdots,X_n)$ be a  finitely generated graded $R$-ideal.

\begin{defn}\label{defEE}
 Set $\EE_0:=\{0\}$ and $\EE_l:=\{\gamma_{i_1}+\cdots+\gamma_{i_l}\ : \ i_1<\cdots<i_l\}$ for $l\neq 0$.
\end{defn}

Observe that if $l<0$ or $l>n$, then $\EE_l=\emptyset$. If $\gamma_i=\gamma$ for all $i$, $\EE_l=\{l\cdot \gamma\}$ when $\EE_l \not= \emptyset$.

\begin{nota}\label{RemShiftPSigma}
 For an $R$-module $M$, we denote by $M[\gamma']$ the shifted module by $\gamma'\in \G$, with $M[\gamma']_{\gamma}:=M_{\gamma'+\gamma}$ for all $\gamma \in \G$. 
\end{nota}

Let $M$ be a graded $R$-module. Write $\k.^M:=\k.(X_1,\hdots,X_n;M)$ for the Koszul complex of the sequence $(X_1,\hdots,X_n)$ with coefficients in $M$. We next establish a relationship between the support of the local cohomologies of its homologies and graded Betti numbers of $M$. 

The Koszul complex $\k.^M$ is graded with $K_l^M:=\bigoplus_{i_1<\cdots<i_l}M[-\gamma_{i_1}-\cdots-\gamma_{i_l}]$. Let $Z_i^M$ and $B_{i}^M$ be the Koszul $i$-th cycles and boundaries modules, with the grading that makes the inclusions $Z_i^M, B_{i}^M\subset K_i^M$ a map of degree $0\in \G$, and set $H_i^M=Z_i^M/ B_{i}^M$.

\begin{thm}\label{ThmLCtoTor}
Let $M$ be a $\G$-graded $R$-module. Then
\[
\Supp_\G(\tor^R_j(M,S))\subset \bigcup_{k\geq 0}(\Supp_\G(H^{k}_B(M))+\EE_{j+k}),
\]
for all $j\geq 0$.
\end{thm}
\begin{proof}
 Notice that $H_j^M\simeq \tor_j^R(M,S)$ is annihilated by $B$, hence has cohomological dimension 0 relatively to $B$. 
 According to Remark \ref{RemSuppCond} (case (1)), Theorem \ref{ThmRegHGral} applies and shows that 
 \[
 \Supp_\G( \tor_j^R(M,S))\subset \bigcup_{\ell\geq 0}\Supp_\G(H^{\ell}_B(K_{j+\ell }))=\bigcup_{k\geq 0}(\Supp_\G(H^{k}_B(M))+\EE_{j+k}).
\]
 \end{proof}

\subsection{From Betti numbers to Local Cohomology}

In this subsection we bound the support of local cohomology modules in terms of the support of Tor modules. This generalizes the fact that for $\ZZ$-graded Castelnuovo-Mumford regularity, if $a_i(M)+i \leq \reg(M) := \max_i\{b_i(M) - i\}$. 

We keep same hypotheses and notation as in Section \ref{LCtoBN}

\medskip
Next result gives an estimate of the support of local cohomology modules of a graded $R$-module $M$ in terms of the supports of those of base ring and the twists in a free resolution. This permits (combined with Lemma \ref{LemResolutions}) to give a bound for the support of local cohomology modules in terms of Betti numbers. 

The key technical point is that Lemma \ref{LemResolutions} part (1) and (2) give a general version of Nakayama Lemma in order to relate `shifts in a resolution'
with support of Tor modules; while part (3) is devoted to give a `base change lemma' in order to pass easily to localization.

\begin{thm}\label{lemSuppHi}
 Let $M$ be a graded $R$-module and $F_\bullet$ be a graded complex of free $R$-modules, with $H_0(F_\bullet)=M$. Write 
 $F_i = \bigoplus_{j\in E_i} R[-\gamma_{ij}]$ and $T_i:=\{ \gamma_{ij}\ \vert\ j\in E_i\}$. Let $\ell \geq 0$ and assume $\cd_B(H_j(F_\bullet ))\leq \ell +j$ for all $j\geq 1$. Then, 
\[
 \Supp_{\G}(H^\ell_B(M))\subset \bigcup_{i\geq 0}(\Supp_{\G}(H^{\ell +i}_B(R))+T_i).
\]
\end{thm}
\begin{proof} Lemma \ref{RemSuppCond} (case (3)) shows that  Theorem \ref{ThmRegHGral} applies for estimating the support of
local cohomologies of $H_0(F_\bullet)$, and provides the quoted result as local cohomology commutes with arbitrary direct sums 
\[
 \Supp_{\G} (H^{p}_B(R[-\gamma ]))=\Supp_{\G} (H^{p}_B(R))+\gamma, \mbox{ and }\Supp_{\G} (\oplus_{i\in E} N_i)=\cup_{i\in E}\Supp_{\G} (N_i)
\]
for any set of graded modules $N_i$, $i\in E$.
\end{proof}

\begin{lem}\label{LemResolutions}
 Let $M$ be a graded $R$-module.
 \begin{enumerate}
  \item Let $S$ be a field and let $F_\bullet$ be a $G$-graded free resolution of a finitely generated module $M$. Then 
   $$
   F_i=\bigoplus_{\gamma\in T_i} R[-\gamma]^{\beta_{i,\gamma}},\quad \mbox{and}\quad T_i= \Supp_G (\tor^R_i(M,S)).
   $$ 
  \item Assume that $(S,\mm ,k)$ is local. Then
   $$
   \Supp_{\G}(\tor_i^R(M,k))\subseteq \bigcup_{j\leq i}\Supp_{\G}(\tor_j^R(M,S)).
   $$
 \end{enumerate}
\end{lem}
\begin{proof}
For Part (1) see \cite{CJR}.
Part (3) follows from the fact that if $(S,\mm ,k)$ is local there is an spectral sequence $\tor_p^{S}(\tor_q^R(M,S),k)\Rightarrow \tor_{p+q}^R(M,k)$ and the fact that $S\subset R_0$. 
\end{proof}

Combining Theorem \ref{lemSuppHi} with Lemma \ref{LemResolutions} (case (1)) one obtains:

\begin{cor}\label{corSuppHi}
 Assume that $S$ is a field and let $M$ be a finitely generated graded $R$-module. Then, for any $\ell$,
\[
 \Supp_{\G}(H^\ell_B(M))\subset \bigcup_{i\geq 0}(\Supp_{\G}(H^{\ell +i}_B(R))+\Supp_{\G}(\tor_i^R(M,S))).
\]
\end{cor}

If $S$ is Noetherian, Lemma \ref{LemResolutions} (case (2)) implies the following:

\begin{cor}\label{corSuppHi2}
 Assume that $(S,\mm ,k)$ is local Noetherian and let $M$ be a finitely generated graded $R$-module. Then, for any $\ell$,
\[
\begin{array}{rl}
 \Supp_{\G}(H^\ell_B(M))&\subset \bigcup_{i\geq 0}(\Supp_{\G}(H^{\ell +i}_B(R))+\Supp_{\G}(\tor_i^R(M,k)))\\
 &\subset \bigcup_{i\geq j\geq 0}(\Supp_{\G}(H^{\ell +i}_B(R))+\Supp_{\G}(\tor_j^R(M,S))).\\
 \end{array}
\]
\end{cor}

After passing to localization, Corollary \ref{corSuppHi2} shows that:

\begin{cor}\label{corSuppHi3}
Let $M$ be a finitely generated graded $R$-module, with $S$ Noetherian. Then, for any $\ell$,
\[
 \Supp_{\G}(H^\ell_B(M))\subset \bigcup_{i\geq j\geq 0}(\Supp_{\G}(H^{\ell +i}_B(R))+\Supp_{\G}(\tor_j^R(M,S))).
 \]
\end{cor}
 
\begin{proof} 
Let $\gamma\in  \Supp_{\G}(H^\ell_B(M))$. Then $H^\ell_B(M)_\gamma \not= 0$, hence there exists 
 $\pp\in {\rm Spec} (S)$ such that $(H^\ell_B(M)_\gamma)\otimes_S S_\pp = H^\ell_{B\otimes_S S_\pp}(M\otimes_S S_\pp )\not= 0$.
 Applying Corollary \ref{corSuppHi2} the result follows since both the local cohomology functor and the
 Tor functor commute with localization in $S$, and preserves grading as $S\subset R_0$.
\end{proof}
 
Notice that taking $G=\ZZ$ and $\deg(X_i)=1$, Corollaries \ref{corSuppHi}, \ref{corSuppHi2} and \ref{corSuppHi3} give the well know bound $a_i(M)+i \leq \max_i\{b_i(M) - i\}$.


\medskip

\section{Castelnuovo-Mumford regularity}

We have mentioned in the beginning of this chapter that one point we are interested in remark is that Castelnuovo-Mumford regularity establishes a relation between the degrees of vanishing of local cohomology modules and the degrees where $\tor$ modules vanish. It is clear that this provides a powerful tool for computing one region of $\ZZ$ in terms of the other. 

In this section we give a definition for a $\G$-graded $R$-module $M$ and $\gamma\in \G$ to be \textsl{weakly $\gamma$-regular} or just \textsl{$\gamma$-regular}, depending if $\gamma$ is or is not on the shifted support of some local cohomology modules of $M$ (cf.\ \ref{defRegLC}). This definition allows us to generalize the classical fact that weak regularity implies regularity.

In the later part of this section, in Theorem \ref{ThmLCtoTor1}, we prove that for $j\geq 0$, the supports of $\tor^R_j(M,S)$ does not meet the support of any shifted regularity region $\reg (M)+\gamma$ for $\gamma$ moving on $\EE_{j}$. As we have mentioned in the introduction of this chapter, this result generalizes the fact that when $\G =\ZZ$ and the grading is standard, $\reg (M)+j\geq \fin (\tor^R_j (M,S))$. 

\medskip

\subsection{Regularity for Local Cohomology modules}

Let $S$ be a commutative ring, $\G$ an abelian group and $R:=S[X_1,\hdots,X_n]$, with $\deg(X_i)=\gamma_i$ and $\deg (s)=0$ for $s\in S$. Let 
$B\subseteq (X_1,\hdots,X_n)$ be a graded $R$-ideal and $\Cc$ be the monoid generated by $\{\gamma_1,\hdots, \gamma_n\}$.

\begin{defn}\label{defRegLC}
For $\gamma\in \G$, a graded $R$-module $M$ is  weakly \textsl{$\gamma$-regular} if 
\[
 \gamma \not\in \bigcup_{i} \Supp_\G(H^i_B(M))+\EE_{i}.
\] 

If further $M$ is weakly $\gamma'$-regular for any $\gamma'\in \gamma +\Cc$, then $M$ is
$\gamma$-regular 
and 
\[
 \reg (M):=\{ \gamma\in \G\ \vert\ M\ {\rm is}\ \gamma {\rm -regular}\} .
\]
\end{defn}

It immediately follows from the definition that $\reg (M)$ is the maximal set
$S$ of elements in $\G$ such that $S+\Cc =S$ and $M$ is weakly $\gamma$-regular for any $\gamma\in S$.

Let $\{ \gamma_1,\ldots ,\gamma_n\}=\{ \mu_1,\ldots ,\mu_p\}$, with $\mu_i\not= \mu_j$ for $i\not= j$.
Denote by $\pp_i$ the ideal generated by the variables of degree $\mu_i$. 

The following lemma generalizes the classical fact that weak regularity implies regularity 
under some extra requirement. 

\begin{lem}\label{wRtoR}
Assume that $B\subset \pp_i$ for every $i$. Let  $M$ be a graded $R$-module. 
If $M$ is weakly $\gamma$-regular and either $H^0_B(M)_{\gamma +\Cc }=0$
or $M$ is generated by elements whose degrees do not belong to $\gamma +\Cc$, then
$M$ is $\gamma$-regular.
\end{lem}

\begin{proof}
We induct on $w(M):=n-m$, where $m$ is the number of variables acting as $0$ on $M$. 

Let $i\in \{ 1,\ldots ,p\}$. We have to show that $M$ is weakly $(\gamma +\mu_i)$-regular if
one of the two conditions of the Lemma is satisfied. Assume that the variables $X_{j}$ for
$j=j_{0i},\ldots ,j_{ti}$ are the ones of degree $\mu_i$. 

If $w(M)=0$, then $M=0:_M B=H^0_B (M)$. Further if $M=H^0_B (M)$, both requirements are equivalent
and the result follows as $H^i_B (M)=0$ for $i>0$.

Our statement is unchanged by faithfully flat extension and the Dedekind-Mertens
Lemma shows that after making a polynomial extension $S':=S[U_1,\ldots U_t]$ of $S$, the element
$f_i:=X_{j_{0i}}+U_1X_{j_{1i}}+\cdots +U_tX_{j_{ti}}$  is a non-zero divisor on $M/H^0_{\pp_i}(M)$, hence on 
$M':=M/H^0_{B}(M)$, as $B\subset \pp_i$ by hypothesis. 

Notice that $w(M/f_i M)<w(M)$ after identifying 
$R/(f_i)$ with $R':=S'[X_1,\ldots ,\widehat{X_{j_{0i}}},\ldots ,X_n]$. For any $\ell$, the exact sequence $0\ra (0:_M (f_i))\ra 
M\ra M(\mu_i )\ra (M/f_iM)(\mu_i )\ra 0$  gives rise to an exact sequence
$$
H^\ell_B (M)\ra H^\ell_B (M)(\mu_i )\ra H^\ell_B (M/f_i M)(\mu_i )\ra H^{\ell +1}_B (M).
$$
The right part of the sequence shows that $M/f_iM$
is weakly $\gamma$-regular, hence, by induction hypothesis, $\gamma$-regular if $M/f_iM$ is generated by elements whose degrees 
do not belong to $\gamma +\Cc$ (for instance if $M$ is so) and $(M/f_iM)/H^0_B(M/f_iM)$ is $\gamma$-regular in any case. 

From  the left part of the sequence, we deduce that $M$ is $(\gamma +\mu_i)$-regular if $M$ is generated by elements whose degrees do not belong to $\gamma +\Cc$ and $M/H^0_B (M)$ is $(\gamma +\mu_i )$-regular in any case, which proves our claim. 
\end{proof}

\begin{thm}\label{ThmLCtoTor1}
Let $M$ be a $\G$-graded $R$-module. Then
\[
 \bigcap_{\gamma\in \EE_{j}}(\reg (M)+\gamma ) \bigcap \Supp_\G(\tor^R_j(M,S)) =\emptyset
\]
for all $j\geq 0$.
\end{thm}

When $\G =\ZZ$ and the grading is standard, this reads with the usual definition of $\reg (M)\in \ZZ$: 
$$
\reg (M)+j\geq \fin (\tor^R_j (M,S)).
$$ 

\begin{proof}
 If $\gamma \in \Supp_\G(\tor^R_j(M,S))$, then it follows from Theorem \ref{ThmLCtoTor} that $\gamma \in \Supp_\G(H^{\ell}_B(M))+\EE_{j+\ell}$ for some 
 $\ell$. Hence 
 $$
 \gamma -\gamma_{i_1}-\cdots -\gamma_{i_{j+\ell}} \in \Supp_\G(H^\ell_B(M))
 $$
 for some $i_1<\cdots <i_{j+\ell}$. By definition it follows that if $\mu\in \reg (M)$ and $t_1< \cdots < t_\ell$, then 
 $$
  \gamma -\gamma_{i_1}-\cdots -\gamma_{i_{j+\ell}} \not= \mu -\gamma_{t_1}-\cdots -\gamma_{t_{\ell}}
  $$
  in particular choosing $t_k:=i_{j+k}$ one has 
   $$
  \gamma -\gamma_{i_1}-\cdots -\gamma_{i_{j}} \not\in \reg (M).
  $$
\end{proof}

On the other hand, Corollary \ref{corSuppHi3} shows that :

\begin{prop}\label{PrTortoLC}
Assume $S$ is Noetherian, let $M$ be a finitely generated $\G$-graded $R$-module and set $T_i:=\Supp_{\G}(\tor_i^R(M,S))$. Then, for any $\ell$,
\[
 \Supp_{\G}(H^\ell_B(M)+\EE_\ell )\subset \bigcup_{i\geq j}(\Supp_{\G}(H^{\ell +i}_B(R))+\EE_\ell +T_j).
 \]
 If further $S$ is a field,
 \[
 \Supp_{\G}(H^\ell_B(M)+\EE_\ell )\subset \bigcup_{i}(\Supp_{\G}(H^{\ell +i}_B(R))+\EE_\ell +T_i).
 \]
 \end{prop}

In some applications it is useful to consider local cohomologies of indices at least equal to some number, for 
instance positive values or values at least two. In view of Lemma \ref{wRtoR}, most of the time weak regularity
and regularity agrees in this case. We set :
$$
\reg^\ell (M):=\{ \gamma\  \vert\ \forall \gamma'\in \Cc ,\  \gamma +\gamma'\not\in \bigcup_{i\geq \ell} \Supp_{\G}(H^i_B (M))+\EE_i\} .
$$

With this notation, Proposition \ref{PrTortoLC} implies the following
\begin{thm}\label{ThmTortoLC}
Assume $S$ is Noetherian, let $M$ be a finitely generated  $\G$-graded $R$-module and set $T_i:=\Supp_{\G}(\tor_i^R(M,S))$. Then, for any $\ell$,
\[
 \reg^\ell (M)\supseteq \bigcap_{j\leq i,\gamma \in T_j,\gamma' \in \EE_i} \reg^{\ell +i} (R)+\gamma -\gamma' \supseteq 
 \reg^{\ell} (R)+ \bigcap_{j\leq i,\gamma \in T_j,\gamma' \in \EE_i}\gamma -\gamma' +\Cc .
 \]
 The above intersection can be restricted to $i\leq \cd_B (R)-\ell$. If further $S$ is a field,
 \[
 \reg^\ell (M)\supseteq \bigcap_{i,\gamma \in T_i,\gamma' \in \EE_i} \reg^{\ell +i} (R)+\gamma -\gamma' \supseteq 
 \reg^{\ell} (R)+ \bigcap_{i,\gamma \in T_i,\gamma' \in \EE_i}\gamma -\gamma' +\Cc .
 \]
 
 \end{thm}

When $\G =\ZZ$ and the grading is standard, this reads with the usual definition of $\reg^\ell (M)\in \ZZ$: 
$$
\reg^\ell (M)\leq \reg^\ell (R)+\max_i \{ \fin (\tor^R_i (M,S))-i\} .
$$

\begin{proof}

If $\mu \not\in \reg^\ell (M)$,
by  Proposition \ref{PrTortoLC}, there exists $i\geq j$ such that 
$$
\mu \in \Supp_{\G}(H^{\ell +i}_B(R))+\EE_\ell +T_j
$$
hence there exists $\gamma' \in \EE_i$ and $\gamma \in T_j$ such that 
$$
\mu+\gamma' -\gamma \in \Supp_{\G}(H^{\ell +i}_B(R))+\EE_{i+\ell} .
$$
Therefore $\mu  \not\in \reg^{\ell +i}(R)+\gamma -\gamma'$.
\end{proof}


\section{Local cohomology of multigraded polynomial rings}

Let $k$ be a commutative ring, $s$ and $m$ be fixed positive integers, $r_1\leq\cdots\leq r_s$ non-negative integers, and write $\x_i=(x_{i,1},\hdots,x_{i,r_i})$ for $1\leq i \leq s$. 

Define $R_i:= k[\x_i]$, the standard $\ZZ$-graded polynomial ring in the variables $\x_i$ for $1\leq i \leq s$, $R=\bigotimes_k R_i$, and $R_{(a_1,\hdots,a_s)}:=\bigotimes_k (R_i)_{a_i}$ stands for its multigraded part of multidegree $(a_1,\hdots,a_s)$. 

\begin{defn}
 We define $\check{R}_{i}:= \frac{1}{x_{i,1}\cdots x_{i,r_i}}k[x_{i,1}^{-1},\hdots,x_{i,r_i}^{-1}]$. Given integers $1\leq i_1<\cdots<i_t \leq s$, take $\alpha=\{i_1,\hdots,i_t\}$, and set $\check{R}_\alpha:=\paren{\bigotimes_{j\in \alpha} \check{R}_{j}} \otimes_k \paren{\bigotimes_{j\notin \alpha} R_{j}}$.
\end{defn}

\begin{rem}
 Observe that $\check{R}_{\{i\}}\cong \check{R}_{i} \otimes_k \bigotimes_{j\neq i} R_j$.
\end{rem}

\begin{defn} Given integers $1\leq i_1<\cdots<i_t \leq s$, take $\alpha=\{i_1,\hdots,i_t\}$. For any integer $j$ write $\sg(j):=1$ if $j\in \alpha$ and $\sg(j):=0$ if $j\notin \alpha$. We define 
\[
 Q_\alpha:= \prod_{1\leq j \leq s} (-1)^{\sg(j)} \NN - \sg(j)r_j\ee_j \subset \ZZ^s,
\]
 the shift of the orthant whose coordinates $\{i_1,\hdots,i_t\}$ are negative and the rest are all positive. We set $\aaa_i$ for the $R$-ideal generated
 by the elements in $\x_i$, $B:=\aaa_1\cdots \aaa_s$, $\aaa_\alpha:=\aaa_{i_1}+\cdots+\aaa_{i_t}$ and $|\alpha|=r_{i_1}+\cdots+r_{i_t}$.
\end{defn}

\begin{lem}\label{lemSuppCheckRalpha}
For every $\alpha\subset \{1,\hdots,s\}$, we have $\Supp_{\ZZ^s}(\check{R}_{\alpha})=Q_\alpha$.
\end{lem}

\begin{rem}\label{remQalphaQbeta} 
For $\alpha,\beta \subset \{1,\hdots,s\}$, if $\alpha \neq \beta$, then
 $Q_\alpha \cap Q_\beta=\emptyset$.
\end{rem}

\begin{lem}\label{lemHalphaRalpha} 
Given integers $1\leq i_1<\cdots<i_t \leq s$, let $\alpha=\{i_1,\hdots,i_t\}$. There are graded isomorphisms of $R$-modules
\begin{equation}\label{eqCheckRalpha}
 H_{\aaa_\alpha}^{|\alpha|} (R)\cong \check{R}_\alpha.
\end{equation}
\end{lem}
\begin{proof}
Recall that for any ring $S$ and any $S$-module $M$, if $x_1,\hdots,x_n$ are variables, then
\begin{equation}\label{etoile}
 H_{(x_1,\hdots,x_n)}^i(M[x_1,\hdots,x_n])= \left\lbrace\begin{array}{ll} 0 & \mbox{if }i\neq n\\
									\frac{1}{x_{1}\cdots x_{n}}M[x_{1}^{-1},\hdots,x_{n}^{-1}]  & \mbox{for } i=n. 
\end{array}\right.
\end{equation}
We induct on $|\alpha|$. The result is obvious for $|\alpha|=1$. Assume that $|\alpha|\geq 2$ and \eqref{eqCheckRalpha} holds for $|\alpha|-1$. Take $I=\aaa_{i_1}\cdots\aaa_{i_{t-1}}$ and $J=\aaa_{i_t}$. There is a spectral sequence $H^{p}_J (H^{q}_I (R)) \Rightarrow H^{p+q}_{I+J} (R)$. By \eqref{etoile}, $H^p_J (R)=0$ for $p\neq r_{i_t}$. Hence, the spectral sequence stabilizes in degree $2$, and gives $H^{r_{i_t}}_J (H^{|\alpha|-r_{i_t}}_I (R)) \cong H^{|\alpha|}_{I+J} (R)$. The result follows by applying \eqref{etoile} with $M=H^{|\alpha|-r_{i_t}}_I (R)$, and inductive hypothesis.
\end{proof}

\begin{lem}\label{LemloccohR} With the above notations,
\begin{equation}\label{loccohR}
H_B^\ell (R) 
\cong \bigoplus_{\mbox{\scriptsize $\begin{array}{c}1\leq i_1<\cdots<i_t \leq s \\ r_{i_1}+\cdots+r_{i_t}-(t-1)=\ell\end{array}$}}H_{\aaa_{i_1}+\cdots+\aaa_{i_t}}^{r_{i_1}+\cdots+r_{i_t}} (R)
\cong \bigoplus_{\mbox{\scriptsize $\begin{array}{c}\alpha \subset \{1,\hdots,s\}\\ |\alpha|-(\#\alpha-1)=\ell\end{array}$}}\check{R}_\alpha.
\end{equation}
\end{lem}
\begin{proof}
The second isomorphism follows from \ref{lemHalphaRalpha}. For proving the first isomorphism, we induct on $s$. The result is obvious for $s=1$. Assume that $s\geq 2$ and \eqref{loccohR} holds for $s-1$. Take $I=\aaa_1\cdots\aaa_{s-1}$ and $J=\aaa_s$. The Mayer-Vietoris long exact sequence of local cohomology for $I$ and $J$ is
\begin{equation}\label{eqSEL-MayerVie}
 \cdots \to H^\ell_{I+J}(R) \nto{\psi_\ell}{\to}  H^\ell_{I}(R)\oplus H^\ell_{J}(R) \to  H^\ell_{IJ}(R) \to  H^{\ell+1}_{I+J}(R) \nto{\psi_{\ell+1}}{\to}  H^{\ell+1}_{I}(R)\oplus H^{\ell+1}_{J}(R) \to \cdots.
\end{equation}
Remark that if $\ell<r_s$, then $H^\ell_{J}(R)=H^\ell_{I+J}(R)=H^{\ell+1}_{I+J}(R)=0$. Hence, $H^\ell_{B}(R)\cong H^\ell_{I}(R)$. Write $\tilde{R}:=R_1\otimes_k\cdots\otimes_k R_{s-1}$. Since the variables $\x_s$ does not appear on $I$, by flatness of $R_s$ and the last isomorphism, we have that $H^\ell_{B}(R)\cong H^\ell_{B}(\tilde{R})\otimes_k R_s$. In this case, the result follows by induction.

Thus, assume $\ell\geq r_s$. We next show that the map $\psi_\ell$ in the sequence \eqref{eqSEL-MayerVie} is the zero map for all $\ell$. Indeed, there is an spectral sequence $H^{p}_J (H^{q}_I (R)) \Rightarrow H^{p+q}_{I+J} (R)$. Since $H^p_J (R)=0$ for $p\neq r_s$, it stabilizes in degree $2$, and gives $H^{r_s}_J (H^{\ell-r_s}_I (R)) \cong H^{\ell}_{I+J} (R)$. We have graded isomorphisms 
\begin{equation}\label{eqIsosHJHI}
 H^{r_s}_J (H^{\ell-r_s}_I (R)) \cong H^{r_s}_J (H^{\ell-r_s}_I (\tilde{R})\otimes_k R_s)\cong (H^{\ell-r_s}_I (\tilde{R}))[\x_s^ {-1}]\cong H^{\ell-r_s}_I (\tilde{R})\otimes_k \check{R}_s,
\end{equation}
where the first isomorphism comes from flatness of $R_s$ over $k$, the second isomorphism follows from equation \eqref{etoile} taking $M=H^{\ell-r_s}_I (\tilde{R})$.
By \eqref{eqIsosHJHI} and the inductive hypothesis we have that.
\begin{equation}\label{eqHjHlR}
 H^{r_s}_J (H^{\ell-r_s}_I (R)) \cong 
\bigoplus_{\mbox{\scriptsize $\begin{array}{c}1\leq i_1<\cdots<i_{t-1} \leq s-1 \\ r_{i_1}+\cdots+r_{i_{t-1}}-(t-2)=\ell-r_s\end{array}$}}H_{\aaa_{i_1}+\cdots+\aaa_{i_{t-1}}}^{r_{i_1}+\cdots+r_{i_{t-1}}} (\tilde{R})\otimes_k \check{R}_s.
\end{equation}

Now, observe that the map $(H^{\ell-r_s}_I (\tilde{R}))[\x_s^ {-1}]\to H^\ell_{I}(R)\oplus H^\ell_{J}(R)$ is graded of degree $0$. Recall from Lemma \ref{lemSuppCheckRalpha} and \ref{lemHalphaRalpha}, and Remark \ref{remQalphaQbeta} we deduce $\Supp_{\ZZ^s}((H^{\ell-r_s}_I (\tilde{R}))[\x_s^ {-1}]) \cap \Supp_{\ZZ^s}(H^\ell_{I}(R)\oplus H^\ell_{J}(R))=\emptyset$. Thus, every homogeneous element on $(H^{\ell-r_s}_I (\tilde{R}))[\x_s^ {-1}]$ is necessary mapped to $0$.

Hence, for each $\ell$, we have a short exact sequence
\begin{equation}\label{secHIHJ}
 0 \to  H^\ell_{I}(R)\oplus H^\ell_{J}(R) \to  H^\ell_{IJ}(R) \to  H^{r_s}_{J} (H^{\ell+1-r_s}_{I} (R))\to 0.
\end{equation}
Observe that this sequence has maps of degree $0$, and for each degree $a\in \ZZ^s$ the homogeneous strand of degree $a$ splits. Moreover,
\[
 \Supp_{\ZZ^s}((H^{\ell+1-r_s}_I (\tilde{R}))[\x_s^ {-1}]) \sqcup \Supp_{\ZZ^s}(H^\ell_{I}(R)\oplus H^\ell_{J}(R))=\Supp_{\ZZ^s}(H^\ell_{IJ}(R)).
\]
Namely, every monomial in $H^\ell_{IJ}(R)$ eater comes from the module $(H^{\ell+1-r_s}_I (\tilde{R}))[\x_s^ {-1}]$ or it is mapped to $H^\ell_{I}(R)\oplus H^\ell_{J}(R)$ injectively, splitting the sequence \eqref{secHIHJ} of $R$-modules. Hence,
\[
 H^\ell_{B}(R)\cong H^\ell_{I}(R)\oplus H^\ell_{J}(R)\oplus H^{r_s}_{J} (H^{\ell+1-r_s}_{I} (R)).
\]
Now, $H^\ell_{I}(R) \cong H^\ell_{B}(\tilde{R})\otimes_k R_s$, $H^\ell_{J}(R)= 0$ if $\ell\neq r_s$ and $H^{r_s}_{J}(R)= \check{R}_s$. The result follows by induction and equation \eqref{eqHjHlR}.
\end{proof}

\begin{cor}\label{corSuppHiMultiGr}
Assume that $(S,\mm ,k)$ is local Noetherian and let  $M$ be a finitely generated $\G$-graded $R$-module. Then, for any $\ell$, 
\begin{equation*}
\begin{array}{rl}
 \Supp_{\G}(H^\ell_B(M))& \subset \underset{i\geq 0}{\bigcup}(\Supp_{\G}(H^{\ell +i}_B(R))+\Supp_{\G}(\tor_i^R(M,k)))\\
			& = \underset{i\geq 0}{\bigcup}\paren{\underset{\mbox{\tiny $\begin{array}{c}1\leq i_1<\cdots<i_t \leq s \\ r_{i_1}+\cdots+r_{i_t}-(t-1)=\ell+i \end{array}$}}{\bigcup} Q_{\{i_1, \hdots,i_t\}}+\Supp_{\G}(\tor_i^R(M,k))}.

\end{array}
\end{equation*}
\end{cor}
\begin{proof}
 Follows from Corollary \ref{corSuppHi2} and Lemma \ref{LemloccohR}.
\end{proof}

Whenever 
$S$ is Noetherian, Corollary \ref{corSuppHi3} provides an estimate of $ \Supp_{\G}(H^\ell_B(M))$ in terms of the sets $\Supp_{\G}(\tor_i^R(M,S))$. 


Recall that we have seen in Theorem \ref{ThmRegHGral} that if $\C.$ is a complex of graded $R$-modules, assuming \eqref{eqDij} we have that for all $i\in \ZZ$
\[
 \Supp_\G(H^i_B(H_j(\C.)))\subset \bigcup_{k\in\ZZ}\Supp_\G(H^{i+k}_B(C_{j+k})).
\]
For $i=1,\hdots,m$, take $f_i\in R$ homogeneous of the same degree $\gamma$ for all $i$. Let $M$ be a graded $R$-module. Denote by $\k.^M$ the Koszul complex $\k.(f_1,\hdots,f_m;R)\otimes_R M$. The Koszul complex $\k.^M$ is graded with $K_i:=\bigoplus_{l_0<\cdots<l_i}R(-i\cdot \gamma)$. Set $H_i^M:=H_i(\k.^M)$ the $i$-th homology module of $\k.^M$.

\begin{cor}\label{Ex1Koszul}
If $\cd_B(H_i^M)\leq 1$ for all $i>0$. Then, for all $j\geq 0$
\[
\Supp_\G(H^i_B(H_j^M))\subset \bigcup_{k\in\ZZ}(\Supp_\G(H^{k}_B(M))+k\cdot\gamma) + (j-i)\cdot\gamma.
\]
\end{cor}
\begin{proof}
This follows by a change of variables in the index $k$ in Lemma \ref{ThmRegHGral}. Since $\C.$ is $\k.^M$ and $K_i^M:=\bigoplus_{l_0<\cdots<l_i}M(-i\cdot \gamma)$, we get that 
\[
\Supp_\G(H^i_B(H_j^M))\subset \bigcup_{k\in\ZZ}\Supp_\G(H^{k}_B(K^M_{k+j-i})) = \bigcup_{k\in\ZZ}(\Supp_\G(H^{k}_B(M)[-(k+j-i)\cdot\gamma]).
\]
 The conclusion follows from \ref{RemShiftPSigma}. 
\end{proof}

\begin{rem}\label{RemEx1Koszul}
 In the special case where $M=R$, we deduce that if $\cd_B(H_i)\leq 1$ for all $i>0$,
\[
 \Supp_\G(H^i_B(H_j))\subset  \bigcup_{k\in\ZZ}(\Supp_\G(H^{k}_B(R))+k\cdot\gamma) + (j-i)\cdot\gamma, \quad \mbox{for all }i,j.
\]
Take $j=0$ and write $I:=(f_1,\hdots,f_m)$, we get 
\[
 \Supp_\G(H^i_B(R/I))\subset  \bigcup_{k\in\ZZ}(\Supp_\G(H^{k}_B(R))+(k-i)\cdot\gamma), \quad \mbox{for all }i.
\]
\end{rem}


\begin{exmp}\label{exmpBigraded}
Let $k$ be a field. Take $R_1:=k[x_1,x_2]$, $R_2:=k[y_1,y_2,y_3,y_4]$, and $\G:=\ZZ^2$. Write $R:=R_1\otimes_k R_2$ and set $\deg(x_i)=(1,0)$ and $\deg(y_i)=(0,1)$ for all $i$. Set $\aaa_1:=(x_1,x_2)$, $\aaa_2:=(y_1,y_2,y_3,y_4)$ and define $B:=\aaa_1\cdot \aaa_2 \subset R$ the irrelevant ideal of $R$, and $\mm:= \aaa_1+\aaa_2\subset R$, the ideal corresponding to the origin in $\Spec(R)$.

From Lemma \ref{LemloccohR}, it follows that
\begin{enumerate}\label{tableLCEx1}
 \item $H^2_B(R) \cong \check{R}_{\{1\}} \cong H^2_{\aaa_1}(R)=\omega_{R_1}^\vee\otimes_k R_2$,
 \item $H^4_B(R) \cong \check{R}_{\{2\}} \cong H^4_{\aaa_2}(R)=R_1\otimes_k\omega_{R_2}^\vee$,
 \item $H^5_B(R) \cong \check{R}_{\{1,2\}} \cong H^6_{\mm}(R)=\omega_{R}^\vee$,
 \item $H^\ell_B(R)=0$ for all $\ell\neq 2,4$ and $5$. 
\end{enumerate}

Hence, we see that
\begin{enumerate}\label{tableSuppEx1} 
 \item $\Supp_\G(H^2_B(R))= \Supp_\G(\check{R}_{1})=  Q_{\{1\}}=-\NN\times \NN+(-2,0)$, .
 \item $\Supp_\G(H^4_B(R))= \Supp_\G(\check{R}_{2})=  Q_{\{2\}}=\NN\times -\NN+(0,-4)$, .
 \item $\Supp_\G(H^5_B(R))= \Supp_\G(\check{R}_{1,2})=  Q_{\{1,2\}}=-\NN\times -\NN+(-2,-4)$, .
\end{enumerate}
\medskip

\begin{center}
 \includegraphics[scale=1]{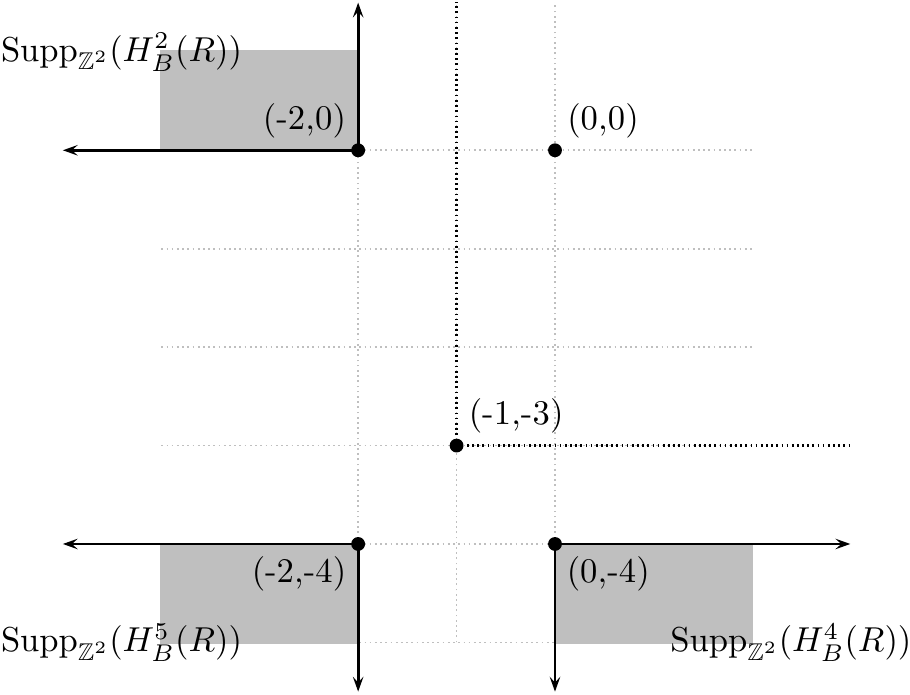}
\end{center}

\medskip
Take $f_1,\hdots,f_m$ homogeneous elements of bidegree $\gamma$, and write $I:=(f_1,\hdots,f_m)$. Assume $\cd_B(R/I)\leq 1$, hence $\cd_B(H_i)\leq 1$ for all $i$. We will compute $\reg(R/I)$.

Define for every $\gamma \in \G$, 
\begin{equation}\label{defSuppSigma}
 \SSup_B(\gamma):= \bigcup_{k\geq 0}(\Supp_\G(H^{k}_B(R))+k\cdot\gamma).
\end{equation}

Thus, in this case, we have
\[
 \SSup_B(\gamma):=(\Supp_\G(H^{2}_B(R))+2\cdot\gamma)\cup(\Supp_\G(H^{4}_B(R))+4\cdot\gamma)\cup(\Supp_\G(H^{5}_B(R))+5\cdot\gamma)
\]

Since $H^\ell_B(R)=0$ for all $\ell\neq 2,4$ and $5$, from \ref{RemEx1Koszul} we get that for all $i$, $\Supp_\G(H^i_B(R/I))\subset  \SSup_B(\gamma) -i\cdot \gamma$. By definition, $\reg(R/I)\supset \complement{\SSup_B(\gamma)}$. Take $\gamma:=(2,5)$ just to draw it.
\begin{center}
 \includegraphics[scale=1]{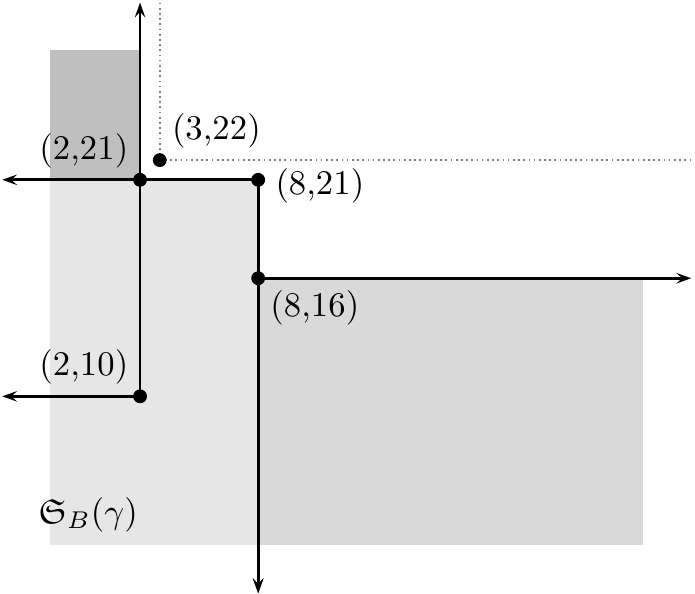}
\end{center}
\end{exmp}


\chapter[Implicit equation of multigraded hypersurfaces]{Implicit equation of multigraded hypersurfaces}
\label{ch:toric-pn}

\section{Introduction}

In this chapter we present a method for computing the implicit equation of a hypersurface given as the image of a rational map $\phi: \Xc \dashrightarrow \PP^n$,  where $\Xc$ is a normal toric variety. In Chapters \ref{ch:toric-emb-pn} and \ref{ch:toric-emb-p1xxp1}, the approach consisted in embedding the space $\Xc$ in a projective space, via a toric embedding. The need of the embedding comes from the necessity of a $\ZZ$-grading in the coordinate ring of $\Xc$, in order to study its regularity. 

The aim of this chapter is to give an alternative to this approach: we study the implicitization problem directly, without an embedding in a projective space, by means of the results of Chapter \ref{ch:CastelMum}. Indeed, we deal with the multihomogeneous structure of the coordinate ring $S$ of $\Xc$, and we adapt the method developed in Chapters \ref{ch:elimination}, \ref{ch:toric-emb-pn} and \ref{ch:toric-emb-p1xxp1} to this setting. The main motivations for our change of perspective are that it is more natural to deal with the original grading on $\Xc$, and that the embedding leads to an artificial homogenization process that makes the effective computation slower, as the number of variables to eliminate increases.

In Definition \ref{defRegion} we introduce the ``good'' region in $\G$ where the approximation complex $\Z.$ and the symmetric algebra $\SIR$ has no $B$-torsion. Indeed, we define for $\gamma\in \G$, $\Region(\gamma):=\bigcup_{0<k< \min\{m,\cd_B(R)\}} (\SSup_B(\gamma)-k\cdot \gamma)\subset \G$. This goes in the direction of proving the main theorem of this chapter, Theorem \ref{CorToricImplicit}. Precisely,  when $\Xc$ is a ($d-1$)-dimensional non-degenerate toric variety over a field $\kk$, and $S$ its Cox ring (cf.\ \ref{CoxRing}). For a rational map $\phi: \Xc \dto \PP^{d}$ defined by $d+1$ homogeneous elements of degree $\rho\in\D$. If $\dim(V(I))\leq 0$ in $\Xc$ and $V(I)$ is almost a local complete intersection off $V(B)$, we prove in Theorem  \ref{CorToricImplicit} that, 
\[
 \det((\Z.)_\gamma)=H^{\deg(\phi)}\cdot G \in \kk[\T],
\]
for all $\gamma\notin \Region(\rho)$, where $H$ stands for the irreducible implicit equation of the image of $\phi$, and $G$ is relatively prime polynomial in $\kk[\T]$. 

This result can be compared with Theorem \ref{teoResGral} and Corollary \ref{mainthT}.

\section{Commutative algebra tools}

\subsection{Regularity for commutative $\G$-graded rings}
Throughout this chapter let $\G$ be a finitely generated abelian group, and let $R$ be a commutative $\G$-graded ring with unity. Let $B$ be an homogeneous ideal of $R$. Take $m$ a positive integer and let $\f:=(f_0,\hdots,f_m)$ be a tuple of homogeneous elements of $R$, with $\deg(f_i)=\gamma_i$, and set $\bfgamma:=(\gamma_0,\hdots,\gamma_m)$. Write $I=(f_0,\hdots,f_m)$ for the homogeneous $R$-ideal generated by the $f_i$.

Our main motivation in Chapter \ref{ch:CastelMum} for considering regularity in general $\G$-gradings comes from toric geometry. Among $\G$-graded rings, homogeneous coordinate rings of a toric varieties are of particular interest in geometry. When $\Xc$ is a toric variety, $\G:=\DT$ is the (torus-invariant) divisor class group of $\Xc$. In this case, the grading can be related geometrically with the action of this group on the toric variety. Thus, as we mentioned in Remark \ref{remGrading} is of particular interest the case where $R$ is a polynomial ring in $n$ variables and $\G=\ZZ^n/K$, is a quotient of $\ZZ^n$ by some subgroup $K$. Note that, if $M$ is a $\ZZ^n$-graded module over a $\ZZ^n$-graded ring, and $\G=\ZZ^n/K$, we can give to $M$ a $\G$-grading coarser than its $\ZZ^n$-grading. For this, define the $\G$-grading on $M$ by setting, for each $\gamma\in \G$, $M_\gamma:=\bigoplus_{d\in \pi^{-1}(\gamma)}M_d$.

In this section we will present several results concerning vanishing of graded parts of certain modules. In our applications we will mainly focus on vanishing of Koszul cycles and homologies. We recall here what the support of a graded modules $M$ is. Recall from Definition \ref{defSuppGP} that for graded $R$-module $M$, we define the support of the module $M$ on $\G$ as $\Supp_\G(M):=\{\gamma \in \G:\ M_\gamma \neq 0\}$.

Recall that from Theorem \ref{ThmRegHGral} that for a complex $\C.$ of graded $R$-modules, for which one of the following holds
\begin{enumerate}
 \item For some $q\in \ZZ$, $H_j(\C.)=0$ for all $j<q$ and, $\cd_B(H_j(\C.))\leq 1$ for all $j>q$.  
 \item $\cd_B(H_j(\C.))\leq 1$ for all $j\in \ZZ$.
\end{enumerate}
we get that, for $i=0,1$, 
\[
 \Supp_\G(H^i_B(H_j(\C.)))\subset \bigcup_{k\in\ZZ}\Supp_\G(H^{i+k}_B(C_{j+k})).
\]

We have seen in that much of the information of the supports of the local cohomologies of the homologies of a complex $\C.$ is obtained from the supports of the local cohomologies of the complex. For instance, if $\C.$ is a free resolution of a graded $R$-module $Q$, the supports of the local cohomologies of $Q$ can be controlled in terms of the supports of the local cohomologies of the base ring $R$, and the shifts appearing in the $C_i$'{}s. 

In order to lighten the reading of this chapter, following equation \eqref{defSuppSigma}, we extend the definition as follows 

Let $P$ be a graded $R$-module. For every $\gamma \in \G$, we define
\begin{equation}
 \SSup_B(\gamma;P):= \bigcup_{k\geq 0}(\Supp_\G(H^{k}_B(P))+k\cdot\gamma).
\end{equation}
We will write $\SSup_B(\gamma):=\SSup_B(\gamma;R)$ as in equation \eqref{defSuppSigma}.

\begin{rem}\label{RemShiftPSigmaToric}
 Recall from Remark \ref{RemShiftPSigma} that for an $R$-module $P$, we denote by $P[\gamma']$ the shifted module by $\gamma'\in \G$, with $P[\gamma']_{\gamma}:=P_{\gamma'+\gamma}$. Hence, $\SSup_B(\gamma;P[\gamma'])=\SSup_B(\gamma;P)-\gamma'$.
\end{rem}

We apply Theorem \ref{ThmRegHGral} and Remark \ref{RemSuppCond} in the particular case where $\C.$ is the Koszul complex of a tuple $\f$ with coefficients in $P$, and we bound the support of the local cohomologies of its homologies in terms of the sets $\SSup_B(\gamma;P)$.

\medskip

Let $P$ be a $\G$-graded $R$-module. Denote by $\k.^P$ the Koszul complex $\k.(\f;R)\otimes_R P$. If the $f_i$ are $\G$-homogeneous of the same degree $\gamma$ for all $i$, the Koszul complex $\k.^P$ is $\G$-graded with $K_i:=\bigoplus_{l_0<\cdots<l_i}R(-i\cdot \gamma)$. Let $Z_i^P$ and $B_{i}^P$ be the Koszul $i$-th cycles and boundaries modules, with the grading that makes the inclusions $Z_i^P, B_{i}^P\subset K_i^P$ a map of degree $0\in \G$, and set $H_i^P=Z_i^P/ B_{i}^P$.

Recall that we have seen in Corollary \ref{Ex1Koszul} that if $\cd_B(H_i^P)\leq 1$ for all $i>0$, then, for all $j\geq 0$
\[
\Supp_\G(H^i_B(H_j^P))\subset \SSup_B(\gamma;P) + (j-i)\cdot\gamma.
\]

Recall from Remark \ref{RemEx1Koszul} that 
\begin{rem}\label{Ex2Koszul}
 If $\cd_B(H_i)\leq 1$ for all $i>0$,
\[
 \Supp_\G(H^i_B(H_j))\subset  \bigcup_{k\in\ZZ}(\Supp_\G(H^{k}_B(R))+k\cdot\gamma) + (j-i)\cdot\gamma, \quad \mbox{for all }i,j.
\]
Take $j=0$ and write $I:=(f_1,\hdots,f_m)$, we get 
\[
 \Supp_\G(H^i_B(R/I))\subset  \bigcup_{k\in\ZZ}(\Supp_\G(H^{k}_B(R))+(k-i)\cdot\gamma), \quad \mbox{for all }i.
\]
\end{rem}

The next result determines the supports of Koszul cycles in terms of the sets $\SSup_B(\gamma)$.

\begin{lem}\label{LemMultiLCRGral} Assume $f_0,\hdots,f_m\in R$ are homogeneous elements of same degree $\gamma$. Write $I=(f_0,\hdots,f_m)$. Fix a positive integer $c$. If $\cd_B(R/I)\leq c$, then the following hold
\begin{enumerate}
 \item $\Supp_\G(H^i(Z_q))\subset  (\SSup_B(\gamma)+(q+1-i)\cdot \gamma)\cup(\bigcup_{k\geq 0} \Supp_\G(H^{i+k}_B(H_{k+q}))\cdot \gamma)$, for $i\leq c$ and all $q\geq 0$.
 \item $\Supp_\G(H^i(Z_q))\subset \SSup_B(\gamma)+(q+1-i)\cdot \gamma$, for $i>c$ and all $q\geq 0$.
\end{enumerate}
\end{lem}
\begin{proof} 
Consider $\k.^{\geq q}: 0\to K_{m+1}\to K_{m}\to \cdots\to K_{q+1}\to Z_{q}\to 0$ the truncated Koszul complex. The double complex $\check \Cc^\bullet_B(\k.^{\geq q})$ gives rise to two spectral sequences. The first one has second screen $_2'E^i_j = H^i_B (H_j)$. This module is $0$ if $i>c$ or if $j>m+1-\grade(I)$.
The other one has as first screen
\[
 _1''E^i_j = \left\lbrace\begin{array}{ll} H^i_B (K_j) & \mbox{for all }i>r, \mbox{ and } j< q\\
						H^i_B (Z_q) & \mbox{for }q=j\\
						0 & \mbox{for all }i\leq r, \mbox{ and } j< q. 
\end{array}\right.
\]

From the second spectral sequence we deduce that if $\gamma'\in \G$ is such that $H^{i+k}_B (K_{q+k+1})_{\gamma'}$ vanishes for all $k\geq 0$, then $( _\infty''E^i_q)_{\gamma'}= H^{i}_B (Z_{q})_{\gamma'}$. Hence, if 
\begin{equation}\label{eqSuppHikK}
\gamma'\notin \bigcup_{k\geq 0}\Supp_\G(H^{i+k}_B(K_{k+q+1})) = \bigcup_{k\geq 0}(\Supp_\G(H^{k+i}_B(R)[-(k+q+1)\cdot\gamma]),
\end{equation}
then $( _\infty''E^i_q)_{\gamma'}= H^{i}_B (Z_{q})_{\gamma'}$

Comparing both spectral sequences, we have that for $\gamma'\notin \bigcup_{k\geq 0}\Supp_\G(H^{i+k}_B(H_{k+q}))$, we get $( _\infty''E^i_q)_{\gamma'}=0$.
This last condition is automatic for $i> c$, because $H^{i+k}_B(H_{k+q})=0$ for all $k\geq 0$. 
\end{proof}

\begin{cor}\label{LemMultiLCRGral2} Assume $f_0,\hdots,f_m\in R$ are homogeneous elements of degree $\gamma$. Write $I=(f_0,\hdots,f_m)$. Fix an integer $q$. If $\cd_B(R/I)\leq 1$, then the following hold
\begin{enumerate}
 \item for $i=0,1$, $\Supp_\G(H^i(Z_q))\subset (\SSup_B(\gamma)+(q-i)\cdot \gamma) \cup (\SSup_B(\gamma)+(q+1-i)\cdot \gamma)$.
 \item for $i>1$, $\Supp_\G(H^i(Z_q))\subset \SSup_B(\gamma)+(q+1-i)\cdot \gamma$.
\end{enumerate}
\end{cor}
\begin{proof} 
Since $\Supp_\G(H^{i+k}_B(H_{k+q}))\subset \SSup_B(\gamma)+(q-i)\cdot \gamma$, for all $k\geq 0$, gathering together this with equation \eqref{eqSuppHikK} and Lemma \ref{LemMultiLCRGral}, the result follows.
\end{proof}

\begin{rem}\label{LemCyclesGral2} 
 We also have empty support for Koszul cycles in the following cases.
\begin{enumerate}
 \item $H_B^{0} (Z_{p})=0$ for all $p$ if $\grade(B)\neq 0$ and
 \item $H^1_B(Z_p)=0$ for all $p$ if $\grade(B)\geq 2$. 
\end{enumerate}
\end{rem}
\begin{proof}
 The first claim follows from the inclusion $Z_p\subset K_p$ and the second from the exact sequence $0 \to Z_{p}\to K_{p}\to B_{p-1}\to 0$ that gives $0\to H_B^{0} (B_{p-1})\to H_B^{1} (Z_{p})\to H_B^{1} (K_{p})$, with $H_B^{0} (B_{p-1})$ as $B_{p-1}\subset K_{p-1}$.
\end{proof}

\subsection{$\G$-graded polynomial rings and approximation complexes}

We treat in this part the case of a finitely generated abelian group $\G$ acting on a polynomial ring $R$. Write $R:=\kk[X_1,\hdots,X_n]$. Take $H\vartriangleleft \ZZ^n$ a normal subgroup of $\ZZ^n$ and assume $\G= \ZZ^n/H$. The group $\G$ defines a grading on $R$ as was mentioned in \ref{remGrading}. 

Take $m+1$ homogeneous elements $\f:=f_0,\hdots,f_m\in R$ of fixed degree $\gamma\in\G$. Set $I=(f_0,\hdots,f_m)$ the homogeneous ideal of $R$ defined by $\f$. Recall that $\RIR:=\bigoplus_{l\geq 0} (It)^l\subset R[t]$. It is however important to observe that the grading in $\RIR$ is taken in such a way that the natural map $\alpha: R[T_0,\hdots,T_m] \to \RIR \subset R[t]: T_i\mapsto f_i t$ is of degree zero, and hence $(It)^l\subset R_{l \gamma}\otimes_\kk \kk[t]_l$.

\medskip

Let $\T:=T_0,\hdots,T_m$ be $m+1$ indeterminates. There is a surjective map of rings $\alpha: R[\T] \twoheadrightarrow \RIR$ with kernel $\pp:=\ker(\alpha)$. 

\begin{rem}
 Observe that $\pp\subset R[\T]$ is $(\G\times\ZZ_{\geq 0})$-graded, hence set $\pp_{(\mu;b)}\subset R_{\mu}\otimes_\kk \kk[\T]_b$, and $\pp_{(*,0)}=0$. Denote $\bb:=(\pp_{(*,1)})=(\set{\sum g_iT_i :g_i\in R, \sum g_if_i=0})$. Usually $\bb$ is called the $R[\T]$-ideal of syzygies and is written $\textnormal{Syz}(\f)$.
\end{rem}

The natural inclusion $\bb\subset \pp$ gives a surjection $\beta:\SIR\cong R[\T]/\bb {\twoheadrightarrow} R[\T]/\pp \cong \RIR$ that makes the following diagram commute
\begin{equation}\label{diagReesSym}
 \xymatrix@1{ 
 0\ar[r] &\bb \ar[r] \ar@{^(->}[d] &R[\T] \ar[r] \ar@{=}[d] &\SIR \ar[r] \ar@{>>}[d]^{\beta} &0 \\
 0\ar[r] &\pp \ar[r] &R[\T] \ar[r]^{\alpha} &\RIR \ar[r] &0
}
\end{equation}

Set $\k.= \k. ^R(\f)$ for the Koszul complex of $\f$ over the ring $R$. Write $K_i:=\bigwedge^i R[-i \gamma]^{m+1}$, and $Z_i$ and $B_i$ for the $i$-th module of cycles and boundaries respectively. We write $H_i=H_i(\f;R)$ for the $i$-th Koszul homology module. 

We write $\Z.$, $\B.$ and $\M.$ for the approximation complexes of cycles, boundaries and homologies (cf. \cite{HSV1}, \cite{HSV2} and \cite{Va94}). Define $\Zc_l=Z_l[l\gamma]\otimes_R R[\T]$, where  $(Z_l[l\gamma])_{\mu}=(Z_l)_{l\gamma+\mu}$. Similarly we define $\Bc_l=B_l[l\gamma]\otimes_R R[\T]$ and $\Mc_l=H_l[l\gamma]\otimes_R R[\T]$,

Let us recall some basic facts about approximation complexes that will be useful in the sequel. In particular, remind from Definition \ref{defLineartype} that the ideal $J\subset R$ is said to be of \textit{linear type} if $\SIR\cong \RIR$.

\begin{defn}
 The sequence $a_1,\hdots,a_l$ in $R$ is said to be a proper sequence if $a_{i+1} H_j(a_1,\hdots,a_i;R)=0$, for all $0\leq i \leq l, 0 < j \leq i$. 
\end{defn}
Notice that an almost complete intersection ideal is generated by a proper sequence.

Henceforward, we will denote $\HH_i:=H_i(\Z.)$ for all $i$.

\begin{lem}\label{LemaboutZ} With the notation above, the following statements hold:
\begin{enumerate}
 \item $\HH_0=\SIR$.
 \item $\HH_i$ is a $\SIR$-module for all $i$.
 \item If the ideal $I$ can be generated by a proper sequence then $\HH_i=0$ for $i>0$.
 \item If $I$ is generated by a $d$-sequence, then it can be generated by a proper sequence, and moreover, $I$ is of linear type.
\end{enumerate}
\end{lem}
\begin{proof}
 For a proof of these facts we refer the reader to \cite{Va94} or \cite{HSV2}.
\end{proof}

Assume the ideal $I=(\f)$ is of linear type out of $V(B)$, that is, for every prime $\qq\not\supset B$, $(\SIR)_\qq=(\RIR)_\qq$. 
The key point of study is the torsion of both algebras as $\kk[\T]$-modules. Precisely we have the following result.

\begin{lem} With the notation above, we have
\begin{enumerate}
 \item $\ann_{\kk[\T]}((\RIR)_{(\nu,*)})=\pp\cap \kk[\T]=\ker(\phi^*)$, if $R_\nu\neq 0$;
 \item if $I$ is of linear type out of $V(B)$ in $\Spec(R)$, then $\SIR/ H^0_B(\SIR)=\RIR$;
\end{enumerate}
\end{lem}
\begin{proof}
 The first part follows from the fact that $\pp$ is $\G\times \ZZ$-homogeneous and as $\RIR$ is a domain, there are no zero-divisors in $R$. By localizing at each point of $\Spec(R)\setminus V(B)$ we have the equality of the second item.
\end{proof}

This result suggest that we can approximate one algebra by the other, when they coincide outside $V(B)$. 

\begin{lem}\label{ToricZacyclic1}
 Assume $B\subset \rad(I)$, then $\HH_i$ is $B$-torsion for all $i>0$.
\end{lem}
\begin{proof} 
 Let $\pp \in \Spec(R)\setminus V(B)$. In particular $\pp \in \Spec(R)\setminus V(I)$, hence, $(H_i)_\pp=0$. This implies that the complex $\M.$ (cf. \cite{HSV2}) is zero, hence acyclic, at localization at $\pp$. It follows that $(\Z.)_\pp$ is also acyclic \cite[Prop. 4.3]{BuJo03}.
\end{proof}

This condition can be carried to a cohomological one, by saying $\cd_B(R/I)=0$. Note that since $V(I)$ is empty in $\Xc$, then $V(I)\subset V(B)$ in $\Spec(R)$, then $H^i_B(R/I)=0$ for $i>0$. Thus, this conditions can be relaxed by bounding $\cd_B(R/I)$. 

We now generalize Lemma \ref{ToricZacyclic1} for the case when $V(I)\nsubseteq V(B)$. 

We will consider  $\cd_B(R/I)\leq 1$ for the sequel in order to have convergence at step $2$ of the horizontal spectral sequence.

Before getting into the next result, recall that $\Zc_q:=Z_q[q\cdot \gamma]\otimes_\kk \kk[\T]$. It follows that 
\begin{equation}\label{EqSuppZcq}
 \Supp_{\G}(H_B^{k}(\Zc_{q+k}))=\Supp_{\G}(H_B^{k}(Z_{q+k}))-q\cdot \gamma\subset 
\left\lbrace\begin{array}{ll} \SSup_B(\gamma)+(1-k)\cdot \gamma & \mbox{for }k>1, \\
 \SSup_B(\gamma)\cup(\SSup_B(\gamma)-\gamma) & \mbox{for }k=1, \\
 (\SSup_B(\gamma)+\gamma) \cup \SSup_B(\gamma) & \mbox{for }k=0\end{array}\right.
\end{equation}

Observe, that any of this sets on the right do not depend on $q$. Furthermore, if $\grade(B)\geq 2$, we have seen in Remark \ref{LemCyclesGral2} 
 $H_B^{0} (Z_{p})=H^1_B(Z_p)=0$ for all $p$. hence, we define:

\begin{defn}\label{defRegion}
For $\gamma\in \G$, set 
\[
 \Region(\gamma):=\bigcup_{0<k< \min\{m,\cd_B(R)\}} (\SSup_B(\gamma)-k\cdot \gamma)\subset \G.
\]
\end{defn}

\begin{thm}\label{ToricZacyclic2}
Assume that $\grade(B)\geq 2$ and $\cd_B(R/I)\leq 1$. Then, if $\mu \notin \Region(\gamma)$,
\[
 H_B^i(\HH_{j})_{\mu}=0, \quad \mbox{for all }i,j.
\]

\end{thm}
\begin{proof} Consider the two spectral sequences that arise from the double complex $\check{\Cc}_B^\bullet\Z.$. Since $\supp (H_p)\subset I$, the first spectral sequence has at second screen $ _2'E^i_j = H^i_B \HH_j $. The condition $\cd_B(R/I)\leq 1$ gives that this spectral sequences stabilizes at the second step with
\[
_\infty'E^i_j =\ _2'E^i_j = H^i_B \HH_j = \left\lbrace\begin{array}{ll} \HH_j & \mbox{for }i=0\ \mbox{and }j> 0, \\
 H^1_B(\HH_j) & \mbox{for }i=1\ \mbox{and }j> 0, \\
 H^i_B (\SIR) & \mbox{for }j=0, \mbox{and all }i\\
0 & \mbox{otherwise.} \end{array}\right.
\]

The second spectral sequence has at first screen $'' _1 E^i_j= H^i_B(\Zc_j)$. Since $R[\T]$ is $R$-flat, $H^i_B(\Zc_j)=H^i_B(Z_j[j\gamma])\otimes_\kk \kk[\T]$. From and Remark \ref{LemCyclesGral2} the top line vanishes for $j>0$, as well as the upper-left part. 

 Comparing both spectral sequences, we deduce that the vanishing of $H_B^{k}(\Zc_{p+k})_{\mu}$ for all $k$, implies the vanishing of $H_B^{k}(\HH_{p+k})_{\mu}$ for all $k$.

Finally, from equation \eqref{EqSuppZcq} we have that if $\mu \notin \Region(\gamma)$ (which do not depend on $p$), then we obtain $H_B^i(\HH_j)_{\mu}=0$.
\end{proof}

\begin{lem}\label{CorToricZacyclic}
Assume $\grade(B)\geq 2$, $\cd_B(R/I)\leq 1$ and $I_\pp$ is almost a local complete intersection for every $\pp\notin V(B)$. Then, for all $\mu\notin \Region(\gamma)$, the complex $(\Z.)_{\mu}$ is acyclic and $H^0_B(\SIR)_{\mu}=0$.
\end{lem}
\begin{proof}
 Since $I_\pp$ is almost a local complete intersection for every $\pp\notin V(B)$, $\Z.$ is acyclic off $V(B)$. Hence, $\HH_q$ is $B$-torsion for all positive $q$. Since $\HH_q$ is $B$-torsion, $H^k_B(\HH_q)=0$ for $k>0$ and $H^0_B(\HH_q)=\HH_q$. From Theorem \ref{ToricZacyclic2} we have that $(\HH_q)_{\mu}=0$, and $H_B^0(\HH_0)_{\mu}=0$.
\end{proof}

\medskip


\section{The implicitization of toric hypersurfaces}

In this part we focus on the study of the closed image of rational maps defined over a toric variety. This subject has been attacked in several articles with many different approaches. The problem of computing the equations defining the closed image of a rational map is an open research area with several applications.

\medskip

Let $\Xc$ be a non-degenerate toric variety over a field $\kk$, $\Delta$ be its fan in the lattice $N\cong \ZZ^d$ corresponding to $\Xc$, and write $\Delta(i)$ for the set of $i$-dimensional cones in $\Delta$ as before. Denote by $S$ the Cox ring of $\Xc$.

Henceforward we will focus on the study of the elimination theory as we have done in Chapters \ref{ch:elimination} \ref{ch:toric-emb-pn} and \ref{ch:toric-emb-p1xxp1} in a different context. This aim brings us to review some basic definitions and properties.

Assume we have a rational map $\phi: \Xc \dto \PP^m$, defined by $m+1$ homogeneous elements $\f:=f_0,\hdots,f_m\in S$ of fixed degree $\rho\in\D$. Precisely, any cone $\sigma\in \Delta$ defines an open affine set $U_\sigma$ (cf. \cite{Cox95}), and two elements $f_i, f_j$ define a rational function $f_i/f_j$ on some affine open set $U_\sigma$, and this $\sigma$ can be determine from the monomials appearing in $f_j$. In particular, if $\Xc$ is a multiprojective space, then $f_i$ stands for a multihomogeneous polynomial of multidegree $\rho\in \ZZ_{\geq 0} \times \cdots \times \ZZ_{\geq 0}$.

We recall that for any $\D$-homogeneous ideal $J$, $\Proj_{\Xc}(S/J)$ simply stands for the gluing of the affine scheme $\Spec((S/J)_\sigma)$ on every affine chart $\Spec(S_\sigma)$, to $\Xc$. It can be similarly done to define from $\D\times \ZZ$-homogeneous ideals of $S\otimes_\kk \kk[\T]$, subschemes of $\Xc \times_\kk \PP^d$, and this projectivization functor will be denoted $\Proj_{\Xc \times \PP^m}(-)$. The graded-ungraded scheme construction will be denoted by $\Proj_{\Xc \times \AA^{m+1}}(-)$. For a deep examination on this subject, we refer the reader to \cite{Fu93}, and \cite{Cox95}. 

\begin{defn}\label{defToricSetting}
 Set $I:=(f_0,\hdots,f_n)$ ideal of $S$. Define $\Sc:=\Proj_{\Xc}(S/I)$ and $\Sc^{\textnormal{red}}:=\Proj_{\Xc}(S/\rad(I))$, the base locus of $\phi$. denote by $\Omega:=\Xc\setminus \Sc$, the domain of definition of $\phi$. 

Let $\Gamma_0$ denote the graph of $\phi$ over $\Omega$, and $\Gamma:=\overline{\Gamma_0}$ its closure in $\Xc \times \PP^m$. Scheme-theoretically we have $\Gamma=\Proj_{\Xc \times \PP^m}(\RIR)$, where $\RIR:=\bigoplus_{l\geq 0} (It)^l\subset S[t]$. 
\end{defn}

Recall that the two surjections, $S[\T]\to \SIR$ and $\beta:\SIR\to \RIR$, established on Diagram \ref{diagReesSym}, correspond to a chain of embedding $\Gamma \subset \Upsilon \subset \Xc \times \PP^m$, where $\Upsilon = \Proj_{\Xc \times \PP^m}(\SIR)$.

Assume the ideal $I$ is of linear type off $V(B)$, that is, for every prime $\qq\not\supset B$, $(\SIR)_\qq=(\RIR)_\qq$. Since Sym and Rees commute with localization, $\Proj_{\Xc \times \PP^m}(\SIR)=\Proj_{\Xc \times \PP^m}(\RIR)$, that is $\Upsilon=\Gamma$ in $ \Xc \times \PP^m$. Moreover, $\Proj_{\Xc\times \AA^{m+1}}(\SIR)$ and $\Proj_{\Xc\times \AA^{m+1}}(\RIR)$ coincide in $ \Xc\times \AA^{m+1}$. Recall that this in general does not imply that $\SIR$ and $\RIR$ coincide, in fact this is almost never true: as $\RIR$ is the closure of the graph of $\phi$ which is irreducible, it is an integral domain, hence, torsion free; on the other hand, $\SIR$ is almost never torsion free.
 
\begin{rem}\label{gradeBinS}
 Observe that it can be assumed without loss of generality that $\grade (B)\geq 2$.
\end{rem}

\begin{lem}\label{lemDimyCDToric}
 If $\dim (V(I)) \leq 0$ in $\Xc$, then $\cd_B(S/I)\leq 1$.
\end{lem}
\begin{proof}
For any finitely generated $S$-module $P$ and all $i>0$, from Equation \eqref{equLC-SCforS} $H^i_{\ast}(\Xc,P^\sim)\cong H^{i+1}_B(P)$.
Applying this to $P=S/I$, for all $\rho\in \D$ we get that
\[
 H^i(\Xc, (S/I)^\sim (\rho))=H^i(V(I), \OO_{V(I)}(\rho)),
\]
 that vanishes for $i>0$, since $\dim V(I) \leq 0$.
\end{proof}

\begin{thm}\label{CorToricImplicit}
 Let $\Xc$ be a ($d-1$)-dimensional non-degenerate toric variety over a field $\kk$, and $S$ its Cox ring. Let $\phi: \Xc \dto \PP^{d}$ be a rational map, defined by $d+1$ homogeneous elements $f_0,\hdots,f_{d}\in S$ of fixed degree $\rho\in\D$. Denote $I=\paren{f_0,\hdots,f_{d}}$. If $\dim(V(I))\leq 0$ in $\Xc$ and $V(I)$ is almost a local complete intersection off $V(B)$, then
\[
 \det((\Z.)_\gamma)=H^{\deg(\phi)}\cdot G \in \kk[\T],
\]
for all $\gamma\notin \Region(\rho)$, where $H$ stands for the irreducible implicit equation of the image of $\phi$, and $G$ is relatively prime polynomial in $\kk[\T]$.
\end{thm}
\begin{proof}
 This result follows in the standard way, similar to the cases of implicitization problems in other contexts.

Recall that $\Gamma$ is the closure of the graph of $\phi$, hence, defined over $\Omega$. The bihomogeneous structure in $S\otimes_\kk \kk[\T]$ gives rise to two natural scheme morphisms $\Xc \nto{\pi_1}{\ot} \Xc \times_{\kk} \PP^d \nto{\pi_2}{\to} \PP^d $. It follows directly that $\pi_2=\pi_1\circ\phi$ over the graph of $\phi$, $\pi_1^{-1}(\Omega)$.

From Corollary \ref{CorToricImplicit}, the complex of $\OO_{\PP^d}$-modules $(\Z.)^\sim$ is acyclic over $\Xc\times_\kk \PP^{d}$. We verify by localization that this complex has support in $\Upsilon$, hence, $H^0(\Xc\times_\kk \PP^{d}, (\Z.)^\sim)=H^0(\Upsilon, (\Z.)^\sim)=\SIR$. Naturally, $G$ defines a divisor in $\PP^{d}$ with support on $\pi_2(\Upsilon\setminus \Gamma)$, and $\Upsilon$ and $\Gamma$ coincide outside $\Sc\times \PP^{d}$. 

Following \cite{KMun}, due to the choice of $\gamma\notin \Region(\rho)$, one has:
\[
 \begin{array}{rl}
[\det((\Z.)_\nu)]&=\div_{\kk[\X]}(H_0(\Z.)_\gamma)=\div_{\kk[\X]}(\SIR_\gamma)\\
&=\sum_{\mbox{\scriptsize $\begin{array}{c}\qq \textnormal{ prime, }\\ \codim_{\kk[\X]}(\qq)=1 \end{array}$}}\length_{\kk[\X]_{\qq}} ((\SIR_\gamma)_\qq)[\qq].
\end{array}
\]
Thus, for all $\gamma\notin \Region(\rho)$, we obtain
\[
 [\det((\Z.)_\gamma)]= \length_{\kk[\X]_{(H)}} ((\SIR_\gamma)_{(H)})[(H)]+\hspace{-0.4cm}\sum_{\mbox{\scriptsize $\begin{array}{c}\qq \textnormal{ prime,}\\ V(\qq) \not\subset V(H)\\ \codim_{\kk[\X]}(\qq)=1 \end{array}$}} \hspace{-0.4cm}\length_{\kk[\X]_{\qq}} ((\SIR_\gamma)_{\qq})[\qq].
\]
It follows that the first summand is the divisor associated to $G$, and the second one, the divisor associated to $H^{\deg(\phi)}$.
\end{proof}

We next give a detailed description of the extra factor $G$, as given in \cite[Prop. 5]{BCJ06}.

\begin{rem}\label{remExtraFactor}
 Let $\Xc$, $S$, $\phi: \Xc \dto \PP^{d}$, $H$ and $G$ be as in Theorem \ref{CorToricImplicit}. If $\kk$ is algebraically closed, then $G$ can be written as 
\[
 G=\prod_{\mbox{\scriptsize $\begin{array}{c}\qq \textnormal{ prime, } V(\qq) \not\subset V(H)\\ \codim_{\kk[\X]}(\qq)=1 \end{array}$}} L_\qq^{e_\qq-l_\qq}.
\]
in $\kk[\T]$, where $e_\qq$ stands for the Hilbert-Samuel multiplicity of $\SIR$ at $\qq$, and $l_{\qq}$ denotes $\length_{\kk[\X]_{\qq}}$.
\end{rem}
\begin{proof}
 The proof follows the same lines of that of \cite[Prop. 5]{BCJ06}. It is just important to observe that \cite[Lemma 6]{BCJ06} is stated for a Cohen-Macaulay ring as is $S$ for us.
\end{proof}

The main idea behind this remark is that only non-complete intersections points in $\Sc$ yield the existence of extra factors as in Chapters \ref{ch:elimination} and \ref{ch:toric-emb-pn}. If $I$ is locally a complete intersection at $\qq\in \Sc$, then $I_\qq$ is of linear type, hence, $(\SIR)_\qq$ and $(\RIR)_\qq$ coincide. Thus, $\Proj_{\Xc \times \PP^m}(\SIR)$ and $\Proj_{\Xc \times \PP^m}(\RIR)$ coincide over $\qq$.


\medskip

\section{Multiprojective spaces and multigraded polynomial rings}

In this section we focus on the better understanding of the multiprojective case. Here we take advantage of the particular structure of the ring. This will permit, as in Chapter \ref{ch:CastelMum}, to precise results to determine the regions of the vanishing of the local cohomology modules.

The problem of computing the implicit equation of a rational multiprojective hypersurface is surely the most important among toric cases of implicitization. The theory follows as a particular case of the one developed in the section before, but many results can be better precise, and better understood. In this case, the grading group is $\ZZ^s$, which permits a deeper insight in the search for a ``good zone'' for $\gamma$. The aim of this paragraph is to show that in this region, approximation complexes behave well enough, allowing the computation of the implicit equation (perhaps with extra factors) as a determinant of a graded branch of a $\Zc$-complex, as we have done in Chapters \ref{ch:elimination} and \ref{ch:toric-emb-pn}.

In what follows for the rest of this section, we will follow the following convention. Let $s$ and $m$ be fixed positive integers, $r_1\leq\cdots\leq r_s$ non-negative integers, and write $\x_i=(x^i_0,\hdots,x^i_{r_i})$ for $1\leq i \leq s$. Let $f_0,\hdots,f_m\in \bigotimes_\kk \kk[\x_i]$ be multihomogeneous polynomials of multidegree $d_i$ on $\x_i$. Assume we are given a rational map
\begin{equation}\label{eqPhiMultiProj}
 \phi: \prod_{1\leq i \leq s}\PP^{r_i} \dto \PP^m: \x:=(\x_1)\times\cdots\times(\x_s) \mapsto (f_0:\cdots:f_m)(\x).
\end{equation}

Take $m$ and $r_i$ such that $m=1+\sum_{1\leq i \leq s}{r_i}$. 
Write $R_i:= \kk[\x_i]$ for $1\leq i \leq s$, $R=\bigotimes_\kk R_i$, and $R_{(a_1,\hdots,a_s)}:=\bigotimes_\kk (R_i)_{a_i}$ stands for its bigraded part of multidegree $(a_1,\hdots,a_s)$. Hence, $\dim R_i=r_i+1,$ and $\dim R=r+s$, and $\prod_{1\leq i \leq s}\PP^{r_i}=\Mproj(R)$. Set $\aaa_i:=(\x_i)$, ideal of $R_i$, and take $\mm:=\sum_{1\leq i \leq s} \aaa_i$ the irrelevant ideal of $R$, and $B:=\bigcap_{1\leq i \leq s} \aaa_i$ the empty locus of $\Mproj(R)$. Set also $I:=(f_0,\hdots,f_m)$ for the multihomogeneous ideal of $R$, and $X=\Mproj(R/I)$ the base locus of $\phi$.

Set-theoretically, write $V(I)$ for the base locus of $\phi$, and $\Omega:=\prod_{1\leq i \leq s}\PP^{r_i}\setminus V(I)$ the domain of definition of $\phi$. Let $\Gamma_0$ denote the graph of $\phi$ over $\Omega$, and $\Gamma:=\overline{\Gamma_0}$ its closure in $(\prod_{1\leq i \leq s}\PP^{r_i}) \times \PP^m$. Scheme-theoretically we have $\Gamma=\Mproj(\RIR)$, where $\RIR:=\bigoplus_{l\geq 0} (It)^l\subset R[t]$. The grading in $\RIR$ is taken in such a way that the natural map $\alpha: R[T_0,\hdots,T_m] \to \RIR \subset R[t]: T_i\mapsto f_i t$ is of degree zero, and hence $(It)^l\subset R_{(ld_1,\hdots,ld_s)}\otimes_\kk \kk[t]_l$.

\begin{rem}\label{lemDimyCDMultiProj}
 From Lemma \ref{lemDimyCDToric} we have that if $\dim (V(I)) \leq 0$ in $\PP^{r_1}\times\cdots\times\PP^{r_s}$, then $\cd_B(R/I)\leq 1$.
\end{rem}
\begin{rem}
It is clear that if $\gamma\in \NN^s$, then, $(\SSup_B(\gamma)-k\cdot \gamma)\supset (\SSup_B(\gamma)-(k+1)\cdot \gamma)$ for all $k\geq 0$. Thus, from Definition \ref{defRegion}, we see that 
for all $\gamma\in \NN^s$, 
\[
 \Region(\gamma)=\SSup_B(\gamma)-\gamma.
\]
\end{rem}

Theorem \ref{CorToricImplicit} and Remark \ref{remExtraFactor} can be applied verbatim since $\prod_{1\leq i \leq s}\PP^{r_i}$ is a toric variety. We have that

\begin{thm}
Let $\phi: \prod_{1\leq i \leq s}\PP^{r_i} \dto \PP^m$ be a rational map, as in \eqref{eqPhiMultiProj}, defined by $m+1$ homogeneous elements $f_0,\hdots,f_{m}\in S$ of the same degree $\rho=(d_0,\hdots,d_m)$. Denote $I=\paren{f_0,\hdots,f_m}$. Assume $\dim V(I)\leq 0$ in $\prod_{1\leq i \leq s}\PP^{r_i}$ and $V(I)$ is almost a local complete intersection off $V(B)$. Then,
\[
 \det((\Z.)_\gamma)=H^{\deg(\phi)}\cdot G \in \kk[\T],
\]
for all $\gamma\notin \Region(\rho)$, where $H$ stands for the irreducible implicit equation of the image of $\phi$, and $G$ is relatively prime polynomial in $\kk[\T]$. 

Moreover, if $\kk$ is algebraically closed, then $G$ can be written as 
\[
 G=\prod_{\mbox{\scriptsize $\begin{array}{c}\qq \textnormal{ prime, } V(\qq) \not\subset V(H)\\ \codim_{\kk[\X]}(\qq)=1 \end{array}$}} L_\qq^{e_\qq-l_\qq}.
\]
in $\kk[\T]$, where $e_\qq$ stands for the Hilbert-Samuel multiplicity of $\SIR$ at $\qq$, and $l_{\qq}$ denotes $\length_{\kk[\X]_{\qq}}$.
\end{thm}
\begin{proof}
 Take $\Region(\rho)$ as in Definition \ref{defRegion}. From Lemma \ref{lemDimyCDMultiProj} we have that $\cd(R/I)\leq 1$. Thus, the result follows by taking $\gamma\notin \Region(\rho)$ and using Theorem \ref{CorToricImplicit} and Remark \ref{remExtraFactor}.
\end{proof}


\section{Examples}

\begin{exmp}
 We will follow Example \ref{exmpBigraded}. Thus, let $k$ be a field. Assume $\Xc$ is the biprojective space $\PP^1_\kk\times \PP^3_\kk$. Take $R_1:=k[x_1,x_2]$, $R_2:=k[y_1,y_2,y_3,y_4]$, and $\G:=\ZZ^2$. Write $R:=R_1\otimes_k R_2$ and set $\deg(x_i)=(1,0)$ and $\deg(y_i)=(0,1)$ for all $i$. Set $\aaa_1:=(x_1,x_2)$, $\aaa_2:=(y_1,y_2,y_3,y_4)$ and define $B:=\aaa_1\cdot \aaa_2 \subset R$ the irrelevant ideal of $R$, and $\mm:= \aaa_1+\aaa_2\subset R$, the ideal corresponding to the origin in $\Spec(R)$.

Recall that
\begin{enumerate}
 \item $H^2_B(R) \cong \check{R}_{\{1\}} \cong H^2_{\aaa_1}(R)=\omega_{R_1}^\vee\otimes_k R_2$,
 \item $H^4_B(R) \cong \check{R}_{\{2\}} \cong H^4_{\aaa_2}(R)=R_1\otimes_k\omega_{R_2}^\vee$,
 \item $H^5_B(R) \cong \check{R}_{\{1,2\}} \cong H^6_{\mm}(R)=\omega_{R}^\vee$,
 \item $H^\ell_B(R)=0$ for all $\ell\neq 2,4$ and $5$. 
\end{enumerate}

Thus,
\begin{enumerate}
 \item $\Supp_\G(H^2_B(R))= \Supp_\G(\check{R}_{1})=  Q_{\{1\}}=-\NN\times \NN+(-2,0)$, .
 \item $\Supp_\G(H^4_B(R))= \Supp_\G(\check{R}_{2})=  Q_{\{2\}}=\NN\times -\NN+(0,-4)$, .
 \item $\Supp_\G(H^5_B(R))= \Supp_\G(\check{R}_{1,2})=  Q_{\{1,2\}}=-\NN\times -\NN+(-2,-4)$, .
\end{enumerate}
\medskip

We have seen that
\begin{center}
 \includegraphics[scale=1]{supportLC-CMReg}
\end{center}

Recall that $f_1,\hdots,f_m$ are homogeneous elements of bidegree $\gamma$, and $I:=(f_1,\hdots,f_m)$. Assume $\cd_B(R/I)\leq 1$, hence $\cd_B(H_i)\leq 1$ for all $i$. We have $\reg(R/I)$, and 
\[
 \SSup_B(\gamma)=(\Supp_\G(H^{2}_B(R))+2\cdot\gamma)\cup(\Supp_\G(H^{4}_B(R))+4\cdot\gamma)\cup(\Supp_\G(H^{5}_B(R))+5\cdot\gamma),
\]
as in the picture
\begin{center}
 \includegraphics[scale=1]{RegionS2}
\end{center}
\[
 \Region(\gamma)=\SSup_B(\gamma)-\gamma
\]
Thus, we have that 
\[
 \complement\Region(2,5)=(\NN^2+(1,17))\cup(\NN^2+(7,12)).
\]
\begin{center}
 \includegraphics[scale=1]{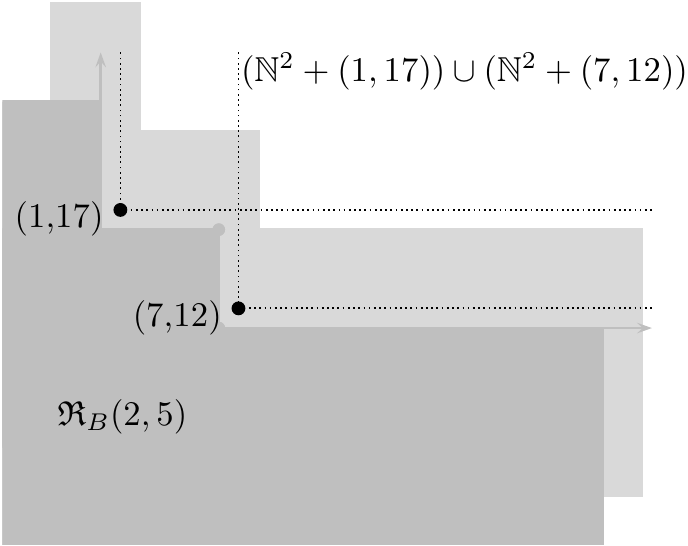}
\end{center}

Consider $\phi:\PP^1\times \PP^3\dto \PP^5 $ given by $f_0,\hdots,f_5\in R$ homogeneous polynomials of bidegree $(2,5)\in \ZZ^2$.

Taking $\mu\notin \Region(2,5)$, the approximation complex of cycles associated to $f_0,\hdots,f_5$ in degree $\nu$ is acyclic and $\Sym(f_0,\hdots,f_5)$ has no $B$-torsion. We conclude that we can compute the implicit equation of $\phi$ as a factor of $\det((\Z.)_{(\mu,\ast)})$ for $\mu\notin \Region(2,5)$.
\end{exmp}

\begin{exmp}\label{exmpBigradedCh7}
Consider the rational map 
\begin{equation}
\begin{array}{ccc}
 \PP^1\times \PP^1 	&\nto{f}{\dto}	& \PP^3\\
(s:u)\times (t:v) 	&\mapsto	& (f_1:f_2:f_3:f_4)
\end{array}
\end{equation}
where the polynomials $f_i=f_i(s,u,t,v)$ are bihomogeneous of bidegree $(2,3)\in \ZZ^2$ given by
\begin{itemize}
 \item $f_1=s^2t^3+2sut^3+3u^2t^3+4s^2t^2v+5sut^2v+6u^2t^2v+7s^2tv^2+8sutv^2+9u^2tv^2+10s^2v^3+suv^3+2u^2v^3$,
 \item $f_2=2s^2t^3-3s^2t^2v-s^2tv^2+sut^2v+3sutv^2-3u^2t^2v+2u^2tv^2-u^2v^3$,
 \item $f_3=2s^2t^3-3s^2t^2v-2sut^3+s^2tv^2+5sut^2v-3sutv^2-3u^2t^2v+4u^2tv^2-u^2v^3$,
 \item $f_4=3s^2t^2v-2sut^3-s^2tv^2+sut^2v-3sutv^2-u^2t^2v+4u^2tv^2-u^2v^3$.
\end{itemize}

Our aim is to get the implicit equation of the hypersurface $\overline{\im(f)}$ of $\PP^3$.
Let us start by defining the parametrization $f$ given by $(f_1,f_2,f_3,f_4)$.

Thus, let $k$ be a field. Assume $\Xc$ is the biprojective space $\PP^1_\kk\times \PP^1_\kk$. Take $R_1:=k[x_1,x_2]$, $R_2:=k[y_1,y_2]$, and $\G:=\ZZ^2$. Write $R:=R_1\otimes_k R_2$ and set $\deg(x_i)=(1,0)$ and $\deg(y_i)=(0,1)$ for all $i$. Set $\aaa_1:=(x_1,x_2)$, $\aaa_2:=(y_1,y_2)$ and define $B:=\aaa_1\cdot \aaa_2 \subset R$ the irrelevant ideal of $R$, and $\mm:= \aaa_1+\aaa_2\subset R$, the ideal corresponding to the origin in $\Spec(R)$.

Recall that
\begin{enumerate}
 \item $H^2_B(R) \cong\omega_{R_1}^\vee\otimes_k\omega_{R_2}^\vee$,
 \item $H^3_B(R) \cong \check{R}_{\{1,2\}} \cong H^4_{\mm}(R)=\omega_{R}^\vee$,
 \item $H^\ell_B(R)=0$ for all $\ell\neq 2$ and $3$. 
\end{enumerate}

Thus,
\begin{enumerate}
 \item $\Supp_\G(H^2_B(R))= \Supp_\G(\check{R}_{1})\cup\Supp_\G(\check{R}_{2})=  Q_{\{1\}}=-\NN\times \NN+(-2,0)\cup\NN\times -\NN+(0,-2)$.
 \item $\Supp_\G(H^3_B(R))= \Supp_\G(\check{R}_{1,2})=  Q_{\{1,2\}}=-\NN\times -\NN+(-2,-2)$, .
\end{enumerate}
\medskip

We have seen that
\begin{center}
 \includegraphics[scale=1]{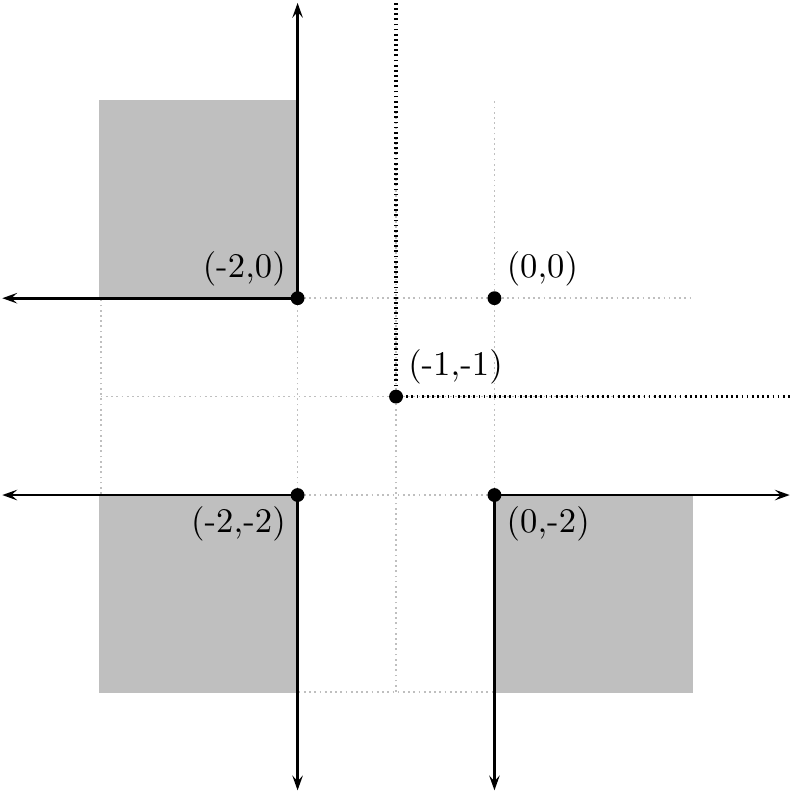}
\end{center}

\[
 \SSup_B(2,3)=(\Supp_\G(H^{2}_B(R))+2\cdot(2,3))\cup(\Supp_\G(H^{3}_B(R))+3\cdot(2,3)).
\]
Hence,
\[
 \Region(2,3)=(\Supp_\G(H^{2}_B(R))+(2,3))\cup(\Supp_\G(H^{3}_B(R))+2\cdot(2,3)).
\]
Thus, 
\[
 \complement\Region(2,3)=(\NN^2+(1,5))\cup(\NN^2+(3,2)).
\]
As we can see in Example \ref{example1Ch9}, a Macaulay2 computation gives exactly this region (illustrated below) as the acyclicity region for $\Z.$.
\begin{center}
 \includegraphics{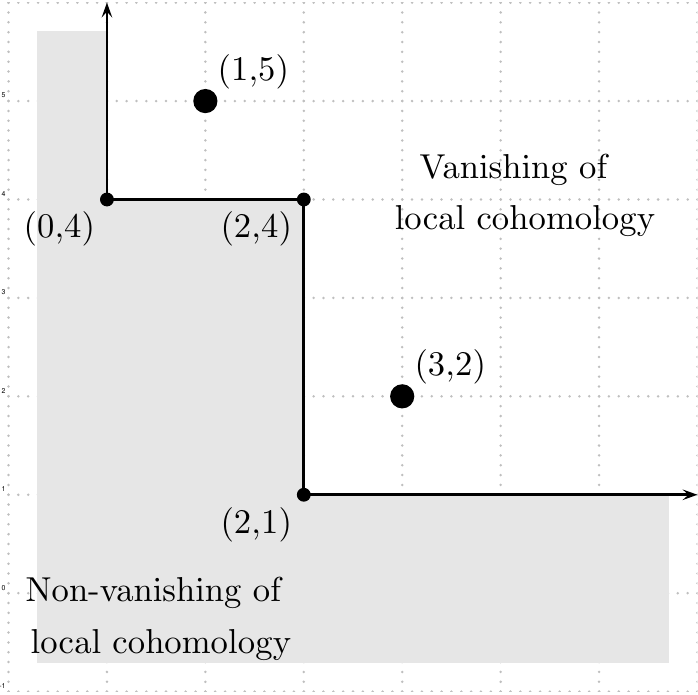}
\end{center}

When $\nu_0=(3,2)$ or $\nu_0=(1,5)$, we get a complex
\[
 (\Z.)_{\nu_0}: 0\to0\to0\to\kk[\X]^{12}\nto{M_{\nu_0}}{\lto}\kk[\X]^{12}\to0.
\]
and, thus, $\det((\Z.)_{\nu_0})=\det(M_{\nu_0})\in \kk[\X]_{12}$ is an homogeneous polynomial of degree $12$ that vanishes on the closed image of $\phi$.

\end{exmp}

\chapter[Algorithm1]{A package for computing implicit equations from toric surfaces}
\label{ch:ch-algor-ToricP2}


\section{Introduction}

Let $\Tc$ be a two-dimensional projective toric variety, and let $f:\Tc \dto \PP^3$ be a generically finite rational map. Hence, $\Sc:=\overline{\im(f)}\subset \PP^3$ is a hypersurface. In Chapter \ref{ch:toric-emb-pn}, following \cite{BDD08} and \cite{Bot09}, we showed how to compute an implicit equation for $\Sc$, assuming that the base locus $X$ of $f$ is finite and locally an almost complete intersection. As we mentioned in Chapter \ref{ch:toric-emb-pn}, this is a further generalization of the results in Chapter \ref{ch:elimination}, which follows \cite{BuJo03, BC05, Ch06}, on implicitization of rational hypersurfaces via approximation complexes; we also generalize \cite{BD07}.

This chapter corresponds to a recent sent article in collaboration with Marc Dohm, entiled \textit{A package for computing implicit equations of parametrizations from toric surfaces} (cf. \cite{BD10})

We showed in Section \ref{sec3Pn} and Section \ref{sec:equation} how to compute a symbolic matrix of linear syzygies $M$, called \textit{representation matrix} of $\Sc$, with the property that, given a point $p\in \PP^3$, the rank of $M(p)$ drops if $p$ lies in the surface $\Sc$. When the base locus $X$ is locally a complete intersection, we get that the rank of $M(p)$ drops if and only if $p$ lies in the surface $S$. 

We begin by recalling the notion of a representation matrix (see Definition \ref{defRepresentationMatrix}).

\begin{defn}
Let $\Sc \subset \PP^n$ be a hypersurface. A matrix $M$ with entries in the polynomial ring $\kk[T_0,\ldots,T_n]$ is called a \textit{representation matrix} of $\Sc$ if it is generically of full rank and if the rank of $M$ evaluated in a point $p$ of $\PP^n$ drops if and only if the point $p$ lies on $\Sc$.
\end{defn}

It follows immediately that a matrix $M$ represents $\Sc$ if and only if the greatest common divisor $D$ of all its minors of maximal size is a power of a homogeneous implicit equation $F \in \kk[T_0,\ldots,T_n]$ of $\Sc$. When the base locus is locally an almost complete intersection, we can construct a matrix $M$ such that $D$ factors as $D=F^\delta G$ where $\delta\in \NN$ and $G \in \kk[T_0,\ldots,T_n]$. In Section \ref{sec:equation} we gave a description of the surface $(D=0)$.
In this chapter we present an implementation of our results in the computer aided software Macaulay2 \cite{M2}. From a practical point of view our results are a major improvement, as it makes the method applicable for a wider range of parametrizations (for example, by avoiding unnecessary base points with bad properties) and
leads to significantly smaller representation matrices.

There are several advantages of this perspective. The method works in a very general setting and makes only minimal assumptions on the parametrization. In particular, as we have mentioned, it works well in the presence of ``nice'' base points. Unlike the method of toric resultants (cf. for example \cite{KD06}), we do not have to extract a maximal minor of unknown size, since the matrices are generically of full rank. The monomial structure of the parametrization is exploited, in Section \ref{defNf}, following \cite{Bot09}, we defined
\begin{defn}
Given a list of polynomials $f_0,\hdots,f_r$, we define 
\[
 \Nc(f_0,\hdots,f_r):=\conv(\bigcup_{i=0}^r \Nc(f_i)),
\]
the convex hull of the union of the Newton polytopes of $f_i$, and we will refer to this polytope as the \textit{Newton polytope} of the list $f_0,\hdots,f_r$. When $f$ denotes the rational map defining $\Sc$, we will write $\Nc(f):=\Nc(f_1,f_2,f_3,f_4)$, and we will refer to it as the Newton polytope of $f$. 
\end{defn}
In these terms, in our algorithm we fully exploit the structure of $\Nc(f)$, so one obtains much better results for sparse parametrizations, both in terms of computation time and in terms of the size of the representation matrix. Moreover, it subsumes the known method of approximation complexes in the case of dense homogeneous parametrizations. One important point is that representation matrices can be efficiently constructed by solving a linear system of relatively small size (in our case $\dim_\kk(A_{\nu+d})$ equations in $4 \dim_\kk(A_\nu)$ variables). This means that their computation is much faster than the computation of the implicit equation and they are thus an interesting alternative as an implicit representation of the surface. 

On the other hand, there are a few disadvantages. Unlike with the toric resultant or the method of moving surfaces (moving plane and quadrics), the matrix representations are not square and the matrices involved are generally bigger than with the method of moving planes and surfaces. It is important to remark that those disadvantages are inherent to the choice of the method: A square matrix built from linear syzygies does not exist in general. It is an automatic consequence of this fact, that if one only uses linear syzygies to construct the matrix, it has to be a bigger matrix which has entries of higher degree (see \cite{BCS09}). The choice of the method to use depends very much on the given parametrization and on what one needs to do with the matrix representation.

\section{Example}

\begin{exmp}\label{interestingexample-algor}
Here we give an example, where we fully exploit the structure of $\Nc(f)$ that we have already seen in \ref{interestingexample}. Take $(f_1,f_2,f_3,f_4) = (st^6+2,st^5-3st^3,st^4+5s^2t^6,2+s^2t^6)$. This is a very sparse parametrization, and we have in this case, there is no smaller lattice homothety of $\Nc(f)$. The coordinate ring is $A=\kk[X_0,\ldots,X_5]/J$, where $J=(X_3^2-X_2X_4, X_2X_3-X_1X_4, X_2^2-X_1X_3, X_1^2-X_0X_5)$ and the new base-point-free parametrization $g$ is given by $(g_1,g_2,g_3,g_4)=(2X_0+X_4,-3X_1+X_3, X_2+5X_5, 2X_0+X_5)$. The Newton polytope looks as follows.
\begin{center}
  \includegraphics[scale=0.9]{ex2bis}
\end{center}

For $\nu_0=2d=2$ we can compute the matrix of the first map of the graded piece of degree $\nu_0$ of the approximation complex of cycles $(\Zc_\bullet)_{\nu_0}$, see for Example \ref{interestingexample} (following \cite[Sec 3.1]{BDD08}), which is a $17 \times 34$-matrix. 
The greatest common divisor of the $17$-minors of this matrix is the homogeneous implicit equation of the surface; it is of degree 6 in the variables 
{\footnotesize 
\[
\begin{array}{rl}
 T_1,\ldots,T_4: & 2809T_1^2T_2^4 + 124002T_2^6 - 5618T_1^3T_2^2T_3 + 66816T_1T_2^4T_3 + 2809T_1^4T_3^2\vspace{0.12cm} \\ \vspace{0.12cm}
 &- 50580T_1^2T_2^2T_3^2  + 86976T_2^4 T_3^2 + 212T_1^3T_3^3  - 14210T_1T_2^2T_3^3  + 3078T_1^2 T_3^4 \\  \vspace{0.12cm}
 & + 13632T_2^2 T_3^4  + 116T_1T_3^5 + 841T_3^6  + 14045T_1^3 T_2^2 T_4 - 169849T_1T_2^4 T_4 \\ \vspace{0.12cm}
 & -14045T_1^4 T_3T_4 + 261327T_1^2 T_2^2 T_3T_4 - 468288T_2^4 T_3T_4 - 7208T_1^3 T_3^2 T_4 \\ \vspace{0.12cm}
 & + 157155T_1T_2^2 T_3^3 T_4 - 31098T_1^2 T_3^3 T_4 - 129215T_2^2 T_3^3 T_4 - 4528T_1T_3^4 T_4  \\ \vspace{0.12cm}
 & - 12673T_3^5 T_4 - 16695T_1^2 T_2^2 T_4^2  + 169600T_2^4 T_4^2  + 30740T_1^3 T_3T_4^2 \\ \vspace{0.12cm}
 & - 433384T_1T_2^2 T_3T_4^2 + 82434T_1^2 T_3^2 T_4^2  + 269745T_2^2 T_3^2 T_4^2  + 36696T_1T_3^3 T_4^2 \\ \vspace{0.12cm}
 &  + 63946T_3^4 T_4^2  + 2775T_1T_2^2 T_4^3  - 19470T_1^2 T_3T_4^4  + 177675T_2^2 T_3T_4^3   \\  \vspace{0.12cm}
 &- 85360T_1T_3^2 T_4^3  - 109490T_3^3 T_4^3  - 125T_2^2 T_4^4  + 2900T_1T_3T_4^4  + 7325T_3^2 T_4^4    \\ 
 &- 125T_3T_4^5
\end{array} 
\]
}

In this example we could have considered the parametrization as a bihomogeneous map either of bidegree $(2,6)$ or of
bidegree $(1,3)$, i.e. we could have chosen the corresponding rectangles instead
of $\Nc(f)$. This leads to a more complicated coordinate ring in $20$ (resp.~$7$) variables and $160$ (resp.~$15$) generators of $J$ and to bigger matrices
(of size $21 \times 34$ in both cases). Even more importantly, the parametrizations will have a non-LCI base point and the matrices do not represent the implicit equation but a multiple of it (of degree $9$). Instead, if we consider the map as a homogeneous map of degree $8$, the results are even worse: For $\nu_0 = 6$, the $28 \times 35$-matrix $M_{\nu_0}$ represents a multiple of the implicit equation of degree $21$.

To sum up, in this example the method of approximation complexes works well for suitable toric varieties, whereas it fails over $\PP^1 \times \PP^1$ and $\PP^2$. This shows that the extension of the method to toric varieties really is a generalization and makes the method applicable to a larger class of parametrizations.\medskip

Interestingly, we can even do better than with $\Nc(f)$ by choosing a smaller polytope. The philosophy is that the choice of the optimal polytope is a compromise between two criteria: keep the simplicity of the polytope in order not to make the ring $A$ too complicated, and respect the sparseness of the parametrization (i.e. keep the polytope close to the Newton polytope) so that no base points appear which are not local complete intersections.

So let us repeat the same example with another polytope $Q$, which is small enough to reduce the size of the matrix but which only adds well-behaved (i.e. local complete intersection) base points:
\begin{center}
\includegraphics[scale=1]{ex3}
\end{center}
 The Newton polytope $\Nc(f)$ is contained in $2 \cdot Q$, so the parametrization
will factor through the toric variety associated to $Q$, more precisely we obtain
a new parametrization defined by
\[
 h=(h_1,h_2,h_3,h_4)=(2X_0^2+X_3X_4,-3X_0X_4+X_2X_4, X_1X_4+5X_4^2,2X_0^2+X_4^2)
\]
over the coordinate ring $A=\kk[X_0,\ldots,X_4]/J$ with $J=(X_2^2-X_1X_3, X_1X_2-X_0X_3, X_1^2-X_0X_2)$ making the following diagram commute:
\[
 \xymatrix@1{ 
  (\AA^\ast)^2 \ar@{-->}[r]^{f} \ar@{^{(}->}[d] &\PP^3 \\
  \Tc_Q\ar@{-->}[ru]^{h}&
} 
\]
The optimal bound is $\nu_0=2$ and in this degree the implicit equation is represented directly without extraneous factors by a $12 \times 19$-matrix, which is smaller than the $17 \times 34$ we had before. 
\end{exmp}

\medskip


\section{Implementation in Macaulay2}\label{codeTtoP3}
In this section we show how to compute a matrix representation and the implicit equation
with the method developed in Chapter \ref{ch:toric-emb-pn}, using the computer algebra system Macaulay2 \cite{M2}. We will explain the code along Example \ref{interestingexample-algor}.
As it is probably the most interesting case from a practical point of view, we restrict our computations to parametrizations of a toric surface.
However, the method can be adapted to the $n$-dimensional toric case. 
Moreover, we are not claiming that our implementation is optimized for efficiency; anyone
trying to implement the method to solve computationally involved examples is well-advised to
give more ample consideration to this issue. For example, in the toric case there are better suited
software systems to compute the generators of the toric ideal $J$, see \cite{4ti2}.

First we load the package ``Maximal minors\footnote{The package ``maxminor.m2" for Macaulay2 can be downloaded from the webpage \texttt{http://mate.dm.uba.ar/\textasciitilde nbotbol/maxminor.m2}.}''
{\footnotesize \begin{verbatim}
i1 : load "maxminor.m2"
\end{verbatim}}

Let us start by defining the parametrization $f$ given by $(f_1,\ldots,f_4)$.

{\footnotesize \begin{verbatim}
i2 : S=QQ[s,u,t,v];
i3 : e1=2;
i4 : e2=6;
i5 : f1=s*u*t^6+2*u^2*v^6
          6     2 6
o5 = s*u*t  + 2u v
i6 : f2=s*u*t^5*v-3*s*u*t^3*v^3
          5          3 3
o6 = s*u*t v - 3s*u*t v
i7 : f3=s*u*t^4*v^2+5*s^2*t^6
       2 6        4 2
o7 = 5s t  + s*u*t v
i8 : f4=2*u^2*v^6+s^2*t^6
      2 6     2 6
o8 = s t  + 2u v

\end{verbatim}}

We construct the matrix associated to the polynomials and we relabel them in order to be able to automatize some procedures.
{\footnotesize \begin{verbatim}

i9 : F=matrix{{f1,f2,f3,f4}}
o9 = | sut6+2u2v6 sut5v-3sut3v3 5s2t6+sut4v2 s2t6+2u2v6 |
             1       4
o9 : Matrix S  <--- S
i10 : f_1=f1;
i11 : f_2=f2;
i12 : f_3=f3;
i13 : f_4=f4;
\end{verbatim}}

We define the associated affine polynomials \texttt{FF$\_$i} by specializing the variables $u$ and $v$ to $1$.
{\footnotesize \begin{verbatim}
i14 : for i from 1 to 4 do (
        FF_i=substitute(f_i,{u=>1,v=>1});
      )
\end{verbatim}}

We just change the polynomials \texttt{FF$\_$i} to the new ring $S2$.

{\footnotesize \begin{verbatim}
i15 : S2=QQ[s,t]
o15 = S2
o15 : PolynomialRing
i16 : for i from 1 to 4 do (
        FF_i=sub(FF_i,S2);
      )
\end{verbatim}}

The reader can experiment with the implementation simply by changing the definition of the polynomials
and their degrees, the rest of the code being identical. We first set up the list $st$ of monomials $s^it^j$ of bidegree $(e'_1,e'_2)$. In the toric case,
this list should only contain the monomials corresponding to points in the Newton polytope $\Nc'(f)$.
{\footnotesize \begin{verbatim}
i17 : use S;
i18 : st={};
i19 : for i from 1 to 4 do (
        st=join(st,flatten entries monomials f_i); 
      )
i20 : l=length(st)-1;
i21 : k=gcd(e1,e2)
o21 = 2
\end{verbatim}}
 
We compute the ideal $J$ and the quotient ring $A$. This is done
by a Gr\"obner basis computation which works well for examples of
small degree, but which should be replaced by a matrix formula
in more complicated examples. In the toric case,
there exist specialized software systems such as \cite{4ti2} to compute the ideal $J$.
{\footnotesize \begin{verbatim}
i24 : SX=QQ[s,u,t,v,w,x_0..x_l,MonomialOrder=>Eliminate 5]
o24 = SX
o24 : PolynomialRing
i25 : X={};
i26 : st=matrix {st};
              1       8
o26 : Matrix S  <--- S
i27 : F=sub(F,SX)
o27 = | sut6+2u2v6 sut5v-3sut3v3 5s2t6+sut4v2 s2t6+2u2v6 |
               1        4
o27 : Matrix SX  <--- SX
i28 : st=sub(st,SX)
o28 = | sut6 u2v6 sut5v sut3v3 s2t6 sut4v2 s2t6 u2v6 |
               1        8
o28 : Matrix SX  <--- SX
i29 : te=1;
i30 : for i from 0 to l do ( te=te*x_i )
i31 : J=ideal(1-w*te)
o31 = ideal(- w*x x x x x x x x  + 1)
                 0 1 2 3 4 5 6 7
o31 : Ideal of SX
i32 : for i from 0 to l do (
          J=J+ideal (x_i - st_(0,i))
          )
i33 : J= selectInSubring(1,gens gb J)
o33 = | x_4-x_6 x_1-x_7 x_3^2-x_6x_7 x_2x_3-x_5^2 x_0x_3-x_2x_5 
      ---------------------------------------------------------
      x_2^2-x_0x_5 x_5^3-x_0x_6x_7 x_3x_5^2-x_2x_6x_7 |
               1        8
o33 : Matrix SX  <--- SX
i34 : R=QQ[x_0..x_l]
o34 = R
o34 : PolynomialRing
i35 : J=sub(J,R)
o35 = | x_4-x_6 x_1-x_7 x_3^2-x_6x_7 x_2x_3-x_5^2 x_0x_3-x_2x_5 
      ---------------------------------------------------------
      x_2^2-x_0x_5 x_5^3-x_0x_6x_7 x_3x_5^2-x_2x_6x_7 |
              1       8
o35 : Matrix R  <--- R
i36 : A=R/ideal(J)  
o36 = A
o36 : QuotientRing
\end{verbatim}}

Next, we set up the list $ST$ of monomials $s^it^j$ of bidegree $(e_1,e_2)$ and the list $X$ of
the corresponding elements of the quotient ring $A$. In the toric case,
this list should only contain the monomials corresponding to points in the Newton polytope $\Nc(f)$.
{\footnotesize \begin{verbatim}
i37 : use SX
o37 = SX
o37 : PolynomialRing
i38 :   ST={};
i39 :   X={};
i40 :   for i from 0 to l do (
          ST=append(ST,st_(0,i)); 
          X=append(X,x_i);
        )
\end{verbatim}}
We can now define the new parametrization $g$ by the polynomials $g_1,\ldots,g_4$. 
 {\footnotesize \begin{verbatim}
i41 : X=matrix {X};
               1        8
o41 : Matrix SX  <--- SX
i42 : X=sub(X,SX)
o42 = | x_0 x_1 x_2 x_3 x_4 x_5 x_6 x_7 |
               1        8
o42 : Matrix SX  <--- SXX=matrix {X};
i43 : (M,C)=coefficients(F,Variables=>{s_SX,u_SX,t_SX,v_SX},Monomials=>ST)
o43 = (| sut6 u2v6 sut5v sut3v3 s2t6 sut4v2 s2t6 u2v6 |, {8} | 1 0  0 0 |)
                                                         {8} | 0 0  0 0 |
                                                         {8} | 0 1  0 0 |
                                                         {8} | 0 -3 0 0 |
                                                         {8} | 0 0  0 0 |
                                                         {8} | 0 0  1 0 |
                                                         {8} | 0 0  5 1 |
                                                         {8} | 2 0  0 2 |
o43 : Sequence
i44 : G=X*C
o44 = | x_0+2x_7 x_2-3x_3 x_5+5x_6 x_6+2x_7 |
               1        4
o44 : Matrix SX  <--- SX
i45 : G=matrix{{G_(0,0),G_(0,1),G_(0,2),G_(0,3)}}
o45 = | x_0+2x_7 x_2-3x_3 x_5+5x_6 x_6+2x_7 |
               1        4
o45 : Matrix SX  <--- SX
i46 : G=sub(G,A)
o46 = | x_0+2x_7 x_2-3x_3 x_5+5x_6 x_6+2x_7 |
              1       4
o46 : Matrix A  <--- A
\end{verbatim}}
In the following, we construct the matrix representation $M$. For simplicity, we
compute the whole module $\Zc_1$, which is not necessary as we only need the graded
part $(\Zc_1)_{\nu_0}$. In complicated examples, one should compute only this graded part
by directly solving a linear system in degree $\nu_0$. Remark that the best
bound $\mathrm{nu}= \nu_0$ depends on the parametrization.
{\footnotesize \begin{verbatim}
i47 : use A
o47 = A
o47 : QuotientRing
i48 : Z0=A^1;
i49 : Z1=kernel koszul(1,G);
i50 : Z2=kernel koszul(2,G);
i51 : Z3=kernel koszul(3,G);
i52 : nu=-1
o52 = -1
i53 : d=1
o53 = 1
i54 : hfnu = 1
o54 = 1
i55 : while hfnu != 0 do (
      nu=nu+1;
      hfZ0nu = hilbertFunction(nu,Z0);
      hfZ1nu = hilbertFunction(nu+d,Z1);
      hfZ2nu = hilbertFunction(nu+2*d,Z2);
      hfZ3nu = hilbertFunction(nu+3*d,Z3);
      hfnu = hfZ0nu-hfZ1nu+hfZ2nu-hfZ3nu;
       );
i56 : nu
o56 = 2
i57 : hfZ0nu
o57 = 17
i58 : hfZ1nu
o58 = 34
i59 : hfZ2nu
o59 = 23
i60 : hfZ3nu
o60 = 6
i61 : hfnu
o61 = 0

i62 : hilbertFunction(nu+d,Z1)-2*hilbertFunction(nu+2*d,Z2)+
      3*hilbertFunction(nu+3*d,Z3)
o62 = 6
i63 : GG=ideal G
o63 = ideal (x  + 2x , x  - 3x , x  + 5x , x  + 2x )
              0     7   2     3   5     6   6     7
o63 : Ideal of A
i64 : GGsat=saturate(GG, ideal (x_0..x_l))
o64 = ideal 1
o64 : Ideal of A
i65 : degrees gens GGsat
o65 = {{{0}}, {{0}}}
o65 : List
i66 : H=GGsat/GG
o66 = subquotient (| 1 |, | x_0+2x_7 x_2-3x_3 x_5+5x_6 x_6+2x_7 |)
                                1
o66 : A-module, subquotient of A
i67 : degrees gens H
o67 = {{{0}}, {{0}}}
o67 : List
i68 : S=A[T1,T2,T3,T4]
o68 = S
o68 : PolynomialRing
i69 : G=sub(G,S);
              1       4
o69 : Matrix S  <--- S
i70 : Z1nu=super basis(nu+d,Z1); 
              4       34
o70 : Matrix A  <--- A
i71 : Tnu=matrix{{T1,T2,T3,T4}}*substitute(Z1nu,S);
              1       34
o71 : Matrix S  <--- S
i72 : 
      lll=matrix {{x_0..x_l}}
o72 = | x_0 x_7 x_2 x_3 x_6 x_5 x_6 x_7 |
              1       8
o72 : Matrix A  <--- A
i73 : lll=sub(lll,S)
o73 = | x_0 x_7 x_2 x_3 x_6 x_5 x_6 x_7 |
              1       8
o73 : Matrix S  <--- S
i74 : ll={}
o74 = {}
o74 : List
i75 : for i from 0 to l do { ll=append(ll,lll_(0,i)) }
i76 : (m,M)=coefficients(Tnu,Variables=>ll,Monomials=>substitute(basis(nu,A),S));
i77 : M;  
              17       34
o77 : Matrix S   <--- S
\end{verbatim}}
The matrix $M$ is the desired matrix representation of the surface $\Sc$.

We can continue by computing the implicit equation and verifying the result by substituting
{\footnotesize \begin{verbatim}
 
i78 : T=QQ[T1,T2,T3,T4]
o78 = T
o78 : PolynomialRing
i79 : ListofTand0 ={T1,T2,T3,T4}
o79 = {T1, T2, T3, T4}
o79 : List
i80 : for i from 0 to l do { ListofTand0=append(ListofTand0,0) };
i81 : p=map(T,S,ListofTand0) 
o81 = map(T,S,{T1, T2, T3, T4, 0, 0, 0, 0, 0, 0, 0, 0})
o81 : RingMap T <--- S
i82 : N=MaxCol(p(M)); 
              17       17
o82 : Matrix T   <--- T
i83 : Eq=det(N); factor Eq
\end{verbatim}}
We verify the result by substituting on the computed equation, the polynomials $f_1$ to $f_4$.

{\footnotesize \begin{verbatim}
i85 :use S; Eq=sub(Eq,S)
o86 : S
i87 : sub(Eq,{T1=>G_(0,0),T2=>G_(0,1),T3=>G_(0,2),T4=>G_(0,3)})
o87 = 0
\end{verbatim}}

\chapter[Algorithm2]{A package for computing implicit equations from toric surfaces without an embedding}
\label{ch:ch-algor-multigr}

\section{Implementation in Macaulay2}

In this section we show how to compute a matrix representation and the implicit equation with the method developed in Chapter \ref{ch:toric-pn}, following \cite{Bot09}, using the computer algebra system Macaulay2 \cite{M2}. As it is probably the most interesting case from a practical point of view, we restrict our computations to parametrizations of a multigraded hypersurface. 

This implementation allows to compute small examples for the better understanding of the theory, but we are not claiming that this implementation is optimized for efficiency; anyone trying to implement the method to solve computationally involved examples is well-advised to give more ample consideration to this issue.

\subsection{Example 1}\label{example1Ch9}

Consider the rational map 
\begin{equation}
\begin{array}{ccc}
 \PP^1\times \PP^1 	&\nto{f}{\dto}	& \PP^3\\
(s:u)\times (t:v) 	&\mapsto	& (f_1:f_2:f_3:f_4)
\end{array}
\end{equation}
where the polynomials $f_i=f_i(s,u,t,v)$ are bihomogeneous of bidegree $(2,3)\in \ZZ^2$ given by
\begin{itemize}
 \item $f_1=s^2t^3+2sut^3+3u^2t^3+4s^2t^2v+5sut^2v+6u^2t^2v+7s^2tv^2+8sutv^2+9u^2tv^2+10s^2v^3+suv^3+2u^2v^3$,
 \item $f_2=2s^2t^3-3s^2t^2v-s^2tv^2+sut^2v+3sutv^2-3u^2t^2v+2u^2tv^2-u^2v^3$,
 \item $f_3=2s^2t^3-3s^2t^2v-2sut^3+s^2tv^2+5sut^2v-3sutv^2-3u^2t^2v+4u^2tv^2-u^2v^3$,
 \item $f_4=3s^2t^2v-2sut^3-s^2tv^2+sut^2v-3sutv^2-u^2t^2v+4u^2tv^2-u^2v^3$.
\end{itemize}

Our aim is to get the implicit equation of the hypersurface $\overline{\im(f)}$ of $\PP^3$.

First we load the package ``Maximal minors''
{\footnotesize \begin{verbatim}
i1 : load "maxminor.m2"
\end{verbatim}}

Let us start by defining the parametrization $f$ given by $(f_1,f_2,f_3,f_4)$.

{\footnotesize \begin{verbatim}
i2 : S=QQ[s,u,t,v,Degrees=>{{1,1,0},{1,1,0},{1,0,1},{1,0,1}}];
i3 : e1=2;
i4 : e2=3;

i5 : f1=1*s^2*t^3+2*s*u*t^3+3*u^2*t^3+4*s^2*t^2*v+5*s*u*t^2*v+6*u^2*t^2*v+ 
        7*s^2*t*v^2+8*s*u*t*v^2+9*u^2*t*v^2+10*s^2*v^3+1*s*u*v^3+2*u^2*v^3;
i6 : f2=2*s^2*t^3-3*s^2*t^2*v-s^2*t*v^2+s*u*t^2*v+3*s*u*t*v^2-3*u^2*t^2*v+
        2*u^2*t*v^2-u^2*v^3;
i7 : f3=2*s^2*t^3-3*s^2*t^2*v-2*s*u*t^3+s^2*t*v^2+5*s*u*t^2*v-3*s*u*t*v^2-
        3*u^2*t^2*v+4*u^2*t*v^2-u^2*v^3;
i8 : f4=3*s^2*t^2*v-2*s*u*t^3-s^2*t*v^2+s*u*t^2*v-3*s*u*t*v^2-u^2*t^2*v+
        4*u^2*t*v^2-u^2*v^3;
\end{verbatim}}

We construct the matrix associated to the polynomials and we relabel them in order to be able to automatize some procedures.

{\footnotesize \begin{verbatim}

i9 : F=matrix{{f1,f2,f3,f4}}

o9 = | s2t3+2sut3+3u2t3+4s2t2v+5sut2v+6u2t2v+7s2tv2+8sutv2+9u2tv2+10s2v3+
     --------------------------------------------------------------------
     suv3+2u2v3 2s2t3-3s2t2v+sut2v-3u2t2v-s2tv2+3sutv2+2u2tv2-u2v3
     --------------------------------------------------------------------
      2s2t3-2sut3-3s2t2v+5sut2v-3u2t2v+s2tv2-3sutv2+4u2tv2-u2v3
     --------------------------------------------------------------------
      -2sut3+3s2t2v+sut2v-u2t2v-s2tv2-3sutv2+4u2tv2-u2v3|

             1       4
o9 : Matrix S  <--- S
\end{verbatim}}

The reader can experiment with the implementation simply by changing the definition of the polynomials
and their degrees, the rest of the code being identical. 

As we mentioned in Example \ref{exmpBigradedCh7}, if $R_1:=k[x_1,x_2]$, $R_2:=k[y_1,y_2,y_3,y_4]$, and $\G:=\ZZ^2$, writing $R:=R_1\otimes_k R_2$ and setting $\deg(x_i)=(1,0)$ and $\deg(y_i)=(0,1)$ for all $i$, with $\aaa_1:=(x_1,x_2)$, $\aaa_2:=(y_1,y_2,y_3,y_4)$ and $B:=\aaa_1\cdot \aaa_2 \subset R$ we have that:
\begin{enumerate}
 \item $\Supp_\G(H^2_B(R))= \Supp_\G(\check{R}_{1})\cup\Supp_\G(\check{R}_{2})=  Q_{\{1\}}=-\NN\times \NN+(-2,0)\cup\NN\times -\NN+(0,-2)$.
 \item $\Supp_\G(H^3_B(R))= \Supp_\G(\check{R}_{1,2})=  Q_{\{1,2\}}=-\NN\times -\NN+(-2,-2)$, .
\end{enumerate}
\begin{center}
 \includegraphics[scale=1]{supp3}
\end{center}
And thus, 
\[
 \Region(2,3)=(\Supp_\G(H^{2}_B(R))+(2,3))\cup(\Supp_\G(H^{3}_B(R))+2\cdot(2,3)).
\]
Obtaining
\[
 \complement\Region(2,3)=(\NN^2+(1,5))\cup(\NN^2+(3,2)).
\]
As we can see in Example \ref{example1Ch9}, a Macaulay2 computation gives exactly this region (illustrated below) as the acyclicity region for $\Z.$.

{\footnotesize \begin{verbatim}
i10 : nu={5,3,2};
\end{verbatim}}

An alternative consists in taking 
{\footnotesize \begin{verbatim}
i10 : nu={6,1,5};
\end{verbatim}}

\begin{center}
 \includegraphics{JSAG-region}
\end{center}

Anyhow, it is interesting to test what happens in different bidegrees $\nu\in \ZZ^2$ by just replacing the desired degree in the code.
\medskip

In the following, we construct the matrix representation $M$. For simplicity, we
compute the whole module $\Zc_1$, which is not necessary as we only need the graded
part $(\Zc_1)_{\nu_0}$. In complicated examples, one should compute only this graded part
by directly solving a linear system in degree $\nu_0$.

{\footnotesize \begin{verbatim}
i11 : Z0=S^1;
i12 : Z1=kernel koszul(1,F);
i13 : Z2=kernel koszul(2,F);
i14 : Z3=kernel koszul(3,F);

i15 : d={e1+e2,e1,e2}

i16 : hfZ0nu = hilbertFunction(nu,Z0)
o16 = 12

i17 : hfZ1nu = hilbertFunction(nu+d,Z1)
o17 = 12

i18 : hfZ2nu = hilbertFunction(nu+2*d,Z2)
o18 = 0

i19 : hfZ3nu = hilbertFunction(nu+3*d,Z3)
o19 = 0

i20 : hfnu = hfZ0nu-hfZ1nu+hfZ2nu-hfZ3nu
o20 = 0
\end{verbatim}}

Thus, when $\nu_0=(3,2)$ or $\nu_0=(1,5)$, we get a complex
\[
 (\Z.)_{\nu_0}: 0\to0\to0\to\kk[\X]^{12}\nto{M_{\nu_0}}{\lto}\kk[\X]^{12}\to0.
\]
and, hence, $\det((\Z.)_{\nu_0})=\det(M_{\nu_0})\in \kk[\X]_{12}$ is an homogeneous polynomial of degree $12$ that vanishes on the closed image of $\phi$. We compute here the degree of the MacRae's invariant which gives the degree of $\det((\Z.)_{\nu_0})$.

{\footnotesize \begin{verbatim}
i21 :hilbertFunction(nu+d,Z1)-2*hilbertFunction(nu+2*d,Z2)+
      3*hilbertFunction(nu+3*d,Z3)

o21 = 12
\end{verbatim}}

{\footnotesize \begin{verbatim}
i22 : GG=ideal F

              2 3       3   2 3   2 2        2    2 2    2   2  
o22 = ideal (s t +2s*u*t +3u t +4s t v+5s*u*t v+6u t v+7s t*v +
      ------------------------------------------------------------
              2   2   2    2 3      3   2 3    2 3   2 2       2   
      8s*u*t*v +9u t*v +10s v +s*u*v *2u v , 2s t -3s t v+s*u*t v-
      ------------------------------------------------------------
        2 2   2   2         2   2   2  2 3    2 3       3   2 2   
      3u t v-s t*v +3s*u*t*v +2u t*v -u v , 2s t -2s*u*t -3s t v+
      ------------------------------------------------------------
            2    2 2   2   2         2   2   2  2 3         3  
      5s*u*t v-3u t v+s t*v -3s*u*t*v +4u t*v -u v , -2s*u*t +
      ------------------------------------------------------------
        2 2       2   2 2   2   2         2   2   2  2 3
      3s t v+s*u*t v-u t v-s t*v -3s*u*t*v +4u t*v -u v )

o22 : Ideal of S

i23 : GGsat=saturate(GG, ideal(s,t)*ideal(u,v))

               2 2        2   2 2    2   2         2   2   2  2 3 
o23 = ideal (3s t v-3s*u*t v-u t v-3s t*v +3s*u*t*v +2u t*v -u v ,
      --------------------------------------------------------------
        2 3        2     2 2     2   2          2    2   2    2 3  
      9u t +42s*u*t v+28u t v+45s t*v -15s*u*t*v +19u t*v +30s v +
      --------------------------------------------------------------
            3    2 3       3       2   2   2         2  2   2   2 3  
      3s*u*v +13u v , s*u*t -2s*u*t v-s t*v +3s*u*t*v -u t*v , s t -
      --------------------------------------------------------------
           2    2 2    2   2         2   2   2  2 3         4  2 4 
      s*u*t v-2u t v-2s t*v +3s*u*t*v +2u t*v -u v , 30s*u*v -u v ,
      --------------------------------------------------------------
         2 4    2 4   2   3  2 4           3  2 4     2   3    2 4 
      15s v +14u v , u t*v -u v , 30s*u*t*v -u v , 15s t*v +14u v ,
      --------------------------------------------------------------
       2 2 2  2 4
      u t v -u v )

o23 : Ideal of S

i24 : degrees gens GGsat

o24 = {{{0, 0, 0}}, {{5, 2, 3}, {5, 2, 3}, {5, 2, 3}, {5, 2, 3}, {6,
      --------------------------------------------------------------
      2, 4}, {6,2, 4}, {6, 2, 4}, {6, 2, 4}, {6, 2, 4}, {6, 2, 4}}}

o24 : List

i25 : H=GGsat/GG

o25 = subquotient (| 3s2t2v-3sut2v-u2t2v-3s2tv2+3sutv2+2u2tv2-u2v3
       9u2t3+42sut2v+28u2t2v+45s2tv2-15sutv2+19u2tv2+30s2v3+3suv3+
       13u2v3 sut3-2sut2v-s2tv2+3sutv2-u2tv2 s2t3-sut2v-2u2t2v-2s2tv2+
       3sutv2+2u2tv2-u2v3 30suv4-u2v4 15s2v4+14u2v4 u2tv3-u2v4 
       30sutv3-u2v4 15s2tv3+14u2v4 u2t2v2-u2v4 |, | s2t3+2sut3+3u2t3+
       4s2t2v+5sut2v+6u2t2v+7s2tv2+8sutv2+9u2tv2+10s2v3+suv3+2u2v3 
       2s2t3-3s2t2v+sut2v-3u2t2v-s2tv2+3sutv2+2u2tv2-u2v3 2s2t3-2sut3-
       3s2t2v+5sut2v-3u2t2v+s2tv2-3sutv2+4u2tv2-u2v3 -2sut3+3s2t2v+
       sut2v-u2t2v-s2tv2-3sutv2+4u2tv2-u2v3 |)

                                1
o25 : S-module, subquotient of S

i26 : degrees gens H

o26 = {{{0, 0, 0}}, {{5, 2, 3}, {5, 2, 3}, {5, 2, 3}, {5, 2, 3}, {6,
      --------------------------------------------------------------
      2, 4}, {6,2, 4}, {6, 2, 4}, {6, 2, 4}, {6, 2, 4}, {6, 2, 4}}}

o26 : List
\end{verbatim}}

Now, we focus on the computation of the implicit equation as the determinant of the right-most map. Precisely, we will build-up this map, and later extract a maximal minor for taking its determinant. It is clear that is in general not the determinant of the approximation complex in degree $\nu$, but a multiple of it. We could get the correct equation by taking several maximal minors and considering the gcd of its determinant. This procedure is much more expensive, hence, we avoid it.

Thus, first, we compute the right-most map of the approximation complex in degree $\nu$

{\footnotesize \begin{verbatim}
i27 : R=S[T1,T2,T3,T4];

i28 : G=sub(F,R);

              1       4
o28 : Matrix R  <--- R
\end{verbatim}} 

We compute a matrix presentation for $(\Zc_1)_\nu$ in $K_1$:
      
{\footnotesize \begin{verbatim}
i29 :Z1nu=super basis(nu+d,Z1); 

              4       12
o29 : Matrix S  <--- S

i30 : Tnu=matrix{{T1,T2,T3,T4}}*substitute(Z1nu,R);

              1       12
o30 : Matrix R  <--- R

i31 : lll=matrix {{s,t,u,v}}

o31 = | s t u v |

              1       4
o31 : Matrix S  <--- S

i32 : lll=sub(lll,R)

o32 = | s t u v |

              1       4
o32 : Matrix R  <--- R

i33 : ll={};

i34 : for i from 0 to 3 do { ll=append(ll,lll_(0,i)) }
\end{verbatim}}

Now, we compute the matrix of the map $(\Zc_1)_\nu \to A_\nu[T_1,T_2,T_3,T_4]$

{\footnotesize \begin{verbatim}
i35 : (m,M)=coefficients(Tnu,Variables=>ll,Monomials=>substitute(
            basis(nu,S),R));

i36 : M;  

              12       12
o36 : Matrix R   <--- R

i37 : T=QQ[T1,T2,T3,T4];

i38 : ListofTand0 ={T1,T2,T3,T4};

i39 : for i from 0 to 3 do { ListofTand0=append(ListofTand0,0) };

i40 : p=map(T,R,ListofTand0)

o40 = map(T,R,{T1, T2, T3, T4, 0, 0, 0, 0})

o40 : RingMap T <--- R

i41 :N=MaxCol(p(M)); 

              12       12
o41 : Matrix T   <--- T
\end{verbatim}}
 
The matrix $M$ is the desired matrix representation of the surface $\Sc$. We can continue by computing the implicit equation by taking determinant. As we mentioned, this is fairly more costly. If we take determinant what we get is a multiple of the implicit equation. One wise way for recognizing which of them is the implicit equation is substituting a few points of the surface, and verifying which vanishes.

Precisely, here there is a multiple of the implicit equation (by taking several minors we erase extra factors):
      
{\footnotesize \begin{verbatim}
i42 :Eq=det(N); factor Eq;
\end{verbatim}}
We verify the result by sustituting on the computed equation, the polynomials $f_1$ to $f_4$. We verify that in this case, this is the implicit equation:
      
{\footnotesize \begin{verbatim}
i44 : use R; Eq=sub(Eq,R);
i46 : sub(Eq,{T1=>G_(0,0),T2=>G_(0,1),T3=>G_(0,2),T4=>G_(0,3)})

o46 = 0

o46 : R
\end{verbatim}}

\printindex

\def\cprime{$'$}
\providecommand{\bysame}{\leavevmode\hbox to3em{\hrulefill}\thinspace}
\providecommand{\MR}{\relax\ifhmode\unskip\space\fi MR }
\providecommand{\MRhref}[2]{%
  \href{http://www.ams.org/mathscinet-getitem?mr=#1}{#2}
}
\providecommand{\href}[2]{#2}


\begin{thebibliography}{HSV83b}

\bibitem[4ti]{4ti2}
Team 4ti2, \emph{4ti2 - a software package for algebraic, geometric and
  combinatorial problems on linear spaces}, http://www.4ti2.de.

\bibitem[AHW05]{AHW05}
William~A Adkins, Jerome~W Hoffman, and Hao~Hao Wang, \emph{Equations of
  parametric surfaces with base points via syzygies}, J. Symbolic Comput
  \textbf{39} (2005), no.~1, 73--101. \MR{MR2168242 (2006h:14066)}

\bibitem[Avr81]{avr}
Luchezar~L Avramov, \emph{Complete intersections and symmetric algebras}, J.
  Algebra \textbf{73} (1981), no.~1, 248--263. \MR{MR641643 (83e:13024)}

\bibitem[BC05]{BC05}
Laurent Bus{\'e} and Marc Chardin, \emph{Implicitizing rational hypersurfaces
  using approximation complexes}, J. Symbolic Comput \textbf{40} (2005),
  no.~4-5, 1150--1168. \MR{MR2172855 (2006g:14097)}

\bibitem[BCD03]{BCD03}
Laurent Bus{\'e}, David Cox, and Carlos D'Andrea, \emph{Implicitization of
  surfaces in {${\PP}\sp 3$} in the presence of base points}, J. Algebra Appl
  \textbf{2} (2003), no.~2, 189--214. \MR{MR1980408 (2004f:14092)}

\bibitem[BCJ09]{BCJ06}
Laurent Bus{\'e}, Marc Chardin, and Jean-Pierre Jouanolou, \emph{Torsion of the
  symmetric algebra and implicitization}, Proc. Amer. Math. Soc \textbf{137}
  (2009), no.~6, 1855--1865. \MR{MR2480264}

\bibitem[BCS09]{BCS09}
Laurent Bus\'e, Marc Chardin, and Aron Simis, \emph{Elimination and nonlinear
  equations of rees algebra}, http://arxiv.org/abs/0911.2569 (2009).

\bibitem[BD07]{BD07}
Laurent Bus{\'e} and Marc Dohm, \emph{Implicitization of bihomogeneous
  parametrizations of algebraic surfaces via linear syzygies}, I{SSAC} 2007,
  ACM, New York, 2007, pp.~69--76. \MR{MR2396186}

\bibitem[BD10]{BD10}
Nicol{\'a}s Botbol and Marc Dohm, \emph{A package for computing implicit
  equations of parametrizations from toric surfaces}, arXiv:1001.1126 (2010).

\bibitem[BDD09]{BDD08}
Nicol{\'a}s Botbol, Alicia Dickenstein, and Marc Dohm, \emph{Matrix
  representations for toric parametrizations}, Comput. Aided Geom. Design
  \textbf{26} (2009), no.~7, 757--771. \MR{MR2569833}

\bibitem[BGT97]{BGN97}
Winfried Bruns, Joseph Gubeladze, and Ng{\^o}~Vi{\^e}t Trung, \emph{Normal
  polytopes, triangulations, and {K}oszul algebras}, J. Reine Angew. Math
  \textbf{485} (1997), 123--160. \MR{MR1442191 (99c:52016)}

\bibitem[BH93]{BH}
Winfried Bruns and J{\"u}rgen Herzog, \emph{Cohen-{M}acaulay rings}, Cambridge
  Studies in Advanced Mathematics, vol.~39, Cambridge University Press,
  Cambridge, 1993. \MR{MR1251956 (95h:13020)}

\bibitem[BJ03]{BuJo03}
Laurent Bus{\'e} and Jean-Pierre Jouanolou, \emph{On the closed image of a
  rational map and the implicitization problem}, J. Algebra \textbf{265}
  (2003), no.~1, 312--357. \MR{MR1984914 (2004e:14024)}

\bibitem[BM93]{BaMum}
Dave Bayer and David Mumford, \emph{What can be computed in algebraic
  geometry?}, Computational algebraic geometry and commutative algebra
  ({C}ortona, 1991), Sympos. Math., XXXIV, Cambridge Univ. Press, Cambridge,
  1993, pp.~1--48. \MR{MR1253986 (95d:13032)}

\bibitem[Bot09a]{BotAlgo3D}
Nicol{\'a}s Botbol, \emph{Code in macaulay2 for computing the toric embedding
  of $(\mathbb{P}^1)^3$ in $\mathbb{P}^{11}$}, http://mate.dm.uba.ar/~nbotbol/
  (2009).

\bibitem[Bot09b]{Bot08}
\bysame, \emph{The implicitization problem for {$\phi\colon\Bbb {P}^n
  \dashrightarrow (\Bbb P^1)^{n+1}$}}, J. Algebra \textbf{322} (2009), no.~11,
  3878--3895. \MR{MR2556128}

\bibitem[Bot10]{Bot09}
\bysame, \emph{Compactifications of rational maps and the implicit equations of
  their images}, To appear in J. Pure and Applied Algebra. arXiv:0910.1340
  (2010).

\bibitem[Bus01]{Buse2}
Laurent Bus\'e, \emph{\'{E}tude du r\'{e}sultant sur une vari\'{e}t\'{e}
  alg\'{e}brique. phd thesis}, Universit\'{e} de Nice Sophia-Antipolis (2001).

\bibitem[Bus06]{Buse1}
\bysame, \emph{Elimination theory in codimension one and applications}, INRIA
  research report 5918. Notes of lectures given at the CIMPA-UNESCO-IRAN school
  in Zanjan, Iran, July 9-22 2005 (2006), 47.

\bibitem[CD07]{CD}
Mar{\'{\i}}a~Ang{\'e}lica Cueto and Alicia Dickenstein, \emph{Some results on
  inhomogeneous discriminants}, Proceedings of the {XVI}th {L}atin {A}merican
  {A}lgebra {C}olloquium ({S}panish), Bibl. Rev. Mat. Iberoamericana, Rev. Mat.
  Iberoamericana, Madrid, 2007, pp.~41--62. \MR{MR2500350}

\bibitem[Cha04]{Ch04}
Marc Chardin, \emph{Regularity of ideals and their powers.}, Insitut de
  Math\'ematiques de Jussieu \textbf{Pr\'epublication 364} (2004).

\bibitem[Cha06]{Ch06}
\bysame, \emph{Implicitization using approximation complexes}, Algebraic
  geometry and geometric modeling, Math. Vis, Springer, Berlin, 2006,
  pp.~23--35. \MR{MR2279841 (2007j:14097)}

\bibitem[CJR11]{CJR}
Marc Chardin, Jean-Pierre Jouanolou, and Ahad Rahimi, \emph{The eventual
  stability of depth, associated primes and cohomology of a graded module}.

\bibitem[CLS]{CoxTV}
David~A Cox, John~B. Little, and Hal Schenck, \emph{Toric varieties},
  http://www.cs.amherst.edu/~dac/toric.html.

\bibitem[Cox95]{Cox95}
David~A Cox, \emph{The homogeneous coordinate ring of a toric variety}, J.
  Algebraic Geom \textbf{4} (1995), no.~1, 17--50. \MR{MR1299003 (95i:14046)}

\bibitem[Cox01]{Co01}
\bysame, \emph{Equations of parametric curves and surfaces via syzygies},
  Symbolic computation: solving equations in algebra, geometry, and engineering
  ({S}outh {H}adley, {MA}, 2000), Contemp. Math, vol. 286, Amer. Math. Soc,
  Providence, RI, 2001, pp.~1--20. \MR{MR1874268 (2002i:14056)}

\bibitem[Cox03a]{Co03}
David Cox, \emph{Curves, surfaces, and syzygies}, Topics in algebraic geometry
  and geometric modeling, Contemp. Math, vol. 334, Amer. Math. Soc, Providence,
  RI, 2003, pp.~131--150. \MR{MR2039970 (2005g:14113)}

\bibitem[Cox03b]{Co03b}
\bysame, \emph{What is a toric variety?}, Topics in algebraic geometry and
  geometric modeling, Contemp. Math, vol. 334, Amer. Math. Soc, Providence, RI,
  2003, pp.~203--223. \MR{MR2039974 (2005b:14089)}

\bibitem[CSC98]{CSC98}
David~A Cox, Thomas~W Sederberg, and Falai Chen, \emph{The moving line ideal
  basis of planar rational curves}, Comput. Aided Geom. Design \textbf{15}
  (1998), no.~8, 803--827. \MR{MR1638732 (99h:14056)}

\bibitem[D'A01]{DA01}
Carlos D'Andrea, \emph{Resultants and moving surfaces}, J. Symbolic Comput.
  \textbf{31} (2001), no.~5, 585--602. \MR{MR1828705 (2002b:65027)}

\bibitem[DFS07]{DFS07}
Alicia Dickenstein, Eva~Maria Feichtner, and Bernd Sturmfels, \emph{Tropical
  discriminants}, J. Amer. Math. Soc \textbf{20} (2007), no.~4, 1111--1133
  (electronic). \MR{MR2328718 (2008j:14095)}

\bibitem[EG84]{EG}
David Eisenbud and Shiro Goto, \emph{Linear free resolutions and minimal
  multiplicity}, J. Algebra \textbf{88} (1984), no.~1, 89--133. \MR{MR741934
  (85f:13023)}

\bibitem[Ful93]{Fu93}
William Fulton, \emph{Introduction to toric varieties}, Annals of Mathematics
  Studies, vol. 131, Princeton University Press, Princeton, NJ, 1993, The
  William H. Roever Lectures in Geometry. \MR{MR1234037 (94g:14028)}

\bibitem[GK03]{GK03}
Ron Goldman and Rimvydas Krasauskas (eds.), \emph{Topics in algebraic geometry
  and geometric modeling}, Contemporary Mathematics, vol. 334, American
  Mathematical Society, Providence, RI, 2003, Papers from the Workshop on
  Algebraic Geometry and Geometric Modeling held at Vilnius University,
  Vilnius, July 29--August 2, 2002. \MR{MR2046835 (2004j:00026)}

\bibitem[GKZ94]{GKZ94}
Israel~M Gel{\cprime}fand, Mikhail~M Kapranov, and Aandrei~V Zelevinsky,
  \emph{Discriminants, resultants, and multidimensional determinants},
  Mathematics: Theory \& Applications, Birkh\"auser Boston Inc, Boston, MA,
  1994. \MR{MR1264417 (95e:14045)}

\bibitem[GS]{M2}
Daniel~R Grayson and Michael~E Stillman, \emph{Macaulay 2, a software system
  for research in algebraic geometry.}, http://www.math.uiuc.edu/Macaulay2/.

\bibitem[Har77]{Hart}
Robin Hartshorne, \emph{Algebraic geometry}, Springer-Verlag, New York, 1977,
  Graduate Texts in Mathematics, No. 52. \MR{MR0463157 (57 \#3116)}

\bibitem[Hof89]{Hof89}
Christoph Hoffmann, \emph{Geometric solid modeling: an introduction}, Morgan
  Kaufmann publishers (1989).

\bibitem[HSV82]{HSV1}
J{\"u}rgen Herzog, Aron Simis, and Wolmer~V Vasconcelos, \emph{Approximation
  complexes of blowing-up rings}, J. Algebra \textbf{74} (1982), no.~2,
  466--493. \MR{MR647249 (83h:13023)}

\bibitem[HSV83a]{HSV}
J{\"u}rgen Herzog, Aaron Simis, and Wolmer~V Vasconcelos, \emph{Koszul homology
  and blowing-up rings}, Commutative algebra (Trento, 1981), Lecture Notes in
  Pure and Appl. Math, vol.~84, Dekker, New York, 1983, pp.~79--169.
  \MR{MR686942 (84k:13015)}

\bibitem[HSV83b]{HSV2}
J{\"u}rgen Herzog, Aron Simis, and Wolmer~V Vasconcelos, \emph{Approximation
  complexes of blowing-up rings. {II}}, J. Algebra \textbf{82} (1983), no.~1,
  53--83. \MR{MR701036 (85b:13015)}

\bibitem[HW04]{HW04}
Jerome~W. Hoffman and Hao~Hao Wang, \emph{Castelnuovo-{M}umford regularity in
  biprojective spaces}, Adv. Geom. \textbf{4} (2004), no.~4, 513--536.
  \MR{MR2096526 (2006b:13032)}

\bibitem[Jou95]{Jou2}
Jean-Pierre Jouanolou, \emph{Aspects invariants de l'\'elimination}, Adv. Math
  \textbf{114} (1995), no.~1, 1--174. \MR{MR1344713 (96m:14001)}

\bibitem[KD06]{KD06}
Amit Khetan and Carlos D'Andrea, \emph{Implicitization of rational surfaces
  using toric varieties}, J. Algebra \textbf{303} (2006), no.~2, 543--565.
  \MR{MR2255122 (2007k:14126)}

\bibitem[KM76]{KMun}
Finn~Faye Knudsen and David Mumford, \emph{The projectivity of the moduli space
  of stable curves. {I}. {P}reliminaries on ``det'' and ``{D}iv''}, Math. Scand
  \textbf{39} (1976), no.~1, 19--55. \MR{MR0437541 (55 \#10465)}

\bibitem[Lat]{Latte}
Team Latte, \emph{Latte - a software dedicated to the problems of counting and
  detecting lattice points inside convex polytopes, and the solution of integer
  programs}, http://www.math.ucdavis.edu/~latte/.

\bibitem[Mac65]{MRae65}
Robert~E MacRae, \emph{On an application of the {F}itting invariants}, J.
  Algebra \textbf{2} (1965), 153--169. \MR{MR0178038 (31 \#2296)}

\bibitem[Mat89]{Mats}
Hideyuki Matsumura, \emph{Commutative ring theory}, second ed., Cambridge
  Studies in Advanced Mathematics, vol.~8, Cambridge University Press,
  Cambridge, 1989, Translated from the Japanese by M. Reid. \MR{MR1011461
  (90i:13001)}

\bibitem[MS04]{MlS04}
Diane Maclagan and Gregory~G Smith, \emph{Multigraded {C}astelnuovo-{M}umford
  regularity}, J. Reine Angew. Math. \textbf{571} (2004), 179--212.
  \MR{MR2070149 (2005g:13027)}

\bibitem[MS05]{MS}
Ezra Miller and Bernd Sturmfels, \emph{Combinatorial commutative algebra},
  Graduate Texts in Mathematics, vol. 227, Springer-Verlag, New York, 2005.
  \MR{MR2110098 (2006d:13001)}

\bibitem[Mum66]{MumRed}
David Mumford, \emph{Lectures on curves on an algebraic surface}, With a
  section by G. M. Bergman. Annals of Mathematics Studies, No. 59, Princeton
  University Press, Princeton, N.J., 1966. \MR{MR0209285 (35 \#187)}

\bibitem[Mus02]{Mus02}
Mircea Musta{\c{t}}a, \emph{Vanishing theorems on toric varieties}, Tohoku
  Math. J. (2) \textbf{54} (2002), no.~3, 451--470. \MR{MR1916637
  (2003e:14013)}

\bibitem[Nor76]{Nor76}
Douglas~G. Northcott, \emph{Finite free resolutions}, Cambridge University
  Press, Cambridge, 1976, Cambridge Tracts in Mathematics, No. 71.
  \MR{MR0460383 (57 \#377)}

\bibitem[PD06]{PeDi06}
Sonia P{\'e}rez-D{\'{\i}}az, \emph{On the problem of proper reparametrization
  for rational curves and surfaces}, Comput. Aided Geom. Design \textbf{23}
  (2006), no.~4, 307--323. \MR{MR2218599 (2007a:14066)}

\bibitem[SC95]{SC95}
Tom Sederberg and Falai Chen, \emph{Implicitization using moving curves and
  surfaces}, Computer Graphics Annual Conference Series \textbf{303} (1995),
  301--308.

\bibitem[Sch03]{Sc03}
Josef Schicho, \emph{Simplification of surface parametrizations---a lattice
  polygon approach}, J. Symbolic Comput. \textbf{36} (2003), no.~3-4, 535--554,
  International Symposium on Symbolic and Algebraic Computation (ISSAC'2002)
  (Lille). \MR{MR2004041 (2004i:65015)}

\bibitem[Sul08]{Sul06}
Seth Sullivant, \emph{Combinatorial symbolic powers}, J. Algebra \textbf{319}
  (2008), no.~1, 115--142. \MR{MR2378064 (2009c:13005)}

\bibitem[SV81]{SV81}
Aron Simis and Wolmer~V Vasconcelos, \emph{The syzygies of the conormal
  module}, Amer. J. Math. \textbf{103} (1981), no.~2, 203--224. \MR{MR610474
  (82i:13016)}

\bibitem[Vas94a]{Vas1}
Wolmer~V Vasconcelos, \emph{Arithmetic of blowup algebras}, London Mathematical
  Society Lecture Note Series, vol. 195, Cambridge University Press, Cambridge,
  1994. \MR{MR1275840 (95g:13005)}

\bibitem[Vas94b]{Va94}
\bysame, \emph{Arithmetic of blowup algebras}, London Mathematical Society
  Lecture Note Series, vol. 195, Cambridge University Press, Cambridge, 1994.
  \MR{MR1275840 (95g:13005)}

\bibitem[ZSCC03]{ZSCC03}
Jianmin Zheng, Thomas~W Sederberg, Eng-Wee Chionh, and David~A Cox,
  \emph{Implicitizing rational surfaces with base points using the method of
  moving surfaces}, Topics in algebraic geometry and geometric modeling,
  Contemp. Math., vol. 334, Amer. Math. Soc., Providence, RI, 2003,
  pp.~151--168. \MR{MR2039971 (2005c:14082)}

\end{thebibliography}
\end{document}